\documentclass[a4paper]{article}
\title{A projective resolution for the Fomin-Kirillov algebra $\FK(4)$}
\author{Estanislao Herscovich and Ziling Li}
\date{}

\usepackage[utf8]{inputenc}

\usepackage[top=1in,bottom=1in,left=1.25in,right=1.25in]{geometry}
\usepackage[font=small,skip=0pt]{caption}

\usepackage{titlesec}
\usepackage{multicolrule}
\usepackage{blindtext}

\usepackage{stmaryrd}
\usepackage{multirow}
\usepackage{float}
\usepackage{diagbox}
\usepackage{tabularx}

\usepackage{enumitem}

\usepackage{verbatimbox}
\usepackage{graphicx}

\usepackage[all]{xy}

\usepackage{tikz}
\usetikzlibrary{
	arrows,
	arrows.meta,
	cd,
	decorations, 
	decorations.markings,
	external,
	fit, 
	intersections,
	matrix,
	decorations.pathreplacing, 
	calc, 
	positioning,
	shapes.geometric,
	trees
}

\usepackage{xcolor,palatino}
\definecolor{ultramarine}{RGB}{0,32,96}
\colorlet{mygreen}{green!20!gray}
\colorlet{myultramarine}{ultramarine!50!gray}

\usepackage{amsmath}
\usepackage{amsthm}
\usepackage{amssymb}
\usepackage{mathrsfs} 
\usepackage{mathtools}
\mathtoolsset{showonlyrefs}

\usepackage[colorlinks=true,citecolor=mygreen, linkcolor=mygreen, urlcolor=myultramarine,breaklinks, naturalnames]{hyperref}

\usepackage[backrefs]{amsrefs}

\usepackage{fancyvrb}
\usepackage{tcolorbox}
\tcbuselibrary{breakable}

\usepackage{titlesec}

\titlespacing\section{0pt}{9.5pt plus 4pt minus 2pt}{6pt plus 2pt minus 2pt}
\titlespacing\subsection{0pt}{7.5pt plus 4pt minus 2pt}{4pt plus 2pt minus 2pt}
\titlespacing\subsubsection{0pt}{6pt plus 4pt minus 2pt}{2pt plus 2pt minus 2pt}

\numberwithin{equation}{section}
\numberwithin{table}{section}

\allowdisplaybreaks[4]

\DeclareFontFamily{U}{BOONDOX-calo}{\skewchar\font=45 }
\DeclareFontShape{U}{BOONDOX-calo}{m}{n}{
   <-> s*[1.05] BOONDOX-r-calo}{}
\DeclareFontShape{U}{BOONDOX-calo}{b}{n}{
   <-> s*[1.05] BOONDOX-b-calo}{}
\DeclareMathAlphabet{\mathcalboondox}{U}{BOONDOX-calo}{m}{n}
\SetMathAlphabet{\mathcalboondox}{bold}{U}{BOONDOX-calo}{b}{n}
\DeclareMathAlphabet{\mathbcalboondox}{U}{BOONDOX-calo}{b}{n}

\newtheorem{thm}{Theorem}[section]

\newtheorem{rk}[thm]{Remark}
\newtheorem{lem}[thm]{Lemma}
\newtheorem{eg}[thm]{Example}
\newtheorem{prop}[thm]{Proposition}
\newtheorem{cor}[thm]{Corollary}
\newtheorem{fact}[thm]{Fact}
\newtheorem*{prop*}{Proposition}

\newcommand\FK{{\operatorname{FK}}}
\newcommand\id{{\operatorname{id}}}
\newcommand\Img{{\operatorname{Im}}}
\newcommand\Ker{{\operatorname{Ker}}}

\newcommand\K{{\operatorname{K}}}

\newcommand\T{{\mathbb{T}}}
\newcommand\NN{{\mathbb{N}}}
\newcommand\RR{{\mathbb{R}}}
\newcommand\ZZ{{\mathbb{Z}}}

\newcommand\I{{\mathcalboondox{I}}}
\newcommand\II{{\mathcalboondox{I}_1}}
\newcommand\J{{\mathcalboondox{J}}}
\newcommand\B{{\mathcalboondox{B}}}

\newcommand\U{{\mathcalboondox{U}}}
\newcommand\C{{\mathcalboondox{C}}}

\newcommand\Q{{\mathcalboondox{Q}}}

\newcommand\RQ{{\mathcalboondox{RQ}}}
\newcommand\Pa{{\mathcalboondox{Pa}}}
\newcommand\MM{{\mathcalboondox{M}}}
\newcommand\Ar{{\mathcalboondox{Ar}}}

\newcommand\hh{{\mathcalboondox{h}}}
\newcommand\bideg{{\operatorname{bideg}}}
\newcommand\ddeg{{\operatorname{dfdeg}}}

\makeatletter
\newcommand{\oset}[3][0ex]{%
  \mathrel{\mathop{#3}\limits^{
    \vbox to#1{\kern-2\ex@
    \hbox{$\scriptstyle#2$}\vss}}}}
\makeatother 

\newcommand\coplus{{\oset[-.3ex]{\text{\begin{tiny}$\hskip -0.25mm \curvearrowleft$\end{tiny}}}{\oplus}}}

\begin{document}

\maketitle

\hrulefill
%%%%%%%%%%%%%%%%%%%%%%%%%%%%%%%%%%%%%%%%%%%%%%%%%%%%%%%%%%%%%%%%%%%%%%%%%%%%%%%%%%%%%%%%%%%%%%%%%%%%%%%%%%%%%%%%%%%%%%%%%%%%%%%%%%%%%%%%%%%%%%%%
\begin{abstract} 
In this article we show that, given a quadratic algebra satisfying some assumptions, which we call having a \textbf{\textcolor{myultramarine}{resolving datum}}, one can construct a projective resolution of the trivial module which is obtained as iterated cones of Koszul complexes, and  
this projective resolution is minimal under some further assumptions. 
We observe that many examples of quadratic algebras studied so far have a resolving datum, and that the (minimal) projective resolutions constructed for all of them in the literature are an example of our construction. 
The second main result of the article is that the Fomin-Kirillov algebra $\FK(4)$ of index $4$ has a resolving datum. 
\end{abstract}

\textbf{Mathematics subject classification 2020:} 16S37, 16T05, 18G15

\textbf{Keywords:} quadratic algebras, homology, Fomin-Kirillov algebras

\hrulefill

%\tableofcontents

%%%%%%%%%%%%%%%%%%%%%%%%%%%%%%%%%%%%%%%%%%%
\section{Introduction}
\label{sec:introduction}

In their study of the Schubert calculus of flag manifolds, S. Fomin and A. Kirillov introduced a family of quadratic algebras over a field $\Bbbk$ of characteristic zero, now called the {\color{ultramarine}{\textbf{Fomin-Kirillov algebras}}} $\operatorname{FK}(n)$, indexed by the positive integers $n \in \NN$ (see \cites{MR1667680, MR1772287, MR3439199}). 
In case the index $n$ takes the values $3$, $4$ or $5$, the Fomin-Kirillov algebras are also Nichols algebras (see \cites{MR1800714,GV16}), which appear in the classification of finite dimensional pointed Hopf algebras (see \cite{MR2630042}). 
Their (co)homological properties have gained some importance, in particular in relation to the conjecture by P. Etingof and V. Ostrik that claims that the Yoneda algebra of every finite dimensional 
Hopf algebra is finitely generated. 

The explicit homological computations of a graded connected $\Bbbk$-algebra $A$ typically involve the construction of a ``small'' projective resolution of the trivial module $\Bbbk$. 
The methods for constructing such projective resolution $P_{\bullet}$ are typically of ``local nature'', \textit{i.e.} for $n \in \NN$ the projective module $P_{n}$ is obtained from the previous steps. 
This is the case for instance of the usual construction of the minimal projective resolution, but also up to some extent of the Anick resolution and variations thereof. 
The obvious drawback of these procedures is that it is in general rather hard to describe the generic module $P_{n}$ of $P_{\bullet}$ or its differential $d_{n} : P_{n} \rightarrow P_{n-1}$. 
We also note that for a quadratic algebra $A$, if there is a choice of monomial order for the generators of $A$ such that the corresponding Anick resolution is minimal then $A$ is Koszul (see \cite{MR2177131}, Chapter 4, Thm. 3.1), which 
gives another motivation for a different method of constructing projective resolutions. 

In this article, given a quadratic algebra $A$ over $\Bbbk$ satisfying some assumptions, we give a method to construct a projective resolution $P_{\bullet}$ of the trivial module $\Bbbk$ that is of ``global nature'' (at least partially). 
By this we mean that the description of every module $P_{n}$ of $P_{\bullet}$ as well as some components of the differential $P_{n} \rightarrow P_{n-1}$ are described in a straightforward manner by means of a quiver associated with $A$, which we call the \textbf{\textcolor{myultramarine}{resolving quiver}} (see Theorem \ref{thm:rd}). 
The method is however not completely satisfactory, since the description of the remaining pieces of the differentials is of local nature, and the projective resolution we get is not minimal in general, even though it is minimal in many interesting examples --for which the Anick resolution cannot be minimal, as observed above--. 
Despite these drawbacks, we remark that many examples of quadratic algebras studied in the literature so far have such a resolving quiver, and that the minimal projective resolutions constructed in the literature for all of them are a specific instance of our construction (see Propositions \ref{proposition:koszul} and \ref{proposition:almost-koszul}, Examples \ref{example:cassidy}-\ref{example:herscovich} and Remark \ref{remark:proj-ex}), which makes us wonder why the method we present here was not noticed previously. 
In any case, the main motivation for introducing this machinery is due to the second main result of the article, which states that the Fomin-Kirillov algebra $\FK(4)$ of index $4$ has a resolving quiver (see Theorem \ref{thm:rd fk4}), so we can construct an explicit projective resolution of its trivial module.

The structure of the article is as follows. 
In the first half of Section \ref{sec:quadratic}, Subsection \ref{subsection:intro-quad},  we recall the basic definitions and properties of quadratic algebras and quadratic modules, whereas in the second half, Subsection \ref{subsection:resolving}, we introduce the novel notion of resolving datum. 
After providing some examples of quadratic algebras having a resolving datum, we prove the first main result of this article, Theorem \ref{thm:rd}. 
In the first subsection of Section \ref{subsec:rd FK4} we recall the definition and basic homological properties of the Fomin-Kirillov algebra $\FK(4)$ of index $4$. 
We then present the second main result of the article, Theorem 
\ref{thm:rd fk4}, which states that $\FK(4)$ has a resolving datum, consisting of quadratic modules $\{ \Bbbk = M^{0}, M^{1}, M^{2}, M^{3} \}$. 
The proof of the theorem is based on several intermediate results that appear in Subsections \ref{subsec:cpx of H4} and \ref{subsec:cpx of H4bis}, 
namely Lemmas \ref{lemma:homology-m0}, \ref{lemma:homology-m1}, 
\ref{lemma:homology-m2} and \ref{lemma:homology-m3}. 
We also present in the appendix several auxiliary computations, most of which are obtained using GAP. 
Namely, in Appendix \ref{subsec:W} we give the GAP code for finding a basis of the Fomin-Kirillov algebra $\FK(4)$ whereas in Appendix \ref{sec:products FK4} we provide several products in the quadratic dual algebra $\FK(4)^{!}$, that are used in the sequel. 
In Appendix \ref{sec:Basis of M1} we provide the explicit GAP code to obtain a basis of the quadratic module $M^{1}$, whereas in Appendix \ref{sec:cpx} we provide the GAP code to compute the homology of the Koszul complex of all the quadratic modules $M^{i}$ over $\FK(4)$ for $i \in \llbracket 0,3 \rrbracket$. 
Finally in Appendix \ref{Appendix:basis m2 d} we provide the GAP code to obtain a basis of $(M^{2})^{!}$, and in Appendix \ref{sec:products M2} some products describing the action of $\FK(4)^{!}$ on $(M^{2})^{!}$. 
We believe these snippets of code could be useful for the reader interested in dealing with other quadratic algebras. 

We will denote by $\NN$ (resp., $\NN_{0}$, $\ZZ$) the set of positive (resp., nonnegative, all) integers. 
Given $i\in\ZZ$, we will denote by $\ZZ_{\leqslant i}$ the set $\{m\in\ZZ|m\leqslant i \}$.
Given $i,j\in \mathbb{Z}$ with $i\leqslant j$, we will denote by $\llbracket i,j \rrbracket =\{m\in \mathbb{Z}|i\leqslant m\leqslant j \}$ the integer interval, and we define $\chi_n=0$ if $n$ is an odd integer and $\chi_n=1$ if $n$ is an even integer. 
Moreover, given $r \in \RR$, we set $\lfloor r\rfloor =\sup \{n\in \ZZ|n\leqslant r \}$. 
Moreover, in the following we will intensively use routines in GAP for many computations, using the GBNP package, so we refer the reader to \cite{cohenknopper}. 

The first author would like to thank Nicolás Andruskiewitsch, for the motivation to study the Fomin-Kirillov algebras and the many references and information on the topic, and Chelsea Walton for several discussions 
concerning the quadratic dual of the Fomin-Kirillov algebra of index $4$ as well as sharing some computations in GAP. 
We are indebted to Jan W. Knopper for assistance with the code involving the computation of Gröbner bases of quadratic modules. 

%%%%%%%%%%%%%%%%%%%%%%%%%%%%%%%%%%%%%%%%%%%
\section{Quadratic algebras and modules}
\label{sec:quadratic}

In the first subsection we recall the basic definitions of quadratic algebras and modules, as well as their Koszul complexes. 
In the second subsection we introduce a new property of quadratic algebras that allows to construct projective resolutions.

%%%%%%%%%%%%%%%%%%%%%%%%%%%%%%%%%%%%%%%%%%%
\subsection{Basic definitions}
\label{subsection:intro-quad}

All of the definitions and results in this subsection are classical and can be found in \cite{MR2177131}.

From now on, let $\Bbbk$ be a field. 
All vector spaces will be over $\Bbbk$, and all maps between vector spaces will be linear unless otherwise stated. 
Moreover, we will denote the usual tensor product of vector spaces simply by $\otimes$. 

Recall that a unitary associative $\Bbbk$-algebra $A$ is said to be \textcolor{myultramarine}{\textbf{nonnegatively graded}} if $A = \oplus_{n \in \NN_{0}} A_{n}$ 
is a direct sum decomposition of vector spaces such that $A_{n} \cdot A_{m} \subseteq A_{n+m}$, for all $n, m \in \NN_{0}$, and $1_{A} \in A_{0}$. 
The grading will be called \textcolor{myultramarine}{\textbf{internal}} or \textcolor{myultramarine}{\textbf{Adams}} to emphasize that it does not intervene in the Koszul sign rule. 
$A$ is said to be \textcolor{myultramarine}{\textbf{connected}} if we have furthermore $A_{0} = \Bbbk$, which we assume from now on. 
Let $A_{>0} = \oplus_{n \in \NN} A_{n}$ and let $V$ be a graded vector subspace of $A_{>0}$ such that the restriction of the canonical projection $A_{>0} \rightarrow A_{>0}/(A_{>0} \cdot A_{>0})$ to $V$ is a bijection. 
We will assume for the rest of this section that $V$ is finite-dimensional. 
We say in this case that $A$ is a \textcolor{myultramarine}{\textbf{finitely generated}} algebra. 
Then, the canonical map $\T (V) \rightarrow A$ induced by the inclusion $V \subseteq A$ is surjective, where $\T (V) = \oplus_{n \in \NN_{0}} V^{\otimes n}$ denotes the tensor algebra. 
We will usually write the product of $\T (V)$ by juxtaposition.  
A nonnegatively graded connected algebra $A$ is said to be \textcolor{myultramarine}{\textbf{quadratic}} if $V=A_1$ and there is a subspace $R \subseteq V^{\otimes 2}$ such that the kernel of $\T (V) \rightarrow A$ is the two-sided ideal generated by $R$. 
By abuse of terminology, we will identify the quadratic algebra $A$ with its presentation $(V,R)$, where $A = \T(V)/(R)$. 

Let $V^{*}$ be the dual vector space of $V$ and for every integer $n \geqslant 2$ define the pairing $\gamma_{n} : (V^{*})^{\otimes n} \otimes V^{\otimes n} \rightarrow \Bbbk$ 
by $\gamma_{n}(f_{1} \otimes \dots \otimes f_{n} , v_{1} \otimes \dots \otimes v_{n}) = f_{1}(v_{1}) \dots f_{n}(v_{n})$, for all $v_{1}, \dots, v_{n} \in V$ and $f_{1}, \dots, f_{n} \in V^{*}$. 
Set $R^{\perp} \subseteq V^{*} \otimes V^{*}$ to be the vector subspace orthogonal to $R$ for $\gamma_{2}$, \textit{i.e.} 
\[     R^{\perp} = \Big\{ \alpha \in (V^{*})^{\otimes 2} \mid \gamma_{2}(\alpha,r) = 0, \text{for all } r \in R \Big\}.     \]
The \textcolor{myultramarine}{\textbf{quadratic dual}} $A^{!}$ of a quadratic algebra $A = \T(V)/(R)$ is the algebra given by $\T(V^{*})/( R^{\perp} )$.
The induced internal grading is denoted by $A^{!} = \oplus_{n \in \NN_{0}} A_{-n}^{!}$. 
Note that $A_{0}^{!} = \Bbbk$ and $A_{-1}^{!} = V^{*}$. 
Moreover, for any integer $n \geqslant 2$, the composition of the isomorphism $(V^{*})^{\otimes n} \overset{\sim}{\rightarrow} (V^{\otimes n})^{*}$ induced by the pairing $\gamma_{n}$ and the dual of the canonical inclusion $\cap_{i=0}^{n-2} V^{\otimes i} \otimes R \otimes V^{\otimes (n-i-2)} \rightarrow V^{\otimes n}$ induces a canonical isomorphism of vector spaces
\begin{equation}
\label{eq:an}
   A_{-n}^{!} \overset{\sim}{\longrightarrow} \bigg(\bigcap_{i=0}^{n-2} V^{\otimes i} \otimes R \otimes V^{\otimes (n-i-2)} \bigg)^{*}. 
\end{equation}          

Recall that the graded dual $(A^{!})^{\#} = \oplus_{n \in \NN_{0}} (A_{-n}^{!})^{*}$ is a graded bimodule over $A^{!}$ via 
$(a \cdot f \cdot b)(c) = f(b c a)$, for all $a, b, c \in A^{!}$ and $f \in (A^{!})^{\#}$. 
Note in particular that $v \cdot f \in (A_{-n}^{!})^{*}$, for all $f \in (A_{-n-1}^{!})^{*}$, $v \in V^{*}$ and $n \in \NN_{0}$.
Since $V^{*} \otimes V \simeq \operatorname{End}_{\Bbbk}(V)$, there is a unique element $\iota \in V^{*} \otimes V$ whose image under the previous isomorphism is the identity of $V$. 
It is easy to prove that, if $\{ v_{i} \}_{i \in I}$ is a basis of $V$ and $\{ f_{i} \}_{i \in I}$ is the dual basis of $V^{*}$, 
then $\iota= \sum_{i \in I} f_{i} \otimes v_{i}$. 
For $n \in \NN_{0}$ set $K_{n}(A) = (A_{-n}^{!})^{*} \otimes A$, provided with the regular (right) $A$-module structure, 
and $d_{n+1} : K_{n+1}(A) \rightarrow K_{n}(A)$ as the multiplication by $\iota$ on the left. 
Furthermore, let $\epsilon : K_{0}(A) \rightarrow \Bbbk$ be the canonical projection from $A$ onto $A_{0} = \Bbbk$.
To reduce space, we will typically denote the composition $f\circ g$ of maps $f$ and $g$ simply by their juxtaposition $f g$.  
It is easy to see that $d_{n} d_{n+1} = 0$, for all $n \in \NN$, and $\epsilon d_{1} = 0$. 
The complex $(K_{\bullet}(A),d_{\bullet})$ 
is called the \textcolor{myultramarine}{\textbf{(right) Koszul complex of $A$}}. 
As usual, we can consider the Koszul complex as a complex indexed by $\ZZ$, with $K_{n}(A) = 0$ for all $n \in \ZZ_{\leqslant -1}$, and $d_{n} = 0$ for all $n \in \ZZ_{\leqslant 0}$. 
Equivalently, if we use the composition of the canonical isomorphism $V^{\otimes n} \overset{\sim}{\rightarrow} (V^{\otimes n})^{**}$ and the dual of \eqref{eq:an} for $A_{-n}^{!}$, then 
$d_{n+1} : K_{n+1}(A) \rightarrow K_{n}(A)$ identifies with the restriction of the map $\tilde{d}_{n+1} : V^{\otimes (n+1)} \otimes A \rightarrow V^{\otimes n} \otimes A$ determined by 
\begin{equation}
\label{eq:anbis}
    (v_{1} \otimes \dots \otimes v_{n+1}) \otimes a \mapsto (v_{1} \otimes \dots \otimes v_{n}) \otimes v_{n+1} a,     
\end{equation}          
for all $v_{1}, \dots, v_{n+1} \in V$, $a \in A$ and $n \in \NN_{0}$. 

The following result is immediate. 
%%%%%%%
\begin{fact}
\label{fact:k-alg}
Let $A$ be a quadratic algebra. Then, $\Ker(\epsilon) = \Img(d_{1})$ and 
$\Ker(d_{1}) = \Img(d_{2})$, and in fact  $(K_{\bullet}(A),d_{\bullet})$ coincides with the minimal projective resolution of the trivial right $A$-module $\Bbbk$ in the category 
of bounded below graded right $A$-modules, up to homological degree $2$. 
\end{fact}
%%%%%%%

Recall that a quadratic algebra $A$ is said to be \textcolor{myultramarine}{\textbf{Koszul}} if its Koszul complex is exact in positive homological degrees. 

Let $M=\oplus_{n \in \ZZ} M_{n}$ 
be a graded right module over a quadratic algebra $A$ such that $\dim(M_{n})$ is finite for all $n \in \ZZ$. 
Given $j\in \mathbb{Z}$, we denote by $M(j)$ the same underlying module with shifted (internal) grading given by $M(j)_i=M_{j+i}$ for $i\in\ZZ$. 
We remark that a \textcolor{ultramarine}{\textbf{morphism of graded right $A$-modules}} $f : M \rightarrow N$ is a homogeneous $A$-linear map of degree zero. 
Moreover, for a nonzero graded module $M$ over $A$, if there exist integers $s \leqslant t$ such that $\dim(M_{n}) = 0$ for all $n \in \ZZ \setminus \llbracket s,t \rrbracket$ 
and $\dim(M_{s}) \cdot \dim(M_{t}) \neq 0$, 
then we say that the \textcolor{ultramarine}{\textbf{dimension vector}} of $M$ is $(\dim(M_{s}),\dots,\dim(M_{t}))$.

Let $M = \oplus_{n \in \NN_0} M_{n}$ be a graded right module over the quadratic algebra $A = \T(V)/(R)$.
We say that $M$ is \textcolor{myultramarine}{\textbf{quadratic}} if the canonical map 
$\rho|_{M_{0} \otimes A} : M_{0} \otimes A \rightarrow M$ given by the restriction of the action $\rho : M \otimes A \rightarrow A$ is surjective and the kernel of $\rho|_{M_{0} \otimes A}$ is the $A$-submodule generated by a vector subspace $R_{M} \subseteq  M_{0} \otimes V$, called the \textcolor{myultramarine}{\textbf{space of relations}} of M. 
We will typically denote $M_{0}$ by $V_{M}$, and call it the \textcolor{myultramarine}{\textbf{space of generators}} of $M$. 
We will assume for the rest of this section that $V_M$ is finite-dimensional. 
Note that, by definition, a quadratic module is generated by a space of generators concentrated in degree zero. 
By abuse of notation we will usually denote a quadratic module by its presentation $(V_{M},R_{M})_{A}$, which we will also be denoted simply by $(V_{M},R_{M})$ if $A$ is clear from the context. 
Note that, if $A$ is a quadratic algebra, then the trivial module $\Bbbk$ is quadratic with $V_{\Bbbk} = \Bbbk$ and $R_{\Bbbk} = V_{\Bbbk} \otimes V$. 

Given a quadratic right module $M$ with presentation $(V_{M},R_{M})_{A}$, we define its \textcolor{myultramarine}{\textbf{quadratic dual}} $M^{!_{m}}$ as the quotient of the graded right $A^{!}$-module $V_{M}^{*} \otimes A^{!}$ by the 
graded submodule generated by $R_{M}^{\perp} \subseteq V_{M}^{*} \otimes V^{*}$, where 
\[     R_{M}^{\perp} = \big\{ \beta \in V_{M}^{*} \otimes V^{*} \mid \gamma_{2,M}(\beta,r) = 0 \text{ for all } r \in R_{M} \big\},      \]
where 
$\gamma_{2, M} : V_{M}^{*} \otimes V^{*} \otimes V_{M} \otimes V \rightarrow \Bbbk$ is the map 
given by $\gamma_{2, M}(f \otimes g \otimes u \otimes v) = f(u) g(v)$ for $f \in V_{M}^{*}$, $g \in V^{*}$, $u \in V_{M}$ and $v \in V$.
Note that $M^{!_{m}} = \oplus_{n \in \NN_{0}} M^{!_{m}}_{-n}$ is a graded right module over $A^{!}$, where $M^{!_{m}}_{-n}$ sits in degree $-n$. 
As a consequence, the graded dual $(M^{!_{m}})^{\#} = \oplus_{n \in \NN_{0}} (M^{!_{m}}_{-n})^{*}$ is a graded left module over $A^{!}$. 
On the other hand, it is easy to see that $\Bbbk^{!_{m}} = A^{!}$. 

Let $M$ be a quadratic right module $M$ over the quadratic algebra $A$. 
For $n \in \NN_{0}$, set $\K_{n}^{\operatorname{mod}}(M) = (M_{-n}^{!_{m}})^{*} \otimes A$, provided with the regular right $A$-module structure, 
and $d_{n+1} : \K_{n+1}^{\operatorname{mod}}(M) \rightarrow \K_{n}^{\operatorname{mod}}(M)$ as the multiplication by $\iota$ on the left. 
Furthermore, let $\epsilon : \K_{0}^{\operatorname{mod}}(M) \rightarrow M$ be the canonical morphism $\rho_{V_{M} \otimes A} : V_{M} \otimes A \rightarrow M$ onto $M$.
It is easy to see that $d_{n} d_{n+1} = 0$, for all $n \in \NN$, and $\rho_{V_{M} \otimes A}  d_{1} = 0$. 
The complex $(\K_{\bullet}^{\operatorname{mod}}(M),d_{\bullet})$ 
is called the \textcolor{myultramarine}{\textbf{(right) Koszul complex of $M$}}. 
As usual, we can consider the Koszul complex as a complex indexed by $\ZZ$, with $\K_{n}^{\operatorname{mod}}(M) = 0$ for all $n \in \ZZ_{\leqslant -1}$, and $d_{n} = 0$ for all $n \in \ZZ_{\leqslant 0}$. 

The following immediate result is the analogous of Fact \ref{fact:k-alg} for right modules. 
%%%%%%%
\begin{fact}
\label{fact:k-mod}
Let $M$ be a quadratic right module over a quadratic algebra $A$. Then, $\Ker(\rho_{V_{M} \otimes A}) = \Img(d_{1})$, and in fact  $(\K_{\bullet}^{\operatorname{mod}}(M),d_{\bullet})$ coincides with the minimal projective resolution of $M$ in the category of bounded below graded right $A$-modules, up to homological degree $1$. 
\end{fact}
%%%%%%%

Let $M$ and $N$ be two quadratic right modules over the quadratic algebra $A = \T(V)/(R)$, with presentations $(V_{M},R_{M})$ and $(V_{N},R_{N})$, respectively. 
Let us denote by $\operatorname{hom}_{A}(M,N)$ the 
vector space formed by all homogeneous morphisms $f : M \rightarrow N$ of right $A$-modules of degree zero, and by 
$\operatorname{Hom}((V_{M},R_{M}),(V_{N},R_{N}))$ 
the vector space formed by all linear morphisms $g : V_{M} \rightarrow V_{N}$ satisfying that $(g \otimes \operatorname{id}_{V})(R_{M}) \subseteq R_{N}$. 
Then, it is clear that the map 
\begin{equation}
\label{eq:rest-qd}
\operatorname{hom}_{A}(M,N) \rightarrow \operatorname{Hom}\big((V_{M},R_{M}),(V_{N},R_{N})\big) 
\end{equation}
sending $f$ to its restriction $f|_{V_{M}} : V_{M} \rightarrow V_{N}$
is an isomorphism. 
This tells us that $f : M \rightarrow N$ is a monomorphism (resp., epimorphism) in the category of quadratic right $A$-modules with homogeneous morphisms of $A$-modules of degree zero if and only if $f|_{M_{0}}:M_0\to N_0$ is injective (resp., surjective). 
In particular, a morphism of the category of quadratic right $A$-modules with homogeneous morphisms of $A$-modules of degree zero is an epimorphism if and only if it is a surjection. 

%%%%%%%
\begin{rk}
Assume the space of generators of the quadratic algebra $A$ has nonzero dimension.
Then, the category of quadratic right $A$-modules with homogeneous morphisms of $A$-modules of degree zero is not abelian, since the canonical projection $A \rightarrow \Bbbk$ is a monomorphism and an epimorphism but it is not an isomorphism. 
In particular, the example shows that monomorphisms of the category of quadratic right modules are not necessarily injective. 
For a less trivial example, consider $\Bbbk$ of characteristic different from $2$, $A = \Bbbk\langle x, y \rangle/(xy-yx) = \Bbbk[x,y]$, $M = e.A$, $M' = (e_{1}.A \oplus e_{2}.A)/(e_{1}.x + e_{2}.x, e_{1}.y - e_{2}.y)$ and the morphism $f : M \rightarrow M'$ of right $A$-modules sending $e$ to $e_{1}$. 
Then $f$ is a non-injective monomorphism of quadratic right modules, since 
\[     f(e.xy + e.yx) = (e_{1}.x + e_{2}.x).y + (e_{1}.y - e_{2}.y).x     \] 
vanishes, but $f|_{M_{0}}$ and 
$f|_{M_{1}}$ are injective.
\end{rk}
%%%%%%%

Given $f \in \operatorname{hom}_{A}(M,N)$, define the homogeneous morphism $f^{!_{m}} : N^{!_{m}} \rightarrow M^{!_{m}}$ of right $A^{!}$-modules of degree zero whose restriction to $V_{N}^{*}$ is precisely the dual $(f|_{V_{M}})^{*}$ of $f|_{V_{M}} : V_{M} \rightarrow V_{N}$. 
Since $((f|_{V_{M}})^{*} \otimes \operatorname{id}_{V^{*}})(R_{N}^{\perp}) \subseteq R_{M}^{\perp}$, the map $f^{!_{m}}$ is well defined. 
By taking the graded dual $(f^{!_{m}})^{\#} : (M^{!_{m}})^{\#} \rightarrow (N^{!_{m}})^{\#}$ we obtain a homogeneous morphism of left $A^{!}$-modules
of degree zero. 
We finally define the morphism 
\begin{equation}
\label{eq:koszul-fun}
\K_{\bullet}^{\operatorname{mod}}(f) : \K_{\bullet}^{\operatorname{mod}}(M) \rightarrow \K_{\bullet}^{\operatorname{mod}}(N)
\end{equation}
of complexes of right $A$-modules by 
$\K_{\bullet}^{\operatorname{mod}}(f) = (f^{!_{m}})^{\#} \otimes \operatorname{id}_{A}$. 
It is clear that $\K_{\bullet}^{\operatorname{mod}}(f  g) = \K_{\bullet}^{\operatorname{mod}}(f) \circ \K_{\bullet}^{\operatorname{mod}}(g)$ and $\K_{\bullet}^{\operatorname{mod}}(\operatorname{id}_{M}) = \operatorname{id}_{\K_{\bullet}^{\operatorname{mod}}(M)}$, for $f \in \operatorname{hom}_{A}(M,N)$, $g \in \operatorname{hom}_{A}(N',M)$ and $N'$ 
a quadratic right $A$-module. 

%%%%%%%
\begin{rk}
If $f$ is injective, then $f|_{V_{M}}$ is also injective, which implies that its dual 
$(f|_{V_{M}})^{*}$ is surjective, so $f^{!_{m}}$ 
is surjective as well, which in turn implies that 
$(f^{!_{m}})^{\#}$ and $\K_{\bullet}^{\operatorname{mod}}(f)$ are injective. 
\end{rk}
%%%%%%%

From now on, by (resp., graded, quadratic) module over a (resp., graded, quadratic) algebra $A$ we will refer to (resp., graded, quadratic) right $A$-module, unless otherwise stated. 

%%%%%%%%%%%%%%%%%%%%%%%%%%%%%%%%%%%%%%%%%%%
\subsection{Resolving data on quadratic algebras}
\label{subsection:resolving}

We introduce the following definition. 
A \textcolor{myultramarine}{\textbf{resolving datum}} on a quadratic algebra $A$ is a finite set $\MM = \{ M^{0}, \dots, M^{N} \}$ of pairwise non-isomorphic quadratic (right) $A$-modules with $N \in \NN_{0}$ such that $M^{0} = \Bbbk$ is the trivial module and a map 
\[     \hh : \llbracket 0 , N \rrbracket^{2} \times \NN^{2} \rightarrow \NN_{0}^{2}      \]
such that 
\begin{enumerate}[label=(A.\arabic*)]
    \item\label{item:a} $\hh$ has finite support (\textit{i.e.} there exists a finite set $S \subseteq \llbracket 0 , N \rrbracket^{2} \times \NN^{2}$ such that $\hh(i,j,k,\ell) = (0,0)$ for all $(i,j,k,\ell) \in (\llbracket 0 , N \rrbracket^{2} \times \NN^{2}) \setminus S$), 
    \item\label{item:b} there are short exact sequences of right $A$-modules
    \begin{small}
    \begin{equation}
    \label{eq:ses}
   \hskip -1cm 
    \begin{tikzcd}[column sep=small]
    0 \ar[r] 
    &
    \bigoplus\limits_{j=0}^{N}\bigoplus\limits_{\ell \in \NN} \Big(M^{j}(-\ell)\Big)^{\pi_{1}(\hh(i,j,k,\ell))}
    \ar[r] 
    &
    \operatorname{H}_{k}(\K_{\bullet}^{\operatorname{mod}}(M^{i}))
    \ar[r] 
    &
    \bigoplus\limits_{j=0}^{N}\bigoplus\limits_{\ell \in \NN} \Big(M^{j}(-\ell)\Big)^{\pi_{2}(\hh(i,j,k,\ell))}
    \ar[r] 
    &
    0
    \end{tikzcd}
    \end{equation}
   \end{small}
    with homogeneous morphisms of degree zero for all $(i,k) \in \llbracket 0 ,  N \rrbracket \times \NN$, where $\pi_{i} : \NN_{0}^{2} \rightarrow \NN_{0}$ is the canonical projection on the $i$-th component for $i \in \{ 1 , 2 \}$, 
    \item If \eqref{eq:ses} splits for some $i_{0} \in \llbracket 0 , N \rrbracket$ and $k_{0} \in \NN$, then $\pi_{1}(\hh(i_{0},j,k_{0},\ell)) = 0$ for all $j \in \llbracket 0 , N \rrbracket$ and $\ell \in \NN$. 
\end{enumerate}

Recall that a \textcolor{myultramarine}{\textbf{quiver}} is the datum of 
a set $Q_{0}$, called \textcolor{myultramarine}{\textbf{set of vertices}}, and a set $Q_{1}$, called \textcolor{myultramarine}{\textbf{set of arrows}}, together with maps $s, t : Q_{1} \rightarrow Q_{0}$ called the \textcolor{myultramarine}{\textbf{source}} and \textcolor{myultramarine}{\textbf{target}} maps of the quiver. 
We say the quiver is \textcolor{myultramarine}{\textbf{bigraded}} if we further have a map $\bideg : Q_{1} \rightarrow \ZZ^{2}$. 
We will denote the bidegree of an arrow $\alpha$ of $Q$ by $\bideg(\alpha)=(\bideg_{1}(\alpha),\bideg_{2}(\alpha)) \in \ZZ^{2}$. 
The \textcolor{myultramarine}{\textbf{difference degree}} of an arrow $\alpha$ is defined as $\ddeg(\alpha)=\bideg_{2}(\alpha)-\bideg_{1}(\alpha) \in \ZZ$. 

We also recall that, given a quiver with set of vertices $Q_{0}$ and set of arrows $Q_{1}$, a \textcolor{myultramarine}{\textbf{path}} of length $n \in \NN_{0}$ is a vertex if $n=0$, and a tuple $\bar{\alpha} =(\alpha_{1}, \dots, \alpha_{n})$ in $Q_{1}^{n}$ for $n \in \NN$ such that $t(\alpha_{i}) = s(\alpha_{i+1})$ for all $i \in \llbracket 1 , n-1 \rrbracket$. 
As usual, we define $s(e) = t(e) = e$ for any vertex $e$, $s(\alpha_{1}, \dots, \alpha_{n}) = s(\alpha_{1})$ and $t(\alpha_{1}, \dots, \alpha_{n}) = t(\alpha_{n})$ for every path $\bar{\alpha} = (\alpha_{1}, \dots, \alpha_{n})$ of length $n \in \NN$. 
Furthermore, if the quiver is bigraded, given a path $\bar{\alpha} = (\alpha_{1}, \dots, \alpha_{n})$ of length $n \in \NN$, we define its bidegree $\bideg(\bar{\alpha}) = (\bideg_{1}(\bar{\alpha}),\bideg_{2}(\bar{\alpha})) \in \ZZ^{2}$ by $(\sum_{i=1}^{n} \bideg_1(\alpha_{i}) , \sum_{i=1}^{n} \bideg_2(\alpha_{i}))$. 
The bidegree of a path of length zero given by a vertex $e$ is defined as
$\bideg(e)=(\bideg_{1}(e),\bideg_{2}(e))=(0,0)$. 
The \textcolor{myultramarine}{\textbf{difference degree}} of a path $\bar{\alpha}$ is defined as $\ddeg(\bar{\alpha})=\bideg_{2}(\bar{\alpha})-\bideg_{1}(\bar{\alpha}) \in \ZZ$. 

Given a quadratic algebra together with a resolving datum as in the first paragraph of this subsection, we define the associated \textcolor{myultramarine}{\textbf{resolving quiver}} $\RQ_{A}$ as the unique bigraded quiver with set of vertices $\{ M^{0}, \dots, M^{N} \}$, and whose set of arrows of degree $(d',d'')$ from $M^{i}$ to $M^{j}$ has cardinality $\pi_{1}(\hh(i,j,d'-1,d'')) + \pi_{2}(\hh(i,j,d'-1,d''))$.
To be able to manipulate these arrows, assume we have chosen a fixed set $\Ar_{i,j,d',d''}'$ of arrows of degree $(d',d'')$ from $M^{i}$ to $M^{j}$ of cardinality $\pi_{1}(\hh(i,j,d'-1,d''))$ and another fixed set $\Ar_{i,j,d',d''}''$ of arrows of degree $(d',d'')$ from $M^{i}$ to $M^{j}$ of cardinality $\pi_{2}(\hh(i,j,d'-1,d''))$, such that $\Ar_{i,j,d',d''}'$ and $\Ar_{i,j,d',d''}''$
are disjoint. 
For every $i \in \llbracket 0,N\rrbracket$ and $d' \in \NN$, we also set a strict partial order on the set of all arrows $\alpha$ of $\RQ_{A}$ such that $s(\alpha)=M^{i}$ and $\bideg_{1}(\alpha) = d'$ by setting precisely that every arrow of $\Ar_{i,j,d',d''}''$ is strictly less than every arrow of 
$\Ar_{i,j',d',d'''}'$ 
for all $j,j' \in \llbracket 0,N\rrbracket$ and $d'',d''' \in \NN$. 
Note that this quiver is finite by \ref{item:a}. 
We will say that the resolving datum is 
\textcolor{myultramarine}{\textbf{connected}} if the associated resolving quiver is connected. 

As we will see in Theorem \ref{thm:rd}, the resolving quiver $\RQ_{A}$ contains some homological information of the algebra $A$. 
The first clues in this direction are given by the following results, the first of which is trivial. 

%%%%%%%
\begin{prop}
\label{proposition:koszul}
A quadratic algebra $A$ is Koszul if and only if the resolving quiver associated to a (equivalently, to every) connected resolving datum on $A$ has no arrows.
\end{prop}
%%%%%%% 

%%%%%%%
\begin{prop}
\label{proposition:almost-koszul}
Let $p, q \geqslant 2$ be integers. 
A quadratic algebra $A$ is $(p,q)$-Koszul (in the sense introduced by S. Brenner, M. Butler and A. King in \cite{MR1930968}) if and only if it is finite dimensional with $\dim(A_{p}) \neq 0$ and $\dim(A_{p+1}) =0$, the Koszul complex of $A$ has finite length $q$ and the resolving quiver associated to a (equivalently, to every) connected resolving datum on $A$ has only one vertex and $\dim(A_{p})\cdot \dim(A_{q}^{!})$ arrows of bidegree $(q+1,q+p)$.
\end{prop}
%%%%%%% 
\begin{proof}
This is precisely Prop. 3.9 of \cite{MR1930968}.
\end{proof}

We also have the following three (families of) examples of resolving quivers. 
They show that the notion of resolving quiver pervades many of the examples of quadratic algebras considered so far (see also Remark \ref{remark:proj-ex}). 

%%%%%%%
\begin{eg}
\label{example:cassidy}
Let $m \geqslant 5$ be an integer. 
Set $V(m)$ as the vector space of dimension $3m$ generated by the set $\cup_{i=1}^{m+1} S_{i}$, 
where $S_{1} = \{ n \}$, 
$S_{2} = \{ p,q,r \}$,
$S_{3} = \{ s, t ,u \}$, 
$S_{4} = \{ v, w, x_{1}, y_{1}, z_{1} \}$, 
$S_{i} = \{ x_{i-3}, y_{i-3}, z_{i-3} \}$ for $i \in \llbracket 5 , m-1 \rrbracket$, 
$S_{m} =\{ x_{m-3}, y_{m-3} \}$ and $S_{m+1} =\{ x_{m-2} \}$. 
Let $R(m)$ be the vector subspace of dimension $3m +4$ of $V(m)^{\otimes 2}$ generated by 
\begin{align*}     
\big\{ &np-nq , np - nr, ps-pt, qt-qu, rs-ru, sv-sw, tw-tx_{1}, uv - u x_{1}, v x_{2}, w x_{2}, sv -s y_{1} ,  
\\
&t w - t y_{1}, u x_{1} - u y_{1}, s z_{1}, t z_{1}, u z_{1} \big\} \cup 
\big\{ x_{i} x_{i+1} , y_{i-1} x_{i} + z_{i-1} y_{i} \mid i \in \llbracket 1 , m-3 \rrbracket \big\} 
\\
&\cup 
\big\{ z_{i} z_{i+1} \mid i \in \llbracket 1 , m-5 \rrbracket \big\},    
\end{align*}
where we denote the tensor product $\otimes$ by simple juxtaposition. 
Then, let $C_{m}$ be the quadratic algebra defined as $\T V(m)/(R(m))$. 
This algebra was defined in Section 2 of \cite{MR2586982}. 
Let $\MM = \{ \Bbbk , M^{1} \}$ where $M^{1}$ is the standard right module $C_{m}$, and let $\hh : \{ 0,1 \}^{2} \times \NN^{2} \rightarrow \NN_{0}^{2}$ be the map given by $\hh(0,1,m-1,m+1) = (0,1)$
and $\hh(i,j,k,\ell) = (0,0)$ if $(i,j,k,\ell) \neq (0,1,m-1,m+1)$. 
Then, \cite{MR2586982}, Thm. 2.7, tells us that this gives us a connected resolving datum on $C_{m}$, whose associated resolving quiver is 
\[
\tikzcdset{arrow style=tikz, diagrams={>=stealth}}
	\begin{tikzcd}[ampersand replacement=\&, column sep=large, row sep = large]
		\Bbbk 
		\arrow[r] 
		\&
		M^{1}
	\end{tikzcd}
 \]
such that its arrow has bidegree $(m,m+1)$. 
\end{eg}
%%%%%%% 

%%%%%%%
\begin{eg}
\label{example:conner-goetz} 
Given $g \in \Bbbk$, let $V$ be the vector space generated by the set $\{ x , y, z \}$ and let $R(g)$ be the vector subspace of $V^{\otimes 2}$ 
generated by the set 
\[     \{ xy - yx, z^{2}, xz-zx - y^{2} - g x^{2} \},     \]
where we denote the tensor product $\otimes$ by simple juxtaposition. 
Define $A(g)$ as the quadratic algebra given by $\T V/(R(g))$. 
This is precisely the algebra
$T(g,0)$ of Lemma 5.2 of \cite{MR4248207}, where we interchanged the role of $x$ and $y$, and we wrote $z$ instead of $w$. 
Let $\MM = \{ \Bbbk , M^{1} \}$ where $M^{1}$ is the quadratic right module $v.A(g)/(v.z)$, and let $\hh : \{ 0,1 \}^{2} \times \NN^{2} \rightarrow \NN_{0}^{2}$ be the map given by $\hh(0,1,3,4) = (0,1)$
and $\hh(i,j,k,\ell) = (0,0)$ if $(i,j,k,\ell) \neq (0,1,3,4)$. 
Then, \cite{MR4248207}, Thm. 5.6, tells us that this gives us a connected resolving datum on $A(g)$, whose associated resolving quiver is 
\[
\tikzcdset{arrow style=tikz, diagrams={>=stealth}}
	\begin{tikzcd}[ampersand replacement=\&, column sep=large, row sep = large]
		\Bbbk 
		\arrow[r] 
		\&
		M^{1}
	\end{tikzcd}
 \]
such that its arrow has bidegree $(3,4)$. 
\end{eg}
%%%%%%% 

%%%%%%%
\begin{eg}
\label{example:herscovich} 
Consider the Fomin-Kirillov algebra $\FK(3)$ on three generators (see \cite{MR4102553}, Section 2.3). 
Let $\MM = \{ \Bbbk \}$ and let $\hh : \{ 0 \}^{2} \times \NN^{2} \rightarrow \NN_{0}^{2}$ be the map given by $\hh(0,0,3,6) = (0,1)$
and $\hh(i,j,k,\ell) = (0,0)$ if $(i,j,k,\ell) \neq (0,0,3,6)$. 
Then, \cite{MR4102553}, Prop. 3.1, 
tells us that this gives a resolving datum on $\FK(3)$ whose associated resolving quiver is 
\[
\tikzcdset{arrow style=tikz, diagrams={>=stealth}}
	\begin{tikzcd}[ampersand replacement=\&, column sep=large, row sep = large, tail]
		\Bbbk 
		\arrow[out = -30, in = 30, loop] 
	\end{tikzcd}
 \]
such that its arrow has bidegree $(4,6)$. 
\end{eg}
%%%%%%%

From the resolving quiver associated to a resolving datum of the form given in the first paragraph of this subsection we can define the \textcolor{myultramarine}{\textbf{set of paths}} $\Pa_{M^{i}}$ given as the set formed by all paths $\bar{\alpha}$ of the quiver $\RQ_{A}$ such that $s(\bar{\alpha}) = M^{i}$. 
Moreover, we will define the following strict partial order on $\Pa_{M^{i}}$ for every $i \in \llbracket 1,n \rrbracket$ as follows. 
First, we set the vertex at $M^{i}$ to be strictly greater than any other path of $\Pa_{M^{i}}$. 
Given $\bar{\alpha} = (\alpha_{1}, \dots, \alpha_{n})$ and $\bar{\beta} = (\beta_{1}, \dots, \beta_{m})$ in $\Pa_{M^{i}}$ with $n, m \in \NN$, we say that $\bar{\alpha} < \bar{\beta}$ if $\alpha_{j}
= \beta_{j}$ for all $j \in \llbracket 1,j_{0}\rrbracket$ for some $j_{0} \in \llbracket 0,\min(n,m)\rrbracket$, and one of the following possibilities holds:
\begin{enumerate}[label=(O.\arabic*)]
    \item $n,m > j_{0}$, $\bideg_{1}(\alpha_{j_{0}+1})
= \bideg_{1}(\beta_{j_{0}+1})$ and $\alpha_{j_{0}+1} < \beta_{j_{0}+1}$;
    \item $n,m > j_{0}$, $\bideg_{1}(\alpha_{j_{0}+1})
<\bideg_{1}(\beta_{j_{0}+1})$;
\item $j_{0} = m < n$.
\end{enumerate} 
It is clear that this defines a strict partial order on $\Pa_{M^{i}}$. 

We now give the first main result of this article, 
which gives a description of a projective resolution of every quadratic module $M^{i}$ in a connected resolving datum $\{ M^{0} , \dots, M^{N} \}$ of a quadratic algebra $A$. 
%%%%%%%
\begin{thm}
\label{thm:rd}
Assume we have a connected resolving datum on a quadratic algebra $A$ with set of quadratic modules $\mathcalboondox{M} = \{ M^{0}, \dots, M^{N} \}$ and whose resolving quiver is denoted by $\RQ_{A}$. 
Then, there exists a projective resolution $P_{\bullet}^{M^{i}}$ of $M^{i}$ in the category of bounded below graded right $A$-modules such that 
\begin{equation}
\label{eq:resolution}
P_{n}^{M^{i}} = \bigoplus_{\text{\begin{tiny}$\begin{matrix}\bar{\alpha} \in \Pa_{M^{i}},\\\bideg_{1}(\bar{\alpha}) \leqslant n \end{matrix}$\end{tiny}}} \bar{\alpha}.\K^{\operatorname{mod}}_{n-\bideg_{1}(\bar{\alpha})}\big(t(\bar{\alpha})\big)\big(-\bideg_{2}(\bar{\alpha})\big)     
\end{equation}
for all $n \in \NN_{0}$ and $i \in \llbracket 0,N\rrbracket$, where the symbol $\bar{\alpha}$ 
%(of homological degree $\bideg_1(\bar{\alpha})$) 
multiplying the Koszul complex on the left is only a formal symbol used as a simple bookkeeping device. 
Moreover, if 
\begin{equation}
\label{eq:minimality}  \ddeg(\bar{\alpha}) \neq \ddeg(\bar{\beta}) - 1
\end{equation}
for all $\bar{\alpha}, \bar{\beta} \in \Pa_{M^{i}}$ such that $\bar{\alpha} < \bar{\beta}$ (\textit{e.g.} if $\ddeg(\alpha)$ is even for all arrows $\alpha $ of $\RQ_{A}$), %$\bideg_{1}(\bar{\alpha}) \leq n$ and $\bideg_{1}(\bar{\beta}) \leq n-1$, 
then the previous projective resolution is minimal.
\end{thm}
%%%%%%%
\begin{proof}
We are going to use the following notation. 
Let 
\[
\begin{tikzcd}
0 
\arrow[r]
&
M'
\arrow[r]
&
M
\arrow[r]
&
M''
\arrow[r]
&
0
\end{tikzcd}
\]
be a short exact sequence of right $A$-modules and let $P'_{\bullet} \rightarrow M'$ and $P''_{\bullet} \rightarrow M''$ be two projective resolutions (resp., up to homological degree $m \in \NN$) with differentials $d'_{\bullet}$ and $d''_{\bullet}$, respectively. 
Then, we will note by 
$P_{\bullet} = P'_{\bullet} \coplus P''_{\bullet} \rightarrow M$ a fixed projective resolution (resp., up to homological degree $m \in \NN$) given by the Horseshoe lemma (see \cite{MR1269324}, Lemma 2.2.8). 
We recall that $P_{n} = P'_{n} \oplus P''_{n}$ for all $n \in \NN_{0}$ (resp., for all $n \in \llbracket 0,m \rrbracket$) with differential $d_{\bullet}$ satisfying that
$d_{\bullet}|_{P'_{\bullet}} = d'_{\bullet}$ and $d_{\bullet}|_{P''_{\bullet}} = d''_{\bullet} + f_{\bullet}$ for some family $\{ f_{n} : P''_{n} \rightarrow P'_{n-1} \mid n \in \NN \}$ (resp., $\{ f_{n} : P''_{n} \rightarrow P'_{n-1} \mid n \in \llbracket 1,m \rrbracket\}$)
of morphisms of $A$-modules. 

Given $i \in \llbracket 0 , N \rrbracket$, let $m_{i} \in \NN$ be the largest positive integer such that $\operatorname{H}_{m_{i}}(\K_{\bullet}^{\operatorname{mod}}(M^{i})) \neq 0$ and $\operatorname{H}_{k}(\K_{\bullet}^{\operatorname{mod}}(M^{i})) = 0$ for all integers $k > m_{i}$. 
If $\operatorname{H}_{k}(\K_{\bullet}^{\operatorname{mod}}(M^{i})) = 0$ for all $k\in \NN$, then we set $m_{i} = 0$ in this case. 

We will denote by $d_{k+1}^{i} : \K_{k+1}^{\operatorname{mod}}(M^{i}) \rightarrow \K_{k}^{\operatorname{mod}}(M^{i})$ the differential of the Koszul complex of $M^{i}$ for $k \in \NN_{0}$ and $i \in \llbracket 0 , N \rrbracket$. 
For every $i \in \llbracket 0 , N \rrbracket$, we will construct a projective resolution $P^{i}_{\bullet}$ of $M^{i}$.
By Fact \ref{fact:k-mod} we will assume that $P^{i}_{n} = \K_{n}^{\operatorname{mod}}(M^{i})$ for $i \in \llbracket 0 , N \rrbracket$ and $n \in \{ 0 , 1 \}$. 
In particular, $P^{i}_{n}$ coincides with \eqref{eq:resolution} 
for all $i \in \llbracket 0 , N \rrbracket$ and $n \in \{ 0 , 1 \}$. 
We will in fact prove that $P^{i}_{n}$ coincides with \eqref{eq:resolution} 
for all $i \in \llbracket 0 , N \rrbracket$ and $n \in \NN_{0}$
by induction on the homological degree $n$. 
If $m_{i} = 0$, we set $P^{i}_{\bullet} = \K_{\bullet}^{\operatorname{mod}}(M^{i})$ for all $\bullet \in \NN_{0}$. 
It is straightforward to see that the resolutions $P^{i}_{\bullet}$ and \eqref{eq:resolution} coincide. 

We will now construct $P^{i}_{\bullet}$ for all $\bullet \in \NN_{0}$ for $i \in \llbracket 0 , N \rrbracket$ such that $m_{i} > 0$. 
Let $m\in\NN$. 
Assume that we have defined $P^{i}_{n}$ 
for all $i \in \llbracket 0 , N \rrbracket$ such that $m_{i} > 0$ and $n \in \llbracket 0 , m \rrbracket$ such that $P^{i}_{n}$ coincides with \eqref{eq:resolution} for all $n \in \llbracket 0 , m \rrbracket$. 
Using the Horseshoe lemma for \eqref{eq:ses}, we get a projective resolution of $\operatorname{H}_{k}(\K_{\bullet}^{\operatorname{mod}}(M^{i}))$ of the form
\[
{}^{m}Q^{i,k}_{\bullet} = 
\bigg(\bigoplus\limits_{j=0}^{N}\bigoplus\limits_{\ell \in \NN} \Big(P^{j}_{\bullet}(-\ell)\Big)^{\pi_{1}(\hh(i,j,k,\ell))}\bigg) \coplus \bigg(\bigoplus\limits_{j=0}^{N}\bigoplus\limits_{\ell \in \NN} \Big(P^{j}_{\bullet}(-\ell)\Big)^{\pi_{2}(\hh(i,j,k,\ell))}\bigg) 
\]
defined for homological degrees $\bullet \in \llbracket 0 , m \rrbracket$, $i\in \llbracket 1,N\rrbracket$ and $k\in \llbracket 1,m_i\rrbracket$. 
We will construct by induction on the index $k \in \llbracket 0 , m_{i}   \rrbracket$ a family of complexes of right $A$-modules ${}^{m}R^{i,k}_{\bullet}$ 
for $\bullet \in \llbracket 0 , m + 1  \rrbracket$ such that ${}^{m}R^{i,k}_{\bullet}$ is a projective resolution of $\Img(d_{m_{i}-k+1}^{i})$ 
up to homological degree $m + 1$. 
For $k = 0$, we set ${}^{m}R^{i,0}_{\bullet}$ as the complex of right $A$-modules given by $(\K_{\bullet+m_{i}+1}^{\operatorname{mod}}(M^{i}),d_{\bullet+m_{i}+1}^{i})_{\bullet \in \NN_{0}}$. 
Note that ${}^{m}R^{i,0}_{\bullet}$ is a projective resolution of $\Img(d_{m_{i}+1}^{i})$ for $i \in \llbracket 0 , N \rrbracket$ such that $m_{i} > 0$, and it is independent of $m$. 
Assume now we have defined a complex of right $A$-modules ${}^{m}R^{i,k-1}_{\bullet}$ 
for some $k \in \llbracket 1 , m_{i}  \rrbracket$ and $\bullet \in \llbracket 0 , m + 1  \rrbracket$ such that ${}^{m}R^{i,k-1}_{\bullet}$ is a projective resolution of $\Img(d_{m_{i}-k+2}^{i})$ 
up to homological degree $m + 1 $.
Then, we define the complex of right $A$-modules ${}^{m}R^{i,k}_{\bullet}$ 
by 
\begin{equation}
\label{eq:cpx-rec}
{}^{m}R^{i,k}_{0} = \K_{m_{i}-k+1}^{\operatorname{mod}}(M^{i}) 
\text{ and } 
{}^{m}R^{i,k}_{\bullet} = {}^{m}R^{i,k-1}_{\bullet-1} \coplus {}^{m}Q^{i,m_{i}-k+1}_{\bullet-1}     
\end{equation} 
for $\bullet \in \llbracket 1, m +1\rrbracket$, the differential $d_{\bullet}^{i,k}$ for $\bullet \geqslant 2$ is induced by that of ${}^{m}R^{i,k-1}_{\bullet-1} \coplus {}^{m}Q^{i,m_{i}-k+1}_{\bullet-1}$ and $d_{1}^{i,k} : {}^{m}R^{i,k}_{1} \rightarrow {}^{m}R^{i,k}_{0}$ is given as the composition of the augmentation ${}^{m}R^{i,k-1}_{\bullet} \coplus {}^{m}Q^{i,m_{i}-k+1}_{\bullet} \rightarrow \Ker(d_{m_{i}-k+1}^{i})$ and the inclusion 
$\Ker(d_{m_{i}-k+1}) \hookrightarrow \K_{m_{i}-k+1}^{\operatorname{mod}}(M^{i})$. 
Using the Horseshoe lemma for 
\[
\begin{tikzcd}
0 
\arrow[r]
&
\Img(d_{m_{i}-k+2}^{i})
\arrow[r]
&
\Ker(d_{m_{i}-k+1}^{i})
\arrow[r]
&
\operatorname{H}_{m_{i}-k+1}\big(\K_{\bullet}^{\operatorname{mod}}(M^{i})\big)
\arrow[r]
&
0
\end{tikzcd}
\]
together with the projective resolutions 
${}^{m}R^{i,k-1}_{\bullet}$ and ${}^{m}Q^{i,m_{i}-k+1}_{\bullet}$ for $\bullet \in \llbracket 0 , m   \rrbracket$, we obtain that the complex
${}^{m}R^{i,k-1}_{\bullet} \coplus {}^{m}Q^{i,m_{i}-k+1}_{\bullet}$ for $\bullet \in \llbracket 0 , m   \rrbracket$ is a projective resolution of $\Ker(d_{m_{i}-k+1}^{i})$ up to homological degree $m$, 
and thus ${}^{m}R^{i,k}_{\bullet}$ for for $\bullet \in \llbracket 0 , m + 1  \rrbracket$ is a projective resolution of $\Img(d_{m_{i}-k+1}^{i})$ 
up to homological degree $m + 1$, as was to be shown. 
In particular, ${}^{m}R^{i,m_i}_{\bullet}$ for $\bullet \in \llbracket 0 , m + 1  \rrbracket$
is a projective resolution of $\Img(d^i_1)$ up to homological degree $m + 1$.  
Let ${}^{m}R^{i}_{\bullet}=\K_0^{\operatorname{mod}}(M^i) $ and $ {}^{m}R^{i}_{\bullet}={}^{m}R^{i,m_{i}}_{\bullet-1}$ for $\bullet \in \llbracket 1,m+2\rrbracket$. 
Then ${}^{m}R^{i}_{\bullet}$ for $\bullet \in \llbracket 0 , m + 2  \rrbracket$
is a projective resolution of $M^i$ up to homological degree $m + 2$.  
A long but straightforward computation shows that ${}^{m}R^{i}_{\bullet}$ 
coincides with \eqref{eq:resolution}
for $\bullet \in \llbracket 0 , m + 2  \rrbracket$, and 
that we can take the complexes ${}^{m}R^{i}_{\bullet}$ and ${}^{m-1}R^{i}_{\bullet}$ to coincide up to homological degree $m+1$. 
Hence, if $i \in \llbracket 0 , N \rrbracket$ such that $m_{i} > 0$, we define the complex $P^{i}_{\bullet}$ to be equal to ${}^{m}R^{i}_{\bullet}$ up to homological degree $m+2$. 
Since this holds for every $m \in \NN$,  the first part of the theorem is proved. 

To prove the last one, let us denote by 
$P_{n,\bar{\alpha}}^{i}$ the direct summand in \eqref{eq:resolution} indexed by $\bar{\alpha} \in \Pa_{M^{i}}$. 
The construction of the projective resolution $P^{i}_{\bullet}$ given in the first part of the proof tells us that, given $\bar{\alpha}, \bar{\beta} \in \Pa_{M^{i}}$, if the component 
\[
d_{n+1}^{\bar{\alpha},\bar{\beta}} \otimes_{A} \operatorname{id}_{\Bbbk}  : P_{n+1,\bar{\alpha}}^{i} \otimes_{A} \Bbbk \rightarrow P_{n,\bar{\beta}}^{i} \otimes_{A} \Bbbk
\]
of the differential of $P^{i}_{\bullet} \otimes_{A} \Bbbk$
is nonzero, then $\bar{\alpha} < \bar{\beta}$ and $\ddeg(\bar{\alpha}) = \ddeg(\bar{\beta}) - 1$. 
The minimality result then follows. 
\end{proof}

%%%%%%%
\begin{rk} 
\label{remark:proj-ex}
It is easy to check that conditions \eqref{eq:minimality} are verified in the case of Proposition \ref{proposition:almost-koszul}, as well as in Examples \ref{example:cassidy}, \ref{example:conner-goetz} and \ref{example:herscovich}, so the corresponding projective resolution \eqref{eq:resolution} is minimal, coinciding with the resolutions constructed in those references. 
We also remark that the Anick resolution for the three previous examples cannot be minimal, regardless of the choice of the order on the generators of the algebras, 
as it is the case for any quadratic algebra that is not Koszul (see \cite{MR2177131}, Chapter 4, Thm. 3.1). 
This in particular implies that our resolution can be minimal even when Anick's resolution is not. 
\end{rk}
%%%%%%%

%%%%%%%%%%%%
\section{\texorpdfstring{Resolving datum on $\FK(4)$}{Resolving datum on FK(4)}}
\label{subsec:rd FK4}

We will prove in this section the second main result of this article, namely that the Fomin-Kirillov algebra $\FK(4)$ of index $4$ has a connected resolving datum 
(see Theorem \ref{thm:rd fk4}). 
In consequence, combining this with Theorem \ref{thm:rd} we obtain immediately a projective resolution of the trivial module in the category of bounded-below graded right $\FK(4)$-modules.

\subsection{\texorpdfstring{Generalities on the Fomin-Kirillov algebra $\FK(4)$ and its quadratic dual}{Generalities on the Fomin-Kirillov algebra FK(4) and its quadratic dual}}
\label{subsec:g FK4} 

From now on we assume that the field $\Bbbk$ has characteristic different from $2$ and $3$. 
For a set $S$, we denote by $\Bbbk S$ the $\Bbbk$-vector space spanned by all elements of $S$. 

Let $\I$ be the set $\{(i,j)\in \llbracket 1,4 \rrbracket ^2 \mid i<j \}$, 
$\II$ the set $\{ (1,2), (1,3), (2,3) \}$ 
and $\J$ the set $ \{ (i,j)\in \llbracket 1,4 \rrbracket^{2} \mid i\neq j \} $.
We recall that the \textbf{\textcolor{myultramarine}{Fomin-Kirillov algebra $\FK(4)$ of index $4$}} is the quadratic $\Bbbk$-algebra generated by the $\Bbbk$-vector space $V$ spanned by $X=\{x_{i,j}\mid (i,j)\in\I \}$, 
modulo the ideal generated by the vector space $R \subseteq V^{\otimes 2}$ spanned by the following $17$ elements
\begin{align*}
   & 
   x_{1,2}^2, 
   x_{1,3}^2,  
   x_{2,3}^2,  
   x_{1,4}^2,   
   x_{2,4}^2,   
   x_{3,4}^2,   
   x_{1,2}x_{2,3}-x_{2,3}x_{1,3}-x_{1,3}x_{1,2},  
   x_{2,3}x_{1,2}-x_{1,2}x_{1,3}-x_{1,3}x_{2,3}, 
   \\
   &
   x_{1,2}x_{2,4}-x_{2,4}x_{1,4}-x_{1,4}x_{1,2},  
   x_{2,4}x_{1,2}-x_{1,2}x_{1,4}-x_{1,4}x_{2,4},  
   x_{1,3}x_{3,4}-x_{3,4}x_{1,4}-x_{1,4}x_{1,3}, 
   \\
   & 
   x_{3,4}x_{1,3}-x_{1,3}x_{1,4}-x_{1,4}x_{3,4}, 
   x_{2,3}x_{3,4}-x_{3,4}x_{2,4}-x_{2,4}x_{2,3},  
   x_{3,4}x_{2,3}-x_{2,3}x_{2,4}-x_{2,4}x_{3,4},
   \\
   &
   x_{1,2}x_{3,4}-x_{3,4}x_{1,2},   
   x_{1,3}x_{2,4}-x_{2,4}x_{1,3},   
   x_{1,4}x_{2,3}-x_{2,3}x_{1,4}.
\end{align*} 
To simplify, we will denote the Fomin-Kirillov algebra $\FK(4)$ of index $4$ simply by $A$. 
Recall that the dimension of $A$ is $576$ and the Hilbert series of $A$ is
\begin{small}
\begin{align*}  
   [2]^2[3]^2[4]^2&=1+6t+19t^2+42t^3+71t^4+96t^5+106t^6+96t^7+71t^8+42t^9+19t^{10}+6t^{11}
   +t^{12},
\end{align*}
\end{small} 
where $[n] = \sum_{i=0}^{n-1} t^{i}$, for $n \in \NN$. 
Note that $A =\oplus_{m\in \llbracket 0,12\rrbracket}A_m$, where $A_m$ is the subspace of $A$ concentrated
in internal degree $m$. 
We refer the reader to \cites{MR1667680, MR1800714} for more information on Fomin-Kirillov algebras.

The previous Hilbert series can be reobtained using GAP code in Appendix \ref{subsec:W}. 
If the free monoid generated by
$X$ is equipped with the homogeneous lexicographic order induced by the well order $x_{1,2} \prec x_{1,3} \prec x_{2,3} \prec x_{1,4} \prec x_{2,4} \prec x_{3,4}$ on $X$, 
then a Gröbner basis $G_A$ of the ideal $(R)$ in the algebra $\T (V)$ is given by the following $30$ elements
\begin{align*}
& x_{1,2}^2,
x_{1,3}^2,
x_{2,3}x_{1,2}-x_{1,3}x_{2,3}-x_{1,2}x_{1,3},
x_{2,3}x_{1,3}+x_{1,3}x_{1,2}-x_{1,2}x_{2,3},
x_{2,3}^2,
x_{1,4}x_{2,3}-x_{2,3}x_{1,4},
\\
&
x_{1,4}^2,
x_{2,4}x_{1,2}-x_{1,4}x_{2,4}-x_{1,2}x_{1,4},
x_{2,4}x_{1,3}-x_{1,3}x_{2,4},
x_{2,4}x_{1,4}+x_{1,4}x_{1,2}-x_{1,2}x_{2,4},
x_{2,4}^2,
\\
&
x_{3,4}x_{1,2}-x_{1,2}x_{3,4},
x_{3,4}x_{1,3}-x_{1,4}x_{3,4}-x_{1,3}x_{1,4},
x_{3,4}x_{2,3}-x_{2,4}x_{3,4}-x_{2,3}x_{2,4},
\\
&
x_{3,4}x_{1,4}+x_{1,4}x_{1,3}-x_{1,3}x_{3,4},
x_{3,4}x_{2,4}+x_{2,4}x_{2,3}-x_{2,3}x_{3,4},
x_{3,4}^2,
x_{1,3}x_{1,2}x_{1,3}+x_{1,2}x_{1,3}x_{1,2},
\\
&
x_{1,4}x_{1,2}x_{1,4}+x_{1,2}x_{1,4}x_{1,2},
x_{1,4}x_{1,3}x_{1,2}-x_{1,4}x_{1,2}x_{2,3}+x_{2,3}x_{1,4}x_{1,3},
\\
&
x_{1,4}x_{1,3}x_{2,3}+x_{1,4}x_{1,2}x_{1,3}-x_{2,3}x_{1,4}x_{1,2},
x_{1,4}x_{1,3}x_{1,4}+x_{1,3}x_{1,4}x_{1,3},
\\
&
x_{2,4}x_{2,3}x_{1,4}+x_{1,4}x_{1,2}x_{2,3}-x_{1,2}x_{2,4}x_{2,3},
x_{2,4}x_{2,3}x_{2,4}+x_{2,3}x_{2,4}x_{2,3},
\\
&
x_{1,4}x_{1,2}x_{1,3}x_{2,3}-x_{2,3}x_{1,4}x_{1,2}x_{2,3},
x_{1,4}x_{1,2}x_{1,3}x_{1,4}+x_{1,3}x_{1,4}x_{1,2}x_{1,3}+x_{1,2}x_{1,3}x_{1,4}x_{1,2},
\\
&
x_{1,4}x_{1,2}x_{2,3}x_{1,4}+x_{1,2}x_{1,4}x_{1,2}x_{2,3},
x_{1,4}x_{1,2}x_{1,3}x_{1,2}x_{2,3}+x_{2,3}x_{1,4}x_{1,2}x_{1,3}x_{1,2},
\\
&
x_{1,4}x_{1,2}x_{1,3}x_{1,2}x_{1,4}x_{1,2}-x_{1,3}x_{1,4}x_{1,2}x_{1,3}x_{1,2}x_{1,4},
\\
&
x_{1,4}x_{1,2}x_{1,3}x_{1,2}x_{1,4}x_{1,3}-x_{1,2}x_{1,4}x_{1,2}x_{1,3}x_{1,2}x_{1,4},
\end{align*}
which are obtained using the GAP code in Appendix \ref{sec:cpx}.
The classes in $A$ of the standard words of $\T (V)$ with respect to $G_A$ thus form a homogeneous $\Bbbk$-basis $\mathcalboondox{B}$ of $A$. 
We set $\mathcalboondox{B}_{m} = \mathcalboondox{B} \cap A_{m}$ for $m \in \llbracket 0 , 12 \rrbracket$. 

We denote by $\{y_{i,j} = x_{i,j}^* \mid (i,j)\in \I  \}$ the basis of $V^*$ dual to the basis $X = \{ x_{i,j} \mid (i,j)\in \I  \}$ of $V$. 
Then, the space of relations $R^{\bot}\subseteq (V^*)^{\otimes 2}$ of the quadratic dual algebra $A^!=\T(V^{*})/(R^{\bot})=\oplus_{n\in\NN_0}A^!_{-n}$ of $A$ is spanned by the following $19$ elements 
\begin{align*}
      & y_{1,2}y_{2,3}+y_{2,3}y_{1,3}, 
      y_{1,3}y_{2,3}+y_{2,3}y_{1,2},
      y_{1,2}y_{2,3}+y_{1,3}y_{1,2}, 
      y_{1,2}y_{1,3}+y_{2,3}y_{1,2}, 
      \\
      & y_{1,2}y_{2,4}+y_{2,4}y_{1,4},  
      y_{1,4}y_{2,4}+y_{2,4}y_{1,2}, 
      y_{1,2}y_{2,4}+y_{1,4}y_{1,2}, 
      y_{1,2}y_{1,4}+y_{2,4}y_{1,2}, 
      \\
      & y_{1,3}y_{3,4}+y_{3,4}y_{1,4}, 
      y_{1,4}y_{3,4}+y_{3,4}y_{1,3},  
      y_{1,3}y_{3,4}+y_{1,4}y_{1,3}, 
      y_{1,3}y_{1,4}+y_{3,4}y_{1,3}, 
      \\
      & y_{2,3}y_{3,4}+y_{3,4}y_{2,4},  
      y_{2,4}y_{3,4}+y_{3,4}y_{2,3},
      y_{2,3}y_{3,4}+y_{2,4}y_{2,3},  
      y_{2,3}y_{2,4}+y_{3,4}y_{2,3}, 
      \\
      & y_{1,2}y_{3,4}+y_{3,4}y_{1,2},  
      y_{1,3}y_{2,4}+y_{2,4}y_{1,3},  
      y_{2,3}y_{1,4}+y_{1,4}y_{2,3}.  
\end{align*}
Using the GAP code in Appendix \ref{sec:cpx}, 
we get a Gröbner basis $G_B$ of the ideal $(R^{\bot})$ in $\T (V^{*})$ given by the following $31$ elements
\begin{equation}
\label{eq:grobner-basis-a-dual}
\begin{split}
 & 
 y_{1,3}y_{1,2} + y_{1,2}y_{2,3}, \hskip 1mm
 y_{1,3}y_{2,3}-y_{1,2}y_{1,3}, \hskip 1mm
 y_{2,3}y_{1,2} + y_{1,3}y_{2,3}, \hskip 1mm
 y_{2,3}y_{1,3} + y_{1,2}y_{2,3}, \hskip 1mm
 y_{1,4}y_{1,2} + y_{1,2}y_{2,4}, 
 \\
 & 
 y_{1,4}y_{1,3} + y_{1,3}y_{3,4}, \hskip 1mm
 y_{1,4}y_{2,3} + y_{2,3}y_{1,4}, \hskip 1mm
 y_{1,4}y_{2,4} - y_{1,2}y_{1,4}, \hskip 1mm
 y_{1,4}y_{3,4} - y_{1,3}y_{1,4}, \hskip 1mm
 y_{2,4}y_{1,2} + y_{1,4}y_{2,4},
 \\
 & 
 y_{2,4}y_{1,3} + y_{1,3}y_{2,4}, \hskip 1mm
 y_{2,4}y_{2,3} + y_{2,3}y_{3,4}, \hskip 1mm
 y_{2,4}y_{1,4} + y_{1,2}y_{2,4}, \hskip 1mm 
 y_{2,4}y_{3,4} - y_{2,3}y_{2,4}, \hskip 1mm 
 y_{3,4}y_{1,2} + y_{1,2}y_{3,4},
 \\
 & 
 y_{3,4}y_{1,3} + y_{1,4}y_{3,4}, \hskip 1mm
 y_{3,4}y_{2,3} + y_{2,4}y_{3,4}, \hskip 1mm 
 y_{3,4}y_{1,4} + y_{1,3}y_{3,4}, \hskip 1mm 
 y_{3,4}y_{2,4} + y_{2,3}y_{3,4}, \hskip 1mm 
 y_{1,2}y_{2,3}^2 - y_{1,2}y_{1,3}^2,
 \\
 & 
 y_{1,2}y_{2,4}^2 - y_{1,2}y_{1,4}^2, \hskip 1mm 
 y_{1,3}y_{3,4}^2 - y_{1,3}y_{1,4}^2, \hskip 1mm 
 y_{2,3}y_{3,4}^2 - y_{2,3}y_{2,4}^2, \hskip 1mm 
 y_{1,2}y_{1,3}^3 - y_{1,2}^3y_{1,3}, \hskip 1mm 
 \\
 &
 y_{1,2}y_{1,3}y_{2,4}^2 - y_{1,2}y_{1,3}y_{1,4}^2, \hskip 1mm 
 y_{1,2}y_{2,3}y_{2,4}^2 - y_{1,2}y_{2,3}y_{1,4}^2, \hskip 1mm 
 y_{1,2}y_{1,4}^3 - y_{1,2}^3y_{1,4}, \hskip 1mm 
 y_{1,3}y_{1,4}^3 - y_{1,3}^3y_{1,4},
 \\
 & 
 y_{2,3}y_{2,4}^3 - y_{2,3}^3y_{2,4}, \hskip 1mm 
 y_{1,2}y_{1,3}^2y_{2,4}^2 - y_{1,2}y_{1,3}^2y_{1,4}^2, \hskip 1mm 
 y_{1,2}y_{2,3}y_{1,4}^3 - y_{1,2}^3y_{2,3}y_{1,4}.
\end{split}    
\end{equation}

Let $\B_0^!=\{1\} \subseteq \Bbbk$, let
$\B_1^!=\{ y_{i,j}\mid (i,j)\in\I \} \subseteq V^{*}$, let $\B_2^! \subseteq A_{-2}^{!}$ be the set formed by the following $17$ elements 
\begin{align*}
 &
 y_{1,2}^2, \hskip 1mm
 y_{1,2}y_{1,3}, \hskip 1mm
 y_{1,2}y_{2,3}, \hskip 1mm
 y_{1,2}y_{1,4}, \hskip 1mm 
 y_{1,2}y_{2,4}, \hskip 1mm
 y_{1,2}y_{3,4}, \hskip 1mm
 y_{1,3}^2, \hskip 1mm
 y_{1,3}y_{1,4}, \hskip 1mm
 y_{1,3}y_{2,4}, \hskip 1mm
 y_{1,3}y_{3,4}, \hskip 1mm
 y_{2,3}^2, \hskip 1mm
 y_{2,3}y_{1,4}, \hskip 1mm
 \\
 &
 y_{2,3}y_{2,4}, \hskip 1mm
 y_{2,3}y_{3,4}, \hskip 1mm
 y_{1,4}^2, \hskip 1mm
 y_{2,4}^2, \hskip 1mm
 y_{3,4}^2,
\end{align*}
let $\B_3^! \subseteq A_{-3}^{!}$ be the set formed by the following $30$ elements 
\begin{align*}
 &
 y_{1,2}^3, \hskip 1mm
 y_{1,2}^2y_{1,3}, \hskip 1mm
 y_{1,2}^2y_{2,3},\hskip 1mm
 y_{1,2}^2y_{1,4},\hskip 1mm
 y_{1,2}^2y_{2,4},\hskip 1mm
 y_{1,2}^2y_{3,4},\hskip 1mm
 y_{1,2}y_{1,3}^2, \hskip 1mm
 y_{1,2}y_{1,3}y_{1,4},\hskip 1mm
 y_{1,2}y_{1,3}y_{2,4},\hskip 1mm
 \\
 &
 y_{1,2}y_{1,3}y_{3,4},\hskip 1mm
 y_{1,2}y_{2,3}y_{1,4},\hskip 1mm
 y_{1,2}y_{2,3}y_{2,4},\hskip 1mm
 y_{1,2}y_{2,3}y_{3,4},\hskip 1mm
 y_{1,2}y_{1,4}^2,\hskip 1mm 
 y_{1,2}y_{3,4}^2,\hskip 1mm
 y_{1,3}^3, \hskip 1mm
 y_{1,3}^2y_{1,4},\hskip 1mm
 y_{1,3}^2y_{2,4},\hskip 1mm
 \\
 &
 y_{1,3}^2y_{3,4},\hskip 1mm
 y_{1,3}y_{1,4}^2,\hskip 1mm
 y_{1,3}y_{2,4}^2,\hskip 1mm
 y_{2,3}^3,\hskip 1mm
 y_{2,3}^2y_{1,4},\hskip 1mm
 y_{2,3}^2y_{2,4},\hskip 1mm
 y_{2,3}^2y_{3,4},\hskip 1mm
 y_{2,3}y_{1,4}^2,\hskip 1mm
 y_{2,3}y_{2,4}^2,\hskip 1mm
 y_{1,4}^3,\hskip 1mm
 y_{2,4}^3,\hskip 1mm
 y_{3,4}^3,
\end{align*}
and let $\B_4^! \subseteq A_{-4}^{!}$ be the set formed by the following $38$ elements 
\begin{align*}
 &
 y_{1,2}^4,\hskip 1mm
 y_{1,2}^3y_{1,3},\hskip 1mm
 y_{1,2}^3y_{2,3},\hskip 1mm 
 y_{1,2}^3y_{1,4},\hskip 1mm
 y_{1,2}^3y_{2,4},\hskip 1mm
 y_{1,2}^3y_{3,4},\hskip 1mm
 y_{1,2}^2y_{1,3}^2,\hskip 1mm
 y_{1,2}^2y_{1,3}y_{1,4},\hskip 1mm 
 y_{1,2}^2y_{1,3}y_{2,4},
 \\
 &
 y_{1,2}^2y_{1,3}y_{3,4},\hskip 1mm
 y_{1,2}^2y_{2,3}y_{1,4},\hskip 1mm 
 y_{1,2}^2y_{2,3}y_{2,4},\hskip 1mm
 y_{1,2}^2y_{2,3}y_{3,4},\hskip 1mm
 y_{1,2}^2y_{1,4}^2,\hskip 1mm
 y_{1,2}^2y_{3,4}^2,\hskip 1mm
 y_{1,2}y_{1,3}^2y_{1,4},\hskip 1mm
 y_{1,2}y_{1,3}^2y_{2,4},
 \\
 &
 y_{1,2}y_{1,3}^2y_{3,4},\hskip 1mm
 y_{1,2}y_{1,3}y_{1,4}^2,\hskip 1mm 
 y_{1,2}y_{2,3}y_{1,4}^2,\hskip 1mm
 y_{1,2}y_{3,4}^3,\hskip 1mm
 y_{1,3}^4,\hskip 1mm
 y_{1,3}^3y_{1,4},\hskip 1mm 
 y_{1,3}^3y_{2,4},\hskip 1mm
 y_{1,3}^3y_{3,4},\hskip 1mm
 y_{1,3}^2y_{1,4}^2,
 \\
 &
 y_{1,3}^2y_{2,4}^2,\hskip 1mm
 y_{1,3}y_{2,4}^3,\hskip 1mm
 y_{2,3}^4,\hskip 1mm
 y_{2,3}^3y_{1,4},\hskip 1mm
 y_{2,3}^3y_{2,4},\hskip 1mm 
 y_{2,3}^3y_{3,4},\hskip 1mm
 y_{2,3}^2y_{1,4}^2,\hskip 1mm
 y_{2,3}^2y_{2,4}^2,\hskip 1mm
 y_{2,3}y_{1,4}^3,\hskip 1mm
 y_{1,4}^4,\hskip 1mm
 y_{2,4}^4,\hskip 1mm
 y_{3,4}^4.
\end{align*}
Moreover, for every integer $n \geqslant 5$, define $\B_n^!=\U^!_n \cup \C^!_n$, where the set $\U^!_n \subseteq A_{-n}^{!}$ consists of the following $24$ elements 
\begin{align*}
   &
   y_{1,2}^{n-1}y_{1,3}, 
   y_{1,2}^{n-1}y_{2,3}, 
   y_{1,2}^{n-1}y_{1,4}, 
   y_{1,2}^{n-1}y_{2,4},  
   y_{1,2}^{n-2}y_{1,3}^2, 
   y_{1,2}^{n-2}y_{1,3}y_{1,4},
   y_{1,2}^{n-2}y_{1,3}y_{2,4}, 
   y_{1,2}^{n-2}y_{1,3}y_{3,4}, 
   \\
   &
   y_{1,2}^{n-2}y_{2,3}y_{1,4}, 
   y_{1,2}^{n-2}y_{2,3}y_{2,4}, 
   y_{1,2}^{n-2}y_{2,3}y_{3,4},  
   y_{1,2}^{n-2}y_{1,4}^2,  
   y_{1,2}^{n-3}y_{1,3}^2y_{1,4}, 
   y_{1,2}^{n-3}y_{1,3}^2y_{2,4}, 
   y_{1,2}^{n-3}y_{1,3}^2y_{3,4}, 
   \\
   &
   y_{1,2}^{n-3}y_{1,3}y_{1,4}^2, 
   y_{1,2}^{n-3}y_{2,3}y_{1,4}^2, 
   y_{1,2}^{n-4}y_{1,3}^2y_{1,4}^2, 
   y_{1,3}^{n-1}y_{1,4}, 
   y_{1,3}^{n-1}y_{3,4},  
   y_{1,3}^{n-2}y_{1,4}^2,  
   y_{2,3}^{n-1}y_{2,4}, 
   y_{2,3}^{n-1}y_{3,4}, 
   y_{2,3}^{n-2}y_{2,4}^2,  
   \end{align*}
and $\C^!_n \subseteq A_{-n}^{!}$ is the set of $3(n+1)$ elements given by
\begin{equation}
\label{eq:basis c}
\begin{split}
\C^!_n=\big\{   
   y_{1,2}^{n-r}y_{3,4}^{r}, \hskip 1mm 
   y_{1,3}^{n-r}y_{2,4}^{r},  \hskip 1mm 
   y_{2,3}^{n-r}y_{1,4}^{r} \mid r\in\llbracket 0,n \rrbracket
\big\}.
\end{split}
\end{equation} 

The following result is proved directly from the explicit description of the Gröbner basis 
$G_B$ given in \eqref{eq:grobner-basis-a-dual} for the ideal $(R^{\bot}) \subseteq \T (V^{*})$ . 
%%%%%%%
\begin{fact}
\label{fact:basis-quadratic-dual-fk4}
The set $\B_n^!$ is a basis of $A^!_{-n}$ for $n\in \NN_{0}$, consisting of standard words with respect to the Gröbner basis $G_B$. 
In consequence, $\# (\B^!_n) =3n+27$ for $n\geqslant 5$, and the Hilbert series $h(t)$ of $A^!$ is given by 
\begin{equation}
   \label{eq:Hilbert series of qd of FK4}
   \begin{split}
h(t)=1+6t+17t^2+30t^3+38t^4+\sum_{n=5}^{\infty} (3n+27)t^n = \frac{1+4 t + 6 t^{2} + 2 t^{3} - 5 t^{4} - 4 t^{5} - t^{6}}{(t-1)^2}.
   \end{split}
\end{equation}
\end{fact}

The following result describes several identities expressing products of the generators of the quadratic dual algebra $A^{!}$ in terms of the basis $\B^!=\cup_{n\in\NN_0}\B^!_{n}$. 
The proof is a straightforward but rather lengthy verification, which we leave to the reader. 

%%%%%%%
\begin{fact}
\label{fact:products-quadratic-dual-fk4}
We have the following identities 
\begin{equation}
   \label{eq:product ijkl}
   \begin{split}
      y_{i,j}^{n-r}y_{k,l}^ry_{i,j}=(-1)^r y_{i,j}^{n-r+1}y_{k,l}^r ,
      \hskip 3mm
      y_{i,j}^{n-r}y_{k,l}^ry_{k,l}= y_{i,j}^{n-r}y_{k,l}^{r+1} 
   \end{split}
\end{equation}
and
\begin{equation}
   \label{eq:product left ijkl}
   \begin{split}
      y_{i,j} y_{i,j}^{n-r}y_{k,l}^r=y_{i,j}^{n-r+1}y_{k,l}^r ,
      \hskip 3mm
      y_{k,l}y_{i,j}^{n-r}y_{k,l}^r=(-1)^{n-r} y_{i,j}^{n-r}y_{k,l}^{r+1} 
   \end{split}
\end{equation}
in $A^{!}$, for all integers $n\geqslant 2$, $r\in \llbracket 1,n-1\rrbracket $, $(i,j)\in \II$, $(k,l)\in \I$ with $\# \{i,j,k,l \}=4$. 
Moreover, we also have the identities 
\begin{align*}
   y_{1,2}^{n-r}y_{3,4}^ry_{1,3} & =\chi_r y_{1,2}^{n-2}y_{1,3}y_{1,4}^2-\chi_{r+1}y_{1,2}^{n-1}y_{1,3}y_{1,4},
   \\
   y_{1,2}^{n-r}y_{3,4}^ry_{2,3} & =\chi_ry_{1,2}^{n-2}y_{2,3}y_{1,4}^2-\chi_{r+1}y_{1,2}^{n-1}y_{2,3}y_{2,4}, 
   \\
   y_{1,2}^{n-r}y_{3,4}^ry_{1,4} & =\chi_r y_{1,2}^{n-2}y_{1,3}^2y_{1,4}-\chi_{r+1}y_{1,2}^{n-1}y_{1,3}y_{3,4},
   \\
   y_{1,2}^{n-r}y_{3,4}^ry_{2,4} & =\chi_r y_{1,2}^{n-2}y_{1,3}^2y_{2,4}-\chi_{r+1}y_{1,2}^{n-1}y_{2,3}y_{3,4},
   \\
   y_{1,3}^{n-r}y_{2,4}^ry_{1,2} & =\chi_n \chi_r y_{1,2}^{n-3}y_{1,3}^2y_{1,4}^2-\chi_{n+1}\chi_r y_{1,2}^{n-2}y_{2,3}y_{1,4}^2+\chi_n\chi_{r+1}y_{1,2}^{n-1}y_{2,3}y_{1,4}
   \\
   & \phantom{= \; }
   -\chi_{n+1}\chi_{r+1}y_{1,2}^{n-2}y_{1,3}^2y_{1,4},
   \\
   y_{1,3}^{n-r}y_{2,4}^ry_{2,3} & =\chi_n \chi_r y_{1,2}^{n-2}y_{2,3}y_{1,4}^2+\chi_{n+1}\chi_r y_{1,2}^{n-2}y_{1,3}y_{1,4}^2-\chi_n\chi_{r+1}y_{1,2}^{n-1}y_{1,3}y_{3,4}
   \\
   & \phantom{= \; }
   -\chi_{n+1}\chi_{r+1}y_{1,2}^{n-1}y_{2,3}y_{3,4}, 
   \\
   y_{1,3}^{n-r}y_{2,4}^ry_{1,4} & =\chi_n \chi_r y_{1,2}^{n-2}y_{1,3}^2y_{1,4}+\chi_{n+1}\chi_r y_{1,2}^{n-1}y_{1,3}y_{1,4}+\chi_n\chi_{r+1}y_{1,2}^{n-1}y_{2,3}y_{2,4}
   \\
   & \phantom{= \; }
   -\chi_{n+1}\chi_{r+1}y_{1,2}^{n-2}y_{1,3}^2y_{2,4}, 
   \stepcounter{equation}\tag{\theequation}\label{eq:y12y34y13}
   \\
   y_{1,3}^{n-r}y_{2,4}^ry_{3,4} & =\chi_n \chi_r y_{1,2}^{n-2}y_{1,3}^2y_{3,4}+\chi_{n+1}\chi_ry_{1,2}^{n-1}y_{1,3}y_{3,4}+\chi_n\chi_{r+1}y_{1,2}^{n-1}y_{1,3}y_{2,4}
   \\
   & \phantom{= \; }
   +\chi_{n+1}\chi_{r+1}y_{1,2}^{n-1}y_{2,3}y_{2,4},
   \\ 
   y_{2,3}^{n-r}y_{1,4}^r y_{1,2} & = \chi_n \chi_r y_{1,2}^{n-3}y_{1,3}^2y_{1,4}^2-\chi_{n+1}\chi_r y_{1,2}^{n-2}y_{1,3}y_{1,4}^2+\chi_n\chi_{r+1}y_{1,2}^{n-1}y_{1,3}y_{2,4}
   \\
   & \phantom{= \; }
   -\chi_{n+1}\chi_{r+1}y_{1,2}^{n-2}y_{1,3}^2y_{2,4}, 
   \\
   y_{2,3}^{n-r}y_{1,4}^r y_{1,3} & = \chi_n \chi_r y_{1,2}^{n-2}y_{1,3}y_{1,4}^2-\chi_{n+1}\chi_r y_{1,2}^{n-2}y_{2,3}y_{1,4}^2+\chi_n \chi_{r+1}y_{1,2}^{n-1}y_{2,3}y_{3,4}
   \\
   & \phantom{= \; }
   -\chi_{n+1}\chi_{r+1}y_{1,2}^{n-1}y_{1,3}y_{3,4}, 
   \\
   y_{2,3}^{n-r}y_{1,4}^r y_{2,4} & = \chi_n \chi_r y_{1,2}^{n-2}y_{1,3}^2y_{2,4}+\chi_{n+1}\chi_r y_{1,2}^{n-1}y_{2,3}y_{2,4}-\chi_n \chi_{r+1}y_{1,2}^{n-1}y_{1,3}y_{1,4}
   \\
   & \phantom{= \; }
   +\chi_{n+1}\chi_{r+1}y_{1,2}^{n-2}y_{1,3}^2y_{1,4}, 
   \\
   y_{2,3}^{n-r}y_{1,4}^r y_{3,4} & = \chi_n \chi_r y_{1,2}^{n-2}y_{1,3}^2y_{3,4}+\chi_{n+1}\chi_r y_{1,2}^{n-1}y_{2,3}y_{3,4}-\chi_n \chi_{r+1}y_{1,2}^{n-1}y_{2,3}y_{1,4}
   \\
   & \phantom{= \; }
   +\chi_{n+1}\chi_{r+1}y_{1,2}^{n-1}y_{1,3}y_{1,4}, 
\end{align*}
and 
\begin{align*}
y_{1,3}y_{1,2}^{n-r}y_{3,4}^r & = \chi_n \chi_r y_{1,2}^{n-2}y_{1,3}y_{1,4}^2
-\chi_n\chi_{r+1}y_{1,2}^{n-1}y_{2,3}y_{3,4}
-\chi_{n+1}\chi_r y_{1,2}^{n-2}y_{2,3}y_{1,4}^2
\\
& \phantom{= \; }
+\chi_{n+1}\chi_{r+1}y_{1,2}^{n-1}y_{1,3}y_{3,4},
\\
y_{2,4}y_{1,2}^{n-r}y_{3,4}^r & = \chi_n \chi_r y_{1,2}^{n-2}y_{1,3}^2y_{2,4} 
-\chi_n\chi_{r+1}y_{1,2}^{n-1}y_{1,3}y_{1,4} 
-\chi_{n+1}\chi_r y_{1,2}^{n-2}y_{1,3}^2y_{1,4}
\\
& \phantom{= \; }
+\chi_{n+1}\chi_{r+1} y_{1,2}^{n-1}y_{2,3}y_{2,4},
\\
y_{2,3}y_{1,2}^{n-r}y_{3,4}^r & = \chi_n \chi_r y_{1,2}^{n-2}y_{2,3}y_{1,4}^2 
-\chi_n \chi_{r+1}y_{1,2}^{n-1}y_{1,3}y_{3,4}
-\chi_{n+1}\chi_r y_{1,2}^{n-2}y_{1,3}y_{1,4}^2
\\
& \phantom{= \; }
+\chi_{n+1}\chi_{r+1} y_{1,2}^{n-1}y_{2,3}y_{3,4},
\\
y_{1,4}y_{1,2}^{n-r}y_{3,4}^r & = \chi_n \chi_r y_{1,2}^{n-2}y_{1,3}^2y_{1,4}
-\chi_n\chi_{r+1} y_{1,2}^{n-1}y_{2,3}y_{2,4}
-\chi_{n+1}\chi_r y_{1,2}^{n-2}y_{1,3}^2y_{2,4}
\\
& \phantom{= \; }
+\chi_{n+1}\chi_{r+1} y_{1,2}^{n-1}y_{1,3}y_{1,4},
\\ 
%%%%%%%%%%%%%%%
y_{1,2}y_{1,3}^{n-r}y_{2,4}^r & = \chi_n \chi_r y_{1,2}^{n-3}y_{1,3}^2y_{1,4}^2 
+\chi_n\chi_{r+1}y_{1,2}^{n-1}y_{1,3}y_{2,4} 
+\chi_{n+1}\chi_r y_{1,2}^{n-2}y_{1,3}y_{1,4}^2 
\\
& \phantom{= \; }
+\chi_{n+1}\chi_{r+1} y_{1,2}^{n-2}y_{1,3}^2y_{2,4}, 
\stepcounter{equation}\tag{\theequation}\label{eq:product y13y12ny34r}
\\
y_{3,4}y_{1,3}^{n-r}y_{2,4}^r & = \chi_n \chi_r y_{1,2}^{n-2}y_{1,3}^2y_{3,4} 
+\chi_n \chi_{r+1} y_{1,2}^{n-1}y_{2,3}y_{1,4} 
-\chi_{n+1}\chi_r y_{1,2}^{n-1}y_{1,3}y_{1,4}
\\
& \phantom{= \; }
-\chi_{n+1}\chi_{r+1} y_{1,2}^{n-1}y_{2,3}y_{3,4}, 
\\
y_{2,3}y_{1,3}^{n-r}y_{2,4}^r & = (-1)^n \chi_r y_{1,2}^{n-2} y_{2,3}y_{1,4}^2 
+(-1)^{n+1} \chi_{r+1} y_{1,2}^{n-1}y_{2,3}y_{2,4} ,
\\
y_{1,4}y_{1,3}^{n-r}y_{2,4}^r & = \chi_{n-r}  y_{1,2}^{n-2}y_{1,3}^2y_{1,4}
+(-1)^n\chi_{n-r+1}  y_{1,2}^{n-1}y_{1,3}y_{3,4}, 
\\
%%%%%%%%%%
y_{1,2} y_{2,3}^{n-r}y_{1,4}^r & = \chi_n \chi_r y_{1,2}^{n-3}y_{1,3}^2y_{1,4}^2 
+\chi_n \chi_{r+1} y_{1,2}^{n-1}y_{2,3}y_{1,4} 
+\chi_{n+1}\chi_r y_{1,2}^{n-2}y_{2,3}y_{1,4}^2 
\\
& \phantom{= \; }
+\chi_{n+1} \chi_{r+1} y_{1,2}^{n-2}y_{1,3}^2y_{1,4},  
\\
y_{3,4} y_{2,3}^{n-r}y_{1,4}^r & = \chi_n \chi_r y_{1,2}^{n-2}y_{1,3}^2 y_{3,4} 
-\chi_n\chi_{r+1} y_{1,2}^{n-1}y_{1,3}y_{2,4} 
-\chi_{n+1}\chi_r y_{1,2}^{n-1}y_{2,3}y_{2,4} 
\\
& \phantom{= \; } 
-\chi_{n+1}\chi_{r+1} y_{1,2}^{n-1}y_{1,3}y_{3,4}, 
\\
y_{1,3} y_{2,3}^{n-r}y_{1,4}^r & = \chi_r y_{1,2}^{n-2}y_{1,3}y_{1,4}^2 
+\chi_{r+1}y_{1,2}^{n-1}y_{1,3}y_{1,4}, 
\\
y_{2,4} y_{2,3}^{n-r}y_{1,4}^r & = (-1)^{n}\chi_{n-r} y_{1,2}^{n-2}y_{1,3}^2y_{2,4} -\chi_{n-r+1} y_{1,2}^{n-1}y_{2,3}y_{3,4}, 
\end{align*}
together with 
\begin{align*}
y_{1,3}y_{1,2}^n & =\chi_n y_{1,2}^ny_{1,3}-\chi_{n+1}y_{1,2}^ny_{2,3},
\quad &
y_{1,3}y_{3,4}^n & = \chi_n y_{1,3}^{n-1}y_{1,4}^2+\chi_{n+1}y_{1,3}^n y_{3,4}, 
\\
y_{2,4}y_{1,2}^n & = \chi_n y_{1,2}^ny_{2,4}-\chi_{n+1} y_{1,2}^ny_{1,4}, 
\quad & 
y_{2,4}y_{3,4}^n & =  y_{2,3}^n y_{2,4}  , 
\\
y_{2,3}y_{1,2}^n & = \chi_n y_{1,2}^n y_{2,3}-\chi_{n+1} y_{1,2}^n y_{1,3},
\quad & 
y_{2,3}y_{3,4}^n & = \chi_n y_{2,3}^{n-1}y_{2,4}^2 +\chi_{n+1} y_{2,3}^n y_{3,4} , 
\\
y_{1,4}y_{1,2}^n & = \chi_n y_{1,2}^ny_{1,4} -\chi_{n+1} y_{1,2}^n y_{2,4} ,
\quad & 
y_{1,4}y_{3,4}^n & =  y_{1,3}^n y_{1,4}  , 
\\
%%%%%%%%%%%%
y_{1,2}y_{1,3}^n & = \chi_n y_{1,2}^{n-1}y_{1,3}^2 +\chi_{n+1}y_{1,2}^ny_{1,3} ,
\quad & 
y_{1,2}y_{2,4}^n & = \chi_n y_{1,2}^{n-1}y_{1,4}^2 +\chi_{n+1} y_{1,2}^ny_{2,4}, 
\\
y_{3,4}y_{1,3}^n & = \chi_n y_{1,3}^n y_{3,4} -\chi_{n+1}y_{1,3}^n y_{1,4} , 
\quad & 
y_{3,4}y_{2,4}^n & = (-1)^n y_{2,3}^n y_{3,4}  , 
\stepcounter{equation}\tag{\theequation}\label{eq:product y13y12n} 
\\
y_{2,3}y_{1,3}^n & = (-1)^n y_{1,2}^n y_{2,3} , 
\quad & 
y_{2,3}y_{2,4}^n & = \chi_n y_{2,3}^{n-1}y_{2,4}^2 +\chi_{n+1} y_{2,3}^n y_{2,4}, 
\\
y_{1,4}y_{1,3}^n & = \chi_n y_{1,3}^n y_{1,4}-\chi_{n+1} y_{1,3}^ny_{3,4}, 
\quad & 
y_{1,4}y_{2,4}^n & = y_{1,2}^n y_{1,4}, 
\\
%%%%%%%%%%%%%%%
y_{1,2}y_{2,3}^n & =\chi_n y_{1,2}^{n-1}y_{1,3}^2+\chi_{n+1}y_{1,2}^n y_{2,3}, 
\quad & 
y_{1,2}y_{1,4}^n & = \chi_n y_{1,2}^{n-1}y_{1,4}^2+\chi_{n+1}y_{1,2}^n y_{1,4}, 
\\
y_{3,4}y_{2,3}^n & = \chi_n y_{2,3}^n y_{3,4} -\chi_{n+1} y_{2,3}^n y_{2,4}, 
\quad & 
y_{3,4} y_{1,4}^n & = (-1)^n  y_{1,3}^n y_{3,4},  
\\
y_{1,3}y_{2,3}^n & = y_{1,2}^n y_{1,3}, 
\quad & 
y_{1,3}y_{1,4}^n & = \chi_n y_{1,3}^{n-1}y_{1,4}^2+\chi_{n+1} y_{1,3}^n y_{1,4}, 
\\
y_{2,4}y_{2,3}^n & = \chi_n y_{2,3}^n y_{2,4}-\chi_{n+1} y_{2,3}^n y_{3,4}, 
\quad & 
y_{2,4} y_{1,4}^n & = (-1)^n y_{1,2}^n y_{2,4}, 
\end{align*}
in $A^{!}$, for all integers $n\geqslant 2$ and $r\in \llbracket 1,n-1\rrbracket $. 
\end{fact}
%%%%%%%

We also provided further identities in $A^!$ given in Appendix \ref{sec:products FK4}, which are also straightforward to verify.

Recall that the graded dual $(A^!)^{\#}=\oplus_{n\in\NN_0}(A_{-n}^!)^*$ is a graded bimodule over $A^!$ via the identity $(ufv)(w)=f(vwu)$ for $u,v,w \in A^!$ and $f\in (A^!)^{\#}$. 
Let $\B_{n}^{!*}$ be the dual basis to the basis $\B_{n}^{!}$ for $n\in\NN_0$. 
We write $\B_{0}^{!*}=\{ \epsilon^!\} $ 
and 
$z^{i_1,j_1}_{n_1}\dots z^{i_r,j_r}_{n_r}=(y_{i_1,j_1}^{n_1}\dots y_{i_r,j_r}^{n_r})^{*}\in \B_{n}^{!*}$ for $y_{i_1,j_1}^{n_1}\dots y_{i_r,j_r}^{n_r}\in \B_{n}^{!}$, 
where 
$n = n_1+\dots+n_r$, 
$n,r, n_1,\dots,n_r\in\NN$
and 
$(i_1,j_1),\dots, (i_r,j_r)\in\I$.
We will omit the index $n_j$ for $j\in\llbracket 1,r\rrbracket $ if $n_j=1$ in the element $z^{i_1,j_1}_{n_1}\dots z^{i_r,j_r}_{n_r}$ or $y_{i_1,j_1}^{n_1}\dots y_{i_r,j_r}^{n_r}$.
Obviously, $y_{i,j}z^{i,j}=\epsilon^!$ for $(i,j)\in \I$ and the other actions of $\B^!_1$ on $\B^{!*}_1$ vanish.

Recall that $(K_{\bullet}(A),d_{\bullet})$ denotes the Koszul complex of $A$ in the category of bounded below graded (right) $A$-modules and $\epsilon:K_0(A) \to \Bbbk$ is the canonical projection. 
The differential $d_n:K_n(A) \to K_{n-1}(A)$ for $n\in\NN$ is given by the multiplication of $\sum_{(i,j)\in\I}y_{i,j}\otimes x_{i,j}$ on the left.
%As usual, we can consider the Koszul complex as a complex indexed by $\ZZ$, with $K_{n} = 0$ for all $n \in \ZZ_{\leqslant -1}$, and $d_{n} = 0$ for all $n \in \ZZ_{\leqslant 0}$. 
To reduce space, we will simply write $K_{n}$ instead of $K_{n}(A)$ for $n \in \NN_{0}$ and we will typically use vertical bars instead of the tensor
product symbols $\otimes$.

The differential $d_\bullet$ of the Koszul complex of $A$ can be explicitly described in the following result. 
Its proof is a straightforward but lengthy verification, using the identities listed in Fact \ref{fact:products-quadratic-dual-fk4} and in Appendix \ref{sec:products FK4}. 

%%%%%%%
\begin{fact}
Let $d_n : K_{n} \rightarrow K_{n-1}$ be the differential of the Koszul complex of $A$ for $n \in \NN$. 
It can be explicitly described as follows. 
First, $d_1(z^{i,j}|1)=\epsilon^!|x_{i,j}$ for $(i,j)\in \I$, 
and 
 \begin{equation}
   \label{eq:differential 4}
   \begin{split}
    d_n(z^{i,j}_{n-r}z^{k,l}_r|1)= (-1)^r z^{i,j}_{n-r-1}z^{k,l}_{r}|x_{i,j}+z^{i,j}_{n-r}z^{k,l}_{r-1}|x_{k,l},
   \end{split}
\end{equation}
for $n\geqslant 2 $, $r\in \llbracket 0,n\rrbracket $, $(i,j)\in \II$, $(k,l)\in \I$ with $\# \{i,j,k,l \}=4$,
where we follow the convention that $z^{i,j}_nz^{k,l}_0=z^{i,j}_n$, $z^{i,j}_0z^{k,l}_n=z^{k,l}_n$, $z^{i,j}_nz^{k,l}_{-1}=0$ and $z^{i,j}_{-1}z^{k,l}_n=0$ for $n\in\NN$. 
Moreover, 
for $n\geqslant 5$, the differential $d_{n+1}$ is given by \eqref{eq:differential 4} and 
\begin{small}
\begin{align*}
   z^{1,2}_nz^{1,3}|1 & \mapsto -(z^{1,2}_{n-1}z^{2,3}+\chi_{n+1}z^{2,3}_n)|x_{1,2}+ (z^{1,2}_n +z^{1,2}_{n-2}z^{1,3}_2+\chi_n z^{2,3}_n )|x_{1,3}
   \\
   & \phantom{ \mapsto \; }
   + (z^{1,2}_{n-1}z^{1,3} +\chi_{n+1}z^{1,3}_n )|x_{2,3},
   \\
   %%%%%%%%%%%%%%%%%
   z^{1,2}_nz^{2,3}|1 & \mapsto - \{ z^{1,2}_{n-1}z^{1,3} + \chi_{n+1} z^{1,3}_n \} |x_{1,2}
   - \{ z^{1,2}_{n-1}z^{2,3}+\chi_{n+1}z^{2,3}_n \} |x_{1,3}
   \\
   & \phantom{ \mapsto \; }
   + \{ z^{1,2}_n +z^{1,2}_{n-2}z^{1,3}_2 +\chi_n z^{1,3}_n \} |x_{2,3} ,
   \\
   %%%%%%%%%%%%%%%%%%%
   z^{1,2}_nz^{1,4}|1 & \mapsto - \{ z^{1,2}_{n-1}z^{2,4}+\chi_{n+1}z^{2,4}_n \}|x_{1,2}
   + \{ z^{1,2}_n + z^{1,2}_{n-2}z^{1,4}_2 +\chi_n z^{2,4}_n \} |x_{1,4} 
   \\
   & \phantom{ \mapsto \; }
   + \{ z^{1,2}_{n-1}z^{1,4} + \chi_{n+1} z^{1,4}_n \} |x_{2,4} 
   , 
   \\
   %%%%%%%%%%%%%%%%%%
   z^{1,2}_nz^{2,4}|1 & \mapsto - \{ z^{1,2}_{n-1}z^{1,4}+\chi_{n+1}z^{1,4}_n \} |x_{1,2}
   - \{ z^{1,2}_{n-1}z^{2,4}+\chi_{n+1} z^{2,4}_n \} |x_{1,4} 
   \\
   & \phantom{ \mapsto \; }
   + \{ z^{1,2}_n +z^{1,2}_{n-2}z^{1,4}_2 + \chi_n z^{1,4}_n \} |x_{2,4} 
   , 
   \\
   %%%%%%%%%%%%%%%%%%
   z^{1,2}_{n-1}z^{1,3}_2 |1 & \mapsto \big\{ z^{1,2}_{n-2}z^{1,3}_2 +\chi_n (z^{1,3}_n +z^{2,3}_n ) \big\}|x_{1,2} + z^{1,2}_{n-1}z^{1,3}|x_{1,3} 
   +z^{1,2}_{n-1}z^{2,3}|x_{2,3} 
   ,
   \\
   %%%%%%%%%%%%%%%%%%%%
   z^{1,2}_{n-1}z^{1,3}z^{1,4}|1 & \mapsto \{ z^{1,2}_{n-2}z^{2,3}z^{2,4} + \chi_n z^{2,3}_{n-1}z^{2,4} \} |x_{1,2} 
   \\
   & \phantom{ \mapsto \; } 
   - \bigg\{ z^{1,2}_{n-3}z^{1,3}_2z^{3,4} +\chi_{n+1} z^{2,3}_{n-1}z^{3,4} 
   + \sum_{s=1}^{\lfloor \frac{n}{2}\rfloor  } z^{1,2}_{n-2s+1}z^{3,4}_{2s-1} 
   \bigg\} |x_{1,3} 
   \\
   & \phantom{ \mapsto \; }
   - \{ z^{1,2}_{n-2}z^{1,3}z^{1,4} + \chi_n z^{1,3}_{n-1}z^{1,4} \} |x_{2,3} 
   \\
   & \phantom{ \mapsto \; } 
   + \bigg\{ 
   z^{1,2}_{n-1}z^{1,3} +z^{1,2}_{n-3}z^{1,3}z^{1,4}_2 
   + \chi_{n+1} \sum_{s=1}^{\frac{n-1}{2}}z^{1,3}_{n-2s}z^{2,4}_{2s} 
   \bigg\}|x_{1,4}
   \\
   & \phantom{ \mapsto \; }
   - \bigg\{ z^{1,2}_{n-2}z^{2,3}z^{1,4}
   +\chi_n \sum_{s=1}^{\frac{n}{2}}z^{2,3}_{n-2s+1}z^{1,4}_{2s-1} 
   \bigg\}|x_{2,4} 
   \\
   & \phantom{ \mapsto \; }
   + \bigg\{ z^{1,2}_{n-1}z^{1,4} + z^{1,2}_{n-3}z^{1,3}_2z^{1,4} 
   +\chi_{n+1} \sum_{s=1}^{\frac{n-1}{2}}z^{2,3}_{n-2s+1}z^{1,4}_{2s-1} 
   \bigg\} |x_{3,4} 
   , 
   \\
   %%%%%%%%%%%%%%%%%%
   z^{1,2}_{n-1}z^{1,3}z^{2,4}|1 & \mapsto 
   \bigg\{ z^{1,2}_{n-2}z^{2,3}z^{1,4}
   +\chi_n \sum_{s=1}^{\frac{n}{2}} z^{2,3}_{n-2s+1}z^{1,4}_{2s-1}
   \bigg\}|x_{1,2} 
   \\
   & \phantom{ \mapsto \; }
   - \{ z^{1,2}_{n-1}z^{2,4} +z^{1,2}_{n-3}z^{1,3}_2z^{2,4} +\chi_{n+1}z^{2,3}_{n-1}z^{2,4} \}|x_{1,3}
   - \{ z^{1,2}_{n-2}z^{1,3}z^{3,4}+ \chi_n z^{1,3}_{n-1}z^{3,4} \} |x_{2,3} 
   \\
   & \phantom{ \mapsto \; }
   + \{ z^{1,2}_{n-2}z^{2,3}z^{2,4} +\chi_n z^{2,3}_{n-1}z^{2,4} \} |x_{1,4} 
   + \{ z^{1,2}_{n-1}z^{1,3} +z^{1,2}_{n-3}z^{1,3}z^{1,4}_2 +\chi_{n+1} z^{1,3}_{n-2}z^{1,4}_2 \} |x_{2,4} 
   \\
   & \phantom{ \mapsto \; }
   + \bigg\{ z^{1,2}_{n-2}z^{1,3}z^{2,4} 
   +\chi_n  \sum_{s=1}^{\frac{n}{2}}z^{1,3}_{n-2s+1}z^{2,4}_{2s-1}
   \bigg\}|x_{3,4} 
   , 
   \\
   %%%%%%%%%%%%%%%%%%
   z^{1,2}_{n-1}z^{1,3}z^{3,4}|1 & \mapsto 
   \{ z^{1,2}_{n-2}z^{2,3}z^{3,4}+\chi_n z^{2,3}_{n-1}z^{3,4} \} |x_{1,2} 
   \\
   & \phantom{ \mapsto \; }
   - \bigg\{ z^{1,2}_{n-1}z^{1,4}+z^{1,2}_{n-3}z^{1,3}_2z^{1,4} 
   +\chi_{n+1} \sum_{s=1}^{\frac{n-1}{2}}z^{2,3}_{n-2s+1}z^{1,4}_{2s-1}
   \bigg\} |x_{1,3} 
   \\
   & \phantom{ \mapsto \; }
   - \bigg\{ 
      z^{1,2}_{n-2}z^{1,3}z^{2,4}
   + \chi_n \sum_{s=1}^{\frac{n}{2}}z^{1,3}_{n-2s+1}z^{2,4}_{2s-1}
   \bigg\} |x_{2,3} 
   \\
   & \phantom{ \mapsto \; }
   - \bigg\{ z^{1,2}_{n-3}z^{1,3}_2z^{3,4} +\chi_{n+1}z^{2,3}_{n-1}z^{3,4} 
   + \sum_{s=1}^{\lfloor \frac{n}{2}\rfloor  } z^{1,2}_{n-2s+1}z^{3,4}_{2s-1} 
   \bigg\} |x_{1,4} 
   \\
   & \phantom{ \mapsto \; }
   - \{ z^{1,2}_{n-2}z^{1,3}z^{3,4} +\chi_n z^{1,3}_{n-1}z^{3,4} \}|x_{2,4} 
   \\
   & \phantom{ \mapsto \; }
   + \bigg\{ z^{1,2}_{n-1}z^{1,3} + z^{1,2}_{n-3}z^{1,3}z^{1,4}_2 
   +\chi_{n+1} \sum_{s=1}^{\frac{n-1}{2}} z^{1,3}_{n-2s}z^{2,4}_{2s} 
   \bigg\} |x_{3,4} 
   , 
   \\
   %%%%%%%%%%%%%%%%%%
   z^{1,2}_{n-1}z^{2,3}z^{1,4}|1 & \mapsto 
   \bigg\{ 
   z^{1,2}_{n-2}z^{1,3}z^{2,4}
   +\chi_n \sum_{s=1}^{\frac{n}{2}}z^{1,3}_{n-2s+1}z^{2,4}_{2s-1}
   \bigg\} |x_{1,2} 
   + \{ z^{1,2}_{n-2}z^{2,3}z^{3,4} +\chi_n z^{2,3}_{n-1}z^{3,4} \} |x_{1,3} 
   \\
   & \phantom{ \mapsto \; }
   - \{ z^{1,2}_{n-1}z^{1,4} +z^{1,2}_{n-3}z^{1,3}_2z^{1,4} +\chi_{n+1}z^{1,3}_{n-1}z^{1,4} \} |x_{2,3} 
   \\
   & \phantom{ \mapsto \; }
   + \{ z^{1,2}_{n-1}z^{2,3} +z^{1,2}_{n-3}z^{2,3}z^{1,4}_2 +\chi_{n+1} z^{2,3}_{n-2}z^{2,4}_2 \} |x_{1,4} 
   - \{ z^{1,2}_{n-2}z^{1,3}z^{1,4}+\chi_n z^{1,3}_{n-1}z^{1,4} \}|x_{2,4} 
   \\
   & \phantom{ \mapsto \; } 
   - \bigg\{ z^{1,2}_{n-2}z^{2,3}z^{1,4}
   +\chi_n \sum_{s=1}^{\frac{n}{2}}z^{2,3}_{n-2s+1}z^{1,4}_{2s-1}
   \bigg\}|x_{3,4} 
   , 
   \\
   %%%%%%%%%%%%%%%%%%
   z^{1,2}_{n-1}z^{2,3}z^{2,4}|1 & \mapsto 
   \{ z^{1,2}_{n-2}z^{1,3}z^{1,4}+\chi_n z^{1,3}_{n-1}z^{1,4} \} |x_{1,2} 
   + \{ z^{1,2}_{n-2}z^{2,3}z^{2,4} +\chi_n z^{2,3}_{n-1}z^{2,4} \} |x_{1,3} 
   \\
   & \phantom{ \mapsto \; }
   - \bigg\{ z^{1,2}_{n-3}z^{1,3}_2z^{3,4} +\chi_{n+1}z^{1,3}_{n-1}z^{3,4} 
   + \sum_{s=1}^{\lfloor \frac{n}{2}\rfloor  } z^{1,2}_{n-2s+1}z^{3,4}_{2s-1} 
   \bigg\} |x_{2,3} 
   \\
   & \phantom{ \mapsto \; }
   + \bigg\{ z^{1,2}_{n-2}z^{1,3}z^{2,4}
   +\chi_n  \sum_{s=1}^{\frac{n}{2}} z^{1,3}_{n-2s+1}z^{2,4}_{2s-1} 
   \bigg\} |x_{1,4} 
   \\
   & \phantom{ \mapsto \; }
   + \bigg\{ z^{1,2}_{n-1}z^{2,3} +z^{1,2}_{n-3}z^{2,3}z^{1,4}_2 
   +\chi_{n+1} \sum_{s=1}^{\frac{n-1}{2}}z^{2,3}_{n-2s}z^{1,4}_{2s}
   \bigg\} |x_{2,4} 
   \\
   & \phantom{ \mapsto \; }
   +\bigg\{ z^{1,2}_{n-1}z^{2,4} + z^{1,2}_{n-3}z^{1,3}_2z^{2,4} 
   +\chi_{n+1} \sum_{s=1}^{\frac{n-1}{2}} z^{1,3}_{n-2s+1}z^{2,4}_{2s-1} 
   \bigg\} |x_{3,4} 
   , 
   \\
   %%%%%%%%%%%%%%%%%%
   z^{1,2}_{n-1}z^{2,3}z^{3,4}|1 & \mapsto \{ z^{1,2}_{n-2}z^{1,3}z^{3,4}+\chi_n z^{1,3}_{n-1}z^{3,4} \} |x_{1,2} 
   + \bigg\{ z^{1,2}_{n-2}z^{2,3}z^{1,4} 
   +\chi_n \sum_{s=1}^{\frac{n}{2}}z^{2,3}_{n-2s+1}z^{1,4}_{2s-1}
   \bigg\}|x_{1,3} 
   \\
   & \phantom{\mapsto \;}
   - \bigg\{ z^{1,2}_{n-1}z^{2,4} +z^{1,2}_{n-3}z^{1,3}_2z^{2,4}  
   +\chi_{n+1} \sum_{s=1}^{\frac{n-1}{2}}z^{1,3}_{n-2s+1}z^{2,4}_{2s-1} 
   \bigg\} |x_{2,3} 
   \\
   & \phantom{\mapsto \;}
   + \{ z^{1,2}_{n-2}z^{2,3}z^{3,4} +\chi_n z^{2,3}_{n-1}z^{3,4} \} |x_{1,4} 
   \\
   & \phantom{\mapsto \;}
   - \bigg\{ z^{1,2}_{n-3}z^{1,3}_2z^{3,4} + \chi_{n+1} z^{1,3}_{n-1}z^{3,4} 
   + \sum_{s=1}^{\lfloor \frac{n}{2}\rfloor  } z^{1,2}_{n-2s+1}z^{3,4}_{2s-1} 
   \bigg\} |x_{2,4} 
   \\
   & \phantom{\mapsto \;}
   + \bigg\{ z^{1,2}_{n-1}z^{2,3} + z^{1,2}_{n-3}z^{2,3}z^{1,4}_2 
   +\chi_{n+1} \sum_{s=1}^{\frac{n-1}{2}}z^{2,3}_{n-2s}z^{1,4}_{2s}
   \bigg\} |x_{3,4} 
   , 
   \\
   %%%%%%%%%%%%%%%%%%
   z^{1,2}_{n-1}z^{1,4}_{2}|1 & \mapsto
   \big\{ z^{1,2}_{n-2}z^{1,4}_{2} +\chi_n ( z^{1,4}_n+z^{2,4}_n ) \big\} |x_{1,2} 
   +z^{1,2}_{n-1}z^{1,4}|x_{1,4} 
   +z^{1,2}_{n-1}z^{2,4}|x_{2,4} 
   ,
   \stepcounter{equation}\tag{\theequation}\label{eq:dn}
   \\
   %%%%%%%%%%%%%%%%%%
   z^{1,2}_{n-2}z^{1,3}_2z^{1,4}|1 & \mapsto 
   - \bigg\{ z^{1,2}_{n-3}z^{1,3}_2z^{2,4} + \chi_{n+1}\bigg( z^{2,3}_{n-1}z^{2,4} 
   + \sum_{s=1}^{\frac{n-1}{2}}z^{1,3}_{n-2s+1}z^{2,4}_{2s-1}
   \bigg)
   \bigg\} |x_{1,2} 
   -z^{1,2}_{n-2}z^{1,3}z^{3,4}|x_{1,3} 
   \\
   & \phantom{ \mapsto \; }
   -z^{1,2}_{n-2}z^{2,3}z^{1,4}|x_{2,3} 
   \\
   & \phantom{ \mapsto \; }
   + \bigg\{ z^{1,2}_{n-2}z^{1,3}_2 +z^{1,2}_{n-4}z^{1,3}_2z^{1,4}_2
   +\chi_n 
   \bigg(
   z^{2,3}_{n-2}z^{2,4}_2 
   +\sum_{s=1}^{\frac{n-2}{2}}z^{1,3}_{n-2s}z^{2,4}_{2s} 
   \bigg)
   + \sum_{s=1}^{\lfloor \frac{n-1}{2}\rfloor  } z^{1,2}_{n-2s}z^{3,4}_{2s} 
   \bigg\} |x_{1,4} 
   \\
   & \phantom{ \mapsto \; }
   + \bigg\{ z^{1,2}_{n-3}z^{1,3}_2z^{1,4} +\chi_{n+1} \bigg( z^{1,3}_{n-1}z^{1,4} 
   + \sum_{s=1}^{\frac{n-1}{2}}z^{2,3}_{n-2s+1}z^{1,4}_{2s-1}
   \bigg) \bigg\} |x_{2,4} 
   +z^{1,2}_{n-2}z^{1,3}z^{1,4}|x_{3,4} 
   , 
   \\
   %%%%%%%%%%%%%%%%%%
   z^{1,2}_{n-2}z^{1,3}_2z^{2,4}|1 & \mapsto 
   - \bigg\{ z^{1,2}_{n-3}z^{1,3}_2z^{1,4}+ \chi_{n+1} \bigg( z^{1,3}_{n-1}z^{1,4} 
   + \sum_{s=1}^{\frac{n-1}{2}} z^{2,3}_{n-2s+1}z^{1,4}_{2s-1} 
   \bigg)\bigg\} |x_{1,2} 
   -z^{1,2}_{n-2}z^{1,3}z^{2,4}|x_{1,3} 
   \\
   & \phantom{ \mapsto \; }
   -z^{1,2}_{n-2}z^{2,3}z^{3,4}|x_{2,3}
   - \bigg\{ z^{1,2}_{n-3}z^{1,3}_2z^{2,4} +\chi_{n+1} \bigg( z^{2,3}_{n-1}z^{2,4} 
   +\sum_{s=1}^{\frac{n-1}{2}} z^{1,3}_{n-2s+1}z^{2,4}_{2s-1} 
   \bigg)\bigg\} |x_{1,4}
   \\
   & \phantom{ \mapsto \; } 
   + \bigg\{ z^{1,2}_{n-2}z^{1,3}_2 +z^{1,2}_{n-4}z^{1,3}_2z^{1,4}_2 
   + \chi_n \bigg( z^{1,3}_{n-2}z^{1,4}_2 
   +\sum_{s=1}^{\frac{n-2}{2}}z^{2,3}_{n-2s}z^{1,4}_{2s}
   \bigg)
   + \sum_{s=1}^{\lfloor \frac{n-1}{2}\rfloor  } z^{1,2}_{n-2s}z^{3,4}_{2s} 
   \bigg\} |x_{2,4} 
   \\
   & \phantom{ \mapsto \; } 
   + z^{1,2}_{n-2}z^{2,3}z^{2,4}|x_{3,4} 
   , 
   \\
   %%%%%%%%%%%%%%%%%%
   z^{1,2}_{n-2}z^{1,3}_2z^{3,4}|1 & \mapsto 
   - \big\{ z^{1,2}_{n-3}z^{1,3}_2z^{3,4}+\chi_{n+1}(z^{1,3}_{n-1}z^{3,4} + z^{2,3}_{n-1}z^{3,4} ) \big\} |x_{1,2}
   -z^{1,2}_{n-2}z^{1,3}z^{1,4}|x_{1,3} 
   \\
   & \phantom{ \mapsto \; }
   -z^{1,2}_{n-2}z^{2,3}z^{2,4}|x_{2,3} 
   -z^{1,2}_{n-2}z^{1,3}z^{3,4}|x_{1,4} 
   -z^{1,2}_{n-2}z^{2,3}z^{3,4}|x_{2,4}
   \\
   & \phantom{ \mapsto \; }
   + \bigg\{ z^{1,2}_{n-2}z^{1,3}_2 + z^{1,2}_{n-2}z^{1,4}_2 + z^{1,2}_{n-4}z^{1,3}_2z^{1,4}_2
   + \chi_n  \sum_{s=1}^{\frac{n-2}{2}} ( z^{1,3}_{n-2s}z^{2,4}_{2s} 
   + z^{2,3}_{n-2s}z^{1,4}_{2s} )
   \bigg\} |x_{3,4} 
   , 
   \\
   %%%%%%%%%%%%%%%%%%
   z^{1,2}_{n-2}z^{1,3}z^{1,4}_2|1 & \mapsto -\bigg\{z^{1,2}_{n-3}z^{2,3}z^{1,4}_2+\chi_{n+1} \bigg( z^{2,3}_{n-2}z^{2,4}_2 +\sum_{s=1}^{\frac{n-1}{2}}z^{2,3}_{n-2s}z^{1,4}_{2s} \bigg) \bigg\} |x_{1,2}
   \\
   & \phantom{ \mapsto \; }
   + \bigg\{ z^{1,2}_{n-2}z^{1,4}_2+z^{1,2}_{n-4}z^{1,3}_2z^{1,4}_2+\chi_n \bigg( z^{2,3}_{n-2}z^{2,4}_2 +\sum_{s=1}^{\frac{n-2}{2}}z^{2,3}_{n-2s}z^{1,4}_{2s} \bigg) + \sum_{s=1}^{\lfloor \frac{n-1}{2}\rfloor  } z^{1,2}_{n-2s}z^{3,4}_{2s} \bigg\}|x_{1,3} 
   \\
   & \phantom{ \mapsto \; }
   + \bigg\{ z^{1,2}_{n-3}z^{1,3}z^{1,4}_2+\chi_{n+1}\bigg( z^{1,3}_{n-2}z^{1,4}_2 +\sum_{s=1}^{\frac{n-1}{2}}z^{1,3}_{n-2s}z^{2,4}_{2s} \bigg)  \bigg\}|x_{2,3}
   \\
   & \phantom{ \mapsto \; }
   + z^{1,2}_{n-2}z^{1,3}z^{1,4}|x_{1,4} + z^{1,2}_{n-2}z^{1,3}z^{2,4}|x_{2,4}
   + z^{1,2}_{n-2}z^{1,3}z^{3,4}|x_{3,4} , 
   \\
   %%%%%%%%%%%%%%%%%
   z^{1,2}_{n-2}z^{2,3}z^{1,4}_2|1 & \mapsto 
   - \bigg\{ z^{1,2}_{n-3}z^{1,3}z^{1,4}_2 +\chi_{n+1} \bigg( 
   z^{1,3}_{n-2}z^{1,4}_2 
   + \sum_{s=1}^{\frac{n-1}{2}}z^{1,3}_{n-2s}z^{2,4}_{2s}
   \bigg)
   \bigg\} |x_{1,2} 
   \\
   & \phantom{ \mapsto \; } 
   - \bigg\{ z^{1,2}_{n-3}z^{2,3}z^{1,4}_2 +\chi_{n+1} \bigg( z^{2,3}_{n-2}z^{2,4}_2 
   +\sum_{s=1}^{\frac{n-1}{2}}z^{2,3}_{n-2s}z^{1,4}_{2s}
   \bigg)\bigg\} |x_{1,3} 
   \\
   & \phantom{ \mapsto \; }
   + \bigg\{ z^{1,2}_{n-2}z^{1,4}_2 +z^{1,2}_{n-4}z^{1,3}_2z^{1,4}_2 +
   \chi_n 
   \bigg( z^{1,3}_{n-2}z^{1,4}_2 
   +  \sum_{s=1}^{\frac{n-2}{2}}z^{1,3}_{n-2s}z^{2,4}_{2s}
   \bigg)
   +  \sum_{s=1}^{\lfloor \frac{n-1}{2}\rfloor  } z^{1,2}_{n-2s}z^{3,4}_{2s} 
   \bigg\} |x_{2,3}
   \\
   & \phantom{ \mapsto \; } 
   +z^{1,2}_{n-2}z^{2,3}z^{1,4}|x_{1,4}  
   +z^{1,2}_{n-2}z^{2,3}z^{2,4}|x_{2,4}
   + z^{1,2}_{n-2}z^{2,3}z^{3,4}|x_{3,4}  
   , 
   \\
   %%%%%%%%%%%%%%%%%%
   z^{1,2}_{n-3}z^{1,3}_{2}z^{1,4}_2|1 & \mapsto \bigg\{ z^{1,2}_{n-4}z^{1,3}_2z^{1,4}_2+\chi_n \bigg(z^{1,3}_{n-2}z^{1,4}_2+z^{2,3}_{n-2}z^{2,4}_2+\sum_{s=1}^{\frac{n-2}{2}}(z^{1,3}_{n-2s}z^{2,4}_{2s}+z^{2,3}_{n-2s}z^{1,4}_{2s}) \bigg) \bigg\}|x_{1,2}
   \\
   & \phantom{ \mapsto \; }
   +z^{1,2}_{n-3}z^{1,3}z^{1,4}_2|x_{1,3}+z^{1,2}_{n-3}z^{2,3}z^{1,4}_2|x_{2,3}+z^{1,2}_{n-3}z^{1,3}_2z^{1,4}|x_{1,4}+z^{1,2}_{n-3}z^{1,3}_2z^{2,4}|x_{2,4}
   \\
   & \phantom{ \mapsto \; }
   +z^{1,2}_{n-3}z^{1,3}_2z^{3,4}|x_{3,4},
   \\
   %%%%%%%%%%%%%%
   z^{1,3}_{n}z^{1,4}|1 & \mapsto 
   - \{ z^{1,3}_{n-1}z^{3,4} +\chi_{n+1} z^{3,4}_n \}|x_{1,3} 
   + \{ z^{1,3}_n +z^{1,3}_{n-2}z^{1,4}_2 +\chi_n z^{3,4}_n \} |x_{1,4} 
   \\
   & \phantom{ \mapsto \; } 
   + \{ z^{1,3}_{n-1}z^{1,4}+\chi_{n+1}z^{1,4}_n \}|x_{3,4} 
   , 
   \\
   %%%%%%%%%%%%%%%%%%
   z^{1,3}_{n}z^{3,4}|1 & \mapsto 
   - \{ z^{1,3}_{n-1}z^{1,4}+\chi_{n+1} z^{1,4}_n \} |x_{1,3} 
   - \{ z^{1,3}_{n-1}z^{3,4}+\chi_{n+1} z^{3,4}_n \}|x_{1,4} 
   \\
   & \phantom{ \mapsto \; }
   + \{ z^{1,3}_n + z^{1,3}_{n-2}z^{1,4}_2 + \chi_n z^{1,4}_n \}|x_{3,4} 
   , 
   \\
   %%%%%%%%%%%%%%%%%%
   z^{1,3}_{n-1}z^{1,4}_2|1 & \mapsto 
   \big\{ z^{1,3}_{n-2}z^{1,4}_2 +\chi_n ( z^{1,4}_n + z^{3,4}_n ) \big\} |x_{1,3} 
   +z^{1,3}_{n-1}z^{1,4}|x_{1,4} 
   +z^{1,3}_{n-1}z^{3,4}|x_{3,4} 
   , 
   \\
   %%%%%%%%%%%%%%%%%%
   z^{2,3}_{n}z^{2,4}|1 & \mapsto
   - \{ z^{2,3}_{n-1}z^{3,4} +\chi_{n+1}z^{3,4}_n\}|x_{2,3} 
   + \{ z^{2,3}_n + z^{2,3}_{n-2}z^{2,4}_2 + \chi_n z^{3,4}_n \} |x_{2,4} 
   \\
   & \phantom{ \mapsto \; } 
   + \{ z^{2,3}_{n-1}z^{2,4} + \chi_{n+1}z^{2,4}_n \} |x_{3,4}
   , 
   \\
   %%%%%%%%%%%%%%%%%%
   z^{2,3}_{n}z^{3,4}|1 & \mapsto 
   - \{ z^{2,3}_{n-1}z^{2,4} +\chi_{n+1}z^{2,4}_n \} |x_{2,3} 
   - \{ z^{2,3}_{n-1}z^{3,4} + \chi_{n+1} z^{3,4}_n \} |x_{2,4} 
   \\
   & \phantom{ \mapsto \; }
   + \{ z^{2,3}_n + z^{2,3}_{n-2}z^{2,4}_2 + \chi_n z^{2,4}_n \} |x_{3,4} 
   , 
   \\
   %%%%%%%%%%%%%%%%%%
   z^{2,3}_{n-1}z^{2,4}_2|1 & \mapsto
   \big\{ z^{2,3}_{n-2}z^{2,4}_2 +\chi_n ( z^{2,4}_n+z^{3,4}_n ) \big\}|x_{2,3}
   +z^{2,3}_{n-1}z^{2,4}|x_{2,4} 
   + z^{2,3}_{n-1}z^{3,4}|x_{3,4} 
   .
\end{align*}
\end{small}
%%%%%%%%%%%
\end{fact}
%%%%%%%%%%%

%%%%%%%%%%%%%%%%%%%%%%%%%%%%%%%%%%%%%%%%%%%%%%%%%%%%%%%%%%%%%%%%%%%%%%%%%%%%%%%%%%%%%%%%%%%%%%
\subsection{\texorpdfstring{The main result about $\FK(4)$}{The main result about FK(4)}}
\label{subsec:Resolving datum fk4}

We will now define some quadratic $A$-modules $M^i$ for $i\in \llbracket 1,3\rrbracket$.  
Let $M^1$ be the $A$-module generated by two homogeneous elements $a_1,a_2$ of degree zero, subject to the following $6$ relations
\begin{equation}
   \label{eq: relations of H4}
   \begin{split}
    a_1x_{1,2}+a_2x_{1,2},
   a_1x_{1,3},
   a_2x_{2,3}, 
   a_2x_{1,4},
   a_1x_{2,4}, 
   a_1x_{3,4}+a_2x_{3,4} . 
   \end{split}
   \end{equation}
Let $M^2$ be the $A$-module  generated by the set $\{ h_i \mid i\in \llbracket 1,7 \rrbracket \}$ of seven homogeneous elements of degree zero, subject to the following $24$ relations
\begin{equation}
\label{eq:rels M2}
\begin{split}
& 
h_1x_{1,2}, 
h_1 x_{1,3},
h_1 x_{2,3}, 
h_2x_{1,2}, 
h_2 x_{1,4}, 
h_2 x_{2,4}, 
h_3x_{1,3}, 
h_3x_{1,4},
h_3x_{3,4},
h_4x_{2,3}, 
h_4x_{2,4},
h_4x_{3,4}, 
\\
&
h_1x_{2,4}-h_3x_{2,4}-h_5x_{1,3}, 
h_2x_{1,3}-h_4x_{1,3}+h_5x_{2,4}, 
h_5x_{3,4}-h_6x_{1,2},
h_1x_{1,4}-h_4x_{1,4}+h_6x_{2,3},
\\
&
h_2x_{2,3}-h_3x_{2,3}-h_6x_{1,4},
h_5x_{1,2}+h_6x_{3,4}, 
h_1x_{3,4}+h_2x_{3,4}+h_7x_{1,2},
h_6x_{2,4}+h_7x_{1,3},
\\
&
h_5x_{1,4}+h_7x_{2,3},
h_5x_{2,3}-h_7x_{1,4},
h_6x_{1,3}-h_7x_{2,4},
h_3x_{1,2}+h_4x_{1,2}+h_7x_{3,4}. 
\end{split}
\end{equation}
Finally, let $M^3$ be the $A$-module generated by the set $\{ e_i \mid i\in \llbracket 1,8\rrbracket \}$ of eight homogeneous elements of degree zero, subject to the following $24$ relations 
\begin{equation}
\label{eq:relation in M3}
\begin{split}
     & e_1x_{1,2}+e_2x_{3,4}, 
      e_1x_{3,4}-e_2x_{1,2}, 
      e_3x_{1,2}-e_4x_{3,4},
      e_3x_{3,4}+e_4x_{1,2},
     e_4x_{1,3}+e_2x_{2,4},
      e_4x_{2,4}-e_2x_{1,3}, 
      \\
     &
      e_3x_{1,3}+e_1x_{2,4}, 
      e_3x_{2,4}-e_1x_{1,3},
      e_1x_{2,3}-e_4x_{1,4},
      e_1x_{1,4}+e_4x_{2,3}, 
      e_3x_{2,3}-e_2x_{1,4}, 
      e_3x_{1,4}+e_2x_{2,3},
      \\
      & 
      e_5x_{1,2}, e_5x_{1,3}, e_5x_{2,3},
      e_6x_{1,2}, e_6x_{1,4}, e_6x_{2,4}, 
      e_7x_{1,3}, e_7x_{1,4}, e_7x_{3,4},
      e_8x_{2,3}, e_8x_{2,4}, e_8x_{3,4}.
\end{split}
\end{equation}
Since the previous modules are finite dimensional, we use GAP to obtain a homogeneous $\Bbbk$-basis of $M^i$, and in particular, the Hilbert series of $M^i$, for $i\in \llbracket 1,3 \rrbracket$. 
See Appendix \ref{sec:Basis of M1} for a basis of $M^1$. 

%%%%%%%
\begin{fact}
\label{fact:hilb-ser-mi}
Given $i\in \llbracket 1,3 \rrbracket$, the Hilbert series $h_{M^i}(t)$ of the quadratic $A$-module $M^i$ introduced in the previous paragraph is given by 
\begin{equation}
\begin{split}
h_{M^1}(t) & = 2+6t+11t^2+12t^3+11t^4+6t^5+2t^6,
\\
h_{M^2}(t) & = 7+18t+32t^2+42t^3+40t^4+30t^5+16t^6+6t^7+1t^8,
\\
h_{M^3}(t) & = 8+24t+48t^2+72t^3+80t^4+72t^5+48t^6+24t^7+8t^8.
\end{split}
\end{equation}
\end{fact}
%%%%%%%

We now present the second main result of this article, which states that the Fomin-Kirillov algebra of index $4$ has a connected resolving datum. 
The proof of this theorem will follow from several intermediate
results that we will provide in the following two subsections (see Subsection \ref{proof:rd fk4}).

%%%%%%%
\begin{thm}
\label{thm:rd fk4}
Let $\MM = \{ M^{0}=\Bbbk, M^1,M^2, M^{3} \}$ be the family of quadratic $A$-modules introduced in the first paragraph of this subsection, and let 
$ \hh : \llbracket 0 , N \rrbracket^{2} \times \NN^{2} \rightarrow \NN_{0}^{2} $ 
be the map given by 
\begin{align*}
\hh(0,2,3,6) & = \hh(0,0,3,6)=\hh(1,2,1,4)=(1,0),
\\
\hh(0,0,3,8) & = \hh(0,1,4,8)=\hh(0,0,5,16)=\hh(1,0,1,6)=\hh(1,0,1,8)=\hh(2,0,1,4)
\\
& =\hh(2,0,1,6)=\hh(2,1,2,6)=\hh(2,3,3,6)=\hh(3,3,3,6)=(0,1), 
\end{align*}
and $\hh(i,j,k,\ell)$ vanishes on other $(i,j,k,\ell)$. 
Then this gives a connected resolving datum on $A$, whose associated resolving quiver is given in Figure \ref{figure:fk4}, where we denote by ${}_{j}\alpha_{i}^{d',d''}$ the unique arrow from $M^{i}$ to $M^{j}$ having bidegree $(d',d'')$. 
In this case, the strict partial order on the arrows is given by ${}_{0}\alpha_{0}^{4,8}<{}_{0}\alpha_{0}^{4,6}, {}_{2}\alpha_{0}^{4,6}$, and ${}_{0}\alpha_{1}^{2,6}, {}_{0}\alpha_{1}^{2,8} <{}_{2}\alpha_{1}^{2,4}$. 
The arrows ${}_{1}\alpha_{0}^{5,8}$ and ${}_{1}\alpha_{2}^{3,6}$ of odd difference degrees appear in red.
\begin{center}
   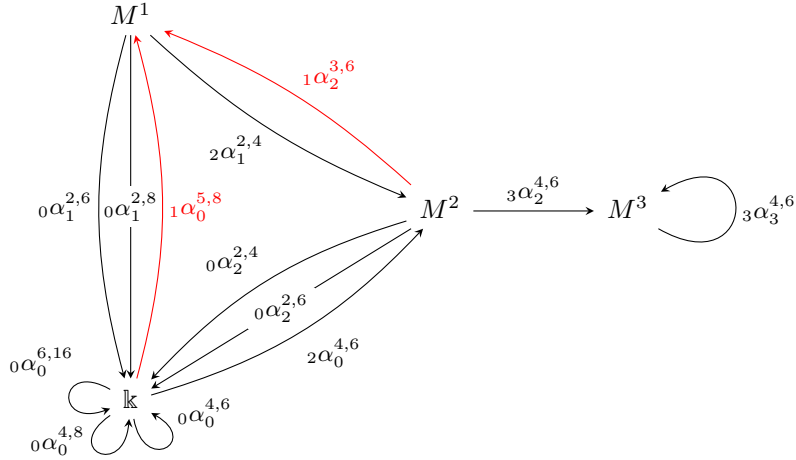
\begin{figure}[htbp] 
   \tikzcdset{arrow style=tikz, diagrams={>=stealth}}
   \tikzcdset{every label/.append style = {font = \small}}
      \begin{tikzcd}[ampersand replacement=\&, column sep=1.6cm, row sep = 2 cm, tail]
           \&
           \&
         M^{1}
         \arrow[dd,bend left = 0, "{}_{0}\alpha_{1}^{2,8}" description]
         \arrow[dd,bend right = 15, "{}_{0}\alpha_{1}^{2,6}"']
           \arrow[drr, bend right = 10, "{}_{2}\alpha_{1}^{2,4}"']
           \&
         \&
         \&
         \\
           \&
           \&
            \&
           \&
           M^{2}
           \arrow[r,bend left = 0, "{}_{3}\alpha_{2}^{4,6}"]
           \arrow[ull,bend right = 10, "{}_{1}\alpha_{2}^{3,6}"', red]
           \arrow[dll,bend left = 0, "{}_{0}\alpha_{2}^{2,6}" description]
           \arrow[dll,bend right = 15, "{}_{0}\alpha_{2}^{2,4}"']
         \&
           M^{3}
           \arrow[out = -30, in = 30, loop, "{}_{3}\alpha_{3}^{4,6}"']
           \\
           \&
           \&
         \Bbbk 
           \arrow[uu, bend right = 15, "{}_{1}\alpha_{0}^{5,8}"', red]
           \arrow[out = 150, in = 200, loop, "{}_{0}\alpha_{0}^{6,16}"', near start]
         \arrow[out = 215, in = 265, loop, "{}_{0}\alpha_{0}^{4,8}"', pos = 0.28] 
           \arrow[out = 280, in = 330, loop, "{}_{0}\alpha_{0}^{4,6}"', near end]
         \arrow[rru, bend right = 15, "{}_{2}\alpha_{0}^{4,6}"']
         \&
         \&
           \&
      \end{tikzcd}
    \caption{Resolving quiver of $\FK(4)$.}
   \label{figure:fk4}
      \end{figure}
    \end{center}
    \vspace{-0.8cm}
\end{thm}
%%%%%%%

%%%%%%%%%%%%%%%%%%%%%%%%%%%%%%%%%%%%%%%%%%%%%%
\subsection{\texorpdfstring{The homology of the Koszul complex of $\FK(4)$}{The homology of the Koszul complex of FK(4)}}
\label{subsec:cpx of H4}

%%%%%%%%%%%%%%%%%%%%%%%%%%%%%%%%%%%%%%%%%%%%%%
\subsubsection{The dimensions of the homology groups}
\label{subsubsec:Homology of Koszul complex of the trivial module}

Recall that $(K_{\bullet},d_{\bullet})$ is the Koszul complex of the trivial module $\Bbbk$ in the category of bounded below graded (right) $A$-modules. 
Let $K_{n,m}=(A^{!}_{-n})^*\otimes A_m$, 
$d_{n,m}=d_n|_{K_{n,m}}:K_{n,m}\to K_{n-1,m+1}$, 
$B_{n,m}=\Img(d_{n+1,m-1})$, $D_{n,m}=\Ker(d_{n,m})$, 
$H_{n,m}=D_{n,m}/B_{n,m}$ 
for $n\in\NN_0$ and $m\in\llbracket 0,12\rrbracket$. 
Let $H_n=\oplus_{m\in \llbracket 0,12 \rrbracket} H_{n,m}$ for $n\in \NN_0$.
We can compute the dimension of $B_{n,m}$ using GAP for $n$ less than some arbitrary positive integer and $m\in\llbracket 1,12\rrbracket$ by using the code in Appendix \ref{sec:cpx} together with the following simple routine.   

\parskip 1ex
\parindent 0in
\begin{tcolorbox}[breakable,colback=white,width=\textwidth ,center,arc=0mm,size=fbox]
\begin{footnotesize}
\begin{Verbatim}[samepage=false]
for j in [0..11] do
    for i in [1..12] do
        Print(j, " ", i, " ", RankMat(FF(0,j,i)), "\n");
    od;
od;
\end{Verbatim}
\end{footnotesize}
\end{tcolorbox}

For the rest of the section, we will only indicate the extra code added to the one in Appendix \ref{sec:cpx} for every computation, and, for the reader's convenience, we will often indicate the output of many of the intermediate commands in the corresponding successive line and preceded by a pound sign $\#$. 

The dimension of $B_{n,m}$ for $n\in\llbracket 0,11\rrbracket $ and $m\in \llbracket 0,12\rrbracket$ is displayed in Table \ref{table:Bnm}. 

\begin{table}[H]
      \begin{center}
       %  \resizebox{\textwidth}{120mm}{
         \begin{tabular}{|c|ccccccccccccc|}
         \hline
          \diagbox[width=14mm,height=5mm]{$n$}{$m$}  & $0$ & $1$  & $2$ & $3$ & $4$ & $5$ & $6$ & $7$ & $8$ & $9$ & $10$ & $11$ & $12$ 
        \\
        \hline
         $0$ & $0$ & $6$  & $19$ & $42$ & $71$ & $96$ & $106$ & $96$ & $71$ & $42$ & $19$ & $6$ & $1$ 
        \\
        $1$ & $0$ & $17$  & $72$ & $181$ & $330$ & $470$ & $540$ & $505$ & $384$ & $233$ & $108$ & $35$ & $6$ 
        \\
        $2$ & $0$ & $30$ & $142$ & $384$  & $737$ & $1092$ & $1297$ & $1248$ & $974$ & $606$ & $288$ & $96$ & $17$ 
        \\
        $3$ & $0$ & $38$ & $186$ & $515$ & $1020$ & $1550$ & $1890$ & $1866$ & $1494$ & $956$ & $468$ & $162$ & $30$   
        \\
        $4$ & $0$ & $42$ & $207$ & $576$ & $1146$ & $1752$ & $2151$ & $2142$ & $1731$ & $1122$ & $558$ & $198$ & $38$
        \\
        $5$ & $0$ & $45$ & $222$ & $618$ & $1230$ & $1881$ & $2310$ & $2301$ & $1860$ & $1206$ & $600$ & $213$ & $42$   
        \\
        $6$ & $0$ & $48$ & $237$ & $660$ & $1314$ & $2010$ & $2469$ & $2460$ & $1989$ & $1290$ & $642$ & $228$ & $45$  
        \\
        $7$ & $0$ & $51$ & $252$ & $702$ & $1398$ & $2139$ & $2628$ & $2619$ & $2118$ & $1374$ & $684$ & $243$ & $48$ 
        \\
        $8$ & $0$ & $54$ & $267$ & $744$ & $1482$ & $2268$ & $2787$ & $2778$ & $2247$ & $1458$ & $726$ & $258$ & $51$ 
        \\
        $9$ & $0$ & $57$ & $282$ & $786$ & $1566$ & $2397$ & $2946$ & $2937$ & $2376$ & $1542$ & $768$ & $273$ & $54$
        \\
        $10$ & $0$ & $60$ & $297$ & $828$ & $1650$ & $2526$ & $3105$ & $3096$ & $2505$ & $1626$ & $810$ & $288$ & $57$ 
        \\
        $11$ & $0$ & $63$ & $312$ & $870$ & $1734$ & $2655$ & $3264$ & $3255$ & $2634$ & $1710$ & $852$ & $303$ & $60$
        \\
         \hline
      \end{tabular}
    %  }
    \vspace{1mm}
      \caption{Dimension of $B_{n,m}$.}	
      \label{table:Bnm}
   \end{center}
   \end{table}
   \vspace{-0.8cm}

Since $\dim D_{n,m}=\dim K_{n,m}-\dim B_{n-1,m+1}$, using Table \ref{table:Bnm} we get the dimension of $D_{n,m}$ for $n\in\llbracket 0,5\rrbracket$ and $m\in \llbracket 0,12\rrbracket$, which is displayed in Table \ref{table:Dnm}.
   \begin{table}[H]
      \begin{center}
       %  \resizebox{\textwidth}{120mm}{
         \begin{tabular}{|c|ccccccccccccc|}
         \hline
          \diagbox[width=14mm,height=5mm]{$n$}{$m$}  & $0$ & $1$  & $2$ & $3$ & $4$ & $5$ & $6$ & $7$ & $8$ & $9$ & $10$ & $11$ & $12$ 
        \\
        \hline
        $0$ & $1$ & $6$ & $19$ & $42$ & $71$ & $96$ & $106$ & $96$ & $71$ & $42$ & $19$ & $6$ & $1$ 
        \\
        $1$ & $0$ & $17$  & $72$ & $181$ & $330$ & $470$ & $540$ & $505$ & $384$ & $233$ & $108$ & $35$ & $6$   
        \\
        $2$ & $0$ & $30$ & $142$ & $384$  & $737$ & $1092$ & $1297$ & $1248$ & $974$ & $606$ & $288$ & $96$ & $17$  
        \\
        $3$ & $0$ & $38$ & $186$ & $523$ & $1038$ & $1583$ & $1932$ & $1906$ & $1524$ & $972$ & $474$ & $163$ & $30$ 
        \\
        $4$ & $0$ & $42$ & $207$ & $576$ & $1148$ & $1758$ & $2162$ & $2154$ & $1742$ & $1128$ & $560$ & $198$ & $38$  
        \\
        $5$ & $0$ & $45$ & $222$ & $618$ & $1230$ & $1881$ & $2310$ & $2301$ & $1860$ & $1206$ & $600$ & $214$ & $42$ 
        \\
         \hline
      \end{tabular}
    %  }
    \vspace{1mm}
      \caption{Dimension of $D_{n,m}$.}	
      \label{table:Dnm}
\end{center}
\end{table}
\vspace{-0.8cm}
Finally, since $\dim H_{n,m}=\dim D_{n,m}-\dim B_{n,m}$, using Tables \ref{table:Bnm} and \ref{table:Dnm} we get the dimension of $H_{n,m}$ for $n\in\llbracket 0,5\rrbracket$ and $m\in \llbracket 0,12\rrbracket$, which appears in Table \ref{table:Hnm0345}. 
\begin{table}[H]
   \begin{center}
    %  \resizebox{\textwidth}{120mm}{
      \begin{tabular}{|c|ccccccccccccc|}
      \hline
       \diagbox[width=14mm,height=5mm]{$n$}{$m$}  & $0$ & $1$  & $2$ & $3$ & $4$ & $5$ & $6$ & $7$ & $8$ & $9$ & $10$ & $11$ & $12$ 
     \\
     \hline
     $0$ & $1$ & $0$ & $0$ & $0$ & $0$ & $0$ & $0$ & $0$ & $0$ & $0$ & $0$ & $0$ & $0$ 
     \\
     $3$ & $0$ & $0$ & $0$ & $8$ & $18$ & $33$ & $42$ & $40$ & $30$ & $16$ & $6$ & $1$ & $0$ 
     \\
     $4$ & $0$ & $0$ & $0$ & $0$ & $2$ & $6$ & $11$ & $12$ & $11$ & $6$ & $2$ & $0$ & $0$  
     \\
     $5$ & $0$ & $0$ & $0$ & $0$ & $0$ & $0$ & $0$ & $0$ & $0$ & $0$ & $0$ & $1$ & $0$ 
     \\
      \hline
   \end{tabular}
 %  }
 \vspace{1mm}
   \caption{Dimension of $H_{n,m}$.}	
   \label{table:Hnm0345}
\end{center}
\end{table}
%\vspace{-0.8cm}

More generally, we have the following result. 

%%%%%%%
\begin{prop}
   \label{prop:dim Bn5}
For $n\geqslant 5$, the dimension of $B_{n,m}$ is given by 
\begin{equation}
      \label{eq:dim Bn5}
          \dim B_{n,m}= \begin{cases} 
          0, &\text{if $m=0$,}
          \\
          3n+30, &\text{if $m = 1$,}
          \\
          15n+147, &\text{if $m =2$,}
          \\
          42n+408, &\text{if $m =3$,}
          \\
          84n+810, &\text{if $m =4$,}
          \\
          129n+1236, &\text{if $m =5$,}
           \\
          159n+1515 , &\text{if $m =6$,}
           \\
          159n+1506, &\text{if $m =7$,}
           \\
           129n+1215, &\text{if $m =8$,}
           \\
          84n+786, &\text{if $m =9$,}
           \\
          42n+390, &\text{if $m =10$,}
           \\
          15n+138, &\text{if $m =11$,}
           \\
          3n+27, &\text{if $m =12$.}
      \end{cases}
      \end{equation}
\end{prop}
%%%%%%%

Before giving the proof of the previous proposition, let us first give a direct consequence. 

%%%%%%%
\begin{cor}
   \label{cor:H}
We have $\dim H_n =0$ for $n\in \NN \setminus \{3,4,5 \}$. 
Moreover, the dimension of $H_{n,m}$ for $n=0,3,4,5$ and $m\in \llbracket 0,12\rrbracket$ is the one given in Table \ref{table:Hnm0345}.   
\end{cor}
%%%%%%%
\begin{proof}
By $\dim D_{n,m}=\dim K_{n,m}-\dim B_{n-1,m+1}$, Proposition \ref{prop:dim Bn5}, together with Tables \ref{table:Bnm} and \ref{table:Dnm}, 
we have $\dim D_{n,m}=\dim B_{n,m}$ for $n\in \NN \setminus\{3,4,5 \}$ and $m\in \llbracket 0,12 \rrbracket $. 
Then the corollary holds. 
\end{proof}

In order to prove Proposition \ref{prop:dim Bn5}, we need some preparatory results. 
Let $\C_n= \cup _{(i,j)\in \II}\C_n^{i,j}$, where  
\[     \C_n^{i,j}= \big\{ z^{i,j}_{n-r}z^{k,l}_{r} \mid (k,l)\in \I \text{ such that } \# \{i,j,k,l \}=4, r\in \llbracket 0,n \rrbracket \big\} \subseteq \B_n^{!*}     \]  
for $(i,j)\in \II$ and $n\in\NN$, and let $\U_n=\B_n^{!*} \setminus \C_n$ for $n\in\NN$. 
Note that the pair $(k,l) \in \I$ is uniquely determined in the definition of $\C_n^{i,j}$. 
Given $m, n \in \NN$, let $C_{n,m}$ be the subspace of $\Bbbk \C_n \otimes A_m$ spanned by $\{ d_{n+1}(z|x)\mid z\in \C_{n+1}, x\in A_{m-1} \}$, 
$C_{n,m}^{i,j}$ be the subspace of $C_{n,m}$ spanned by $\{ d_{n+1}(z|x)\mid z\in \C_{n+1}^{i,j}, x\in A_{m-1} \}$ for $(i,j)\in \II$, 
and $U_{n,m}$ be the subspace of $B_{n,m}$ spanned by $\{ d_{n+1}(z|x)\mid z\in \U_{n+1}, x\in A_{m-1} \}$.

Fixing the order $x_{1,2} \prec x_{3,4} \prec x_{1,3} \prec x_{2,3} \prec x_{1,4} \prec x_{2,4}$ (resp., $x_{1,3} \prec x_{2,4} \prec x_{1,2} \prec x_{2,3} \prec x_{1,4} \prec x_{3,4}$, $x_{2,3} \prec x_{1,4} \prec x_{1,2} \prec x_{1,3} \prec x_{2,4} \prec x_{3,4}$), 
the corresponding basis of $A$  consisting of standard words will be denoted by $W^{1,2}$ (resp., $W^{1,3}$, $W^{2,3}$). 
It can be explicitly computed using GAP (see Appendix \ref{subsec:W} for $W^{1,2}$). 
For $(i,j)\in \II$, let $(k,l)\in \I$ such that $\# \{i,j,k,l \}=4$, 
set $W^{i,j}_m=W^{i,j}\cap A_m$. 
Set $E^{i,j}_m$ as the subset of $W^{i,j}_m$ containing elements whose first element is not $x_{i,j}$, and set $\tilde{E}^{i,j}_m$ as the subset of $W^{i,j}_m$ containing elements whose first element is neither $x_{i,j}$ nor $x_{k,l}$. 
Let $\mathfrak{a}^{i,j}_m=\# E^{i,j}_m$ and $\mathfrak{b}^{i,j}_m=\# \tilde{E}^{i,j}_m$ for $m\in \llbracket 0,11\rrbracket$.  
The integers $\mathfrak{a}^{i,j}_m$ and $\mathfrak{b}^{i,j}_m$ are easily computed from the explicit description of the bases $W_{m}^{i,j}$, they are independent of $(i,j)$, so they will be denoted simply by $\mathfrak{a}_{m}$ and $\mathfrak{b}_{m}$, respectively, and are given in Table \ref{table:a b}. 
\begin{table}[H]
   \begin{center}
    %  \resizebox{\textwidth}{120mm}{
      \begin{tabular}{|c|cccccccccccc|}
      \hline
       \diagbox[width=14mm,height=5mm]{ }{$m$}  & $0$ & $1$  & $2$ & $3$ & $4$ & $5$ & $6$ & $7$ & $8$ & $9$ & $10$ & $11$  
     \\
     \hline
     $\mathfrak{a}_m$ & $1$ & $5$ & $14$ & $28$ & $43$ & $53$ & $53$ & $43$ & $28$ & $14$ & $5$ & $1$  
     \\
     $\mathfrak{b}_m$ & $1$ & $4$ & $10$ & $18$ & $25$ & $28$ & $25$ & $18$ & $10$ & $4$ & $1$ & $0$  
     \\
      \hline
   \end{tabular}
 %  }
 \vspace{1mm}
   \caption{Values of $\mathfrak{a}_m$ and $\mathfrak{b}_m$.}	
   \label{table:a b}
\end{center}
\end{table}
\vspace{-0.8cm}

%%%%%%%%%%%%%%%%%%%%%%%%%%
\begin{lem}
\label{lem:dim changing in image}
We have
$C_{n,m}=\bigoplus_{(i,j)\in \II} C^{i,j}_{n,m}$ 
and the dimension of $C^{i,j}_{n,m}$ is given by  
\begin{equation}
   \label{eq:dim chang}
       \dim C^{i,j}_{n,m}= \begin{cases} 
       n+2, &\text{if $m = 1$,}
       \\
       5n+9, &\text{if $m =2$,}
       \\
       14n+24, &\text{if $m =3$,}
       \\
       28n+46, &\text{if $m =4$,}
       \\
       43n+68, &\text{if $m =5$,}
        \\
       53n+81 , &\text{if $m =6$,}
        \\
       53n+78, &\text{if $m =7$,}
        \\
        43n+61, &\text{if $m =8$,}
        \\
       28n+38, &\text{if $m =9$,}
        \\
       14n+18, &\text{if $m =10$,}
        \\
       5n+6, &\text{if $m =11$,}
        \\
       n+1, &\text{if $m =12$,}
   \end{cases}
   \end{equation}
for all $(i,j)\in \II $ and $n\in \NN$. 
Else $\dim C^{i,j}_{n,m} = 0$. 
\end{lem}
\begin{proof}
Given $(i,j)\in \II$, fix $(k,l)\in \I$ such that $\# \{i,j,k,l \}=4$. 
Then, the maps $\Bbbk E^{i,j}_{m-1} \rightarrow A_m$ and $\Bbbk \tilde{E}^{i,j}_{m-1} \rightarrow A_m$ given by left multiplication by 
$x_{i,j}$ and by left multiplication by $x_{k,l}$, respectively, are injective for $m\in \llbracket 1,12\rrbracket$. 
Hence, using \eqref{eq:differential 4}, we see that the set formed by the elements $ (-1)^r z^{i,j}_{n-r}z^{k,l}_{r}|x_{i,j}x+z^{i,j}_{n-r+1}z^{k,l}_{r-1}|x_{k,l}x$, for $x\in E^{i,j}_{m-1}$ and $r\in \llbracket 0, n \rrbracket $, 
together with the elements $z^{k,l}_n|x_{k,l}y$ for $y\in \tilde{E}^{i,j}_{m-1}$ 
gives a basis of $ C^{i,j}_{n,m}$. Then, $\dim C^{i,j}_{n,m}=\mathfrak{a}_{m-1}(n+1)+\mathfrak{b}_{m-1}$, which together with Table \ref{table:a b} proves the claim. 
\end{proof}

%%%%%%%%%%%%%%%%%%%%%%%%%
\begin{lem}
\label{lem:dim non changing in image}
We have $\dim U_{n,m}=\dim U_{n+2,m}$ and $\dim (U_{n,m}\cap C_{n,m})=\dim (U_{n+2,m} \cap C_{n+2,m}) $ for $n\geqslant 5$ and $m\in \llbracket 1,12 \rrbracket$.
\end{lem}

\begin{proof}
For $n\geqslant 5$, set 
\begin{equation*}
   \label{eq:elements u v}
   \begin{split}
      u^{i,j}_n=\underset{\text{\begin{tiny}$\begin{matrix} r\in \llbracket 1, n-1\rrbracket, \\ r \; \text{odd} \end{matrix}$\end{tiny}}}{\sum} z^{i,j}_{n-r}z^{k,l}_{r} 
      \hskip 2mm 
      \text{ and }  
      \hskip 2mm      v^{i,j}_n=\underset{\text{\begin{tiny}$\begin{matrix} r\in \llbracket 1, n-1\rrbracket, \\ r \; \text{even} \end{matrix}$\end{tiny}}}{\sum} z^{i,j}_{n-r}z^{k,l}_{r} ,
   \end{split}
\end{equation*}
for $(i,j)\in \II$, $(k,l)\in \I$, $\# \{i,j,k,l \}=4$, and 
\begin{equation*}
   \label{eq:non change}
   \begin{split}
      \Q_{n} & =\U_n 
      \cup 
      \big\{ z^{i,j}_n\mid (i,j)\in \I\big\}  
      \cup 
      \big\{ u^{i,j}_n,  
      v^{i,j}_n \mid (i,j) \in \II \big\} \subseteq (A^!_{-n})^*.
   \end{split}
\end{equation*}
Let $\U^{i,j}_n$ be the subset of $\U_n$ consisting of elements whose first element is $z^{i,j}$ for $(i,j)\in \II $. 
There is an isomorphism $\mathfrak{f}_n:\Bbbk \Q_{n}\to \Bbbk \Q_{n+2}$ of vector spaces defined by $\mathfrak{f}_n(z)=z^{i,j}_2z$ for $z\in \U^{i,j}_n$ and $(i,j)\in \II $, 
$\mathfrak{f}_n(u^{i,j}_n)=u^{i,j}_{n+2}$, 
$\mathfrak{f}_n(v^{i,j}_n)=v^{i,j}_{n+2}$, 
and 
$\mathfrak{f}_n(z^{i,j}_n)=z^{i,j}_{n+2}$ for $(i,j)\in\I$. 
Then, the map $\mathfrak{g}_n =\mathfrak{f}_n  \otimes \id_{A}:\Bbbk \Q_{n}\otimes A \to \Bbbk \Q_{n+2}\otimes A$ is a linear isomorphism. 
By \eqref{eq:dn}, $U_{n,m}\subseteq \Bbbk \Q_n\otimes A_m$ and $\mathfrak{g}_n(U_{n,m})=U_{n+2,m}$, giving an isomorphism $U_{n,m}\cong U_{n+2,m}$ of vector spaces for $n\geqslant 5$ and $m\in \llbracket 1,12 \rrbracket$. 
This proves the first part of the lemma. 

Set $F_{n,m}=(\Bbbk \Q_n\otimes A_m ) \cap C_{n,m}$ and define $L^{i,j}_{n,m}= \Bbbk \{z^{i,j}_n,z^{k,l}_n, u_n^{i,j}, v_n^{i,j} \} \otimes A_m $ as the subspace of $\Bbbk \C^{i,j}_{n}\otimes A_m$, where $(i,j)\in \II$, $(k,l)\in \I$ with $\# \{i,j,k,l \}=4$.
It is clear that
$F_{n,m}=\oplus_{(i,j)\in \II} (L^{i,j}_{n,m}\cap C^{i,j}_{n,m})$.
Fix $(i,j)\in \II$, $(k,l)\in \I$ with $\# \{i,j,k,l \}=4$.
Let $\xi^{i,j}\in C^{i,j}_{n,m}$. Then $\xi^{i,j}$ is of the form 
\begin{equation}
\label{eq: inter C}
\begin{split}
   \xi^{i,j}=
   \underset{\text{\begin{tiny}$\begin{matrix} r\in \llbracket 0, n \rrbracket, \\ x\in E_{m-1}^{i,j} \end{matrix}$\end{tiny}}}{\sum}
   \lambda_{r,x}
   \big\{
   (-1)^r z^{i,j}_{n-r}z^{k,l}_{r}|x_{i,j}x+z^{i,j}_{n-r+1}z^{k,l}_{r-1}|x_{k,l}x
   \big\}
   + 
   \sum_{y\in \tilde{E}^{i,j}_{m-1}}
   \mu_{y}
   z^{k,l}_n|x_{k,l}y
\end{split}
\end{equation}
for $\lambda_{r,x},\mu_y \in \Bbbk$. 
If $\xi^{i,j} \in L^{i,j}_{n,m}$, then $\xi^{i,j}$ is of the form 
\begin{equation}
   \label{eq: inter L}
   \begin{split}
      \xi^{i,j}= 
      \sum_{w\in W^{i,j}_m} 
      (\alpha_{w} z_n^{i,j}|w
      +
      \beta_w z_n^{k,l}|w
      +
      \gamma_w u_n^{i,j}|w
      +
      \eta_w v_n^{i,j}|w)
   \end{split}
\end{equation}
for $\alpha_w, \beta_w, \gamma_w, \eta_w \in \Bbbk$. 
Comparing the coefficients in \eqref{eq: inter C} and \eqref{eq: inter L}, we obtain 
\begin{equation}
\label{eq:coefficients}
\begin{split}
\alpha_{x_{i,j}x}& =\lambda_{0,x}, 
\hskip 2mm 
\alpha_{x_{k,l}y} =\lambda_{1,y},
\hskip 2mm 
\beta_{x_{i,j}x} =(-1)^n\lambda_{n,x},
\hskip 2mm
\beta_{x_{k,l}y} =\mu_y,
\\
\gamma_{x_{i,j}x}& =-\lambda_{p,x} 
\text{ for $p\in \llbracket 1,n-1 \rrbracket $ with $p$ odd}, 
\\
\gamma_{x_{k,l}y}& =\lambda_{q,y} 
\text{ for $q\in \llbracket 2,n \rrbracket $ with $q$ even}, 
\\
\eta_{x_{i,j}x}& =\lambda_{q,x}  
\text{ for $q\in \llbracket 2,n-1 \rrbracket $ with $q$ even}, 
\\
\eta_{x_{k,l}y}& =\lambda_{p,y} 
\text{ for $p\in \llbracket 3 , n \rrbracket $ with $p$ odd},  
\end{split}
\end{equation}
where $x\in E^{i,j}_{m-1}$ and $y\in \tilde{E}^{i,j}_{m-1}$. 
Hence, if $n$ is even, the space $L^{i,j}_{n,m}\cap C^{i,j}_{n,m}$ is spanned by 
$z^{i,j}_n|x_{i,j}x$, 
$(v^{i,j}_n|x_{k,l}-u^{i,j}_n|x_{i,j}+z^{i,j}_n|x_{k,l})x$ for $x\in E^{i,j}_{m-1}$, 
$z^{k,l}_n|x_{k,l}y$, 
$(u^{i,j}_n|x_{k,l}+v^{i,j}_n|x_{i,j}+ z^{k,l}_n|x_{i,j})y$ for $y\in \tilde{E}^{i,j}_{m-1}$, 
$v^{i,j}_n|x_{i,j}w$  
and 
$z^{k,l}_n|x_{i,j}w$ for $w\in E^{i,j}_{m-1}\setminus \tilde{E}^{i,j}_{m-1}$. 
If $n$ is odd, the space $L^{i,j}_{n,m}\cap C^{i,j}_{n,m}$ is spanned by 
$z^{i,j}_n|x_{i,j}x$, 
$(u^{i,j}_n|x_{k,l}+v^{i,j}_n|x_{i,j})x$ for $x\in E^{i,j}_{m-1}$, 
$z^{k,l}_n|x_{k,l}y$, 
$(v^{i,j}_n|x_{k,l}-u^{i,j}_n|x_{i,j}+z^{i,j}_n|x_{k,l}-z^{k,l}_n|x_{i,j})y$ for $y\in \tilde{E}^{i,j}_{m-1}$, 
$u^{i,j}_n|x_{i,j}w$ 
and 
$z^{k,l}_n|x_{i,j}w$ for $w\in E^{i,j}_{m-1}\setminus \tilde{E}^{i,j}_{m-1}$.
We finally note that $U_{n,m}\cap C_{n,m}= U_{n,m}\cap F_{n,m}$ and $\mathfrak{g}_n(F_{n,m})=F_{n+2,m}$. 
Hence, $U_{n,m}\cap C_{n,m}\cong U_{n+2,m}\cap C_{n+2,m}$ as vector spaces. 
This proves the second part of the lemma. 
\end{proof}

\begin{proof}[Proof of Proposition \ref{prop:dim Bn5}]
   By Table \ref{table:Bnm}, we obtain that \eqref{eq:dim Bn5} holds for $(n,m)\in \llbracket 5,6\rrbracket \times \llbracket 0,12\rrbracket$. 
   On the other hand, by Lemma \ref{lem:dim non changing in image}, we get that $\dim B_{n+2,m}-\dim B_{n,m}=\dim C_{n+2,m}-\dim C_{n,m}$ for $n\geqslant 5$ and $m\in \llbracket 1,12 \rrbracket$. 
   The statement then follows. 
\end{proof}

Using GAP, we can easily compute the dimension of $U_{n,m}$ for $n\geqslant 3$ and $m\in \llbracket 1, 12 \rrbracket$, which is given in Table \ref{table:Unm}. 
\begin{table}[H]
   \begin{center}
      \resizebox{\textwidth}{9.2mm}{
      \begin{tabular}{|c|cccccccccccc|}
      \hline
       \diagbox[width=32mm,height=4.5mm]{$n$}{$m$}   & $1$  & $2$ & $3$ & $4$ & $5$ & $6$ & $7$ & $8$ & $9$ & $10$ & $11$ & $12$ 
     \\
     \hline
     $3$ &  $23$ & $138$ & $422$ & $896$ & $1428$ & $1800$ & $1815$ & $1468$ & $947$ & $466$ & $162$ & $30$ 
     \\
     $n\geqslant 4$ with $n$ even  &  $24$ & $136$ & $408$ & $850$ & $1344$ & $1690$ & $1716$ & $1406$ & $924$ & $466$ & $168$ & $34$ 
     \\
     $n\geqslant 5$ with $n$ odd & $24$ & $144$ & $434$ & $912$ & $1452$ & $1836$ & $1872$ & $1536$ & $1008$ & $504$ & $180$ & $36$  
     \\
      \hline
   \end{tabular}
   }
   \vspace{1mm}
   \caption{Dimension of $U_{n,m}$.}	
   \label{table:Unm}
\end{center}
\end{table}
\vspace{-0.8cm}

%%%%%%%
\begin{rk}
\label{rk:exact of subcomplex kCA}
Lemma \ref{lem:dim changing in image} tells us that the subcomplex of $\operatorname{K}_{\bullet}$ formed by the submodules $\Bbbk \C_n\otimes A $ for $n \in \NN$ is exact for $n\geqslant 2$. 
\end{rk}
%%%%%%%

%%%%%%%%%%%%%%%%%%%%%%%%%%%%%%%%%%
%%%%%%%%%%%%%%%%%%%%%%%%%%%%%%%%%%
\subsubsection{\texorpdfstring{The $A$-module structure of the homology groups}{The A-module structure of the homology groups}}

%%%%%%%
\begin{lem}
\label{lemma:homology-m0}
We have the following isomorphisms 
\begin{equation}
\label{eq:HK k}
\operatorname{H}_{n}(\Bbbk) \cong
\begin{cases} 
   %\operatorname{H}_{3}(\Bbbk), & \text{if $n = 3$,}
   %\\
   M^{1}(-8), & \text{if $n = 4$,}
   \\
   \Bbbk(-16), & \text{if $n = 5$, }
   \\
   0, & \text{if $n \in \NN \setminus \llbracket 3,5  \rrbracket $,}
\end{cases}
\end{equation}
of graded $A$-modules, 
as well as the non-split short exact sequence 
\begin{equation}
      \label{eq:sp2}
   \begin{split}
   0 \to M^2(-6)\oplus \Bbbk(-6) \to \operatorname{H}_3(\Bbbk) \to \Bbbk(-8) \to 0
   \end{split} 
\end{equation}
of graded $A$-modules. 
\end{lem}
%%%%%%%
\begin{proof}
The isomorphism in \eqref{eq:HK k} for $n \in \NN \setminus \llbracket 3,5  \rrbracket$ 
follows from Corollary \ref{cor:H}. 
Similarly, the isomorphism in \eqref{eq:HK k} for $n = 5$ follows immediately from Table \ref{table:Hnm0345}. 
Recall that we write $H_{n}$ instead of $\operatorname{H}_{n}(\Bbbk)$ for $n \in \NN$ to simplify the notation. 

Let us prove the isomorphism in \eqref{eq:HK k} for $n = 4$.
The following GAP code shows that the dimension vector of the submodule of $H_4$ generated by two basis elements $a'_1,a'_2$ of $H_{4,4}$ is $(2,6,11,12,11,6,2)$.
So, Table \ref{table:Hnm0345} tells us that $H_4$ is generated by $a'_1,a'_2$ as an $A$-module.
\parskip 1ex
\parindent 0in
\begin{tcolorbox}[breakable,colback=white,width=\textwidth ,center,arc=0mm,size=fbox]
\begin{footnotesize}
\begin{Verbatim}[samepage=false]
Imm:=Im(0,4,4);;
RankMat(Imm);
# 1146 
gene:=geneMH(0,4,4);;
Append(Imm,gene);
RankMat(Imm);
# 1148
Uh:=UU(gene,4);; Vh:=VV(gene,4);; Wh:=WW(gene,4);;
for r in [5..10] do 
   hxr:=HXR(0,Uh,Vh,Wh,4,4,r-4); 
   Im4r:=Im(0,4,r);
   Append(Im4r, hxr);
   Print(r, " ", RankMat(Im4r)-RankMat(Im(0,4,r)), "\n");
od;
# 5 6
# 6 11
# 7 12
# 8 11
# 9 6
# 10 2
\end{Verbatim}
\end{footnotesize}
\end{tcolorbox}

On the other hand, it is direct to check that the generators $a'_1,a'_2$ of $H_4$ satisfy the quadratic relations \eqref{eq: relations of H4} defining $M^1$. 
Indeed, the following code shows that the dimension of the subspace generated by $B_{4,5}$ together with the elements of the form \eqref{eq: relations of H4} with $a'_i$ instead of $a_i$ coincides with the dimension of $B_{4,5}$.
\parskip 1ex
\parindent 0in
\begin{tcolorbox}[breakable,colback=white,width=\textwidth ,center,arc=0mm,size=fbox]
\begin{footnotesize}
\begin{Verbatim}[samepage=false]
gene:=geneMH(0,4,4);;
Uh:=UU(gene,4);; Vh:=VV(gene,4);; Wh:=WW(gene,4);;
hx:=HXR(0,Uh,Vh,Wh,4,4,1);;
cc:=0*[1..6];;
cc[1]:=hx[1]+hx[7];; cc[2]:=hx[2];; cc[3]:=hx[5];; cc[4]:=hx[6]+hx[12];;
cc[5]:=hx[9];; cc[6]:=hx[10];;
Imm:=Im(0,4,5);;
RankMat(Imm);
# 1752 
Append(Imm,cc);
RankMat(Imm);
# 1752 
\end{Verbatim}
\end{footnotesize}
\end{tcolorbox}

Hence, there is a surjective morphism $M^1(-8)\to H_4$ of graded $A$-modules. 
Since the dimension vector of $M^1$ is $(2,6,11,12,11,6,2)$ by Fact \ref{fact:hilb-ser-mi}, we have $H_4\cong M^1(-8)$ as graded $A$-modules, as claimed.

Let us now prove the existence of the short exact sequence \eqref{eq:sp2}.  
The following GAP code shows that the dimension vector of the submodule of $H_3$ generated by the basis elements $c'_i, i\in \llbracket 1,8 \rrbracket$ of $H_{3,3}$ is $(8, 18, 32, 42, 40, 30, 16, 6, 1)$.
%The space $H_{3,5}$ is spanned by $c_i.A_2$ and one element $c_9$. 
%The dimension vector of the submodule of $H_3$ generated by $c_9$ is $1$, $6$, $19$, $30$, $16$, $6$, $1$. 
\parskip 1ex
\parindent 0in
\begin{tcolorbox}[breakable,colback=white,width=\textwidth ,center,arc=0mm,size=fbox]
\begin{footnotesize}
\begin{Verbatim}[samepage=false]
Imm:=Im(0,3,3);;
RankMat(Imm);
# 515
gene:=geneMH(0,3,3);;
Append(Imm,gene);
RankMat(Imm);
# 523
Uh:=UU(gene,3);; Vh:=VV(gene,3);; Wh:=WW(gene,3);;
for r in [4..11] do 
   hxr:=HXR(0,Uh,Vh,Wh,3,3,r-3);
   Im3r:=Im(0,3,r);
   Append(Im3r, hxr);
   Print(r, " ", RankMat(Im3r)-RankMat(Im(0,3,r)), "\n");
od;
# 4 18
# 5 32
# 6 42
# 7 40
# 8 30
# 9 16
# 10 6
# 11 1
\end{Verbatim}
\end{footnotesize}
\end{tcolorbox}
Let $M^{4}$ be the quadratic module generated by the set $\{ c_i \mid i\in \llbracket 1,8 \rrbracket \}$ of eight homogeneous elements
of degree zero, subject to the following $30$ relations 
\begin{equation}
\label{eq:rels M4}
\begin{split}
      & 
      c_1x_{1,2}, c_1x_{1,3}, c_1x_{2,3}, 
      c_2x_{1,2}, c_2x_{1,4}, c_2x_{2,4},
      c_3x_{1,3}, c_3x_{1,4}, c_3x_{3,4}, 
      c_4x_{2,3}, c_4x_{2,4}, c_4x_{3,4}, 
      \\
      & 
      c_5x_{1,3}-c_1x_{2,4}+c_3x_{2,4}, 
      c_5x_{2,4}+c_2x_{1,3}-c_4x_{1,3}, 
      c_6x_{2,3}+c_1x_{1,4}-c_4x_{1,4}, 
      \\
      &
      c_6x_{1,4}-c_2x_{2,3}+c_3x_{2,3}, 
      c_7x_{1,2}+c_1x_{3,4}+c_2x_{3,4}, 
      c_7x_{3,4}+c_3x_{1,2}+c_4x_{1,2}, 
      c_5x_{1,2}+c_6x_{3,4}, 
      \\
      &
      c_5x_{3,4}-c_6x_{1,2}, 
      c_6x_{1,3}-c_7x_{2,4}, 
      c_6x_{2,4}+c_7x_{1,3},
      c_5x_{1,4}+c_7x_{2,3}, 
      c_5x_{2,3}-c_7x_{1,4}, 
      \\
      &
      c_8x_{1,2}, 
      c_8x_{1,3}, 
      c_8x_{2,3}, 
      c_8x_{1,4}, 
      c_8x_{2,4}, 
      c_8x_{3,4}.
\end{split}
\end{equation}
Using GAP we get that the dimension vector of $M^{4}$ is $(8, 18, 32, 42, 40, 30, 16, 6, 1)$. 
It is direct to check that the elements $c'_i, i\in \llbracket 1,8 \rrbracket$ of $H_{3}$ satisfy the quadratic relations \eqref{eq:rels M4}. 
Indeed, the following code shows that the dimension of the subspace generated by $B_{3,4}$ together with the elements of the form \eqref{eq:rels M4} with $c'_i$ instead of $c_i$ coincides with the dimension of $B_{3,4}$.

\parskip 1ex
\parindent 0in
\begin{tcolorbox}[breakable,colback=white,width=\textwidth ,center,arc=0mm,size=fbox]
\begin{footnotesize}
\begin{Verbatim}[samepage=false]
gene:=geneMH(0,3,3);;
Uh:=UU(gene,3);; Vh:=VV(gene,3);; Wh:=WW(gene,3);;
hx:=HXR(0,Uh,Vh,Wh,3,3,1);;
cc:=0*[1..30];;
cc[1]:=hx[1];; cc[2]:=hx[2];; cc[3]:=hx[3];; cc[4]:=hx[7];; cc[5]:=hx[10];; 
cc[6]:=hx[11];; cc[7]:=hx[14];; cc[8]:=hx[16];; cc[9]:=hx[18];; cc[10]:=hx[21];;
cc[11]:=hx[23];; cc[12]:=hx[24];; cc[13]:=hx[5]-hx[17]-hx[26];; 
cc[14]:=hx[8]-hx[20]+hx[29];; cc[15]:=hx[30]-hx[31];; cc[16]:=hx[4]-hx[22]+hx[33];; 
cc[17]:=hx[9]-hx[15]-hx[34];; cc[18]:=hx[25]+hx[36];; cc[19]:=hx[6]+hx[12]+hx[37];;
cc[20]:=hx[35]+hx[38];; cc[21]:=hx[28]+hx[39];; cc[22]:=hx[27]-hx[40];;
cc[23]:=hx[32]-hx[41];; cc[24]:=hx[13]+hx[19]+hx[42];; cc[25]:=hx[43];; 
cc[26]:=hx[44];; cc[27]:=hx[45];; cc[28]:=hx[46];; cc[29]:=hx[47];; cc[30]:=hx[48];; 
Imm:=Im(0,3,4);;
RankMat(Imm);
# 1020  
Append(Imm,cc);
RankMat(Imm);
# 1020  
\end{Verbatim}
\end{footnotesize}
\end{tcolorbox}
Hence, there is a morphism $M^4(-6) \rightarrow H_3$ of graded $A$-modules whose image is the submodule of $H_{3}$ generated by $c'_i, i\in \llbracket 1,8 \rrbracket$. 
Since the dimension vectors of $M^4$ and the submodule of $H_3$ generated by $c'_i, i\in \llbracket 1,8 \rrbracket$ are the same, the previous morphism is injective. 
Moreover, the submodule of $M^4$ generated by $c_i,i\in \llbracket 1,7 \rrbracket$ is isomorphic to $M^2$ via the map given by $c_i\mapsto h_i$ for $i\in \llbracket 1,7 \rrbracket$, and the submodule of $M^4$ generated by $c_8$ is isomorphic to the trivial $A$-module $\Bbbk$. 
It is direct to check that these submodules have trivial intersection, by degree reasons. 
By comparing the Hilbert series of $M^4$, $M^{2}$ and $\Bbbk$ we obtain the isomorphism 
$M^4\cong M^2\oplus \Bbbk$ of graded $A$-modules. 
%Each element of $M^4$ can be written in only one way as the sum of an element of the submodule generated by $c_i,i\in \llbracket 1,7 \rrbracket$ and an element of the submodule generated by $c_8$. 
In consequence, there is an injective morphism $M^2(-6)\oplus \Bbbk (-6)\to H_3$.
By a direct dimension and grading argument using Table \ref{table:Hnm0345}, its cokernel is exactly $\Bbbk (-8)$. 

Finally, we prove that the short exact sequence \eqref{eq:sp2} is non-split. 
Let $\mathsf{c}_i$ for $i\in \llbracket 1,33\rrbracket$ be the basis elements of space $H_{3,5}$ and $p:\operatorname{H}_3(\Bbbk)\to  \Bbbk(-8)$ the surjection in \eqref{eq:sp2}, satisfying that $p(\mathsf{c}_i)=0$ for $i\in\llbracket 1,32\rrbracket$, 
and $p(\mathsf{c}_{33})=\mathsf{e}_1$, where $\mathsf{e}_1$ is the identity element of $\Bbbk(-8)$. 
The short exact sequence \eqref{eq:sp2} is split if and only if there exists a morphism $s:  \Bbbk(-8) \to \operatorname{H}_3(\Bbbk)$ of graded $A$-modules such that the composition $ps$ is the identity map.  
Assume that there exists such a map $s$. 
Let $m=s(\mathsf{e}_1)\in H_{3,5}$. 
Then $m$ is of the form $\sum_{i=1}^{32}\lambda_i\mathsf{c}_i+\mathsf{c}_{33}$ for $\lambda_i\in \Bbbk$, and $m.x=s(\mathsf{e}_1).x=s(\mathsf{e}_1.x)=s(0)=0$ for all $x\in A_{+}$. 
In particular, $\sum_{i=1}^{32}\lambda_i\mathsf{c}_ix_{1,2}+\mathsf{c}_{33}x_{1,2}=0$ for some $\lambda_i\in \Bbbk$, \textit{i.e.} $\mathsf{c}_{33}x_{1,2}$ is a linear combination of $\mathsf{c}_ix_{1,2}$ for $i\in \llbracket 1,32\rrbracket$. 
Using GAP, we choose suitable representative elements $\mathsf{c}'_i\in D_{3,5}$ of $\mathsf{c}_i$ for $i\in \llbracket 1,33\rrbracket$, and get that the dimension of the space spanned by $\mathsf{c}'_ix_{1,2}$ for $i\in \llbracket 1,33\rrbracket$ and elements in $B_{3,6}$, is strictly larger than the dimension of the space spanned by $\mathsf{c}'_ix_{1,2}$ for $i\in \llbracket 1,32\rrbracket$ and elements in $B_{3,6}$, as the following code shows. 
\parskip 1ex
\parindent 0in
\begin{tcolorbox}[breakable,colback=white,width=\textwidth ,center,arc=0mm,size=fbox]
\begin{footnotesize}
\begin{Verbatim}[samepage=false]
gene:=geneMH(0,3,3);;
Uh:=UU(gene,3);; Vh:=VV(gene,3);; Wh:=WW(gene,3);;
hx:=HXR(0,Uh,Vh,Wh,3,3,2);;
hxx:=0*[1..33];;
hxx[1]:=hx[14];; hxx[2]:=hx[15];; hxx[3]:=hx[16];; hxx[4]:=hx[17];; hxx[5]:=hx[18];;
hxx[6]:=hx[19];; hxx[7]:=hx[25];; hxx[8]:=hx[26];; hxx[9]:=hx[27];; hxx[10]:=hx[29];;
hxx[11]:=hx[31];; hxx[12]:=hx[32];; hxx[13]:=hx[39];; hxx[14]:=hx[40];; 
hxx[15]:=hx[41];; hxx[16]:=hx[42];; hxx[17]:=hx[50];; hxx[18]:=hx[51];;
hxx[19]:=hx[58];; hxx[20]:=hx[59];; hxx[21]:=hx[60];; hxx[22]:=hx[61];;
hxx[23]:=hx[65];; hxx[24]:=hx[67];; hxx[25]:=hx[77];; hxx[26]:=hx[78];;
hxx[27]:=hx[79];; hxx[28]:=hx[80];; hxx[29]:=hx[88];; hxx[30]:=hx[89];;
hxx[31]:=hx[91];; hxx[32]:=hx[93];; hxx[33]:=Ker(0,3,5)[79];; 
Imm:=Im(0,3,5);;
RankMat(Imm);
# 1550
Append(Imm, hxx);
RankMat(Imm);
# 1583 
gene:=hxx;;
Uh:=UU(gene,5);; Vh:=VV(gene,5);; Wh:=WW(gene,5);;
cc:=HXR(0,Uh,Vh,Wh,3,5,1);;
cc12:=0*[1..32];;
for i in [1..32] do
    cc12[i]:=cc[6*i-5];
od;
Imm:=Im(0,3,6);;
RankMat(Imm);
# 1890
Append(Imm,cc12);
RankMat(Imm);
# 1910
Append(Imm,[cc[6*33-5]]);
RankMat(Imm);
# 1911
\end{Verbatim}
\end{footnotesize}
\end{tcolorbox}
This shows that $\mathsf{c}_{33}x_{1,2}\neq 0$, and it is not a linear combination of $\mathsf{c}_ix_{1,2}$ for $i\in \llbracket 1,32\rrbracket$, which is a contradiction. 
So, \eqref{eq:sp2} is non-split. 
\end{proof}

%%%%%%%%%%%%%%%%%%%%%%%%%%%%%%%%%%%%%%%%%%%%%%
\subsection{\texorpdfstring{The homology of the Koszul complex of $M^{i}$ for $i \in \{ 1,2,3\}$}{The homology of the Koszul complex of Mi for i equal to 1,2,3}}
\label{subsec:cpx of H4bis}

For a quadratic $A$-module $M$, 
we write the quadratic dual module $M^{!_{m}}$ simply by $M^!$.
Let $(\operatorname{K}_{\bullet}(M),d_{\bullet}(M))$ be the Koszul complex of $M$. 
Let $\operatorname{K}_{n,m}(M)=(M^!_{-n})^*\otimes A_{m}$, 
$d_{n,m}(M)=d_n(M)|_{\operatorname{K}_{n,m}(M)}:\operatorname{K}_{n,m}(M)\to \operatorname{K}_{n-1,m+1}(M)$,
$B^{M}_{n,m}=\Img ( d_{n+1,m-1}(M) )$, 
$D^{M}_{n,m}=\Ker ( d_{n,m}(M) )$
and 
$\operatorname{H}_{n,m}(M)= D^{M}_{n,m}/ B^{M}_{n,m}$ for $n\in\NN_0$ and $m\in \llbracket 0, 12 \rrbracket $. 
Let $\operatorname{H}_{n}(M)=\oplus_{m\in \llbracket 0, 12 \rrbracket} \operatorname{H}_{n,m}(M) $
for $n\in\NN_0$. 
Using GAP, we can also compute the dimension of the $\operatorname{H}_{n,m}(M^i)$ for $n$ less than some arbitrary positive integer, $m\in\llbracket 1,12\rrbracket$ and $i\in\llbracket 1,3\rrbracket$. 

%%%%%%%%%%%%
\subsubsection{\texorpdfstring{Homology of the Koszul complex of $M^1$}{Homology of the Koszul complex of M1}}
\label{subsubsec:HnM1}

In this subsubsection, we compute $\operatorname{H}_n(M^1)$ for all $n\in \NN_0$. 

%%%%%%%%%%%%%%%%%%%%%%%%%%%%%%%%%%%%%%%%%%%%%%%%%%%%%%%%%%%%%%%%%%%%
\paragraph{The dimensions of the homology groups}
\phantom{x}
\\
Recall that $M^{1}= (W \otimes A)/(I)$, where $W$ is the $2$-dimensional vector space spanned by $a_1,a_2$, and $I$ is the subspace of $W\otimes V$ spanned by \eqref{eq: relations of H4}.
The quadratic dual $(M^{1})^!=\oplus_{n\in \NN_0} (M^{1})^{!}_{-n}=(U \otimes A^! )/ (J)$ of $M^{1}$ is an $A^!$-module, where $U$ is the $2$-dimensional vector space spanned by $b_1,b_2$ (for $\{ b_1,b_2\}$ the dual basis to $\{ a_1,a_2 \}$), and $J$ is the subspace of $U\otimes V^*$ spanned by 
\begin{equation}
   \label{eq: relations of dual H4}
   \begin{split}
\{ b_1y_{1,2}-b_2y_{1,2}, b_2y_{1,3}, b_1y_{2,3}, b_1y_{1,4}, b_2y_{2,4}, b_1y_{3,4}-b_2y_{3,4}   \}. 
   \end{split}
\end{equation}

\begin{lem}
\label{lem:productss}
Recall that $\B^!=\cup_{n\in\NN_0}\B^!_{n}$ is the basis of $A^!$. 
Let $u,v\in \B^!$ and 
\begin{equation}
\label{eq:Yij}
Y_{1,2}=\{ \pm y_{1,2}^{r_1}y_{3,4}^{r_2} |r_1,r_2\in \NN_0 \},
Y_{1,3}=\{ \pm y_{1,3}^{r_1}y_{2,4}^{r_2} |r_1,r_2\in \NN_0 \},
Y_{2,3}=\{ \pm y_{2,3}^{r_1}y_{1,4}^{r_2} |r_1,r_2\in \NN_0 \}.
\end{equation}
If $uv \in Y_{i,j}$ for $(i,j)\in \II$, then $u,v\in Y_{i,j}$. 
\end{lem}

\begin{proof}
We will prove the lemma by induction on the degree of $v$. 
Let $u\in \B^!_{m}$ and $v\in \B^!_{n}$ for $m,n\in \NN_0$. 
Obviously, the lemma holds for $n=0$ and $m\in\NN_0$. 
Assume that $v=v'y$ for $y\in \{y_{s,t}|(s,t)\in \I \}$ and $v'\in \B^!_{n-1}$.
Note that $uv'=\pm c$, where $c\in \B^!_{m+n-1}$. 
By Tables \ref{table:yy} - \ref{table:yn456} together with \eqref{eq:product ijkl} and \eqref{eq:y12y34y13}, $cy\in Y_{i,j}$ implies that $c,y\in Y_{i,j}$. 
Then, by induction hypothesis we get that $u, v'\in Y_{i,j}$. 
In consequence, $v=v'y\in Y_{i,j}$, as was to be shown.
\end{proof}

\begin{lem}
Set $T_n=\{ b_1y_{1,2}^ky_{3,4}^{n-k}, b_1y_{1,3}^ky_{2,4}^{n-k}, b_2y_{2,3}^ky_{1,4}^{n-k}  \mid k\in \llbracket 0, n \rrbracket \} \subseteq (M^{1})^!_{-n}$ for $n \in \NN_{0}$. 
Note that $T_{n}$ has cardinal $3(n+1)$ for $n \in \NN$, and cardinal $2$ for $n = 0$, since $T_0= \{ b_1,b_2 \}$. 
Then, $T_{n}$ is a basis of the space $(M^{1})^!_{-n}$ for $n\in \NN_{0}$.  
\end{lem}

\begin{proof}  
Note that the space $(M^{1})^!_{-n}$ is spanned by $\{ b_1y,b_2y\mid y\in \B^!_{n} \}$ for $n\in \NN_0$. 
It is easy to check that  
\begin{align*}
b_jy_{1,2}^{m}y_{1,3}&= \chi_{m}b_1y_{2,3}^{m}y_{1,3}-\chi_{m+1} b_1 y_{2,3}^{m}y_{1,2}=0, 
\\
b_jy_{1,2}^{m}y_{2,3}&= \chi_{m}b_2y_{1,3}^{m}y_{2,3}-\chi_{m+1} b_2y_{1,3}^{m}y_{1,2}=0, 
\\
b_jy_{1,2}^{m}y_{1,4}&= \chi_{m}b_2y_{2,4}^{m}y_{1,4}-\chi_{m+1} b_2y_{2,4}^{m}y_{1,2}=0, 
\\
b_jy_{1,2}^{m}y_{2,4}&= \chi_{m}b_1y_{1,4}^{m}y_{2,4}-\chi_{m+1} b_1 y_{1,4}^{m}y_{1,2}=0, 
\\
b_1y_{1,3}^{m}y_{1,4}&= \chi_m b_1 y_{3,4}^m y_{1,4}+\chi_{m+1}b_1y_{1,4}^my_{3,4} 
= \chi_m b_2 y_{3,4}^m y_{1,4}+\chi_{m+1}b_1y_{1,4}^my_{3,4} 
\\
& = \chi_m b_2 y_{1,3}^m y_{1,4}+\chi_{m+1}b_1y_{1,4}^my_{3,4} 
=0, 
\\
b_1y_{1,3}^{m}y_{3,4}&= \chi_m b_1 y_{1,4}^m y_{3,4}-\chi_{m+1}b_1y_{1,4}^my_{1,3}=0,
\\
b_2y_{2,3}^{m}y_{2,4}&= \chi_m b_2 y_{3,4}^m y_{2,4}+\chi_{m+1}b_2y_{2,4}^my_{3,4} 
= \chi_m b_1 y_{3,4}^m y_{2,4}+\chi_{m+1}b_2y_{2,4}^my_{3,4} 
\\
& = \chi_m b_1 y_{2,3}^m y_{2,4}+\chi_{m+1}b_2y_{2,4}^my_{3,4} =0, 
\\
b_2y_{2,3}^{m}y_{3,4}&= \chi_m b_2 y_{2,4}^m y_{3,4}-\chi_{m+1}b_2y_{2,4}^my_{2,3}=0, 
\end{align*}
for $j\in \llbracket 1,2 \rrbracket$ and $m\in \NN$.
Together with \eqref{eq: relations of dual H4}, we get that the space $(M^{1})^!_{-n}$ is spanned by $T_n$ for $n\in \NN_{0}$. 

It is clear that $T_{0}$ is linearly independent. 
Next, we prove that the elements in $T_n$ are linearly independent for $n\in\NN$.
Suppose that 
\begin{equation}
   \label{eq:linearly independent 1}
\begin{split}
   \sum_{k\in \llbracket 0,n \rrbracket } \alpha_k b_1y_{1,2}^ky_{3,4}^{n-k}
   + \sum_{k\in \llbracket 0,n \rrbracket } \beta_k b_1y_{1,3}^ky_{2,4}^{n-k} 
   + \sum_{k\in \llbracket 0,n \rrbracket } \gamma_k b_2y_{2,3}^ky_{1,4}^{n-k} 
   = 0 
\end{split}
\end{equation}
in $(M^{1})^!_{-n}$, 
where $\alpha_k,\beta_k,\gamma_k\in \Bbbk$ for $k\in \llbracket 0,n \rrbracket$. 
Then 
\begin{equation}
\label{eq:linearly independent 2}
\begin{split}
& \phantom{= \;} \sum_{k\in \llbracket 0,n \rrbracket } \alpha_k b_1y_{1,2}^ky_{3,4}^{n-k}
+ \sum_{k\in \llbracket 0,n \rrbracket } \beta_k b_1y_{1,3}^ky_{2,4}^{n-k} 
+ \sum_{k\in \llbracket 0,n \rrbracket } \gamma_k b_2y_{2,3}^ky_{1,4}^{n-k} 
\\
& = \sum_{u\in \B^!_{n-1}} \lambda_{1,u} ( b_1y_{1,2}-b_2y_{1,2}) u
+ \sum_{u\in \B^!_{n-1}} \lambda_{2,u} b_2y_{1,3} u  
+ \sum_{u\in \B^!_{n-1}} \lambda_{3,u} b_1y_{2,3} u  
\\
& \phantom{= \;} 
+ \sum_{u\in \B^!_{n-1}} \lambda_{4,u} b_1y_{1,4} u 
+ \sum_{u\in \B^!_{n-1}} \lambda_{5,u} b_2y_{2,4} u 
+ \sum_{u\in \B^!_{n-1}} \lambda_{6,u} (b_1y_{3,4}-b_2y_{3,4} ) u 
\in U \otimes A^!, 
\end{split}
\end{equation}
where $\lambda_{i,u}\in \Bbbk$ for $i\in \llbracket 1,6\rrbracket$ and $u\in \B^!_{n-1}$. 
So, 
\begin{equation}
   \label{eq:linearly independent 3}
\begin{split}
\sum_{k\in \llbracket 0,n \rrbracket } \gamma_k b_2y_{2,3}^ky_{1,4}^{n-k} & = 
-\sum_{u\in \B^!_{n-1}} \lambda_{1,u} b_2y_{1,2} u
+ \sum_{u\in \B^!_{n-1}} \lambda_{2,u} b_2y_{1,3} u  
+ \sum_{u\in \B^!_{n-1}} \lambda_{5,u} b_2y_{2,4} u 
\\
& \phantom{= \; }
- \sum_{u\in \B^!_{n-1}} \lambda_{6,u} b_2y_{3,4}  u 
\in \Bbbk \{b_2 \} \otimes A^! \cong A^!
\end{split}
\end{equation}
and 
\begin{equation}
   \label{eq:linearly independent 4}
\begin{split}
& \phantom{= \;} 
\sum_{k\in \llbracket 0,n \rrbracket } \alpha_k b_1y_{1,2}^ky_{3,4}^{n-k}
+ \sum_{k\in \llbracket 0,n \rrbracket } \beta_k b_1y_{1,3}^ky_{2,4}^{n-k} 
\\  
& 
=
\sum_{u\in \B^!_{n-1}} \lambda_{1,u}  b_1y_{1,2} u 
+ \sum_{u\in \B^!_{n-1}} \lambda_{3,u} b_1y_{2,3} u  
+ \sum_{u\in \B^!_{n-1}} \lambda_{4,u} b_1y_{1,4} u 
+ \sum_{u\in \B^!_{n-1}} \lambda_{6,u} b_1y_{3,4} u 
\\
&
\in \Bbbk\{ b_1 \}\otimes A^! \cong A^!.
\end{split}
\end{equation}
Lemma \ref{lem:productss} and \eqref{eq:linearly independent 3} imply that 
%each term of right side of \eqref{eq:linearly independent 3} is not of the form $c b_1.y_{1,3}^ky_{2,4}^{n-1-k}$, where $0\neq c\in\Bbbk$.
\begin{equation}
   \label{eq:linearly independent 5}
\begin{split}
\sum_{k\in \llbracket 0,n \rrbracket } \gamma_k y_{2,3}^ky_{1,4}^{n-k}  = 
\sum_{u\in \B^!_{n-1}\cap Y_{1,2}} \lambda_{1,u}  y_{1,2} u 
+ \sum_{u\in \B^!_{n-1}\cap Y_{1,2}} \lambda_{6,u} y_{3,4} u = 0
\end{split}
\end{equation}
in $A^!$, whereas  
Lemma \ref{lem:productss} and \eqref{eq:linearly independent 4} imply that
\begin{equation}
   \label{eq:linearly independent 6}
\begin{split}
\sum_{k\in \llbracket 0,n \rrbracket } \alpha_k y_{1,2}^ky_{3,4}^{n-k} = 
\sum_{u\in \B^!_{n-1}\cap Y_{1,2}} \lambda_{1,u} y_{1,2} u 
+ \sum_{u\in \B^!_{n-1}\cap Y_{1,2}} \lambda_{6,u} y_{3,4}  u ,
\hspace*{3mm}
\sum_{k\in \llbracket 0,n \rrbracket } \beta_k y_{1,3}^ky_{2,4}^{n-k} =0
\end{split}
\end{equation}
in $A^!$.
Hence, $\alpha_k=\beta_k=\gamma_k=0$ for $k\in \llbracket 0,n \rrbracket$. 
The lemma is thus proved. 
\end{proof}

% \begin{rk}
% Under the order $b_1 \succ b_2$, a Gröbner basis $G$ of the submodule $(J)$ is given by the following elements 
% \begin{equation} \begin{split}
% & b_1.y_{1,2}-b_2.y_{1,2}, b_2.y_{1,3}, b_1.y_{2,3}, b_1.y_{1,4}, b_2.y_{2,4}, b_1.y_{3,4}-b_2.y_{3,4}, b_2.y_{1,2}^{m}y_{1,3}, b_2.y_{1,2}^{m}y_{2,3}, b_2.y_{1,2}^{m}y_{1,4}, \\
% & b_2.y_{1,2}^{m}y_{2,4}, b_2.y_{2,3}^{m}y_{2,4}, b_2.y_{2,3}^{m}y_{3,4}, b_1.y_{1,3}^{m}y_{1,4}, b_1.y_{1,3}^{m}y_{3,4},
% \end{split} \end{equation}
% where $m\in\NN$, 
% \textit{i.e.} $\{ \Lt(u)|0\neq u\in(J) \}=\{ \Lt(g)w|g\in G, w\in \B^! \}$ in $U\otimes A^!$, where $\Lt(x)$ is the leading term (with the coefficient) of $x\in U\otimes A^! $.
% \end{rk}

Given $n \in \NN$, we will denote by $T_{n}^{*} = \{ x^{*} \mid x \in T_{n} \}$ the dual basis of $T_{n}$. 
Note that the differential $d_{1}(M^{1}): \operatorname{K}_1(M^{1})\to \operatorname{K}_{0}(M^{1})$ is given by 
\begin{equation}
\begin{split}
   (b_1 y_{1,2})^*|1 & \mapsto b_1^*|x_{1,2}+b_2^*|x_{1,2}, 
   (b_1 y_{1,3})^*|1 \mapsto b_1^*|x_{1,3},
   (b_1 y_{2,4})^*|1 \mapsto b_1^*|x_{2,4},
   \\
   (b_1 y_{3,4})^*|1 & \mapsto b_1^*|x_{3,4}+b_2^*|x_{3,4}, 
   (b_2 y_{2,3})^*|1 \mapsto b_2^*|x_{2,3},
   (b_2 y_{1,4})^*|1 \mapsto b_2^*|x_{1,4},
\end{split}
\end{equation}
where $b_sy_{i,j}\in T_1$ and $(b_sy_{i,j})^* \in T_{1}^{*}$ is the dual element of $b_sy_{i,j}$.
The differential $d_n(M^{1}):\operatorname{K}_n(M^{1})\to \operatorname{K}_{n-1}(M^{1})$ for $n\geqslant 2$ is given by 
\begin{equation}
\label{eq:differential H4}
\begin{split}
(b_sy_{i,j}^{n-r}y_{k,l}^{r})^* |1 \mapsto (-1)^r (b_s y_{i,j}^{n-1-r}y_{k,l}^{r})^*|x_{i,j}+(b_sy_{i,j}^{n-r}y_{k,l}^{r-1})^*|x_{k,l}, 
\end{split}
\end{equation}
where $s\in\llbracket 1,2 \rrbracket $, 
$r\in \llbracket 0,n\rrbracket $, 
$(i,j)\in \II$, $(k,l)\in \I$ with $\# \{i,j,k,l \}=4$, 
$b_sy_{i,j}^{n-r}y_{k,l}^{r}\in T_n $ 
and $x^*\in ((M^{1})^!_{-n})^* \in T_{n}^{*}$ is the dual element of $x\in T_{n} \subseteq (M^{1})^!_{-n}$.  

%%%%
\begin{prop}
\label{prop:dim H M1}
We have $\dim \operatorname{H}_{n}(M^{1})=0$ for integers $n \geqslant 2$. 
\end{prop}
%%%
\begin{proof}
It is clear that there is an isomorphism $((M^{1})^!_{-n})^*\otimes A \to \Bbbk \C_{n} \otimes A$ of chain complex of graded $A$-modules given by 
$ (b_sy_{i,j}^{n-r}y_{k,l}^{r})^* |x \mapsto  z^{i,j}_{n-r}z^{k,l}_{r} |x $, 
where $x\in A$ and $n\in \NN$. 
So, $\dim B^{M^{1}}_{n,m}=\dim C_{n,m}$ for 
$m\in \llbracket 0, 12 \rrbracket $ and $n\in \NN$, where $\dim C_{n,m}$ is given by Lemma \ref{lem:dim changing in image}. 
The result now follows from the fact that the Koszul complex $\K_{\bullet}(M^{1})$ is isomorphic to the complex $\Bbbk \C_{\bullet} \otimes A$ for $\bullet \in\NN$ and Remark \ref{rk:exact of subcomplex kCA}.
\end{proof}

%%%%%%%%%%%%%
\begin{cor}
\label{cor:dim BMAnm H4}
The dimension of $B^{M^{1}}_{n,m}$ for $n\in\NN_0 $ and $m\in\llbracket 0,12 \rrbracket$ is given by 
\begin{table}[H]
   \begin{center}
     \resizebox{\textwidth}{5.5mm}{
      \begin{tabular}{|c|ccccccccccccc|}
      \hline
       \diagbox[width=14mm,height=5mm]{$n$ }{$m$}  & $0$ & $1$  & $2$ & $3$ & $4$ & $5$ & $6$ & $7$ & $8$ & $9$ & $10$ & $11$ & $12$
     \\
     \hline
     $n=0$ & $0$ & $6$ & $27$ & $72$ & $131$ & $186$ & $210$ & $192$ & $142$ & $84$ & $38$ & $12$  & $2 $
     \\
     $n\in \NN $ & $0$ & $3n+6$ & $15n+27$ & $42n+72$ & $84n+138$ & $129n+204$ & $159n+243$ & $159n+234$ & $129n+183$ & $84n+114$ & $42n+54$ & $15n+18$  & $3n+3$
     \\
      \hline
   \end{tabular}
  }
  \vspace{1mm}
   \caption{Dimension of $B^{M^{1}}_{n,m}$.}	
   \label{table:dim BMAnm H4}
\end{center}
\end{table}
\vspace{-0.8cm}
\end{cor}

\begin{proof}
The last row of Table \ref{table:dim BMAnm H4} follows from Lemma \ref{lem:dim changing in image}, since $\dim B^{M^{1}}_{n,m}=\dim C_{n,m}$ for $n\in \NN$ and $m \in \llbracket 0, 12 \rrbracket$, 
as explained in the proof of Proposition \ref{prop:dim H M1}. 
For the remaining case, note that $\dim B^{M^{1}}_{0,m}=\dim D^{M^{1}}_{0,m}= \dim (((M^{1})^!_{0})^*\otimes A_{m})- \dim H_{4,m+4}
= 2\dim A_m-\dim H_{4,m+4}$ %by the exactness in Proposition \ref{prop:dim H M1} 
for $m \in \llbracket 0, 12 \rrbracket$. 
The result now follows. 
\end{proof}

\begin{cor}
\label{cor:dim D of Koszul cpx of H4}
The dimension of $D^{M^{1}}_{1,m}$ for $m\in\llbracket 0,12 \rrbracket$ is given by 
\begin{table}[H]
   \begin{center}
     %\resizebox{\textwidth}{5.5mm}{
      \begin{tabular}{|c|ccccccccccccc|}
      \hline
       $m$  & $0$ & $1$  & $2$ & $3$ & $4$ & $5$ & $6$ & $7$ & $8$ & $9$ & $10$ & $11$ & $12$
     \\
     \hline
     $\dim D^{M^{1}}_{1,m}$ & $0$ & $9$ & $42$ & $121$ & $240$ & $366$ & $444$ & $434$ & $342$ & $214$ & $102$ & $34$  & $6 $
     \\
      \hline
   \end{tabular}
  %}
  \vspace{1mm}
   \caption{Dimension of $D^{M^{1}}_{1,m}$.}	
   \label{table:dim D of Koszul cpx of H4}
\end{center}
\end{table}
\vspace{-0.8cm}
Hence, the dimension of $\operatorname{H}_{1,m}(M^{1})$ for $m\in\llbracket 0,12 \rrbracket$ is exactly given in Table \ref{table:dim Koszul cpx of M1}, by $\dim \operatorname{H}_{1,m}(M^{1}) = \dim D^{M^{1}}_{1,m}- \dim B^{M^{1}}_{1,m}$. 
In particular, $\dim \operatorname{H}_{1}(M^{1})=194$. 
\end{cor}

\begin{proof}
The result follows directly from $\dim D^{M^{1}}_{1,m}= \dim (((M^{1})^!_{-1})^*\otimes A_{m})- \dim B^{M^{1}}_{0,m+1} 
= 6\dim A_m-\dim B^{M^{1}}_{0,m+1}$ for $m \in \llbracket 0, 12 \rrbracket$, together with Corollary \ref{cor:dim BMAnm H4}.
\end{proof}

We can also use GAP to get the dimension of the homology $\operatorname{H}_n(M^1)$ of the Koszul complex of $M^1$ for $n$ less than some positive integer. 
In particular, using Appendix \ref{sec:cpx} and the following routine in GAP 
\parskip 1ex
\parindent 0in
\begin{tcolorbox}[breakable,colback=white,width=\textwidth ,center,arc=0mm,size=fbox]
\begin{footnotesize}
\begin{Verbatim}[samepage=false]
for j in [0..8] do
    for i in [1..12] do
        Print(j, " ", i, " ", RankMat(FF(1,j,i)), "\n");
    od;
od;
\end{Verbatim}
\end{footnotesize}
\end{tcolorbox}
we obtain the dimension of $B^{M^1}_{n,m}$ for $n\in \llbracket 0,8\rrbracket$ and $m\in\llbracket 1,12\rrbracket$. 
The homology of the Koszul complex of $M^{1}$ for $n = 1$ and $m\in\llbracket 0,12 \rrbracket $ is given in Table \ref{table:dim Koszul cpx of M1}. 
The dimensions that are not listed in the following table are zeros. 
\begin{table}[H]
   \begin{center}
     %\resizebox{\textwidth}{5.5mm}{
      \begin{tabular}{|c|ccccccccccccc|}
      \hline
       \diagbox[width=14mm,height=5mm]{$n$}{$m$}   & $0$ & $1$  & $2$ & $3$ & $4$ & $5$ & $6$ & $7$ & $8$ & $9$ & $10$ & $11$ & $12$
     \\
     \hline
     $1$  &  &  &  & $7$ & $18$ & $33$ & $42$ & $41$ & $30$ & $16$ & $6$ & $1$  & 
     \\
      \hline
   \end{tabular}
  %}
  \vspace{1mm}
   \caption{Dimension of $\operatorname{H}_{n,m}(M^1)$.}	
   \label{table:dim Koszul cpx of M1}
\end{center}
\end{table}
\vspace{-0.8cm}

%%%%%%%%%%%%%%%%%%%%%%%%%%%%%%%%%%%%%%%%%%%%%%%%%%%%%%%%%%%%%%%%%%%%
\paragraph{\texorpdfstring{The $A$-module structure of the homology groups}{The A-module structure of the homology groups}}

%%%%%%%
\begin{lem}
\label{lemma:homology-m1}
We have the isomorphism
\begin{equation}
   \label{eq:HK M}
   \operatorname{H}_{n}(M^1) =
   %\begin{cases} 
   %   \operatorname{H}_{1}(M^1), & \text{if $n = 1$,}
   %   \\
      0 %& \text{if $n\geqslant 2$,}
   %\end{cases}
\end{equation}
of graded $A$-modules for $n\geqslant 2$, as well as
the non-split short exact sequence of graded $A$-modules of the form
\begin{equation}
   \label{eq:sp1}
   \begin{split}
   0 \to M^2(-4) \to \operatorname{H}_1(M^1) \to  \Bbbk(-6)\oplus \Bbbk(-8) \to 0. 
\end{split} 
\end{equation}
\end{lem}
%%%%%%%
\begin{proof}
The isomorphism in \eqref{eq:HK M} for $n \in \NN \setminus \{ 1 \}$ follows from Proposition \ref{prop:dim H M1}. 
It remains to show the existence of the non-split short exact sequence. 

The following GAP code shows that the dimension vector of the $A$-submodule of $\operatorname{H}_{1}(M^1)$ generated by basis elements $h'_i, i\in \llbracket 1,7 \rrbracket$ of $\operatorname{H}_{1,3}(M^{1})$ is $(7,18, 32, 42, 40, 30, 16, 6, 1)$.  
\parskip 1ex
\parindent 0in
\begin{tcolorbox}[breakable,colback=white,width=\textwidth ,center,arc=0mm,size=fbox]
\begin{footnotesize}
\begin{Verbatim}[samepage=false]
Imm:=Im(1,1,3);;
RankMat(Imm);
# 114
gene:=geneMH(1,1,3);;
Append(Imm,gene);
RankMat(Imm);
# 121
Uh:=UU(gene,3);; Vh:=VV(gene,3);; Wh:=WW(gene,3);;
for r in [4..11] do 
   hxr:=HXR(1,Uh,Vh,Wh,1,3,r-3);
   Im1r:=Im(1,1,r);
   Append(Im1r, hxr);
   Print(r, " ", RankMat(Im1r)-RankMat(Im(1,1,r)), "\n");
od;
# 4 18
# 5 32
# 6 42
# 7 40
# 8 30
# 9 16
# 10 6
# 11 1
\end{Verbatim}
\end{footnotesize}
\end{tcolorbox}
Moreover, it is direct to check that the elements $h'_i, i\in \llbracket 1,7 \rrbracket$ of $\operatorname{H}_{1}(M^1)$ satisfy the quadratic relations \eqref{eq:rels M2} defining $M^2$. 
Indeed, the following code shows that the dimension of the subspace generated by $B_{1,4}^{M^1}$ together with the elements of the form \eqref{eq:rels M2} with $h'_i$ instead of $h_i$ coincides with the dimension of $B_{1,4}^{M^1}$.
\parskip 1ex
\parindent 0in
\begin{tcolorbox}[breakable,colback=white,width=\textwidth ,center,arc=0mm,size=fbox]
\begin{footnotesize}
\begin{Verbatim}[samepage=false]
gene:=geneMH(1,1,3);;
Uh:=UU(gene,3);; Vh:=VV(gene,3);; Wh:=WW(gene,3);;
hx:=HXR(1,Uh,Vh,Wh,1,3,1);;
cc:=0*[1..24];;
cc[1]:=hx[1];; cc[2]:=hx[2];; cc[3]:=hx[3];; cc[4]:=hx[7];; cc[5]:=hx[10];; 
cc[6]:=hx[11];; cc[7]:=hx[14];; cc[8]:=hx[16];; cc[9]:=hx[18];; cc[10]:=hx[21];;
cc[11]:=hx[23];; cc[12]:=hx[24];; cc[13]:=hx[5]-hx[17]-hx[26];; 
cc[14]:=hx[8]-hx[20]+hx[29];; cc[15]:=hx[30]-hx[31];; cc[16]:=hx[4]-hx[22]+hx[33];; 
cc[17]:=hx[9]-hx[15]-hx[34];; cc[18]:=hx[25]+hx[36];; cc[19]:=hx[6]+hx[12]+hx[37];;
cc[20]:=hx[35]+hx[38];; cc[21]:=hx[28]+hx[39];; cc[22]:=hx[27]-hx[40];;
cc[23]:=hx[32]-hx[41];; cc[24]:=hx[13]+hx[19]+hx[42];; 
Imm:=Im(1,1,4);;
RankMat(Imm);
# 222
Append(Imm,cc);
RankMat(Imm);
# 222
\end{Verbatim}
\end{footnotesize}
\end{tcolorbox}
Hence, there is a surjective morphism from $M^2(-4)$ to the submodule of $\operatorname{H}_{1}(M^1)$ generated by $h'_i, i\in \llbracket 1,7 \rrbracket$, which is an isomorphism of graded $A$-modules since the dimension vector of $M^{2}$ is also $(7, 18, 32, 42, 40, 30, 16, 6, 1)$.
Namely, there is an injective morphism $M^{2}(-4)\to \operatorname{H}_{1}(M^1)$ of graded modules.
%The space $H(M)_{1,5}$ is spanned by $h_i.A_2$ and $h_8$. 
%The dimension vector of the submodule of $H(M)_1$ generated by $h_8$ is $1$, $6$, $19$, $30$, $16$, $6$, $1$. 
%The space $H(M)_{1,7}$ is spanned by $h_i.A_4$ and $h_9$. 
%The dimension vector of the submodule of $H(M)_1$ generated by $h_9$ is $1$, $3$, $4$, $3$, $1$. 
%Moreover, $h_8.A_{+}\subseteq M_{2}$. 
A simple argument using dimensions and grading together with Table \ref{table:dim Koszul cpx of M1} tells us that the cokernel of this injective morphism is exactly the graded $A$-module $\Bbbk(-6)\oplus \Bbbk(-8)$, as was to be shown. 

We finally show that \eqref{eq:sp1} is non-split. 
Let $\mathsf{c}_i$ for $i\in \llbracket 1,33\rrbracket$ be the basis elements of space $\operatorname{H}_{1,5}(M^1)$ and $p:\operatorname{H}_1(M^1)\to \Bbbk(-6)\oplus \Bbbk(-8)$ the surjection in \eqref{eq:sp1}, satisfying that $p(\mathsf{c}_i)=0$ for $i\in\llbracket 1,32\rrbracket$, 
and $p(\mathsf{c}_{33})=\mathsf{e}_1$, where $\mathsf{e}_1$ is the identity element of $\Bbbk(-6)$. 
The short exact sequence \eqref{eq:sp1} is split if and only if there exists a morphism $s:  \Bbbk(-6)\oplus \Bbbk(-8) \to \operatorname{H}_1(M^1)$ of graded $A$-modules such that the composition $ps$ is the identity map. 
Assume there is such a map $s$. 
Let $m=s(\mathsf{e}_1)\in \operatorname{H}_{1,5}(M^1)$. 
Then $m$ is of the form $\sum_{i=1}^{32}\lambda_i\mathsf{c}_i+\mathsf{c}_{33}$ for $\lambda_i\in \Bbbk$, and $m.x=s(\mathsf{e}_1).x=s(\mathsf{e}_1.x)=s(0)=0$ for all $x\in A_{+}$. 
In particular, $\sum_{i=1}^{32}\lambda_i\mathsf{c}_i x_{1,2}+\mathsf{c}_{33} x_{1,2}=0$ for some $\lambda_i\in \Bbbk$, \textit{i.e.} $\mathsf{c}_{33}x_{1,2}$ is a linear combination of $\mathsf{c}_ix_{1,2}$ for $i\in \llbracket 1,32\rrbracket$. 
Using GAP, we choose suitable representative elements $\mathsf{c}'_i\in D^{M^1}_{1,5}$ of $\mathsf{c}_i$ for $i\in \llbracket 1,33\rrbracket$, and get that the dimension of the space spanned by $\mathsf{c}'_ix_{1,2}$ for $i\in \llbracket 1,33\rrbracket$ and elements in $B^{M^1}_{1,6}$, is strictly larger than the dimension of the space spanned by $\mathsf{c}'_ix_{1,2}$ for $i\in \llbracket 1,32\rrbracket$ and elements in $B^{M^1}_{1,6}$, as the following code shows.  
\parskip 1ex
\parindent 0in
\begin{tcolorbox}[breakable,colback=white,width=\textwidth ,center,arc=0mm,size=fbox]
\begin{footnotesize}
\begin{Verbatim}[samepage=false]
gene:=geneMH(1,1,3);;
Uh:=UU(gene,3);; Vh:=VV(gene,3);; Wh:=WW(gene,3);;
hx:=HXR(1,Uh,Vh,Wh,1,3,2);;
hxx:=0*[1..33];;
hxx[1]:=hx[14];; hxx[2]:=hx[15];; hxx[3]:=hx[16];; hxx[4]:=hx[17];; hxx[5]:=hx[18];;
hxx[6]:=hx[19];; hxx[7]:=hx[25];; hxx[8]:=hx[26];; hxx[9]:=hx[27];; hxx[10]:=hx[29];;
hxx[11]:=hx[31];; hxx[12]:=hx[32];; hxx[13]:=hx[39];; hxx[14]:=hx[40];; 
hxx[15]:=hx[41];; hxx[16]:=hx[42];; hxx[17]:=hx[50];; hxx[18]:=hx[51];;
hxx[19]:=hx[58];; hxx[20]:=hx[59];; hxx[21]:=hx[60];; hxx[22]:=hx[61];;
hxx[23]:=hx[65];; hxx[24]:=hx[67];; hxx[25]:=hx[77];; hxx[26]:=hx[78];;
hxx[27]:=hx[79];; hxx[28]:=hx[80];; hxx[29]:=hx[88];; hxx[30]:=hx[89];;
hxx[31]:=hx[91];; hxx[32]:=hx[93];; hxx[33]:=Ker(1,1,5)[76];;
Imm:=Im(1,1,5);;
RankMat(Imm);
# 333
Append(Imm, hxx);
RankMat(Imm);
# 366 
gene:=hxx;;
Uh:=UU(gene,5);; Vh:=VV(gene,5);; Wh:=WW(gene,5);;
cc:=HXR(1,Uh,Vh,Wh,1,5,1);;
cc12:=0*[1..32];;
for i in [1..32] do
    cc12[i]:=cc[6*i-5];
od;
Imm:=Im(1,1,6);;
RankMat(Imm);
# 402
Append(Imm,cc12);
RankMat(Imm);
# 422
Append(Imm,[cc[6*33-5]]);
RankMat(Imm);
# 423 
\end{Verbatim}
\end{footnotesize}
\end{tcolorbox}
This shows that $\mathsf{c}_{33}x_{1,2}\neq 0$, and it is not a linear combination of $\mathsf{c}_ix_{1,2}$ for $i\in \llbracket 1,32\rrbracket$, which is a contradiction. 
So, \eqref{eq:sp1} is non-split. 
\end{proof}

%%%%%%%%%%%%
\subsubsection{\texorpdfstring{Homology of the Koszul complex of $M^2$}{Homology of the Koszul complex of M2}}
\label{subsubsec:HnM2}

\paragraph{The dimensions of the homology groups}
\phantom{x}
\\
Recall the definition of the quadratic module $M^{2}$ given in Subsection \ref{subsec:Resolving datum fk4}. 
Let $\{ g_i \mid i \in\llbracket 1,7 \rrbracket\}$ be the dual basis to the basis $\{ h_i \mid i \in\llbracket 1,7 \rrbracket\}$ of the space of generators of $M^{2}$. 
Then, it is easy to see that the $A^!$-module $(M^2)^!$ is generated by $g_i,i\in\llbracket 1,7 \rrbracket$, subject to the following $18$ relations  
\begin{equation}
\label{eq:dual rels M2}
\begin{split}
   &
   g_1y_{3,4}-g_2y_{3,4}, 
   g_3y_{1,2}-g_4y_{1,2}, 
   g_5y_{1,2}-g_6y_{3,4}, 
   g_5y_{3,4}+g_6y_{1,2},
   g_1y_{3,4}-g_7y_{1,2}, 
   g_3y_{1,2}-g_7y_{3,4},
   \\
   &
   g_1y_{2,4}+g_3y_{2,4}, 
   g_2y_{1,3}+g_4y_{1,3}, 
   g_6y_{1,3}+g_7y_{2,4}, 
   g_6y_{2,4}-g_7y_{1,3},
   g_1y_{2,4}+g_5y_{1,3}, 
   g_2y_{1,3}-g_5y_{2,4}, 
   \\
   &
   g_1y_{1,4}+g_4y_{1,4}, 
   g_2y_{2,3}+g_3y_{2,3}, 
   g_5y_{2,3}+g_7y_{1,4}, 
   g_5y_{1,4}-g_7y_{2,3}, 
   g_1y_{1,4}-g_6y_{2,3}, 
   g_2y_{2,3}+g_6y_{1,4}.
\end{split}
\end{equation}

Using GAP we get the basis of $(M^2)^!_{-n}$ for $n\in \llbracket 0,3\rrbracket$ given in Appendix \ref{Appendix:basis m2 d}.
Let $\U^{!,M^2}_n$ be the subset of $(M^2)^!_{-n} $ consisting of the following $24$ elements
\begin{equation}
\label{eq: basis1 M2}
\begin{split}
   & 
   g_1y_{1,2}^{n-1}y_{1,3},
    g_1y_{1,2}^{n-1}y_{2,3},
    g_1y_{1,2}^{n-1}y_{1,4},
    g_1y_{1,2}^{n-1}y_{2,4}, 
    g_1y_{1,2}^{n-2}y_{1,3}^2, 
    g_1y_{1,2}^{n-2}y_{1,3}y_{1,4},
    g_1y_{1,2}^{n-2}y_{1,3}y_{2,4},
    \\
    & 
    g_1y_{1,2}^{n-2}y_{1,3}y_{3,4}, 
    g_1y_{1,2}^{n-2}y_{2,3}y_{1,4}, 
    g_1y_{1,2}^{n-2}y_{2,3}y_{2,4},
    g_1y_{1,2}^{n-2}y_{2,3}y_{3,4},
    g_1y_{1,2}^{n-2}y_{1,4}^2, 
    g_1y_{1,2}^{n-3}y_{1,3}^2y_{3,4},
    \\
    &
    g_1y_{1,2}^{n-3}y_{1,3}y_{1,4}^2, 
    g_1y_{1,2}^{n-3}y_{2,3}y_{1,4}^2,  
    g_2y_{1,2}^{n-1}y_{1,4}, 
    g_2y_{1,2}^{n-1}y_{2,4}, 
    g_2y_{1,2}^{n-2}y_{1,4}^2, 
    g_3y_{1,3}^{n-1}y_{1,4}, 
    g_3y_{1,3}^{n-1}y_{3,4}, 
    \\
    &
    g_3y_{1,3}^{n-2}y_{1,4}^2,
    g_4y_{2,3}^{n-1}y_{2,4}, 
    g_4y_{2,3}^{n-1}y_{3,4}, 
    g_4y_{2,3}^{n-2}y_{2,4}^2,
\end{split}
\end{equation}
and $\C^{!,M^2}_n$ the subset of $(M^2)^!_{-n} $  consisting of the following $3n+21$ elements
\begin{equation}
\label{eq: basis2 M2}
\begin{split}
   &
   g_1 y_{1,2}^n, 
   g_1 y_{1,2}^{n-r}y_{3,4}^{r},
   g_1 y_{3,4}^n, 
   g_2y_{1,2}^{n}, 
   g_3y_{1,2}y_{3,4}^{n-1}, 
   g_3y_{3,4}^{n}, 
   g_4y_{3,4}^{n}, 
   g_5y_{1,2}^{n}, 
   g_5y_{1,2}^{n-1}y_{3,4}, 
   \\
   &
   g_1 y_{1,3}^n,
   g_1 y_{1,3}^{n-r}y_{2,4}^{r}, 
   g_1 y_{2,4}^n ,
   g_2y_{1,3}y_{2,4}^{n-1}, 
   g_2y_{2,4}^{n}, 
   g_3y_{1,3}^{n}, 
   g_4y_{2,4}^{n}, 
   g_6y_{1,3}^{n}, 
   g_6y_{1,3}^{n-1}y_{2,4}, 
   \\
   &
   g_1 y_{2,3}^n, 
   g_1 y_{2,3}^{n-r}y_{1,4}^{r}, 
   g_1 y_{1,4}^n, 
   g_2y_{2,3}y_{1,4}^{n-1}, 
   g_2y_{1,4}^{n}, 
   g_3y_{1,4}^{n}, 
   g_4y_{2,3}^{n}, 
   g_5y_{2,3}^{n}, 
   g_5y_{2,3}^{n-1}y_{1,4}, 
\end{split}
\end{equation}
where $r\in \llbracket 1,n-1\rrbracket$ and $n\geqslant 4$.

\begin{lem}
\label{lem:basis of dual M2}
The set $T_n^{M^2}=\U^{!,M^2}_n\cup \C^{!,M^2}_n$ is a basis of $(M^2)^{!}_{-n}$ for $n\geqslant 4$. 
Moreover, $\dim (M^2)^{!}_{0} =7$, $\dim (M^2)^{!}_{-1} =24$, $\dim (M^2)^{!}_{-2} =43$ and $\dim (M^2)^{!}_{-n} =3n+45$ for $n\geqslant 3$. 
\end{lem}

\begin{proof}  
We will prove that the set $T_n^{M^2}$ is a basis of $(M^2)^{!}_{-n}$ for $n\geqslant 4$. 
Firstly, using GAP, $T_n^{M^2}$ is a basis of $(M^2)^{!}_{-n}$ for $n\in \llbracket 4,7\rrbracket$. 
Note that the space $(M^2)^!_{-n}$ is spanned by $\{ g_iy\mid i\in \llbracket 1,7\rrbracket , y\in \B^!_{n} \}$ for $n\in \NN_0$. 
Moreover, the following identities are straightforward to verify and are left to the reader: 
\begin{align*}
g_1y_{1,2}^{n-3}y_{1,3}^2y_{1,4} & = g_1 y_{1,2}^{n-3}y_{1,4}y_{1,3}^2 
=-\chi_n g_1 y_{2,4}^{n-3}y_{1,2} y_{1,3}^2  +\chi_{n+1} g_1 y_{2,4}^{n-3} y_{1,4} y_{1,3}^2
\\
&
= \chi_n g_5 y_{1,3}y_{2,4}^{n-4}y_{1,2}y_{1,3}^2 -\chi_{n+1}g_5 y_{1,3}y_{2,4}^{n-4}y_{1,4}y_{1,3}^2 
=- g_5 y_{1,2}^3y_{2,3}y_{2,4}^{n-4} 
\\
&
=- g_6 y_{3,4} y_{1,2}^2 y_{2,3} y_{2,4}^{n-4} 
= g_6 y_{2,3}y_{1,2}^2y_{2,4}^{n-3} 
= g_1 y_{1,4} y_{1,2}^2y_{2,4}^{n-3} 
= g_1 y_{1,2}^{n-1}y_{1,4}, 
\\
%%%%%%%%%%%%%%%%%%%%%%%%%%%%%%%%
g_1y_{1,2}^{n-3}y_{1,3}^2y_{2,4} & = g_1 y_{1,2}^{n-3} y_{2,4} y_{1,3}^2 
= -\chi_n g_1 y_{1,4}^{n-3}y_{1,2} y_{1,3}^2 +\chi_{n+1} g_1 y_{1,4}^{n-3}y_{2,4} y_{1,3}^2
\\
& 
= -\chi_n g_6 y_{2,3}y_{1,4}^{n-4}y_{1,2} y_{1,3}^2 +\chi_{n+1} g_6 y_{2,3}y_{1,4}^{n-4}y_{2,4} y_{1,3}^2 
\\
&
=\chi_n g_6 y_{1,2}^{3} y_{1,3}y_{1,4}^{n-4}- \chi_{n+1} g_6 y_{1,2}^{n-2}y_{1,3}y_{1,4} 
\\
&
= -\chi_n g_5 y_{3,4}y_{1,2}^{2} y_{1,3}y_{1,4}^{n-4}+ \chi_{n+1} g_5 y_{3,4}y_{1,2}^{n-3}y_{1,3}y_{1,4} 
\\
&
=\chi_n g_5 y_{1,3}y_{1,2}^{2}y_{1,4}^{n-3}- \chi_{n+1} g_5 y_{1,3}y_{1,4}y_{1,2}^{n-3}y_{1,4}
\\
&
=  -\chi_n g_1 y_{2,4}y_{1,2}^{2}y_{1,4}^{n-3}+ \chi_{n+1} g_1 y_{2,4}y_{1,4}y_{1,2}^{n-3}y_{1,4}
=g_1 y_{1,2}^{n-1}y_{2,4}, 
\\
%%%%%%%%%%%%%%%%%%%%%%
g_1 y_{1,3}^{n-1}y_{1,4} & = -\chi_n g_1 y_{3,4}^{n-1}y_{1,3}+ \chi_{n+1} g_1 y_{3,4}^{n-1}y_{1,4} 
= -\chi_n g_7 y_{1,2}y_{3,4}^{n-2}y_{1,3}+ \chi_{n+1} g_7 y_{1,2}y_{3,4}^{n-2}y_{1,4} 
\\
&
=- \chi_n g_7 y_{1,4}^2 y_{1,2}^{n-3}y_{1,3} +\chi_{n+1}g_7 y_{2,3}y_{1,2}y_{3,4}y_{1,2}^{n-3}
\\
&
=\chi_n g_5 y_{2,3}y_{1,4} y_{1,2}^{n-3} y_{1,3} + \chi_{n+1}g_5 y_{1,4}y_{1,2}y_{3,4}y_{1,2}^{n-3}
\\
&
=-\chi_n g_5 y_{1,3}^2 y_{1,2}y_{2,4}y_{1,2}^{n-4}-\chi_{n+1}g_5 y_{1,3}y_{1,4}y_{1,2}^{n-2}
\\
&
=\chi_n g_1 y_{2,4}y_{1,3}y_{1,2}y_{2,4}y_{1,2}^{n-4} + \chi_{n+1}g_1 y_{2,4}y_{1,4}y_{1,2}^{n-2}
\\
&
=\chi_n g_1 y_{1,2}^{n-2}y_{1,3}y_{1,4}+\chi_{n+1}g_1 y_{1,2}^{n-1}y_{1,4}, 
\\
%%%%%%%%%%%%%%%%%%%%%%%%%%%%%%%%%
g_1 y_{1,3}^{n-1}y_{3,4} & = -\chi_n g_1 y_{3,4}y_{1,4}y_{1,3}^{n-2}+\chi_{n+1}g_1 y_{3,4}y_{1,3}^{n-1} 
= -\chi_n g_7 y_{1,2}y_{1,4}y_{1,3}^{n-2}+\chi_{n+1}g_7 y_{1,2}y_{1,3}^{n-1} 
\\
&
= \chi_n g_7 y_{2,4}y_{1,2}y_{1,3}^{n-2}+\chi_{n+1}g_7 y_{1,3}^2 y_{1,2}^{n-2} 
\\
&
= -\chi_n g_6 y_{1,3}y_{1,2}y_{1,3}^{n-2}+\chi_{n+1}g_6 y_{2,4}y_{1,3} y_{1,2}^{n-2} 
\\
&
=-\chi_n g_6 y_{2,3}^{3}y_{1,3}^{n-3}+\chi_{n+1}g_6 y_{2,3}y_{1,3}y_{1,4}y_{1,2}^{n-3}
\\
&
=-\chi_n g_1 y_{1,4}y_{2,3}^{2}y_{1,3}^{n-3}+\chi_{n+1}g_1 y_{1,4}y_{1,3}y_{1,4}y_{1,2}^{n-3}
\\
&
= \chi_n g_1 y_{1,2}^{n-2}y_{1,3}y_{3,4}+\chi_{n+1}g_1 y_{1,2}^{n-3}y_{1,3}^2y_{3,4}, 
\\
%%%%%%%%%%%%%%%%%%%%%%%%%%%
g_1 y_{2,3}^{n-1}y_{2,4} & = g_1 y_{2,4}y_{3,4}^{n-1} 
= -g_5 y_{1,3}y_{3,4}^{n-1}
=\chi_n g_5 y_{1,4}y_{1,3}^{n-1}-\chi_{n+1}g_5 y_{1,4}^2 y_{1,3}^{n-2} 
\\
&
=\chi_n g_7 y_{2,3}y_{1,3}^{n-1}-\chi_{n+1}g_7 y_{2,3}y_{1,4}y_{1,3}^{n-2} 
=-\chi_n g_7 y_{1,2}^{n-1}y_{2,3}-\chi_{n+1}g_7 y_{1,2}^{n-2}y_{2,3}y_{3,4} 
\\
& 
=-\chi_n g_1 y_{3,4}y_{1,2}^{n-2}y_{2,3}-\chi_{n+1}g_1 y_{3,4}y_{1,2}^{n-3}y_{2,3}y_{3,4}
\\
&
=\chi_n g_1 y_{1,2}^{n-2}y_{2,3}y_{2,4}+\chi_{n+1}g_1 y_{1,2}^{n-3}y_{1,3}^2y_{2,4}
=\chi_n g_1 y_{1,2}^{n-2}y_{2,3}y_{2,4}+\chi_{n+1}g_1 y_{1,2}^{n-1}y_{2,4}, 
\\
%%%%%%%%%%%%%%%%%%%%%%%%%%%%%%%%%%%%%%%%%%%%%
g_1 y_{2,3}^{n-1}y_{3,4} & = (-1)^{n+1}g_1 y_{3,4}y_{2,4}^{n-1} 
=(-1)^{n+1} g_7y_{1,2}y_{2,4}^{n-1} 
\\
&
=-\chi_n g_7 y_{2,4}^2 y_{1,2}y_{2,4}^{n-3}+\chi_{n+1}g_7 y_{2,4}^{n-1}y_{1,2}
\\
&
=\chi_n g_6 y_{1,3}y_{2,4}y_{1,2}y_{2,4}^{n-3}-\chi_{n+1}g_6 y_{1,3}y_{2,4}^{n-2}y_{1,2}
\\
&
=-\chi_n g_6 y_{2,3}y_{1,3}y_{1,4}y_{2,4}^{n-3}+\chi_{n+1}g_6 y_{2,3}y_{1,3}^{n-2}y_{1,4}
\\
&
=-\chi_n g_1 y_{1,4}y_{1,3}y_{1,4}y_{2,4}^{n-3}+\chi_{n+1}g_1 y_{1,4}y_{1,3}^{n-2}y_{1,4}
\\
&
=\chi_n g_1 y_{1,2}^{n-2}y_{2,3}y_{3,4}+\chi_{n+1}g_1 y_{1,3}^{n-1}y_{3,4}
=\chi_n g_1 y_{1,2}^{n-2}y_{2,3}y_{3,4}+\chi_{n+1}g_1 y_{1,2}^{n-3}y_{1,3}^2y_{3,4}, 
\\
%%%%%%%%%%%%%%%%%%%%%%%%%%%%%%%%%%%%%%%
g_2 y_{1,3}^2 & = g_5 y_{2,4} y_{1,3} =-g_5 y_{1,3}y_{2,4}=g_1 y_{2,4}^2, 
\hskip 1mm 
g_2 y_{2,3}^2  =-g_6 y_{1,4}y_{2,3}=g_6 y_{2,3}y_{1,4}=g_1 y_{1,4}^2, 
\\
g_2 y_{1,2}^{n-1}y_{1,3} & = -\chi_n g_2 y_{2,3}^3 y_{1,2}^{n-3}+\chi_{n+1}g_2 y_{2,3}^2 y_{1,3}y_{1,2}^{n-3} 
\\
&
=-\chi_n g_1 y_{1,4}^2 y_{2,3}y_{1,2}^{n-3}+\chi_{n+1}g_1 y_{1,4}^2 y_{1,3}y_{1,2}^{n-3}
=g_1 y_{1,2}^{n-3}y_{1,3}y_{1,4}^2, 
\\
g_2 y_{1,2}^{n-1}y_{2,3} & = -\chi_n g_2 y_{1,3}^3 y_{1,2}^{n-3}+\chi_{n+1}g_2 y_{1,3}^2 y_{2,3}y_{1,2}^{n-3} 
\\
&
=-\chi_n g_1 y_{2,4}^2y_{1,3}y_{1,2}^{n-3}+\chi_{n+1}g_1 y_{2,4}^2 y_{2,3}y_{1,2}^{n-3} 
=g_1 y_{1,2}^{n-3}y_{2,3}y_{1,4}^2, 
\\
g_3 y_{1,2}^2 & = g_7 y_{3,4}y_{1,2}=-g_7 y_{1,2}y_{3,4}=-g_1 y_{3,4}^2, 
\\
g_5 y_{3,4}^2 & = -g_6 y_{1,2}y_{3,4}=g_6 y_{3,4}y_{1,2}=g_5 y_{1,2}^2 , 
\hskip 1mm 
g_5 y_{1,4}^2  = g_7 y_{2,3}y_{1,4}=-g_7 y_{1,4}y_{2,3}=g_5 y_{2,3}^2, 
\\
g_5 y_{1,2}^{n-1}y_{1,3} & = \chi_n g_5 y_{1,3}y_{2,3}y_{1,2}^{n-2}+\chi_{n+1}g_5 y_{1,3}y_{1,2}^{n-1} 
=-\chi_n g_1 y_{2,4}y_{2,3}y_{1,2}^{n-2}-\chi_{n+1}g_1 y_{2,4}y_{1,2}^{n-1}
\\
&
=\chi_n g_1y_{1,2}^{n-2}y_{2,3}y_{3,4}-\chi_{n+1}g_1 y_{1,2}^{n-1}y_{2,4}, 
\\
g_5 y_{1,2}^{n-1}y_{2,3} & = -\chi_n g_5 y_{1,3}y_{1,2}^{n-1} +\chi_{n+1}g_5 y_{1,3}^2 y_{2,3}y_{1,2}^{n-3} 
\\
&
=\chi_n g_1 y_{2,4}y_{1,2}^{n-1}-\chi_{n+1}g_1 y_{2,4}y_{1,3}y_{2,3}y_{1,2}^{n-3}
\\
& 
=-\chi_ng_1y_{1,2}^{n-1}y_{1,4}-\chi_{n+1}g_1 y_{1,2}^{n-2}y_{1,3}y_{3,4}, 
\\
g_5 y_{1,2}^{n-1}y_{1,4} & = \chi_n g_5 y_{3,4}^2 y_{1,2}^{n-3}y_{1,4}
=\chi_n g_5 y_{1,3}^2 y_{1,4}y_{2,4}y_{1,2}^{n-4}+\chi_{n+1}g_5 y_{1,3}^2y_{1,4}y_{1,2}^{n-3}
\\
&
=-\chi_n g_1y_{2,4} y_{1,3} y_{1,4}y_{2,4}y_{1,2}^{n-4}-\chi_{n+1}g_1y_{2,4} y_{1,3}y_{1,4}y_{1,2}^{n-3}
\\
&
=\chi_n g_1 y_{1,2}^{n-3}y_{2,3}y_{1,4}^2+\chi_{n+1}g_1 y_{1,2}^{n-2}y_{2,3}y_{2,4}, 
\\
g_5 y_{1,2}^{n-1}y_{2,4} & = \chi_n g_5 y_{3,4}^2 y_{1,2}^{n-3}y_{2,4} 
=-\chi_n g_5 y_{1,3}^2 y_{1,4}y_{1,2}^{n-3}-\chi_{n+1}g_5 y_{1,3}^2 y_{1,4}y_{2,4}y_{1,2}^{n-4}
\\
& 
=\chi_n g_1y_{2,4} y_{1,3}y_{1,4}y_{1,2}^{n-3}+\chi_{n+1}g_1y_{2,4} y_{1,3}y_{1,4}y_{2,4}y_{1,2}^{n-4}
\\
&
=-\chi_n g_1 y_{1,2}^{n-2}y_{1,3}y_{1,4}+\chi_{n+1}g_1 y_{1,2}^{n-3}y_{1,3}y_{1,4}^2, 
\\
g_5 y_{2,3}^{n-1}y_{2,4} & = g_5 y_{1,4}^2 y_{2,3}^{n-3}y_{2,4} 
=\chi_n g_5 y_{1,3}^2 y_{2,3}y_{2,4}y_{2,3}^{n-4}-\chi_{n+1}g_5 y_{1,3}^2 y_{2,3}y_{3,4}y_{2,3}^{n-4}
\\
& 
=-\chi_n g_1y_{2,4} y_{1,3} y_{2,3}y_{2,4}y_{2,3}^{n-4}+\chi_{n+1}g_1y_{2,4} y_{1,3} y_{2,3}y_{3,4}y_{2,3}^{n-4}
\\
&
=\chi_n g_1 y_{1,2}^{n-2}y_{1,3}y_{3,4}+\chi_{n+1}g_1 y_{1,2}^{n-3}y_{1,3}y_{1,4}^2, 
\\
g_5 y_{2,3}^{n-1}y_{3,4} & = g_5 y_{1,4}^2 y_{2,3}^{n-3}y_{3,4} 
= \chi_n g_5 y_{1,3}^2 y_{2,3}y_{3,4}y_{2,3}^{n-4}+\chi_{n+1}g_5 y_{1,3}^2 y_{3,4}y_{2,3}^{n-3}
\\
&
=-\chi_n g_1y_{2,4} y_{1,3} y_{2,3}y_{3,4}y_{2,3}^{n-4}-\chi_{n+1}g_1y_{2,4} y_{1,3} y_{3,4}y_{2,3}^{n-3}
\\
&
=-\chi_n g_1 y_{1,2}^{n-3}y_{1,3}y_{1,4}^2+\chi_{n+1}g_1 y_{1,2}^{n-2}y_{1,3}y_{2,4},
\\
g_6 y_{2,4}^2 & = g_7 y_{1,2}y_{2,4}=-g_7 y_{2,4}y_{1,3}=g_6 y_{1,3}^2, 
\\
g_6 y_{1,3}^{n-1}y_{1,4} & = -\chi_n g_6 y_{3,4}y_{1,3}^{n-1}+\chi_{n+1}g_6 y_{1,4}y_{1,3}^{n-1} 
=-\chi_n g_5 y_{1,2}y_{1,3}^{n-1}-\chi_{n+1}g_2 y_{2,3}y_{1,3}^{n-1}
\\
&
=-\chi_n g_1 y_{1,2}^{n-2}y_{2,3}y_{3,4}-\chi_{n+1}g_1 y_{1,2}^{n-3}y_{2,3}y_{1,4}^2, 
\\
g_6 y_{1,3}^{n-1}y_{3,4} & = -\chi_n g_6 y_{1,4}y_{1,3}^{n-1}+\chi_{n+1}g_6 y_{3,4}y_{1,3}^{n-1} 
= \chi_n g_2 y_{2,3}y_{1,3}^{n-1}+\chi_{n+1}g_5 y_{1,2}y_{1,3}^{n-1}
\\
&
=-\chi_n g_1 y_{1,2}^{n-3}y_{2,3}y_{1,4}^2+\chi_{n+1}g_1 y_{1,2}^{n-2}y_{2,3}y_{1,4},
\end{align*}
for $n\geqslant 5$. 
Using the previous identities together with 
\eqref{eq:dual rels M2} we see that the space $(M^2)^!_{-n}$ is spanned by $T^{M^2}_n$ for $n\geqslant  8$.

We will next prove that the elements in $T_n^{M^2}$ for $n\geqslant 8$ are linearly independent. 
Suppose that we have the identity
\begin{equation}
\label{eq:lin-ind-tnm2}
\begin{split}
\sum_{i\in\llbracket 1,24 \rrbracket} \alpha_i t_i 
+ \sum_{i\in \llbracket 1,n+7 \rrbracket} \alpha_i^{1,2} t_i^{1,2} 
+ \sum_{i\in \llbracket 1,n+7 \rrbracket} \alpha_i^{1,3} t_i^{1,3} 
+ \sum_{i\in \llbracket 1,n+7 \rrbracket} \alpha_i^{2,3} t_i^{2,3} 
= \underset{\text{\begin{tiny}$\begin{matrix} i\in \llbracket 1, 18 \rrbracket, \\ u\in \B^!_{n-1} \end{matrix}$\end{tiny}}}{\sum} \lambda^i_u r_i u, 
\end{split}
\end{equation}
in $\Bbbk \{ g_i|i\in\llbracket 1,7\rrbracket\} \otimes A^!$, where $t_i$ is the $i$-th element in \eqref{eq: basis1 M2} for $i\in \llbracket 1,24\rrbracket$, 
$t^{1,2}_i$ is the $i$-th element in the first line of \eqref{eq: basis2 M2}, 
$t^{1,3}_i$ is the $i$-th element in the second line of \eqref{eq: basis2 M2},
and 
$t^{2,3}_i$ is the $i$-th element in the last line of \eqref{eq: basis2 M2} for $i\in\llbracket 1, n+7\rrbracket$, 
$r_i$ is the $i$-th element in \eqref{eq:dual rels M2}, 
and $\alpha_i,\alpha_i^{1,2}, \alpha_i^{1,3}, \alpha_i^{2,3}, \lambda^i_{u} \in \Bbbk$. 
We need to prove that the coefficients $\alpha_{i}$ vanish for all $i \in \llbracket 1, 24 \rrbracket$, as well as that $\alpha_{i}^{1,2}$, 
$\alpha_{i}^{1,3}$ and $\alpha_{i}^{2,3}$ vanish for all $i \in \llbracket 1, n+7 \rrbracket$ . 
By Lemma \ref{lem:productss}, \eqref{eq:lin-ind-tnm2} implies that 
\begin{equation}
\label{eq:12}
\begin{split}
\sum_{i\in \llbracket 1,n+7 \rrbracket} \alpha_i^{1,2} t_i^{1,2} 
= 
\underset{\text{\begin{tiny}$\begin{matrix} i\in \llbracket 1, 6 \rrbracket, \\ u\in \B^!_{n-1}\cap Y_{1,2} \end{matrix}$\end{tiny}}}{\sum}
\lambda^i_u r_i u, 
\end{split}
\end{equation}
\begin{equation}
\label{eq:13}
\begin{split}
\sum_{i\in \llbracket 1,n+7 \rrbracket} \alpha_i^{1,3} t_i^{1,3} 
= 
\underset{\text{\begin{tiny}$\begin{matrix} i\in \llbracket 7, 12 \rrbracket, \\ u\in \B^!_{n-1}\cap Y_{1,3} \end{matrix}$\end{tiny}}}{\sum} \lambda^i_u r_i u,
\end{split}
\end{equation}
\begin{equation}
\label{eq:23}
\begin{split}
\sum_{i\in \llbracket 1,n+7 \rrbracket} \alpha_i^{2,3} t_i^{2,3} 
= 
\underset{\text{\begin{tiny}$\begin{matrix} i\in \llbracket 13, 18 \rrbracket, \\ u\in \B^!_{n-1}\cap Y_{2,3} \end{matrix}$\end{tiny}}}{\sum}\lambda^i_u r_i u,
\end{split}
\end{equation}
and 
\begin{equation}
\label{eq:4}
\begin{split}
\sum_{i\in\llbracket 1,24 \rrbracket} \alpha_i t_i  
=
\underset{\text{\begin{tiny}$\begin{matrix} i\in \llbracket 1, 6 \rrbracket, \\ u\in \B^!_{n-1}\setminus Y_{1,2} \end{matrix}$\end{tiny}}}{\sum} \lambda^i_u r_i u +\underset{\text{\begin{tiny}$\begin{matrix} i\in \llbracket 7, 12 \rrbracket, \\ u\in \B^!_{n-1}\setminus Y_{1,3} \end{matrix}$\end{tiny}}}{\sum} \lambda^i_u r_i u +\underset{\text{\begin{tiny}$\begin{matrix} i\in \llbracket 13, 18 \rrbracket, \\ u\in \B^!_{n-1}\setminus Y_{2,3} \end{matrix}$\end{tiny}}}{\sum}\lambda^i_u r_i u,
\end{split}
\end{equation}
in $\Bbbk \{ g_i|i\in\llbracket 1,7\rrbracket\} \otimes A^!$. 
By \eqref{eq:12}, we get $\alpha_i^{1,2}=0$ for $i\in \llbracket 1,n+7 \rrbracket$. 
Indeed, since there is no $g_2 y_{1,2}^n$, $g_3 y_{3,4}^n$, $g_4 y_{3,4}^n$ on the right side of \eqref{eq:12}, we get that $\alpha^{1,2}_{n+2}=\alpha^{1,2}_{n+4}=\alpha^{1,2}_{n+5}=0$. 
Furthermore, as there is no $g_4y_{1,2}^{n-r}y_{3,4}^r$ for $n-r\in \NN$ on the left side of \eqref{eq:12}, we see that $\lambda^2_u=0$ for $u\in \B^!_{n-1}\cap Y_{1,2}$. 
Moreover, since there is no $g_7 y_{3,4}^n$ on the left side of \eqref{eq:12}, we obtain that $\lambda^6_{y_{3,4}^{n-1}}=0$.
This implies that $\alpha^{1,2}_{n+3}=0$ and $\lambda^6_u=0$ for $u\in \B^!_{n-1}\cap Y_{1,2}$. 
Finally, since there is no $g_2u$ and $g_7u$ for $u\in\B^!_n$ on the left side of \eqref{eq:12}, we have that $\lambda^1_u=\lambda^5_u=0$ for $u\in \B^!_{n-1}\cap Y_{1,2}$. 
In consequence, we get $\alpha^{1,2}_i=0$ for $i\in \llbracket 1, n+1\rrbracket$. 
Now, we have that 
\[   
\alpha^{1,2}_{n+6}g_5y_{1,2}^n+\alpha^{1,2}_{n+7}g_5 y_{1,2}^{n-1}y_{3,4}=\sum_{r\in \llbracket 0,n-1\rrbracket}\mathcalboondox{a}_r(g_5y_{1,2}-g_6y_{3,4})u+\sum_{r\in \llbracket 0,n-1\rrbracket}\mathcalboondox{b}_r(g_5y_{3,4}+g_6y_{1,2})u
\] 
in $\Bbbk \{ g_i|i\in\llbracket 1,7\rrbracket\} \otimes A^!$, 
where $\mathcalboondox{a}_r=\lambda^3_{y_{1,2}^{n-1-r}y_{3,4}^r}$ and $\mathcalboondox{b}_r=\lambda^4_{y_{1,2}^{n-1-r}y_{3,4}^r}$ for $r\in\llbracket 0, n-1\rrbracket$. 
Hence, 
\begin{equation}
\begin{split}
&\alpha^{1,2}_{n+6}g_5y_{1,2}^n+\alpha^{1,2}_{n+7}g_5 y_{1,2}^{n-1}y_{3,4} 
\\
&
=   
\mathcalboondox{a}_0 g_5 y_{1,2}^n + \sum_{r\in \llbracket 1, n-1\rrbracket} \bigg(\mathcalboondox{a}_r+\big((-1)^{r}\chi_n+(-1)^{r-1}\chi_{n+1}\big)\mathcalboondox{b}_{r-1} \bigg) g_5y_{1,2}^{n-r}y_{3,4}^r+\mathcalboondox{b}_{n-1} g_5 y_{3,4}^n 
\\
& \phantom{= \;}
+ \mathcalboondox{b}_0 g_6 y_{1,2}^n + \sum_{r\in \llbracket 1, n-1\rrbracket} \bigg(\mathcalboondox{b}_r+\big( (-1)^{r-1}\chi_n +(-1)^{r}\chi_{n+1} \big)\mathcalboondox{a}_{r-1} \bigg) g_6y_{1,2}^{n-r}y_{3,4}^r-\mathcalboondox{a}_{n-1} g_6 y_{3,4}^n 
\end{split}
\end{equation}
in $\Bbbk \{ g_i|i\in\llbracket 1,7\rrbracket\} \otimes A^!$. 
Comparing the coefficients, it is easy to see that $\alpha^{1,2}_{n+6}=\alpha^{1,2}_{n+7}=0$ and $\mathcalboondox{a}_r=\mathcalboondox{b}_r=0$ for $r\in \llbracket 0,n-1\rrbracket$. 
Similarly, \eqref{eq:13} implies $\alpha_i^{1,3}=0$ for $i\in \llbracket 1,n+7 \rrbracket$, and \eqref{eq:23} implies $\alpha_i^{2,3}=0$ for $i\in \llbracket 1,n+7 \rrbracket$. 
By regarding the coefficients of $g_i$ in \eqref{eq:4} for $i\in \llbracket 1,7 \rrbracket$, we get that \eqref{eq:4} is tantamount to
\begin{equation}
   \label{eq:5}
\begin{split}
g_1\big(y_{3,4}\Delta^1+y_{3,4}\Delta^5+y_{2,4}\Delta^7+y_{2,4}\Delta^{11}+y_{1,4}\Delta^{13}+y_{1,4}\Delta^{17} \big)& =\sum_{i\in\llbracket 1,15\rrbracket}\alpha_it_i , 
\\
g_2\big(-y_{3,4}\Delta^1+y_{1,3}\Delta^{8}+y_{1,3}\Delta^{12}+y_{2,3}\Delta^{14}+y_{2,3}\Delta^{18} \big) & = \sum_{i\in\llbracket 16,18\rrbracket}\alpha_it_i , 
\\
g_3\big( y_{1,2}\Delta^2+y_{1,2}\Delta^{6}+y_{2,4}\Delta^{7}+y_{2,3}\Delta^{14}  \big) & = \sum_{i\in\llbracket 19,21\rrbracket}\alpha_it_i , 
\\
g_4\big( -y_{1,2}\Delta^{2}+y_{1,3}\Delta^{8}+y_{1,4}\Delta^{13} \big) & = \sum_{i\in\llbracket 22,24\rrbracket}\alpha_it_i , 
\\
g_5\big( y_{1,2}\Delta^{3}+y_{3,4}\Delta^{4}+y_{1,3}\Delta^{11}-y_{2,4}\Delta^{12}+y_{2,3}\Delta^{15}+y_{1,4}\Delta^{16} \big) & = 0, 
\\
g_6\big( -y_{3,4}\Delta^{3}+y_{1,2}\Delta^{4}+y_{1,3}\Delta^{9}+y_{2,4}\Delta^{10}-y_{2,3}\Delta^{11}+y_{1,4}\Delta^{12} \big) & = 0, 
\\
g_7\big( -y_{1,2}\Delta^{5}-y_{3,4}\Delta^{6}+y_{2,4}\Delta^{9}-y_{1,3}\Delta^{10}+y_{1,4}\Delta^{15}-y_{2,3}\Delta^{16} \big) & = 0 , 
\end{split}
\end{equation}
in $\Bbbk \{g_i \}\otimes A^!$ for $i\in \llbracket 1,7\rrbracket$ respectively, 
where $\Delta^j=\sum_{u\in \B^!_{n-1}\setminus Y_{1,2}}\lambda^j_u u $ for $j\in \llbracket 1,6\rrbracket $, 
$\Delta^j=\sum_{u\in \B^!_{n-1}\setminus Y_{1,3}}\lambda^j_u u $ for $j\in \llbracket 7,12\rrbracket $,  
$\Delta^j=\sum_{u\in \B^!_{n-1}\setminus Y_{2,3}}\lambda^j_u u $ for $j\in \llbracket 13,18\rrbracket $ 
and $Y_{i,j}$ is defined in \eqref{eq:Yij}.
In consequence, we see that the elements in $T_n^{M^2}$ are linearly independent if and only if equation \eqref{eq:5} implies that $\alpha_i=0$ for $i\in\llbracket 1, 24 \rrbracket$.

Let 
\begin{align*}
a^i_0 & =\lambda^i_{y_{1,2}^{n-1}}, 
a'^i_{0}=\lambda^i_{y_{3,4}^{n-1}},
a^{i}_1=\underset{\text{\begin{tiny}$\begin{matrix} r\in \llbracket 1, n-2 \rrbracket, \\ r \; odd \end{matrix}$\end{tiny}}}{\sum}
\lambda^i_{y_{1,2}^{n-1-r}y_{3,4}^{r}}, 
a^{i}_2=\underset{\text{\begin{tiny}$\begin{matrix} r\in \llbracket 1, n-2 \rrbracket, \\ r \; even \end{matrix}$\end{tiny}}}{\sum}
\lambda^i_{y_{1,2}^{n-1-r}y_{3,4}^{r}} ,
\\
b^j_0 & =\lambda^j_{y_{1,3}^{n-1}}, 
b'^j_{0}=\lambda^j_{y_{2,4}^{n-1}},
b^{j}_1=\underset{\text{\begin{tiny}$\begin{matrix} r\in \llbracket 1, n-2 \rrbracket, \\ r \; odd \end{matrix}$\end{tiny}}}{\sum}
\lambda^j_{y_{1,3}^{n-1-r}y_{2,4}^{r}}, 
b^{j}_2=\underset{\text{\begin{tiny}$\begin{matrix} r\in \llbracket 1, n-2 \rrbracket, \\ r \; even \end{matrix}$\end{tiny}}}{\sum}
\lambda^j_{y_{1,3}^{n-1-r}y_{2,4}^{r}} ,
\\
c^k_0 & =\lambda^k_{y_{2,3}^{n-1}}, 
c'^k_{0}=\lambda^k_{y_{1,4}^{n-1}},
c^{k}_1=\underset{\text{\begin{tiny}$\begin{matrix} r\in \llbracket 1, n-2 \rrbracket, \\ r \; odd \end{matrix}$\end{tiny}}}{\sum}
\lambda^k_{y_{2,3}^{n-1-r}y_{1,4}^{r}}, 
c^{k}_2=\underset{\text{\begin{tiny}$\begin{matrix} r\in \llbracket 1, n-2 \rrbracket, \\ r \; even \end{matrix}$\end{tiny}}}{\sum}
\lambda^k_{y_{2,3}^{n-1-r}y_{1,4}^{r}},
\end{align*}
for $i\in \llbracket 7,18 \rrbracket $, 
$j\in \llbracket 1,6 \rrbracket \cup \llbracket 13,18\rrbracket$
and 
$k\in \llbracket 1,12 \rrbracket $.
From \eqref{eq:5} as well as the products \eqref{eq:product y13y12ny34r} and \eqref{eq:product y13y12n} in $A^!$, we get a system $E_n$ of linear equations in the field $\Bbbk$, which contains $24\times 7 = 168$ linear equations and $24+24\times 18+ 4\times 12 \times 3 = 600$ variables $\alpha_i$, $\lambda^j_u$ for $u\in \U^!_{n-1}$, $a^i_0$, $a'^i_0$, $a^i_1$, $a^i_2$, $b^j_0$, $b'^j_0$, $b^j_1$, $b^j_2$, $c^k_0$, $c'^k_0$, $c^k_1$, $c^k_2$. 
Moreover, the linear independence of $T_n^{M^2}$ (or, equivalently, the fact that \eqref{eq:5} implies that $\alpha_i=0$ for $i\in \llbracket 1, 24\rrbracket$) is equivalent to the fact that the linear system $E_n$ implies that $\alpha_i=0$ for $i\in \llbracket 1, 24\rrbracket$.
Note that $E_n$ has the same form when $n$ increases by $2$. 
Using GAP, the elements in $T^{M^2}_{n}$ are linearly independent for $n\in \{ 8,9 \}$, so the lemma holds for all integers $n\geqslant 8$. 
\end{proof}

Let $\U^{M^2}_n$ be the dual basis to $\U^{!,M^2}_n$, $\C^{M^2}_n$ the dual basis to $\C^{!,M^2}_n$, 
and $\C^{M^2}_n= \cup _{(i,j)\in \II}\C_n^{i,j,M^2}$, where  
$\C_n^{i,j,M^2}$ is the subset of $\C^{M^2}_n$ consisting of elements of the form $(g_s y^{n-\bullet}_{i,j}y^{\bullet}_{k,l})^{*} $ for $(i,j)\in \II$, $(k,l)\in \I$ with $\# \{i,j,k,l \}=4$. 
%We denote by $t^{1,2,n}_{r}$ the dual element of the $r$-th element of the first line in \eqref{eq: basis2 M2}, 
%$t^{1,3,n}_{r}$ the dual element of the $r$-th element of the second line in \eqref{eq: basis2 M2}, 
%and $t^{2,3,n}_{r}$ the dual element of the $r$-th element of the last line in \eqref{eq: basis2 M2} for $r\in \llbracket 1,n+7\rrbracket$. 
Given $n, m \in \NN$, let $C^{M^2}_{n,m}$ be the subspace of $\Bbbk \C^{M^2}_n \otimes A_m$ spanned by 
\[ \{ d_{n+1}(M^2)(z|x)\mid z\in \C^{M^2}_{n+1}, x\in A_{m-1} \} ,\] 
$C_{n,m}^{i,j,M^2}$ the subspace of $C^{M^2}_{n,m}$ spanned by 
\[ \{ d_{n+1}(M^2)(z|x) \mid z\in \C_{n+1}^{i,j,M^2}, x\in A_{m-1} \} \]  
for $(i,j)\in \II$,  
and $U^{M^2}_{n,m}$ the subspace of $B^{M^2}_{n,m}$ spanned by 
\[ \{ d_{n+1}(M^2)(z|x)\mid z\in \U^{M^2}_{n+1}, x\in A_{m-1} \}. \] 
Using the actions listed in Appendix \ref{sec:products M2}, 
it is direct but lengthy to check that the differential in the subcomplex $\Bbbk \C^{1,2,M^2}_{n+1}\otimes A$ of the Koszul complex is given by 
\begin{align*}
(g_1y_{1,2}^{n+1})^*|1 & \mapsto (g_1 y_{1,2}^n)^*|x_{1,2}, 
\\
(g_1 y_{1,2}^{n}y_{3,4})^*|1 & \mapsto -(g_1 y_{1,2}^{n-1}y_{3,4})^*|x_{1,2}+(g_1y_{1,2}^n)^*|x_{3,4}+(g_2 y_{1,2}^n)^*|x_{3,4}, 
\\
(g_1 y_{1,2}^{n+1-r}y_{3,4}^{r})^* |1& \mapsto
(-1)^r(g_1 y_{1,2}^{n-r}y_{3,4}^r)^*|x_{1,2}+(g_1 y_{1,2}^{n+1-r}y_{3,4}^{r-1})^*|x_{3,4}
\text{ for $r\in \llbracket 2, n-1 \rrbracket$}, 
\\
(g_1 y_{1,2}y_{3,4}^n)^*|1 & \mapsto (-1)^n (g_1 y_{3,4}^n)^*|x_{1,2}+(g_1 y_{1,2}y_{3,4}^{n-1})^*|x_{3,4}, 
\\
(g_1 y_{3,4}^{n+1})^*|1 & \mapsto (g_1 y_{3,4}^n)^*|x_{3,4}+(-1)^n (g_3 y_{1,2}y_{3,4}^{n-1})^*|x_{1,2}, 
\\
(g_2 y_{1,2}^{n+1})^*|1 & \mapsto (g_2 y_{1,2}^n)^*|x_{1,2}, 
\\ 
(g_3 y_{1,2}y_{3,4}^n)^*|1 & \mapsto (-1)^n(g_3 y_{3,4}^n)^*|x_{1,2}+(-1)^n (g_4 y_{3,4}^n)^*|x_{1,2}+(g_3 y_{1,2}y_{3,4}^{n-1})^*|x_{3,4}, 
\\
(g_3 y_{3,4}^{n+1})^*|1 &\mapsto  (g_3 y_{3,4}^n)^*|x_{3,4}, 
\\
(g_4 y_{3,4}^{n+1})^*|1 & \mapsto (g_4 y_{3,4}^n)^*|x_{3,4}, \\
(g_5 y_{1,2}^{n+1})^*|1 & \mapsto (g_5 y_{1,2}^n)^*|x_{1,2}+(g_5 y_{1,2}^{n-1}y_{3,4})^*|x_{3,4},
\\
(g_5 y_{1,2}^n y_{3,4})^*|1 & \mapsto -(g_5 y_{1,2}^{n-1}y_{3,4})^*|x_{1,2}+(g_5 y_{1,2}^n)^*|x_{3,4},
\end{align*}
for $n\geqslant 4$.
Similarly, the differential in $\Bbbk \C^{1,3,M^2}_{n+1}\otimes A$ is given by \begin{align*}
(g_1y_{1,3}^{n+1})^*|1 & \mapsto (g_1 y_{1,3}^n)^*|x_{1,3}, 
\\
(g_1 y_{1,3}^{n}y_{2,4})^*|1 & \mapsto -(g_1 y_{1,3}^{n-1}y_{2,4})^*|x_{1,3}+(g_1y_{1,3}^n)^*|x_{2,4} -(g_3 y_{1,3}^n)^* |x_{2,4}, 
\\
(g_1 y_{1,3}^{n+1-r}y_{2,4}^{r})^* |1& \mapsto
(-1)^r(g_1 y_{1,3}^{n-r}y_{2,4}^r)^*|x_{1,3}+(g_1 y_{1,3}^{n+1-r}y_{2,4}^{r-1})^*|x_{2,4}
\text{ for $r\in \llbracket 2, n-1 \rrbracket$}, 
\\
(g_1 y_{1,3}y_{2,4}^n)^*|1 & \mapsto (-1)^n (g_1 y_{2,4}^n)^*|x_{1,3}+(g_1 y_{1,3}y_{2,4}^{n-1})^*|x_{2,4}, 
\\
(g_1 y_{2,4}^{n+1})^*|1 & \mapsto (g_1 y_{2,4}^n)^*|x_{2,4} - (-1)^{n}(g_2 y_{1,3}y_{2,4}^{n-1})^*|x_{1,3}, 
\\
(g_2y_{1,3}y_{2,4}^{n})^*|1 & \mapsto (-1)^n(g_2y_{2,4}^n)^*|x_{1,3} -(-1)^n (g_4y_{2,4}^n)^*|x_{1,3} + (g_2y_{1,3}y_{2,4}^{n-1})^*|x_{2,4}, 
\\
(g_2 y_{2,4}^{n+1})^*|1 & \mapsto (g_2 y_{2,4}^n)^*|x_{2,4}, 
\\
(g_3 y_{1,3}^{n+1})^* |1 & \mapsto (g_3 y_{1,3}^n)^*|x_{1,3}, 
\\
(g_4 y_{2,4}^{n+1})^* |1 & \mapsto (g_4 y_{2,4}^n)^*|x_{2,4}, 
\\
(g_6 y_{1,3}^{n+1})^* |1 & \mapsto (g_6 y_{1,3}^n)^*|x_{1,3} +(g_6 y_{1,3}^{n-1}y_{2,4})^*|x_{2,4}, 
\\
(g_6 y_{1,3}^{n}y_{2,4})^*|1 & \mapsto -(g_6 y_{1,3}^{n-1}y_{2,4})^*|x_{1,3} + (g_6y_{1,3}^n )|x_{2,4},
\end{align*}
for $n\geqslant 4$, 
whereas the differential in $\Bbbk \C^{2,3,M^2}_{n+1}\otimes A$ is given by 
\begin{align*}
(g_1y_{2,3}^{n+1})^*|1 & \mapsto (g_1 y_{2,3}^n)^*|x_{2,3}, 
\\
(g_1 y_{2,3}^{n}y_{1,4})^*|1 & \mapsto -(g_1 y_{2,3}^{n-1}y_{1,4})^*|x_{2,3}+(g_1y_{2,3}^n)^*|x_{1,4} -(g_4 y_{2,3}^n)^*|x_{1,4}, 
\\
(g_1 y_{2,3}^{n+1-r}y_{1,4}^{r})^* |1& \mapsto
(-1)^r(g_1 y_{2,3}^{n-r}y_{1,4}^r)^*|x_{2,3}+(g_1 y_{2,3}^{n+1-r}y_{1,4}^{r-1})^*|x_{1,4}
\text{ for $r\in \llbracket 2, n-1 \rrbracket$}, 
\\
(g_1 y_{2,3}y_{1,4}^n)^*|1 & \mapsto (-1)^n (g_1 y_{1,4}^n)^*|x_{2,3}+(g_1 y_{2,3}y_{1,4}^{n-1})^*|x_{1,4}, 
\\
(g_1 y_{1,4}^{n+1})^*|1 & \mapsto (g_1 y_{1,4}^n)^*|x_{1,4} -(-1)^n (g_2 y_{2,3}y_{1,4}^{n-1})^*|x_{2,3}, 
\\
(g_2 y_{2,3}y_{1,4}^n)^*|1 & \mapsto (-1)^n(g_2y_{1,4}^n)^*|x_{2,3}-(-1)^n (g_3y_{1,4}^n)^*|x_{2,3} +(g_2 y_{2,3}y_{1,4}^{n-1})^*|x_{1,4}, 
\\
(g_2 y_{1,4}^{n+1})^*|1 & \mapsto (g_2 y_{1,4}^{n})^*|x_{1,4},
\\
(g_3 y_{1,4}^{n+1})^*|1 & \mapsto (g_3 y_{1,4}^n)^*|x_{1,4},
\\
(g_4 y_{2,3}^{n+1})^*|1 & \mapsto (g_4 y_{2,3}^n)^*|x_{2,3} , 
\\
(g_5 y_{2,3}^{n+1})^*|1 & \mapsto (g_5 y_{2,3}^n)^*|x_{2,3} +(g_5 y_{2,3}^{n-1}y_{1,4})^*|x_{1,4} , 
\\
(g_5 y_{2,3}^n y_{1,4})^*|1 & \mapsto -(g_5 y_{2,3}^{n-1}y_{1,4})^*|x_{2,3} +(g_5 y_{2,3}^n)^*|x_{1,4},
\end{align*}
for $n\geqslant 4$. 

Recall that the sets $W_m^{i,j}$, $E_m^{i,j}$ and $\tilde{E}_m^{i,j}$ for $(i,j)\in \II$ are defined in the paragraph before Table \ref{table:a b}. 
For $(i,j)\in \II$, $(k,l)\in \I$ with $\# \{i,j,k,l \}=4$, 
let $\hat{E}^{i,j}_m$ be the subset of $W^{i,j}_m$ containing elements whose first element is $x_{i,j}$ and second element is not $x_{k,l}$. 
Let $E'^{i,j}_m$ be the subset of $W^{i,j}_m$ containing elements whose first element is $x_{i,j}$ and the second element is $x_{k,l}$.
The left multiplication of $x_{k,l}$ from $\Bbbk\hat{E}^{i,j}_{m-1}$ to $\Bbbk E'^{i,j}_m $ is isomorphic.
It is easy to check that $\# (\hat{E}^{i,j}_m \cup \tilde{E}^{i,j}_m)=\mathfrak{a}_m $, where $\mathfrak{a}_m$ is given in Table \ref{table:a b}. 
A basis of $C^{1,2,M^2}_{n,m}$ is given by 
$d_{n+1}(M^2)( t_i|x)$ for $i\in \llbracket 1,n+1 \rrbracket \cup\{ n+3, n+7,n+8 \} $ and $x\in E^{1,2}_{m-1}$, 
$d_{n+1}(M^2)( t_i|x)$ for $i\in \llbracket n+4,n+6 \rrbracket $ and $x\in \hat{E}^{1,2}_{m-1}\cup \tilde{E}^{1,2}_{m-1}$, 
$d_{n+1}(M^2)( (g_1 y_{3,4}^{n+1})^*|x)$ for $x\in \tilde{E}^{1,2}_{m-1}$, 
where $t_i\in \C^{1,2,M^2}_{n+1}$ is the $i$-th element in the following sequence 
\begin{equation}
\label{eq:squ1}
\begin{split}
& (g_1y_{1,2}^{n+1})^*, 
(g_1 y_{1,2}^{n+1-r}y_{3,4}^{r})^* 
\text{ for $r\in \llbracket 1, n \rrbracket$}, 
(g_1 y_{3,4}^{n+1})^*, 
(g_2 y_{1,2}^{n+1})^*, 
(g_3 y_{1,2}y_{3,4}^n)^*, 
(g_3 y_{3,4}^{n+1})^*, 
\\
& (g_4 y_{3,4}^{n+1})^*, 
(g_5 y_{1,2}^{n+1})^*,
(g_5 y_{1,2}^n y_{3,4})^*.
\end{split}
\end{equation}
A basis of $C^{1,3,M^2}_{n,m}$ is given by 
$d_{n+1}(M^2)( t_i|x)$ for $i\in \llbracket 1,n+1 \rrbracket \cup \{n+5, n+7,n+8 \}$ and $x\in E^{1,3}_{m-1}$, 
$d_{n+1}(M^2)( t_i|x)$ for $i\in \{n+3, n+4,n+6 \} $ and $x\in \hat{E}^{1,3}_{m-1}\cup \tilde{E}^{1,3}_{m-1}$, 
$d_{n+1}(M^2)( (g_1 y_{2,4}^{n+1})^*|x)$ for $x\in \tilde{E}^{1,3}_{m-1}$, 
where $t_i\in \C^{1,3,M^2}_{n+1}$ is the $i$-th element in the following sequence 
\begin{equation}
\label{eq:squ2}
\begin{split}
& (g_1y_{1,3}^{n+1})^*, 
(g_1 y_{1,3}^{n+1-r}y_{2,4}^{r})^*
\text{ for $r\in \llbracket 1, n \rrbracket$}, 
(g_1 y_{2,4}^{n+1})^*, 
(g_2y_{1,3}y_{2,4}^{n})^*, 
(g_2 y_{2,4}^{n+1})^*, 
(g_3 y_{1,3}^{n+1})^*, 
\\
&
(g_4 y_{2,4}^{n+1})^*, 
(g_6 y_{1,3}^{n+1})^*, 
(g_6 y_{1,3}^{n}y_{2,4})^*.
\end{split}
\end{equation}
A basis of $C^{2,3,M^2}_{n,m}$ is given by 
$d_{n+1}(M^2)( t_i|x)$ for $i\in \llbracket 1,n+1 \rrbracket \cup \llbracket n+6,n+8 \rrbracket $ and $x\in E^{2,3}_{m-1}$, 
$d_{n+1}(M^2)( t_i|x)$ for $i\in \llbracket n+3,n+5 \rrbracket $ and $x\in \hat{E}^{2,3}_{m-1}\cup \tilde{E}^{2,3}_{m-1}$, 
$d_{n+1}(M^2)( (g_1 y_{1,4}^{n+1})^*|x)$ for $x\in \tilde{E}^{2,3}_{m-1}$, 
where $t_i\in\C^{2,3,M^2}_{n+1}$ is the $i$-th element in the following sequence
\begin{equation}
\label{eq:squ3}
\begin{split}
& (g_1y_{2,3}^{n+1})^*, 
(g_1 y_{2,3}^{n+1-r}y_{1,4}^{r})^* 
\text{ for $r\in \llbracket 1, n \rrbracket$}, 
(g_1 y_{2,3}y_{1,4}^n)^*, 
(g_1 y_{1,4}^{n+1})^*, 
(g_2 y_{1,4}^{n+1})^*,
(g_3 y_{1,4}^{n+1})^*,
\\
&
(g_4 y_{2,3}^{n+1})^*, 
(g_5 y_{2,3}^{n+1})^*, 
(g_5 y_{2,3}^n y_{1,4})^*.
\end{split}
\end{equation}
So, $\dim C^{i,j,M^2}_{n,m}=\mathfrak{a}_{m-1}(n+7)+\mathfrak{b}_{m-1}$, where $\mathfrak{a}_m$ and $\mathfrak{b}_m$ are given in Table \ref{table:a b}.

\begin{lem}
\label{lem: dim Cnm M2}
We have $C^{M^2}_{n,m}=\bigoplus_{(i,j)\in \II} C^{i,j,M^2}_{n,m}$ 
and the dimension of $C^{i,j,M^2}_{n,m}$ is given by  
\begin{equation}
   \label{eq:dim chang M2}
       \dim C^{i,j,M^2}_{n,m}= \begin{cases} 
       n+8, &\text{if $m = 1$,}
       \\
       5n+39, &\text{if $m =2$,}
       \\
       14n+108, &\text{if $m =3$,}
       \\
       28n+214, &\text{if $m =4$,}
       \\
       43n+326, &\text{if $m =5$,}
        \\
       53n+399 , &\text{if $m =6$,}
        \\
       53n+396, &\text{if $m =7$,}
        \\
        43n+319, &\text{if $m =8$,}
        \\
       28n+206, &\text{if $m =9$,}
        \\
       14n+102, &\text{if $m =10$,}
        \\
       5n+36, &\text{if $m =11$,}
        \\
       n+7, &\text{if $m =12$,}
   \end{cases}
   \end{equation}
for all $(i,j)\in \II $ and $n\geqslant 4$. 
Moreover, if $(i,j)\in \II $, $n\geqslant 4$ and $m \geqslant 13$, $\dim C^{i,j,M^2}_{n,m}= 0$. 
\end{lem}

\begin{lem}
\label{lem: dim Unm M2}
We have $\dim U^{M^2}_{n,m}=\dim U^{M^2}_{n+2,m}$ and $\dim (U^{M^2}_{n,m}\cap C^{M^2}_{n,m})=\dim (U^{M^2}_{n+2,m} \cap C^{M^2}_{n+2,m}) $ for $n\geqslant 4$ and $m\in \llbracket 1,12 \rrbracket$.
\end{lem}

\begin{proof}
Let 
\[ 
u_n^{i,j}=\underset{\text{\begin{tiny}$\begin{matrix} r\in \llbracket 1, n-1\rrbracket, \\ r \; odd \end{matrix}$\end{tiny}}}{\sum} (g_1y_{i,j}^{n-r}y_{k,l}^r)^* ,  
\hskip 2mm 
v^{i,j}_n=\underset{\text{\begin{tiny}$\begin{matrix} r\in \llbracket 1, n-1\rrbracket, \\ r \; even \end{matrix}$\end{tiny}}}{\sum} (g_1y_{i,j}^{n-r}y_{k,l}^r)^*, 
\]
for $(i,j)\in\II$, $(k,l)\in \I$ with $\#\{i,j,k,l \}=4$.
Let 
\[  
\Q^{i,j}_n= \big(\C^{i,j,M^2}_n \setminus \{ t_{r}^{i,j,n} |r\in \llbracket 2,n\rrbracket \} \big)
\cup \{ u_n^{i,j}, v_n^{i,j} \}
\]
for $(i,j)\in \I_1$, where $t^{i,j,n}_r=(g_1 y_{i,j}^{n-r+1}y_{k,l}^{r-1})^*\in \C^{i,j,M^2}_n$ for $r\in \llbracket 2,n\rrbracket$.  
Let 
\[ 
\Q_n=\U^{M^2}_n \cup %\bigg( \underset{(i,j)\in \I_1}{\bigcup} 
(\cup_{(i,j)\in \II}\Q^{i,j}_n )
. \]
It is clear that there is an isomorphism $\mathfrak{f}_n:\Bbbk \Q_n \to \Bbbk \Q_{n+2}$ of vector spaces. 
Consider thus  the linear isomorphism $\mathfrak{g}_n =\mathfrak{f}_n\otimes \id_{A}:\Bbbk \Q_n \otimes A\to \Bbbk \Q_{n+2}\otimes A$. 
Then $U^{M^2}_{n,m}\subseteq \Bbbk\Q_n\otimes A_m $ and $\mathfrak{g}_n(U^{M^2}_{n,m})=U^{M^2}_{n+2,m}$ for $n\geqslant 4$. 
Hence, $U^{M^2}_{n,m}\cong U^{M^2}_{n+2,m}$ as vector spaces for $n\geqslant 4$.

Let $F^{i,j}_{n,m}=(\Bbbk \Q^{i,j}_n \otimes A_m)\cap C_{n,m}^{i,j,M^2}$ for $(i,j)\in \II$.
Then 
\[     U^{M^2}_{n,m}\cap C^{M^2}_{n,m}= U^{M^2}_{n,m} \cap \big(\Bbbk \Q_n \otimes A_m\big)\cap C^{M^2}_{n,m}= U^{M^2}_{n,m}\cap \bigg(\bigoplus_{(i,j)\in \I_1} F^{i,j}_{n,m}\bigg).     \]
To prove that $\dim (U^{M^2}_{n,m}\cap C^{M^2}_{n,m})=\dim (U^{M^2}_{n+2,m} \cap C^{M^2}_{n+2,m}) $, 
it is sufficient to show that $\mathfrak{g}_n(F^{i,j}_{n,m})=F^{i,j}_{n+2,m}$ for $(i,j)\in \II$.
This follows directly from the next simple facts, whose proof is left to the reader. 
If $n$ is even, $F^{i,j}_{n,m}$ is spanned by the elements 
\begin{enumerate}[label = (E.\arabic*)]
    \item $(g_1 y_{i,j}^n)^*|x_{i,j}x$, $(v_n^{i,j}|x_{k,l}-u_n^{i,j}|x_{i,j}+(g_1y_{i,j}^n)^*|x_{k,l}+\xi^{i,j}|x_{k,l})x $ for $x\in E^{i,j}_{m-1}$, 
    \item $((g_1 y_{k,l}^n)^*|x_{k,l}+\eta^{i,j}|x_{i,j})y $, 
$(v^{i,j}_n|x_{i,j}+u^{i,j}_n|x_{k,l}+(g_1 y_{k,l}^n)^*|x_{i,j})y$ for $y\in \tilde{E}^{i,j}_{m-1}$,
    \item $v^{i,j}_n|x_{i,j}w $,
$(g_1 y_{k,l}^n)^*|x_{i,j}w$ for $w\in E^{i,j}_{m-1}\setminus \tilde{E}^{i,j}_{m-1}$, 
\item $d_{n+1}(M^2)(t^{i,j}_r|x)$ for $(r,x)\in \Delta^{i,j}$, 
\end{enumerate}
whereas, if $n$ is odd, $F^{i,j}_{n,m}$ is spanned by the elements
\begin{enumerate}[label = (O.\arabic*)]
    \item $(g_1y_{i,j}^n)^*|x_{i,j}x$, 
$(u^{i,j}_n|x_{k,l}+v^{i,j}_n|x_{i,j})x$ for $x\in E^{i,j}_{m-1}$, 
\item $((g_1y_{k,l}^n)^*|x_{k,l}+\eta^{i,j}|x_{i,j} )y$, 
$(v^{i,j}_n|x_{k,l}-u^{i,j}_n|x_{i,j}+(g_1y_{i,j}^n)^*|x_{k,l}+\xi^{i,j}|x_{k,l}-(g_1y_{k,l}^n)^*|x_{i,j})y$ for $y\in \tilde{E}^{i,j}_{m-1}$, 
\item $u^{i,j}_n|x_{i,j}w$, 
$(g_1y_{k,l}^n)^*|x_{i,j}w$ for $w\in E^{i,j}_{m-1}\setminus \tilde{E}^{i,j}_{m-1}$, 
\item $d_{n+1}(M^2)(t^{i,j}_r|x)$ for $(r,x)\in \Delta^{i,j}$.
\end{enumerate}
Here, 
$t^{1,2}_r$ (resp., $t^{1,3}_r$, $t^{2,3}_r$) is the $r$-th element in \eqref{eq:squ1} (resp., \eqref{eq:squ2}, \eqref{eq:squ3}),
and 
$\xi^{1,2}=(g_2y_{1,2}^n)^*$, 
$\xi^{1,3}=-(g_3y_{1,3}^n)^*$, 
$\xi^{2,3}=-(g_4y_{2,3}^n)^*$, 
$\eta^{1,2}=(-1)^n(g_3y_{1,2}y_{3,4}^{n-1})^*$, 
$\eta^{1,3}=(-1)^{n+1}(g_2y_{1,3}y_{2,4}^{n-1})^*$,  
$\eta^{2,3}=(-1)^{n+1}(g_2y_{2,3}y_{1,4}^{n-1})^*$,
$\Delta^{1,2}= (\{n+3,n+7,n+8 \} \times E^{1,2}_{m-1}) \cup ( \llbracket n+4,n+6 \rrbracket \times ( \hat{E}^{1,2}_{m-1} \cup \tilde{E}^{1,2}_{m-1}) )$,   
$\Delta^{1,3}= (\{n+5,n+7,n+8 \} \times E^{1,3}_{m-1}) \cup ( \{ n+3, n+4,n+6 \} \times ( \hat{E}^{1,3}_{m-1} \cup \tilde{E}^{1,3}_{m-1}) )$, and 
$\Delta^{2,3}= (\llbracket n+6,n+8 \rrbracket \times E^{2,3}_{m-1}) \cup ( \llbracket n+3,n+5 \rrbracket \times ( \hat{E}^{2,3}_{m-1} \cup \tilde{E}^{2,3}_{m-1}) )$.
\end{proof}

Recall that $B^{M^2}_{n,m}$ (resp., $D^{M^2}_{n,m}$) is the image (resp., kernel) concentrated in homological degree $n$ and internal degree $m+n$ of the Koszul complex of $M^2$.

\begin{prop}
   \label{prop:dim Bn M2}
The dimension of $B_{n,m}^{M^2}$ is given by 
\begin{equation}
      \label{eq:dim Bn M2}
          \dim B^{M^2}_{n,m}= \begin{cases} 
          0, &\text{if $m=0$,}
          \\
          3n+48, &\text{if $m = 1$,}
          \\
          15n+237, &\text{if $m =2$,}
          \\
          42n+660, &\text{if $m =3$,}
          \\
          84n+1314, &\text{if $m =4$,}
          \\
          129n+2010, &\text{if $m =5$,}
           \\
          159n+2469 , &\text{if $m =6$,}
           \\
          159n+2460, &\text{if $m =7$,}
           \\
           129n+1989, &\text{if $m =8$,}
           \\
          84n+1290, &\text{if $m =9$,}
           \\
          42n+642, &\text{if $m =10$,}
           \\
          15n+228, &\text{if $m =11$,}
           \\
          3n+45, &\text{if $m =12$,}
      \end{cases}
      \end{equation}
for $n\geqslant 3$. 
\end{prop}

\begin{proof}
By Lemma \ref{lem: dim Unm M2}, we have $\dim B^{M^2}_{n+2,m}-\dim B^{M^2}_{n,m}=\dim C^{M^2}_{n+2,m}-\dim C^{M^2}_{n,m}$ for $n\geqslant 4$ and $m\in \llbracket 1,12 \rrbracket$.  
Using GAP we get the value of $\dim B^{M^2}_{n,m}$ for $n\in \llbracket 3,5\rrbracket$ and $m\in \llbracket 1,12 \rrbracket$. 
\end{proof}

\begin{cor}
   \label{cor:M2}
We have $\operatorname{H}_n(M^2)=0$ for $n\geqslant 4$. 
\end{cor}

\begin{proof}
The result follows from $\dim D^{M^2}_{n,m}=(3n+45)\dim A_m -\dim B^{M^2}_{n-1,m+1}$ for $n\geqslant 4$ and $m\in \llbracket 0,12\rrbracket$, and $\dim \operatorname{H}_{n,m}(M^2)=\dim D^{M^2}_{n,m}- \dim B^{M^2}_{n,m}$. 
\end{proof}

By Appendix \ref{sec:cpx} and the following code 
\parskip 1ex
\parindent 0in
\begin{tcolorbox}[breakable,colback=white,width=\textwidth ,center,arc=0mm,size=fbox]
\begin{footnotesize}
\begin{Verbatim}[samepage=false]
for j in [0..8] do
    for i in [1..12] do
        Print(j, " ", i, " ", RankMat(FF(2,j,i)), "\n");
    od;
od;
\end{Verbatim}
\end{footnotesize}
\end{tcolorbox}
we obtain the dimension of $B^{M^2}_{n,m}$ for $n\in \llbracket 0,8\rrbracket$ and $m\in\llbracket 1,12\rrbracket$. 
Then the dimension of $\operatorname{H}_{n,m}(M^2)$ for $n\in \llbracket 1,8 \rrbracket$ and $m\in\llbracket 0,12 \rrbracket $ is given by Table \ref{table:H M2}. 
The dimensions that are not listed in the following table are zeros. 

\begin{table}[H]
   \begin{center}
    %  \resizebox{\textwidth}{120mm}{
      \begin{tabular}{|c|ccccccccccccc|}
      \hline
       \diagbox[width=14mm,height=5mm]{$n$}{$m$}  & $0$ & $1$  & $2$ & $3$ & $4$ & $5$ & $6$ & $7$ & $8$ & $9$ & $10$ & $11$ & $12$ 
     \\
     \hline
     $1$ &  &  &  & $1$ &  & $1$ &  &  &  &  &  &  &  
     \\
     $2$ &  &  &  &  & $2$ & $6$ & $11$ & $12$ & $11$ & $6$ & $2$ &  &  
     \\
     $3$ &  &  &  & $8$ & $24$ & $48$ & $72$ & $80$ & $72$ & $48$ & $24$ & $8$ &   
     \\
      \hline
   \end{tabular}
 %  }
 \vspace{1mm}
   \caption{Dimension of $\operatorname{H}_{n,m}(M^2)$.}	
   \label{table:H M2}
\end{center}
\end{table}
\vspace{-0.8cm}

\paragraph{\texorpdfstring{The $A$-module structure of the homology groups}{The A-module structure of the homology groups}}

%%%%%%%
\begin{lem}
\label{lemma:homology-m2}
We have the following isomorphisms of graded $A$-modules 
\begin{equation}
   \label{eq:HK M2}
\operatorname{H}_{n}(M^2) \cong
\begin{cases} 
   \Bbbk(-4) \oplus \Bbbk(-6), & \text{if $n = 1$,}
      \\
   M^1(-6), & \text{if $n = 2$,}
      \\
   M^3(-6) , & \text{if $n = 3$,}
   \\
   0, & \text{if $n\geqslant 4$.}
\end{cases}
\end{equation}
\end{lem}
%%%%%%%
\begin{proof}
The isomorphisms in \eqref{eq:HK M2} for all integers $n \geqslant 4$ follow immediately from Corollary \ref{cor:M2}. 
A simple argument using dimension and grading together with Table 
\ref{table:H M2} gives the isomoprhism in \eqref{eq:HK M2} for $n = 1$.

We prove that the space $\operatorname{H}_{2}(M^2)$ is a quadratic module, which is isomorphic to $M^1(-6)$. 
The following GAP code shows that the dimension vector of the submodule of $\operatorname{H}_{2}(M^2)$ generated by two basis elements $a''_1,a''_2$ of $\operatorname{H}_{2,4}(M^2)$ is $(2,6,11,12,11,6,2)$.
So, $\operatorname{H}_{2}(M^2)$ is generated by the two elements as an $A$-module.
\parskip 1ex
\parindent 0in
\begin{tcolorbox}[breakable,colback=white,width=\textwidth ,center,arc=0mm,size=fbox]
\begin{footnotesize}
\begin{Verbatim}[samepage=false]
Imm:=Im(2,2,4);;
RankMat(Imm);
# 1474
gene:=geneMH(2,2,4);;
Append(Imm,gene);
RankMat(Imm);
# 1476
Uh:=UU(gene,4);; Vh:=VV(gene,4);; Wh:=WW(gene,4);;
for r in [5..10] do 
   hxr:=HXR(2,Uh,Vh,Wh,2,4,r-4);
   Im2r:=Im(2,2,r);
   Append(Im2r, hxr);
   Print(r, " ", RankMat(Im2r)-RankMat(Im(2,2,r)), "\n");
od;
# 5 6
# 6 11
# 7 12
# 8 11
# 9 6
# 10 2
\end{Verbatim}
\end{footnotesize}
\end{tcolorbox}
Furthermore, it is direct to check that the generators $a''_1,a''_2$ of $\operatorname{H}_{2}(M^2)$ satisfy the quadratic relations \eqref{eq: relations of H4} defining $M^1$. 
Indeed, the following code shows that the dimension of the subspace generated by $B_{2,5}^{M^2}$ together with the elements of the form \eqref{eq: relations of H4} with $a''_i$ instead of $a_i$ coincides with the dimension of $B_{2,5}^{M^2}$.

\parskip 1ex
\parindent 0in
\begin{tcolorbox}[breakable,colback=white,width=\textwidth ,center,arc=0mm,size=fbox]
\begin{footnotesize}
\begin{Verbatim}[samepage=false]
gene:=geneMH(2,2,4);;
Uh:=UU(gene,4);; Vh:=VV(gene,4);; Wh:=WW(gene,4);;
hx:=HXR(2,Uh,Vh,Wh,2,4,1);;
cc:=0*[1..6];;
cc[1]:=hx[1]+hx[7];; cc[2]:=hx[2];; cc[3]:=hx[5];; cc[4]:=hx[6]+hx[12];;
cc[5]:=hx[9];; cc[6]:=hx[10];;
Imm:=Im(2,2,5);;
RankMat(Imm);
# 2244
Append(Imm,cc);
RankMat(Imm);
# 2244
\end{Verbatim}
\end{footnotesize}
\end{tcolorbox}
Hence, there is a surjective morphism $M^1(-6)\to \operatorname{H}_{2}(M^2)$ of graded $A$-modules. 
Since the dimension vector of $M^1$ is $(2,6,11,12,11,6,2)$, we have $\operatorname{H}_{2}(M^2)\cong M^1(-6)$ as graded $A$-modules, as claimed.

Next, we prove that the space $\operatorname{H}_{3}(M^2)$ is also a quadratic module, which is isomorphic to $M^3(-6)$. 
The following code shows that the dimension vector of the submodule of $\operatorname{H}_{3}(M^2)$ generated by basis elements $e'_i,i\in \llbracket 1,8\rrbracket $ of $\operatorname{H}_{3,3}(M^2)$ is $(8,24,48,72,80,72,48,24,8)$.
So, $\operatorname{H}_{3}(M^2)$ is generated by the eight elements $e'_i,i\in \llbracket 1,8\rrbracket $ as an $A$-module.
\parskip 1ex
\parindent 0in
\begin{tcolorbox}[breakable,colback=white,width=\textwidth ,center,arc=0mm,size=fbox]
\begin{footnotesize}
\begin{Verbatim}[samepage=false]
Imm:=Im(2,3,3);;
RankMat(Imm);
# 786
gene:=geneMH(2,3,3);;
Append(Imm,gene);
RankMat(Imm);
# 794
Uh:=UU(gene,3);; Vh:=VV(gene,3);; Wh:=WW(gene,3);;
for r in [4..11] do 
   hxr:=HXR(2,Uh,Vh,Wh,3,3,r-3);
   Im3r:=Im(2,3,r);
   Append(Im3r, hxr);
   Print(r, " ", RankMat(Im3r)-RankMat(Im(2,3,r)), "\n");
od;
# 4 24
# 5 48
# 6 72
# 7 80
# 8 72
# 9 48
# 10 24
# 11 8
\end{Verbatim}
\end{footnotesize}
\end{tcolorbox}
Moreover, it is direct to check that the generators $e'_i, i\in \llbracket 1,8 \rrbracket$ of $\operatorname{H}_{3}(M^2)$ satisfy the quadratic relations \eqref{eq:relation in M3}. 
Indeed, the following code shows that the dimension of the subspace generated by $B_{3,4}^{M^2}$ together with the elements of the form \eqref{eq:relation in M3} with $e'_i$ instead of $e_i$ coincides with the dimension of $B_{3,4}^{M^2}$.

\parskip 1ex
\parindent 0in
\begin{tcolorbox}[breakable,colback=white,width=\textwidth ,center,arc=0mm,size=fbox]
\begin{footnotesize}
\begin{Verbatim}[samepage=false]
gene:=geneMH(2,3,3);;
Uh:=UU(gene,3);; Vh:=VV(gene,3);; Wh:=WW(gene,3);;
hx:=HXR(2,Uh,Vh,Wh,3,3,1);;
cc:=0*[1..24];;
cc[1]:=hx[1]+hx[12];; cc[2]:=hx[6]-hx[7];; cc[3]:=hx[13]-hx[24];; 
cc[4]:=hx[18]+hx[19];; cc[5]:=hx[20]+hx[11];; cc[6]:=hx[23]-hx[8];;
cc[7]:=hx[14]+hx[5];; cc[8]:=hx[17]-hx[2];; cc[9]:=hx[3]-hx[22];;
cc[10]:=hx[4]+hx[21];; cc[11]:=hx[15]-hx[10];; cc[12]:=hx[16]+hx[9];;
cc[13]:=hx[25];; cc[14]:=hx[26];; cc[15]:=hx[27];; cc[16]:=hx[31];; cc[17]:=hx[34];; 
cc[18]:=hx[35];; cc[19]:=hx[38];; cc[20]:=hx[40];; cc[21]:=hx[42];; cc[22]:=hx[45];; 
cc[23]:=hx[47];; cc[24]:=hx[48];; 
Imm:=Im(2,3,4);;
RankMat(Imm);
# 1566
Append(Imm,cc);
RankMat(Imm);
# 1566
\end{Verbatim}
\end{footnotesize}
\end{tcolorbox}
Hence, there is a surjective morphism $M^3(-6)\to \operatorname{H}_{3}(M^2)$ of graded $A$-modules. 
Since the dimension vector of $M^3$ is $(8,24,48,72,80,72,48,24,8)$, we have $\operatorname{H}_{3}(M^2)\cong M^3(-6)$ as graded $A$-modules, as claimed.
\end{proof}

%%%%%%%%%%%%
\subsubsection{\texorpdfstring{Homology of the Koszul complex of $M^3$}{Homology of the Koszul complex of M3}}
\label{subsubsec:HnM3}

\paragraph{The dimensions of the homology groups}
\phantom{x}
\\
Note first that $M^3 \cong N\oplus (\oplus _{k\in \llbracket 1,4 \rrbracket}S_k )$ as graded $A$-modules, where $N$ is the submodule of $M^3$ generated by $e_i,i\in\llbracket 1,4\rrbracket$, and $S_k$ is the submodule generated by $e_{k+4}$ for $k\in \llbracket 1,4 \rrbracket$.
Let $\{ f_i\mid  i\in \llbracket 1,8\rrbracket \}$ be the dual basis to $\{ e_i\mid i\in \llbracket 1,8\rrbracket \}$. 
It is easy to see that the $A^!$-module $(M^3)^!$ is generated by $f_i$ for $i\in \llbracket 1,8\rrbracket$, subject to the following $24$ relations 
\begin{equation}
\label{eq: dual relations in M3}
\begin{split}
   & 
   f_1y_{1,2}-f_2y_{3,4},
   f_1y_{3,4}+f_2y_{1,2},
   f_3y_{1,2}+f_4y_{3,4},
   f_3y_{3,4}-f_4y_{1,2},
   f_4y_{1,3}-f_2y_{2,4},
   f_4y_{2,4}+f_2y_{1,3},
   \\
   &
   f_3y_{1,3}-f_1y_{2,4},
   f_3y_{2,4}+f_1y_{1,3},
   f_1y_{2,3}+f_4y_{1,4},
   f_1y_{1,4}-f_4y_{2,3},
   f_3y_{2,3}+f_2y_{1,4},
   f_3y_{1,4}-f_2y_{2,3},
   \\
   &
   f_5y_{1,4}, f_5y_{2,4}, f_5y_{3,4},
   f_6y_{1,3}, f_6y_{2,3}, f_6y_{3,4},
   f_7y_{1,2}, 
   f_7y_{2,3},
   f_7y_{2,4},
   f_8y_{1,2}, f_8y_{1,3}, f_8y_{1,4}.
\end{split}
\end{equation}

Using GAP, a basis of $(M^3)^!_{-1}$ is given by the $24$ elements 
\begin{align*}
&
f_1y_{1,2}, 
f_1y_{1,3},
f_1y_{2,3},
f_1y_{1,4},
f_1y_{2,4},
f_1y_{3,4}, 
f_2 y_{1,3}, 
f_2 y_{2,3}, 
f_2 y_{1,4}, 
f_2 y_{2,4}, 
f_3 y_{1,2}, 
f_3 y_{3,4}, 
f_5 y_{1,2}, 
f_5 y_{1,3}, 
\\
&
f_5 y_{2,3}, 
f_6 y_{1,2},
f_6 y_{1,4}, 
f_6 y_{2,4},
f_7 y_{1,3},
f_7 y_{1,4},
f_7 y_{3,4}, 
f_8 y_{2,3},
f_8 y_{2,4},
f_8 y_{3,4}, 
\end{align*}
%%%%%
and a basis of $(M^3)^!_{-2}$ is given by the $40$ elements 
\begin{align*}
& 
f_1 y_{1,2}^2, 
f_1y_{1,2}y_{1,3},
f_1y_{1,2}y_{2,3},
f_1 y_{1,2}y_{1,4}, 
f_1y_{1,2}y_{2,4}, 
f_1 y_{1,2}y_{3,4}, 
f_1 y_{1,3}^2, 
f_1 y_{1,3}y_{1,4}, 
f_1y_{1,3}y_{2,4}, 
\\
&
f_1 y_{1,3}y_{3,4}, 
f_1 y_{2,3}^2, 
f_1 y_{2,3}y_{1,4}, 
f_1 y_{2,3}y_{2,4}, 
f_1 y_{2,3}y_{3,4}, 
f_2 y_{1,3}^2, 
f_2 y_{1,3}y_{2,4}, 
f_2 y_{2,3}^2, 
f_2 y_{2,3}y_{1,4},  
f_3 y_{1,2}^2, 
\\
&
f_3 y_{1,2} y_{3,4}, 
f_5 y_{1,2}^2,
f_5 y_{1,2}y_{1,3}, 
f_5 y_{1,2}y_{2,3}, 
f_5 y_{1,3}^2, 
f_5 y_{2,3}^2, 
f_6 y_{1,2}^2, 
f_6 y_{1,2}y_{1,4}, 
f_6 y_{1,2}y_{2,4},
f_6 y_{1,4}^2, 
f_6 y_{2,4}^2, 
\\
&
f_7 y_{1,3}^2, 
f_7 y_{1,3}y_{1,4}, 
f_7 y_{1,3}y_{3,4}, 
f_7 y_{1,4}^2, 
f_7 y_{3,4}^2, 
f_8 y_{2,3}^2, 
f_8 y_{2,3}y_{2,4}, 
f_8 y_{2,3}y_{3,4},
f_8 y_{2,4}^2,
f_8 y_{3,4}^2. 
\end{align*}
%%%%%
\begin{rk}
\label{rk:S1234}
Let $k\in \llbracket 1, 4\rrbracket$ and $f_{k+4}y_{i_1,j_1}\dots y_{i_s,j_s}$ be a monomial in $(S_k)^!$, where $s\in \NN$ and $(i_1,j_1),\dots, (i_s,j_s)\in \J$.
If $5-k\in \{i_1,j_1,\dots, i_s,j_s \}$, then $f_{k+4}y_{i_1,j_1}\dots y_{i_s,j_s}=0\in (S_k)^!$.
\end{rk}

\begin{rk}
\label{rk:rea}
Let $k\in \llbracket 1, 4\rrbracket$, $y_{i_1,j_1}\dots y_{i_s,j_s}=y_{i'_1,j'_1}\dots y_{i'_s,j'_s}$ in $A^!$ for $(i_1,j_1),\dots, (i_s,j_s)\in \J$, $(i'_1,j'_1),\dots, (i'_s,j'_s)\in \J$ and $s\in \NN$.
If $k\in \{i_1,j_1,\dots, i_s,j_s \}$, then $k\in \{i'_1,j'_1,\dots, i'_s,j'_s \}$. 
\end{rk}

\begin{lem}
\label{lem:basis dual S1234}
Let 
\begin{equation}
\begin{split}
T^{S_1}_n & = \{ f_5y_{1,2}^n, 
   f_5y_{1,2}^{n-1}y_{1,3},
   f_5y_{1,2}^{n-1}y_{2,3},
   f_5y_{1,2}^{n-2}y_{1,3}^2,
   f_5y_{1,3}^{n},
   f_5y_{2,3}^{n}
   \}, 
\\
T^{S_2}_n & = \{ f_6y_{1,2}^n, 
   f_6y_{1,2}^{n-1}y_{1,4},
   f_6y_{1,2}^{n-1}y_{2,4},
   f_6y_{1,2}^{n-2}y_{1,4}^2,
   f_6y_{1,4}^{n},
   f_6y_{2,4}^{n}
   \}, 
\\
T^{S_3}_n & = \{ f_7y_{1,3}^n, 
   f_7y_{1,3}^{n-1}y_{1,4},
   f_7y_{1,3}^{n-1}y_{3,4},
   f_7y_{1,3}^{n-2}y_{1,4}^2,
   f_7y_{1,4}^{n},
   f_7y_{3,4}^{n}
   \}, 
\\
T^{S_4}_n & = \{ f_8y_{2,3}^n, 
   f_8y_{2,3}^{n-1}y_{2,4},
   f_8y_{2,3}^{n-1}y_{3,4},
   f_8y_{2,3}^{n-2}y_{2,4}^2,
   f_8y_{2,4}^{n},
   f_8y_{3,4}^{n}
   \},
\end{split}
\end{equation}
for $n\geqslant 3$. 
Then $T_n^{S_k} $
is a basis of $(S_k)_{-n}^!$ for $k\in \llbracket 1,4 \rrbracket$ and $n\geqslant 3$. 
Note that $T^{S_k}_n$ has cardinal $6$ for $k \in \llbracket 1, 4 \rrbracket$ and $n\geqslant 3$.
\end{lem}

\begin{proof}
The space $(S_k)^!_{-n}$ is spanned by $\{ f_{k+4}y \mid y\in \B^!_{n} \}$. 
By Remark \ref{rk:S1234}, 
the space $(S_k)^!_{-n}$ is spanned by $T^{S_k}_n$ for $n\geqslant 3$.
By Remark \ref{rk:rea}, it is easy to see that the elements of $T^{S_k}_n$ are linearly independent. 
Indeed, let $\sum_{i\in\llbracket 1,6\rrbracket}\alpha_i q_i= \sum_{u\in \B^!_{n-1}}\lambda^1_u f_5 y_{1,4}u
+ \sum_{u\in \B^!_{n-1}}\lambda^2_u f_5 y_{2,4}u
+\sum_{u\in \B^!_{n-1}}\lambda^3_u f_5 y_{3,4}u
$ 
in $\Bbbk\{f_5\}\otimes A^!$,
where $\alpha_i,\lambda^j_u\in \Bbbk$, and $q_i$ is the $i$-th element of $T_n^{S_1}$. 
By Remark \ref{rk:rea},
the right side of the equation is a linear combination of elements of form $f_5 y_{i_1,j_1}\dots y_{i_n,j_n}\in f_5 \B^!_{n}$ for $4\in \{i_1,j_1,\dots, i_n,j_n \}$. 
This implies $\alpha_i=0$ for $i\in \llbracket 1,6\rrbracket$. 
Hence, $T_n^{S_1}$ are linearly independent. 
The other cases are similar. 
\end{proof}

%%%%%
\begin{lem}
   \label{lem:basis dual NNN}
The set $T_n^{N}$ consisting of the following $24$ elements
\begin{equation}
\label{eq:basis dual NNN}
\begin{split}
& f_1y_{1,2}^n, 
f_1y_{1,2}^{n-1}y_{3,4},
f_1y_{1,3}^n,
f_1y_{1,3}^{n-1}y_{2,4},
f_1y_{2,3}^n, 
f_1y_{2,3}^{n-1}y_{1,4},
f_2y_{1,3}^n,
f_2y_{1,3}^{n-1}y_{2,4},
f_2y_{2,3}^n,
\\
&
f_2y_{2,3}^{n-1}y_{1,4},
f_3y_{1,2}^n,
f_3y_{1,2}^{n-1}y_{3,4},
f_1y_{1,2}^{n-1}y_{1,3}, 
f_1y_{1,2}^{n-1}y_{2,3},
f_1y_{1,2}^{n-1}y_{1,4},
f_1y_{1,2}^{n-1}y_{2,4},
\\
&
f_1y_{1,2}^{n-2}y_{1,3}^2,
f_1y_{1,2}^{n-2}y_{1,3}y_{1,4},
f_1y_{1,2}^{n-2}y_{1,3}y_{2,4},
f_1y_{1,2}^{n-2}y_{1,3}y_{3,4},
f_1y_{1,2}^{n-2}y_{2,3}y_{1,4},
\\
&
f_1y_{1,2}^{n-2}y_{2,3}y_{2,4},
f_1y_{1,2}^{n-2}y_{2,3}y_{3,4},
f_1y_{1,2}^{n-3}y_{1,3}^2y_{3,4}
\end{split}
\end{equation}
is a basis of $N^!_{-n}$ for $n\geqslant 3$. 
\end{lem}

\begin{proof}
Firstly, using GAP, $T_n^{N}$ is a basis of $N^!_{-n}$ for $n\in \llbracket 3,5\rrbracket$.
Note that the space $N^!_{-n}$ is spanned by $\{ f_iy\mid i\in \llbracket 1,4\rrbracket , y\in \B^!_{n} \}$ for $n\in \NN_0$. 
By the dual relations, it is easy to see that $N^!_{-n}$ is spanned by 
\begin{equation}
\begin{split}
& f_1y_{1,2}^n,
f_1y_{1,2}^{n-1}y_{3,4},
f_1y_{1,3}^n,
f_1y_{1,3}^{n-1}y_{2,4},
f_1y_{2,3}^n,
f_1y_{2,3}^{n-1}y_{1,4},
f_2y_{1,3}^n,
f_2y_{1,3}^{n-1}y_{2,4},
f_2y_{2,3}^n,
f_2y_{2,3}^{n-1}y_{1,4},
\\
& 
f_3y_{1,2}^n,
f_3y_{1,2}^{n-1}y_{3,4},
f_iy,
\end{split}
\end{equation}
for $i\in \llbracket 1,3\rrbracket$, $y\in \U^!_{n}$ and $n\geqslant 2$.
Note that $y_{i,j}^2$ is central in $A^!$ and 
$f_{s}y_{i,j}^2=f_{s}y_{k,l}^2$ for $s\in\llbracket 1,4\rrbracket$ and $(i,j),(k,l)\in \I$ with $\# \{i,j,k,l \}=4$.
For $n \geqslant 5$ and $i\in \llbracket 1,3\rrbracket$, 
\begin{align*}
f_iy_{1,2}^{n-2}y_{1,4}^2 & = f_iy_{1,2}^{n-2}y_{2,3}^2
= f_i y_{1,2}^{n-2}y_{1,3}^2 , 
\hskip 1mm
f_iy_{1,2}^{n-3}y_{1,3}^2 y_{1,4}  = f_iy_{1,2}^{n-3}y_{2,4}^2 y_{1,4}
=f_iy_{1,2}^{n-1}y_{1,4} , 
\\
f_iy_{1,2}^{n-3}y_{1,3}^2 y_{2,4} & = f_iy_{1,2}^{n-3}y_{2,4}^3 
=f_i y_{1,2}^{n-1}y_{2,4}, 
\hskip 1mm
f_iy_{1,2}^{n-3}y_{1,3} y_{1,4}^2 = f_iy_{1,2}^{n-3}y_{1,3}y_{2,3}^2 
= f_iy_{1,2}^{n-1}y_{1,3}, 
\\
f_iy_{1,2}^{n-3}y_{2,3} y_{1,4}^2 & = f_iy_{1,2}^{n-3}y_{2,3}^3 
= f_i y_{1,2}^{n-1}y_{2,3} , 
\hskip 1mm
f_iy_{1,2}^{n-4}y_{1,3}^2 y_{1,4}^2  = f_iy_{1,2}^{n-4}y_{1,3}^2 y_{2,3}^2 
=  f_iy_{1,2}^{n-2}y_{1,3}^2 , 
\\
f_iy_{1,3}^{n-1}y_{1,4} & = \chi_n f_iy_{1,3}y_{2,4}^{n-2}y_{1,4}+\chi_{n+1}f_iy_{2,4}^{n-1}y_{1,4} 
= \chi_n f_i y_{1,3} y_{1,2}^{n-2}y_{1,4}+\chi_{n+1}f_iy_{1,2}^{n-1}y_{1,4} 
\\
& 
=\chi_n f_iy_{1,2}^{n-2}y_{1,3}y_{1,4}+\chi_{n+1}f_iy_{1,2}^{n-1}y_{1,4},
\\
f_iy_{1,3}^{n-1}y_{3,4} & = f_iy_{1,3}^{n-3}y_{2,4}^2y_{3,4} 
= f_iy_{1,3}^{n-3}y_{2,3}^2y_{3,4} 
= \chi_n f_i y_{1,2}^{n-3}y_{1,3}y_{2,3}y_{3,4}+\chi_{n+1}f_iy_{1,2}^{n-3}y_{2,3}^2y_{3,4} 
\\
&
=\chi_n f_i y_{1,2}^{n-2}y_{1,3}y_{3,4}+\chi_{n+1}f_iy_{1,2}^{n-3}y_{1,3}^2y_{3,4} , 
\\
f_iy_{2,3}^{n-1}y_{2,4} & = \chi_n f_iy_{2,3}y_{1,4}^{n-2}y_{2,4}+\chi_{n+1}f_iy_{1,4}^{n-1}y_{2,4} 
= \chi_n f_iy_{2,3}y_{1,2}^{n-2}y_{2,4} +\chi_{n+1}f_iy_{1,2}^{n-1}y_{2,4} 
\\
&
= \chi_n f_i y_{1,2}^{n-2}y_{2,3}y_{2,4}+\chi_{n+1}f_iy_{1,2}^{n-1}y_{2,4}, 
\\
f_iy_{2,3}^{n-1}y_{3,4} & = f_iy_{2,3}^{n-3}y_{1,4}^2y_{3,4}=f_iy_{2,3}^{n-3}y_{1,3}^2y_{3,4} 
= - \chi_n f_i y_{1,2}^{n-3}y_{2,3}y_{1,3}y_{3,4}+\chi_{n+1} f_i y_{1,2}^{n-3}y_{1,3}^2y_{3,4}
\\
& 
= \chi_n f_i y_{1,2}^{n-2}y_{2,3}y_{3,4}+\chi_{n+1} f_i y_{1,2}^{n-3}y_{1,3}^2y_{3,4}.
\end{align*}
Moreover, by the dual relation $f_2y_{1,2}=-f_1y_{3,4}$, and 
\begin{align*}
f_3y_{1,2}^{n-1}y_{1,3} 
& = \chi_n f_3y_{1,2}y_{1,3}y_{1,2}^{n-2}+\chi_{n+1}f_3y_{1,3}y_{1,2}^{n-1} 
= -\chi_n f_3y_{2,3}y_{1,2}y_{1,2}^{n-2}+\chi_{n+1}f_3y_{1,3}y_{1,2}^{n-1} 
\\
&
=\chi_n f_2 y_{1,4}y_{1,2}^{n-1}+\chi_{n+1}f_1y_{2,4}y_{1,2}^{n-1}
% =-\chi_n f_2 y_{1,2}y_{2,4}y_{1,2}^{n-2}+\chi_{n+1}f_1y_{1,2}^{n-1}y_{2,4}
% =\chi_n f_1 y_{3,4}y_{2,4}y_{1,2}^{n-2}+\chi_{n+1}f_1y_{1,2}^{n-1}y_{2,4}
% =-\chi_n f_1y_{1,2}^{n-2}y_{2,3}y_{3,4} +\chi_{n+1}f_1y_{1,2}^{n-1}y_{2,4}
, 
\\
f_3y_{1,2}^{n-1}y_{2,3} 
& = \chi_n f_3 y_{1,2}y_{2,3}y_{1,2}^{n-2}+\chi_{n+1}f_3y_{2,3}y_{1,2}^{n-1}
=-\chi_n f_3y_{2,3}y_{1,3}y_{1,2}^{n-2}+\chi_{n+1}f_3y_{2,3}y_{1,2}^{n-1} 
\\
& 
=\chi_n f_2y_{1,4}y_{1,3}y_{1,2}^{n-2}-\chi_{n+1}f_2y_{1,4}y_{1,2}^{n-1}, 
\\
f_3y_{1,2}^{n-1}y_{1,4} 
& = \chi_n f_3y_{1,2}y_{1,4}y_{1,2}^{n-2}+\chi_{n+1}f_3y_{1,4}y_{1,2}^{n-1}
= -\chi_n f_3 y_{2,4}y_{1,2}y_{1,2}^{n-2}+\chi_{n+1}f_3y_{1,4}y_{1,2}^{n-1}
\\
& = \chi_n f_1y_{1,3}y_{1,2}^{n-1}+\chi_{n+1}f_2y_{2,3}y_{1,2}^{n-1}, 
\\
f_3y_{1,2}^{n-1}y_{2,4}
& =  \chi_n f_3y_{1,2}y_{2,4}y_{1,2}^{n-2}+\chi_{n+1}f_3y_{2,4}y_{1,2}^{n-1} 
= -\chi_n f_3y_{1,4}y_{1,2}y_{1,2}^{n-2}+\chi_{n+1}f_3y_{2,4}y_{1,2}^{n-1} 
\\
& 
=-\chi_n f_2y_{2,3}y_{1,2}y_{1,2}^{n-2}-\chi_{n+1}f_1y_{1,3}y_{1,2}^{n-1} ,
\end{align*}
for $n\geqslant 3$, 
the space $N^!_{-n}$ is spanned by $T^N_n$ for $n\geqslant 5$. 

Next, we prove that the elements in $T_n^N$ for $n\geqslant 6$ are linearly independent. 
Suppose that we have the identity 
\begin{equation}
\label{eq:1}
\begin{split}
\sum_{i\in \llbracket 1,24\rrbracket}\alpha_{i}t_i
& = 
\sum_{u\in \B^!_{n-1}} \lambda^1_{u} (f_1y_{1,2}-f_2y_{3,4}) u
+\sum_{u\in \B^!_{n-1}} \lambda^2_{u} (f_1y_{3,4}+f_2y_{1,2} ) u
\\
& \phantom{= \;}
+\sum_{u\in \B^!_{n-1}} \lambda^3_{u} 
(f_3y_{1,2}+f_4y_{3,4})u
+\sum_{u\in \B^!_{n-1}} \lambda^4_{u} 
(f_3y_{3,4}-f_4y_{1,2})u
\\
& \phantom{= \;} 
+\sum_{u\in \B^!_{n-1}} \lambda^5_{u} 
(f_4y_{1,3}-f_2y_{2,4})u
+\sum_{u\in \B^!_{n-1}} \lambda^6_{u} 
(f_4y_{2,4}+f_2y_{1,3} )u 
\\
& \phantom{= \;} 
+\sum_{u\in \B^!_{n-1}} \lambda^7_{u} 
(f_3y_{1,3}-f_1y_{2,4})u
+\sum_{u\in \B^!_{n-1}} \lambda^8_{u} 
(f_3y_{2,4}+f_1y_{1,3})u
\\
& \phantom{= \;} 
+\sum_{u\in \B^!_{n-1}} \lambda^9_{u} 
(f_1y_{2,3}+f_4y_{1,4})u
+\sum_{u\in \B^!_{n-1}} \lambda^{10}_{u} 
(f_1y_{1,4}-f_4y_{2,3})u
\\
& \phantom{= \;} 
+\sum_{u\in \B^!_{n-1}} \lambda^{11}_{u} 
(f_3y_{2,3}+f_2y_{1,4})u
+\sum_{u\in \B^!_{n-1}} \lambda^{12}_{u} 
(f_3y_{1,4}-f_2y_{2,3})u, 
\end{split}
\end{equation}
in $\Bbbk\{f_1,f_2,f_3,f_4 \} \otimes A^!$, where $\alpha_i, \lambda^{j}_{u} \in \Bbbk$, and $t_i$ is the $i$-th element in \eqref{eq:basis dual NNN}. 
We need to show that $\alpha_i = 0$ for all $i \in \llbracket 1, 24 \rrbracket$. 
By inspecting the coefficients of the term $f_sy_{i,j}^{n-r}y_{k,l}^r$ for $\# \{i,j,k,l \}=4$, it is easy to see that $\alpha_i=0$ for $i\in \llbracket 1, 12\rrbracket$.
Then \eqref{eq:1} is equivalent to 
\begin{equation}
\label{eq:2}
\begin{split}
f_1\big( y_{1,2} \Delta^1+y_{3,4}\Delta^2+y_{1,3}\Delta^8-y_{2,4}\Delta^7 +y_{2,3}\Delta^9+y_{1,4}\Delta^{10} \big) & =\sum_{i\in\llbracket 13,24 \rrbracket} \alpha_i t_i, 
   \\
   f_2\big( y_{1,2} \Delta^2-y_{3,4}\Delta^1+y_{1,3}\Delta^6-y_{2,4}\Delta^5 -y_{2,3}\Delta^{12}+y_{1,4}\Delta^{11} \big) & = 0, 
   \\
   f_3\big( y_{1,2} \Delta^3+y_{3,4}\Delta^4+y_{1,3}\Delta^7+y_{2,4}\Delta^8 +y_{2,3}\Delta^{11}+y_{1,4}\Delta^{12} \big) & = 0,
   \\
   f_4\big( -y_{1,2} \Delta^4+y_{3,4}\Delta^3+y_{1,3}\Delta_5+y_{2,4}\Delta^6 -y_{2,3}\Delta^{10}+y_{1,4}\Delta^{9} \big) & = 0,
\end{split}
\end{equation}
in $\Bbbk\{f_i\}\otimes A^!$ for $i\in \llbracket 1,4\rrbracket $ respectively,  
where $\Delta^j=\sum_{u\in \B^!_{n-1}\setminus Y_{1,2}}\lambda^j_u u $ for $j\in \llbracket 1,4\rrbracket $, 
$\Delta^j=\sum_{u\in \B^!_{n-1}\setminus Y_{1,3}}\lambda^j_u u $ for $j\in \llbracket 5,8\rrbracket $, 
$\Delta^j=\sum_{u\in \B^!_{n-1}\setminus Y_{2,3}}\lambda^j_u u $ for $j\in \llbracket 9,12\rrbracket $ 
and $Y_{i,j}$ is defined in \eqref{eq:Yij}. 
In particular, we see that the elements in $T_n^N$ are linearly independent if and only if 
equation \eqref{eq:2} implies that $\alpha_i=0$ for all $i\in \llbracket 13, 24\rrbracket$.

Let 
\begin{small}
\begin{align*}
a^j_0 & =\lambda^j_{y_{1,2}^{n-1}}, 
a'^j_{0}=\lambda^j_{y_{3,4}^{n-1}},
a^{j}_1=\underset{\text{\begin{tiny}$\begin{matrix} r\in \llbracket 1, n-2 \rrbracket, \\ r \; odd \end{matrix}$\end{tiny}}}{\sum}
\lambda^j_{y_{1,2}^{n-1-r}y_{3,4}^{r}}, 
a^{j}_2=\underset{\text{\begin{tiny}$\begin{matrix} r\in \llbracket 1, n-2 \rrbracket, \\ r \; even \end{matrix}$\end{tiny}}}{\sum}
\lambda^j_{y_{1,2}^{n-1-r}y_{3,4}^{r}} \text{ for $j\in \llbracket 5,12 \rrbracket $},
\\
b^j_0 & =\lambda^j_{y_{1,3}^{n-1}}, 
b'^j_{0}=\lambda^j_{y_{2,4}^{n-1}},
b^{j}_1=\underset{\text{\begin{tiny}$\begin{matrix} r\in \llbracket 1, n-2 \rrbracket, \\ r \; odd \end{matrix}$\end{tiny}}}{\sum}
\lambda^j_{y_{1,3}^{n-1-r}y_{2,4}^{r}}, 
b^{j}_2=\underset{\text{\begin{tiny}$\begin{matrix} r\in \llbracket 1, n-2 \rrbracket, \\ r \; even \end{matrix}$\end{tiny}}}{\sum}
\lambda^j_{y_{1,3}^{n-1-r}y_{2,4}^{r}} \text{ for $j\in \llbracket 1,4 \rrbracket \cup \llbracket 9,12\rrbracket $},
\\
c^j_0 & =\lambda^j_{y_{2,3}^{n-1}}, 
c'^j_{0}=\lambda^j_{y_{1,4}^{n-1}},
c^{j}_1=\underset{\text{\begin{tiny}$\begin{matrix} r\in \llbracket 1, n-2 \rrbracket, \\ r \; odd \end{matrix}$\end{tiny}}}{\sum}
\lambda^j_{y_{2,3}^{n-1-r}y_{1,4}^{r}}, 
c^{j}_2=\underset{\text{\begin{tiny}$\begin{matrix} r\in \llbracket 1, n-2 \rrbracket, \\ r \; even \end{matrix}$\end{tiny}}}{\sum}
\lambda^j_{y_{2,3}^{n-1-r}y_{1,4}^{r}} \text{ for $j\in \llbracket 1,8 \rrbracket $}.
\end{align*}
\end{small}
\hskip -1mm Using \eqref{eq:2} together with the products \eqref{eq:product y13y12ny34r} and \eqref{eq:product y13y12n},  in $A^!$, we get a system of linear equations $E_n$, which contains $24\times 4 = 96$ linear equations and $12+24\times 12+ 4\times 8\times 3 = 396$ variables $\alpha_i$, $\lambda^j_u$ for $u\in \U^!_{n-1}$, $a^j_0$, $a'^j_0$, $a^j_1$, $a^j_2$, $b^j_0$, $b'^j_0$, $b^j_1$, $b^j_2$, $c^j_0$, $c'^j_0$, $c^j_1$, $c^j_2$. 
Hence, the linear independence of $T_{n}^{N}$ (or, equivalently, the fact that equation \eqref{eq:2} implies that $\alpha_i=0$ for all $i\in \llbracket 13, 24\rrbracket$) is tantamount to the fact that the linear system $E_n$ implies that $\alpha_i=0$ for all $i\in \llbracket 13, 24\rrbracket$.
Furthermore, it is easy to see that $E_n$ has the same form as $E_{n+2}$. 
We then use GAP to check that the elements in $T^N_{n}$ are linearly independent for $n\in \llbracket 6, 7 \rrbracket$, and conclude that 
the lemma holds for all integers $n\geqslant 6$. 
\end{proof}

\begin{cor}
\label{cor:M3}
We have $\operatorname{H}_n(M^3)=0$ for $n\in \NN \setminus \{ 3 \}$. 
\end{cor}

\begin{proof}
By Tables \ref{table:yn123} and \ref{table:yn456}, and the reductions in the proof of Lemma \ref{lem:basis dual NNN}, 
the differential at homological degree $n$ in the Koszul complex $N$ or $S_k$ has the same form when $n\geqslant 4$ increases by $2$. 
Then $\operatorname{H}_{n+2}(M^3)=\operatorname{H}_{n}(M^3)$ for $n\geqslant 4$. 
Using GAP, $\operatorname{H}_n(M^3)=0$ for $n\in \llbracket 1,5 \rrbracket  \setminus \{ 3 \}$. 
By induction on $n$, $\operatorname{H}_n(M^3)=0$ for $n\in \NN \setminus \{ 3 \}$. 
\end{proof}

By Appendix \ref{sec:cpx} and 
the code 
\parskip 1ex
\parindent 0in
\begin{tcolorbox}[breakable,colback=white,width=\textwidth ,center,arc=0mm,size=fbox]
\begin{footnotesize}
\begin{Verbatim}[samepage=false]
for j in [0..8] do
    for i in [1..12] do
        Print(j, " ", i, " ", RankMat(FF(3,j,i)), "\n");
    od;
od;
\end{Verbatim}
\end{footnotesize}
\end{tcolorbox}
we obtain the dimension of $B^{M^3}_{n,m}$ for $n\in \llbracket 0,8\rrbracket$ and $m\in\llbracket 1,12\rrbracket$. 
Then the dimension of $\operatorname{H}_{n,m}(M^3)$ is given by Table \ref{table:HHH}
for $n\in \llbracket 1,8 \rrbracket   $ and $m\in\llbracket 0,12 \rrbracket $.
The dimensions that are not listed in the following table are zeros. 
\begin{table}[H]
   \begin{center}
    %  \resizebox{\textwidth}{120mm}{
      \begin{tabular}{|c|ccccccccccccc|}
      \hline
       \diagbox[width=14mm,height=5mm]{$n$}{$m$}  & $0$ & $1$  & $2$ & $3$ & $4$ & $5$ & $6$ & $7$ & $8$ & $9$ & $10$ & $11$ & $12$ 
     \\
     \hline
     $3$ &  &  &  & $8$ & $24$ & $48$ & $72$ & $80$ & $72$ & $48$ & $24$ & $8$ &   
     \\
      \hline
   \end{tabular}
 %  }
 \vspace{1mm}
   \caption{Dimension of $\operatorname{H}_{n,m}(M^3)$.}	
   \label{table:HHH}
\end{center}
\end{table}
\vspace{-0.8cm}

%%%%%%%%%%%%%%%%%%%%%%%%%%%%%%%%%%%%%%%%%%%%%%%%%%%%%%%%%%%%%%%%%%%%
\paragraph{\texorpdfstring{The $A$-module structure of the homology groups}{The A-module structure of the homology groups}}

%%%%%%%
\begin{lem}
\label{lemma:homology-m3}
We have the following isomorphisms of graded $A$-modules 
\begin{equation}
   \label{eq:HK M3}
\operatorname{H}_{n}(M^3) \cong
\begin{cases} 
   M^3(-6) , & \text{if $n = 3$,}
   \\
   0, & \text{if $n \in \NN \setminus \{ 3 \}$.}
\end{cases}
\end{equation}
\end{lem}
%%%%%%%
\begin{proof}
The isomorphism in \eqref{eq:HK M3} for $n \in \NN \setminus \{ 3 \}$ immediately follows from Corollary \ref{cor:M3}.
We prove that the space $\operatorname{H}_{3}(M^3)$ is isomorphic to $M^{3}(-6)$. 
The following GAP code shows that the dimension vector of the submodule of $\operatorname{H}_{3}(M^3)$ generated by basis elements $e''_i, i\in \llbracket 1,8 \rrbracket$ of $\operatorname{H}_{3,3}(M^3)$ is $(8,24,48,72,80,72,48,24,8)$.
So, $\operatorname{H}_{3}(M^3)$ is generated by the eight elements as an $A$-module.
\parskip 1ex
\parindent 0in
\begin{tcolorbox}[breakable,colback=white,width=\textwidth ,center,arc=0mm,size=fbox]
\begin{footnotesize}
\begin{Verbatim}[samepage=false]
Imm:=Im(3,3,3);;
RankMat(Imm);
# 672
gene:=geneMH(3,3,3);;
Append(Imm,gene);
RankMat(Imm);
# 680
Uh:=UU(gene,3);; Vh:=VV(gene,3);; Wh:=WW(gene,3);;
for r in [4..11] do 
   hxr:=HXR(3,Uh,Vh,Wh,3,3,r-3);
   Im3r:=Im(3,3,r);
   Append(Im3r, hxr);
   Print(r, " ", RankMat(Im3r)-RankMat(Im(3,3,r)), "\n");
od;
# 4 24
# 5 48
# 6 72
# 7 80
# 8 72
# 9 48
# 10 24
# 11 8
\end{Verbatim}
\end{footnotesize}
\end{tcolorbox}
Furthermore, it is direct to check that the generators $e''_i, i\in \llbracket 1,8 \rrbracket$ of $\operatorname{H}_{3}(M^3)$ satisfy the quadratic relations \eqref{eq:relation in M3}. 
Indeed, the following code shows that the dimension of the subspace generated by $B_{3,4}^{M^3}$ together with the elements of the form \eqref{eq:relation in M3} with $e''_i$ instead of $e_i$ coincides with the dimension of $B_{3,4}^{M^3}$.
\parskip 1ex
\parindent 0in
\begin{tcolorbox}[breakable,colback=white,width=\textwidth ,center,arc=0mm,size=fbox]
\begin{footnotesize}
\begin{Verbatim}[samepage=false]
gene:=geneMH(3,3,3);;
Uh:=UU(gene,3);; Vh:=VV(gene,3);; Wh:=WW(gene,3);;
hx:=HXR(3,Uh,Vh,Wh,3,3,1);;
cc:=0*[1..24];;
cc[1]:=hx[1]+hx[12];; cc[2]:=hx[6]-hx[7];; cc[3]:=hx[13]-hx[24];; 
cc[4]:=hx[18]+hx[19];; cc[5]:=hx[20]+hx[11];; cc[6]:=hx[23]-hx[8];;
cc[7]:=hx[14]+hx[5];; cc[8]:=hx[17]-hx[2];; cc[9]:=hx[3]-hx[22];;
cc[10]:=hx[4]+hx[21];; cc[11]:=hx[15]-hx[10];; cc[12]:=hx[16]+hx[9];;
cc[13]:=hx[25];; cc[14]:=hx[26];; cc[15]:=hx[27];; cc[16]:=hx[31];; cc[17]:=hx[34];; 
cc[18]:=hx[35];; cc[19]:=hx[38];; cc[20]:=hx[40];; cc[21]:=hx[42];; cc[22]:=hx[45];; 
cc[23]:=hx[47];; cc[24]:=hx[48];; 
Imm:=Im(3,3,4);;
RankMat(Imm);
# 1344
Append(Imm,cc);
RankMat(Imm);
# 1344
\end{Verbatim}
\end{footnotesize}
\end{tcolorbox}
Hence, we see that there is a surjective morphism $M^3(-6)\to \operatorname{H}_{3}(M^3)$ of graded $A$-modules. 
Since the dimension vector of $M^3$ is $(8,24,48,72,80,72,48,24,8)$, we have $\operatorname{H}_{3}(M^3)\cong M^3(-6)$ as graded $A$-modules, as claimed.
\end{proof}

\subsection{\texorpdfstring{Proof of Theorem \ref{thm:rd fk4}}{Proof of Theorem 3.5}}

\begin{proof}[Proof of Theorem \ref{thm:rd fk4}] 
\label{proof:rd fk4}
The result is a direct consequence of Lemmas \ref{lemma:homology-m0}, \ref{lemma:homology-m1}, 
\ref{lemma:homology-m2} and \ref{lemma:homology-m3}. 
\end{proof}

%%%%%%%%%%%%%%%%%%%%%%%%%%%%%%%%%%%%%%%%%%%%%%%%%%%%%%%%%%%%%%%%%%%%%%%%%%%%%%%%%%%%%%%%%%%%%%%%%%%
\appendix

\section{\texorpdfstring{Some computations}{Some computations}}
\label{sec:appendix}

In this Appendix, we list some computations about the Fomin-Kirillov algebra $\FK(4)$ of index $4$. 
As before, we will usually denote $\FK(4)$ simply by $A$. 

\subsection{\texorpdfstring{A basis of $\FK(4)$}{A basis of FK(4)}}
\label{subsec:W}

We present here the GAP code as well the result to compute the basis $W^{1,2}$ (consisting of standard words) of $A$ 
under the order $x_{1,2} \prec x_{3,4} \prec x_{1,3} \prec x_{2,3} \prec x_{1,4} \prec x_{2,4}$. 

\parskip 1ex
\parindent 0in
\begin{tcolorbox}[breakable,colback=white,width=\textwidth ,center,arc=0mm,size=fbox]
\begin{footnotesize}
\begin{Verbatim}[samepage=false]
LoadPackage("GBNP");
A:=FreeAssociativeAlgebraWithOne(Rationals,"x12","x34","x13","x23","x14","x24");;
x12:=A.x12;; x13:=A.x13;; x23:=A.x23;; x14:=A.x14;; x24:=A.x24;; x34:=A.x34;; 
oA:=One(A);;
relationsA:=[x12^2, x13^2, x23^2, x14^2, x24^2, x34^2, x12*x23-x23*x13-x13*x12, 
 x23*x12-x12*x13-x13*x23, x12*x24-x24*x14-x14*x12, x24*x12-x12*x14-x14*x24, 
 x13*x34-x34*x14-x14*x13, x34*x13-x13*x14-x14*x34, x23*x34-x34*x24-x24*x23, 
 x34*x23-x23*x24-x24*x34, x12*x34-x34*x12, x13*x24-x24*x13, x14*x23-x23*x14];;
relsANP:=GP2NPList(relationsA);;
GA:=Grobner(relsANP);;
GBNP.ConfigPrint(A);
PrintNPList(GA);

 x12^2
 x34x12 - x12x34
 x34^2
 x13^2
 x23x12 - x13x23 - x12x13
 x23x13 + x13x12 - x12x23
 x23^2
 x14x34 + x13x14 - x34x13
 x14x13 - x13x34 + x34x14
 x14x23 - x23x14
 x14^2
 x24x12 - x14x24 - x12x14
 x24x34 + x23x24 - x34x23
 x24x13 - x13x24
 x24x23 - x23x34 + x34x24
 x24x14 + x14x12 - x12x24
 x24^2
 x13x12x13 + x12x13x12
 x13x34x13 - x34x13x34
 x23x34x23 - x34x23x34
 x14x12x34 + x13x14x12 - x34x13x12
 x14x12x13 - x23x14x12 + x13x34x23 - x34x23x14
 x14x12x23 + x23x34x14 + x34x14x12 - x12x23x34
 x14x12x14 + x12x14x12
 x23x34x13x12 - x13x34x23x34 + x34x23x34x13 - x12x13x34x23
 x23x34x13x34 + x13x12x34x13 - x12x23x34x13
 x23x34x13x23 + x13x23x34x13 - x34x13x23x34
 x13x12x34x13x12 + x34x13x12x34x13
 x13x12x34x13x34 + x12x13x12x34x13

W:=BaseQA(GA,6,0);;
PrintNPList(W); 
\end{Verbatim}
\end{footnotesize}
\end{tcolorbox}
The basis $W^{1,2}$ is given by the following $576$ elements


%
%\begin{multicols}{2}
%\parskip 1ex
%\parindent 0in
%\begin{tcolorbox}[breakable,colback=white,width=0.49\textwidth ,center,arc=0mm,size=fbox]
%\begin{footnotesize}
%\begin{Verbatim}[samepage=false]
 %\end{multicols}

\subsection{\texorpdfstring{Products in $\FK(4)$}{Products in FK(4)}}
\label{sec:products FK4}

It is easy to check the products in $A^!$, listed in Table \ref{table:yy}-\ref{table:yyyyy}, by using GAP or by computing them directly, and to check the products listed in \ref{table:yn123}-\ref{table:y'yn odd} by induction on integers $n\geqslant 5$. 
In Tables \ref{table:yy}-\ref{table:yyyyy}, \ref{table:yn123} and \ref{table:yn456}, the entry appearing in the row indexed by $y$ and the column indexed by $y'$ is the product $yy'$. 
In Tables \ref{table:y'yn even} and \ref{table:y'yn odd}, the entry appearing in the column indexed by $y'$ and the row indexed by $y$ is the product $y'y$. 
To reduce space, in Table \ref{table:yyyy}, we write the product $yy'$ by $\pm m$, where $m\in \llbracket 55 ,92\rrbracket$ is the integer appearing in the first column of Table \ref{table:yyyyy}, and indicating the element in the second column of Table \ref{table:yyyyy} that is in the same row as it. 
In Table \ref{table:yyyyy}, we write the product $yy'$ by $\pm m$, where $m\in \llbracket 93,134\rrbracket$ is the integer appearing in the first column of Table \ref{table:y5}, and indicating the element in the second column of Table \ref{table:y5} that is in the same row as it. 
In Table \ref{table:y'yn even} and \ref{table:y'yn odd}, we write the product $yy'$ by $\pm m$, where $m\in \llbracket 1,24\rrbracket$ is the integer appearing in the first column of Table \ref{table:y'yn even} (or \ref{table:y'yn odd}), and indicating the element $\mathsf{a}^{n+1}_m$, where 
\begin{align*}
\mathsf{a}^n_{1} & =y_{1,2}^{n-1}y_{1,3}, & 
\mathsf{a}^n_{2}&=y_{1,2}^{n-2}y_{1,3}^2, &
\mathsf{a}^n_{3}&=y_{1,2}^{n-1}y_{2,3}, &
\mathsf{a}^n_{4}&=y_{1,2}^{n-1}y_{1,4}, 
\\
\mathsf{a}^n_{5}&=y_{1,2}^{n-2}y_{1,3}y_{1,4}, &
\mathsf{a}^n_{6} & =y_{1,2}^{n-3}y_{1,3}^2y_{1,4},&
\mathsf{a}^n_{7}&=y_{1,3}^{n-1}y_{1,4},&
\mathsf{a}^n_{8}&=y_{1,2}^{n-2}y_{2,3}y_{1,4}, 
\\
\mathsf{a}^n_{9}&=y_{1,2}^{n-2}y_{1,4}^2, &
\mathsf{a}^n_{10}&=y_{1,2}^{n-3}y_{1,3}y_{1,4}^2, &
\mathsf{a}^n_{11} & =y_{1,2}^{n-4}y_{1,3}^2y_{1,4}^2, &
\mathsf{a}^n_{12}&=y_{1,3}^{n-2}y_{1,4}^2, 
\\
\mathsf{a}^n_{13}&=y_{1,2}^{n-3}y_{2,3}y_{1,4}^2, &
\mathsf{a}^n_{14}&=y_{1,2}^{n-1}y_{2,4}, &
\mathsf{a}^n_{15}&=y_{1,2}^{n-2}y_{1,3}y_{2,4}, &
\mathsf{a}^n_{16} & =y_{1,2}^{n-3}y_{1,3}^2y_{2,4}, 
\\
\mathsf{a}^n_{17}&=y_{1,2}^{n-2}y_{2,3}y_{2,4}, &
\mathsf{a}^n_{18}&=y_{2,3}^{n-1}y_{2,4}, &
\mathsf{a}^n_{19}&=y_{2,3}^{n-2}y_{2,4}^2, &
\mathsf{a}^n_{20}&=y_{1,2}^{n-2}y_{1,3}y_{3,4}, 
\\
\mathsf{a}^n_{21} & =y_{1,2}^{n-3}y_{1,3}^2y_{3,4},& 
\mathsf{a}^n_{22}&=y_{1,3}^{n-1}y_{3,4}, &
\mathsf{a}^n_{23}&=y_{1,2}^{n-2}y_{2,3}y_{3,4}, &
\mathsf{a}^n_{24}&=y_{2,3}^{n-1}y_{3,4}, 
\end{align*}
for $n\geqslant 5$ and $m\in\llbracket 1,24\rrbracket$. 

%%%%%%%%%%%%%%%%%%%%%%%%%%%%%%%%%
\begin{table}[H]
\begin{center}
\begin{tabular}{c|c|cccccc}
		\hline
		$1$ & \diagbox[width=19mm,height=5mm]{$y$}{$y'$}  & $y_{1,2}$ & $y_{1,3}$  & $y_{2,3}$ & $y_{1,4}$ & $y_{2,4}$ & $y_{3,4}$ 
      \\
      \hline
      $2$ & $y_{1,2}$ & $y_{1,2}^2$ & $y_{1,2}y_{1,3}$ & $y_{1,2}y_{2,3}$ & $y_{1,2}y_{1,4}$ & $y_{1,2}y_{2,4}$ & $y_{1,2}y_{3,4}$
        \\
      $3$ & $y_{1,3}$ & $-y_{1,2}y_{2,3}$ & $y_{1,3}^2$ & $y_{1,2}y_{1,3}$ & $y_{1,3}y_{1,4}$ & $y_{1,3}y_{2,4} $ &  $y_{1,3}y_{3,4}$
        \\
      $4$ & $y_{2,3}$ & $-y_{1,2}y_{1,3}$ & $-y_{1,2}y_{2,3}$ & $y_{2,3}^2$  & $y_{2,3}y_{1,4}$ & $y_{2,3}y_{2,4}$ & $y_{2,3}y_{3,4}$
        \\
      $5$ & $y_{1,4}$ & $-y_{1,2}y_{2,4}$ & $-y_{1,3}y_{3,4} $ & $-y_{2,3}y_{1,4}$  & $y_{1,4}^2$ & $y_{1,2}y_{1,4}$ & $y_{1,3}y_{1,4}$
        \\
      $6$ & $y_{2,4}$ & $-y_{1,2}y_{1,4}$ & $-y_{1,3}y_{2,4} $ & $-y_{2,3}y_{3,4}$  & $-y_{1,2}y_{2,4}$ & $y_{2,4}^2$ & $y_{2,3}y_{2,4}$
        \\
      $7$ & $y_{3,4}$ & $-y_{1,2}y_{3,4}$ & $-y_{1,3}y_{1,4}$ & $-y_{2,3}y_{2,4}$ & $-y_{1,3}y_{3,4}$ & $-y_{2,3}y_{3,4}$ & $y_{3,4}^2$
      \\
      \hline
	\end{tabular}
   % }
   \vspace{1mm}
	\caption{Products $yy'$.}	
	\label{table:yy}
\end{center}
\end{table}
%\vskip -1cm

\begin{table}[H]
\begin{center}
\resizebox{\textwidth}{38mm}{
\begin{tabular}{c|c|cccccc}
		\hline
		$1$ & \diagbox[width=19mm,height=5mm]{$y$}{$y'$}  & $y_{1,2}$ & $y_{1,3}$  & $y_{2,3}$ & $y_{1,4}$ & $y_{2,4}$ & $y_{3,4}$ 
      \\
      \hline
      $8$ & $y_{1,2}^2$ & $y_{1,2}^3$ & $y_{1,2}^2y_{1,3}$ & $y_{1,2}^2y_{2,3}$ &  $y_{1,2}^2y_{1,4}$ & $y_{1,2}^2y_{2,4}$ & $y_{1,2}^2y_{3,4}$
      \\
      $9$ & $y_{1,2}y_{1,3}$ & $-y_{1,2}^2y_{2,3}$ & $y_{1,2}y_{1,3}^2$ & $y_{1,2}^2y_{1,3}$ & $y_{1,2}y_{1,3}y_{1,4}$ & $y_{1,2}y_{1,3}y_{2,4}$ & $y_{1,2}y_{1,3}y_{3,4}$
      \\
      $10$ & $y_{1,2}y_{2,3}$ & $-y_{1,2}^2y_{1,3} $ & $-y_{1,2}^2y_{2,3}$ & $y_{1,2}y_{1,3}^2$ & $y_{1,2}y_{2,3}y_{1,4}$ &$y_{1,2}y_{2,3}y_{2,4}$ & $y_{1 ,2}y_{2,3}y_{3,4}$
      \\
      $11$ & $y_{1,2}y_{1,4}$ & $-y_{1,2}^2y_{2,4} $ & $-y_{1,2}y_{1,3}y_{3,4}$ & $-y_{1,2}y_{2,3}y_{1,4}$ & $y_{1,2}y_{1,4}^2$ & $y_{1,2}^2y_{1,4}$ & $y_{1,2}y_{1,3}y_{1,4}$
      \\
      $12$ & $y_{1,2}y_{2,4}$ & $-y_{1,2}^2y_{1,4}$ & $-y_{1,2}y_{1,3}y_{2,4}$ & $-y_{1,2}y_{2,3}y_{3,4}$ & $-y_{1,2}^2y_{2,4}$ & $y_{1,2}y_{1,4}^2$ & $y_{1,2}y_{2,3}y_{2,4}$
      \\
      $13$ & $y_{1,2}y_{3,4}$ & $-y_{1,2}^2y_{3,4} $ & $-y_{1,2}y_{1,3}y_{1,4}$ & $-y_{1,2}y_{2,3}y_{2,4}$ & $-y_{1,2}y_{1,3}y_{3,4}$ & $-y_{1,2}y_{2,3}y_{3,4}$ & $y_{1,2}y_{3,4}^2$ 
      \\
      $14$ & $y_{1,3}^2$ & $y_{1,2}y_{1,3}^2$ & $y_{1,3}^3$ & $y_{1,2}^2y_{2,3}$ & $y_{1,3}^2y_{1,4}$ & $y_{1,3}^2y_{2,4}$ & $y_{1,3}^2y_{3,4}$ 
      \\
\hline
	\end{tabular}
   }
   \vspace{1mm}
	\caption{Products $yy'$.}	
	\label{table:yyy}
\end{center}
\end{table}
%\vskip -1cm

\begin{table}[H]
\begin{center}
\begin{tabular}{c|c|cccccc}
        \hline
            $1$ & \diagbox[width=21mm,height=5mm]{$y$}{$y'$}  & $y_{1,2}$ & $y_{1,3}$  & $y_{2,3}$ & $y_{1,4}$ & $y_{2,4}$ & $y_{3,4}$ 
            \\
            \hline
      $25$ & $y_{1,2}^3$ & $55$ & $56$ & $57$ & $58$ & $59$ & $60$ 
      \\
      $26$ & $y_{1,2}^2y_{1,3}$ & $-57$ & $61$ & $56$  & $62$ & $63$ & $64$ 
      \\
      $27$ & $y_{1,2}^2y_{2,3}$ & $-56$ & $-57$ & $61$ & $65$ & $66$ & $67$
      \\
      $28$ & $y_{1,2}^2y_{1,4}$ & $-59$ & $-64$ & $-65$ & $68$ & $58$ & $62$
      \\
      $29$ & $y_{1,2}^2y_{2,4}$ & $-58$ & $-63$ & $-67$ & $-59$ & $68$ & $66$     
      \\
      $30$ & $y_{1,2}^2y_{3,4}$ & $-60$ & $-62$ & $-66$ & $-64$ & $-67$ & $69$
      \\
      $31$ & $y_{1,2}y_{1,3}^2$ & $61$ & $56$ & $57$ & $70$ & $71$ & $72$
      \\
\hline
\end{tabular}
%}}
\vspace{1mm}
\caption{Products $yy'$.}	
\label{table:yyyy}
\end{center}
\end{table}

%%%%%
\begin{minipage}{\textwidth}
\begin{minipage}[t]{0.66\textwidth}
\begin{table}[H]
\begin{center}
\setlength{\tabcolsep}{0.8mm}{
%\resizebox{\textwidth}{80mm}{
\begin{tabular}{c|c|cccccc}
		\hline
		$1$ & \diagbox[width=21mm,height=5mm]{$y$}{$y'$}  & $y_{1,2}$ & $y_{1,3}$  & $y_{2,3}$ & $y_{1,4}$ & $y_{2,4}$ & $y_{3,4}$ 
      \\
      \hline
      $55$ & $y_{1,2}^4$ & $93$ & $94$ & $95$ & $96$ & $97$ & $98$
      \\
      $56$ & $y_{1,2}^3y_{1,3}$ & $-95$ & $99$ & $94$ & $100$ & $101$ & $102$
      \\
      $57$ & $y_{1,2}^3y_{2,3}$ & $-94$ & $-95$ & $99$ & $103$ & $104$ & $105$
      \\
      $58$ & $y_{1,2}^3y_{1,4}$ & $-97$ & $-102$ & $-103$ & $106$ & $96$ & $100$
      \\
      $59$ & $y_{1,2}^3y_{2,4}$ & $-96$ & $-101$ & $-105$ & $-97$ & $106$ & $104$
      \\
      $60$ & $y_{1,2}^3y_{3,4}$ & $-98$ & $-100$ & $-104$ & $-102$ & $-105$ & $107$
      \\
      $61$ & $y_{1,2}^2y_{1,3}^2$ & $99$ & $94$ & $95$ & $108$ & $109$ & $110$
      \\
\hline
\end{tabular}
}
%}
\vspace{1mm}
\captionof{table}{Products $yy'$.}	
\label{table:yyyyy}
\end{center}
\end{table}
\end{minipage}
\begin{minipage}[t]{0.33\textwidth}
\begin{table}[H]
\begin{center}
\begin{tabular}{c|c}
		\hline
		$1$ & 
      \\
      \hline
      $93$ & $y_{1,2}^5$
      \\
      $94$ & $y_{1,2}^4y_{1,3}$  
      \\
      $95$ & $y_{1,2}^4y_{2,3}$   
      \\
      $96$ & $y_{1,2}^4y_{1,4}$  
      \\
      $97$ & $y_{1,2}^4y_{2,4}$ 
      \\
      $98$ & $y_{1,2}^4y_{3,4}$ 
      \\
      $99$ & $y_{1,2}^3y_{1,3}^2$ 
      \\
      $100$ & $y_{1,2}^3y_{1,3}y_{1,4}$ 
      \\
      $101$ & $y_{1,2}^3y_{1,3}y_{2,4}$ 
      \\
      $102$ & $y_{1,2}^3y_{1,3}y_{3,4}$ 
      \\
      $103$ & $y_{1,2}^3y_{2,3}y_{1,4}$ 
      \\
      $104$ & $y_{1,2}^3y_{2,3}y_{2,4}$ 
      \\
      $105$ & $y_{1,2}^3y_{2,3}y_{3,4}$ 
      \\
      $106$ & $y_{1,2}^3y_{1,4}^2$ 
      \\
      $107$ & $y_{1,2}^3y_{3,4}^2$ 
      \\
      $108$ & $y_{1,2}^2y_{1,3}^2y_{1,4}$ 
      \\
      $109$ & $y_{1,2}^2y_{1,3}^2y_{2,4}$ 
      \\
      $110$ & $y_{1,2}^2y_{1,3}^2y_{3,4}$ 
      \\
      $111$ & $y_{1,2}^2y_{1,3}y_{1,4}^2$ 
      \\
      $112$ & $y_{1,2}^2y_{2,3}y_{1,4}^2$ 
      \\
      $113$ & $y_{1,2}^2y_{3,4}^3$ 
      \\
      $114$ & $y_{1,2}y_{1,3}^2y_{1,4}^2$ 
      \\
      $115$ & $y_{1,2}y_{3,4}^4$ 
      \\
      $116$ & $y_{1,3}^5$ 
      \\
      $117$ & $y_{1,3}^4y_{1,4}$ 
      \\
      $118$ & $y_{1,3}^4y_{2,4}$ 
      \\
      $119$ & $y_{1,3}^4y_{3,4}$ 
      \\
      $120$ & $y_{1,3}^3y_{1,4}^2$ 
      \\
      $121$ & $y_{1,3}^3y_{2,4}^2$ 
      \\
      $122$ & $y_{1,3}^2y_{2,4}^3$ 
      \\
      $123$ & $y_{1,3}y_{2,4}^4$ 
      \\
      $124$ & $y_{2,3}^5$ 
      \\
      $125$ & $y_{2,3}^4y_{1,4}$ 
      \\
      $126$ & $y_{2,3}^4y_{2,4}$ 
      \\
      $127$ & $y_{2,3}^4y_{3,4}$ 
      \\
      $128$ & $y_{2,3}^3y_{1,4}^2$ 
      \\
      $129$ & $y_{2,3}^3y_{2,4}^2$ 
      \\
      $130$ & $y_{2,3}^2y_{1,4}^3$ 
      \\
      $131$ & $y_{2,3}y_{1,4}^4$ 
      \\
      $132$ & $y_{1,4}^5$ 
      \\
      $133$ & $y_{2,4}^5$ 
      \\
      $134$ & $y_{3,4}^5$
      \\
      \hline
\end{tabular}
%}
\vspace{1mm}
\captionof{table}{Elements in $\B^!_5$.}	
\label{table:y5}
\end{center}
\end{table}
\end{minipage}
\end{minipage}
%%%%%%%

\begin{table}[H]
\begin{center}
      \resizebox{\textwidth}{58.5mm}{
\begin{tabular}{|c|ccc|}
		\hline
		  \diagbox[width=23mm,height=5mm]{$y$}{$y'$}  & $y_{1,2}$ & $y_{1,3}$  & $y_{2,3}$ 
        % & $y_{1,4}$ & $y_{2,4}$ & $y_{3,4}$ 
      \\
      \hline
        $y_{1,2}^n$ & $y_{1,2}^{n+1}$ & $y_{1,2}^ny_{1,3}$ & $y_{1,2}^ny_{2,3}$ 
        % & $y_{1,2}^ny_{1,4}$ & $y_{1,2}^ny_{2,4}$ & $y_{1,2}^ny_{3,4}$
      \\
         $y_{1,2}^{n-1}y_{1,3}$ & $-y_{1,2}^ny_{2,3}$ & $y_{1,2}^{n-1}y_{1,3}^2$ & $y_{1,2}^{n}y_{1,3}$ 
         % & $y_{1,2}^{n-1}y_{1,3}y_{1,4}$ & $y_{1,2}^{n-1}y_{1,3}y_{2,4}$ & $y_{1,2}^{n-1}y_{1,3}y_{3,4}$
      \\
        $y_{1,2}^{n-1}y_{2,3}$ & $-y_{1,2}^{n}y_{1,3}$ & $-y_{1,2}^{n}y_{2,3}$ & $y_{1,2}^{n-1}y_{1,3}^2$ 
        % & $y_{1,2}^{n-1}y_{2,3}y_{1,4}$ & $y_{1,2}^{n-1}y_{2,3}y_{2,4}$ & $y_{1,2}^{n-1}y_{2,3}y_{3,4}$
      \\
        $y_{1,2}^{n-1}y_{1,4}$ & $-y_{1,2}^{n}y_{2,4}$ & $-y_{1,2}^{n-1}y_{1,3}y_{3,4}$ & $-y_{1,2}^{n-1}y_{2,3}y_{1,4}$ 
        % & $y_{1,2}^{n-1}y_{1,4}^2$ & $y_{1,2}^{n}y_{1,4}$ & $y_{1,2}^{n-1}y_{1,3}y_{1,4}$
      \\
        $y_{1,2}^{n-1}y_{2,4}$ & $-y_{1,2}^{n}y_{1,4}$ & $-y_{1,2}^{n-1}y_{1,3}y_{2,4}$ & $-y_{1,2}^{n-1}y_{2,3}y_{3,4}$ 
        % & $-y_{1,2}^{n}y_{2,4}$ & $y_{1,2}^{n-1}y_{1,4}^2$ & $y_{1,2}^{n-1}y_{2,3}y_{2,4}$
      \\
        % $y_{1,2}^{n-1}y_{3,4}$ & $-y_{1,2}^{n}y_{3,4}$ & $-y_{1,2}^{n-1}y_{1,3}y_{1,4}$ & $-y_{1,2}^{n-1}y_{2,3}y_{2,4}$ 
         % & $-y_{1,2}^{n-1}y_{1,3}y_{3,4}$ & $-y_{1,2}^{n-1}y_{2,3}y_{3,4}$ & $y_{1,2}^{n-1}y_{3,4}^2$
      % \\
         $y_{1,2}^{n-2}y_{1,3}^2$ & $y_{1,2}^{n-1}y_{1,3}^2$ & $y_{1,2}^{n}y_{1,3}$ & $y_{1,2}^{n}y_{2,3}$ 
         % & $y_{1,2}^{n-2}y_{1,3}^2y_{1,4}$ & $y_{1,2}^{n-2}y_{1,3}^2y_{2,4}$ & $y_{1,2}^{n-2}y_{1,3}^2y_{3,4}$
      \\
        $y_{1,2}^{n-2}y_{1,3}y_{1,4}$ & $y_{1,2}^{n-1}y_{2,3}y_{2,4}$ & $-y_{1,2}^{n-2}y_{1,3}^2y_{3,4}$ & $-y_{1,2}^{n-1}y_{1,3}y_{1,4}$ 
        % & $y_{1,2}^{n-2}y_{1,3}y_{1,4}^2$ & $-y_{1,2}^{n-1}y_{2,3}y_{1,4}$ & $y_{1,2}^{n-2}y_{1,3}^2y_{1,4}$
      \\
\hline
	\end{tabular}
   }
	\caption{Products $yy'$ for $n\geqslant 5$.}	
	\label{table:yn123}
\end{center}
\end{table}
\vspace{-0.8cm}

\begin{table}[H]
\begin{center}
      \resizebox{\textwidth}{58.5mm}{
\begin{tabular}{|c|ccc|}
		\hline
		  \diagbox[width=23mm,height=5mm]{$y$}{$y'$}   & $y_{1,4}$ & $y_{2,4}$ & $y_{3,4}$ 
      \\
      \hline
        $y_{1,2}^n$  & $y_{1,2}^ny_{1,4}$ & $y_{1,2}^ny_{2,4}$ & $y_{1,2}^ny_{3,4}$
      \\
         $y_{1,2}^{n-1}y_{1,3}$  & $y_{1,2}^{n-1}y_{1,3}y_{1,4}$ & $y_{1,2}^{n-1}y_{1,3}y_{2,4}$ & $y_{1,2}^{n-1}y_{1,3}y_{3,4}$
      \\
        $y_{1,2}^{n-1}y_{2,3}$  & $y_{1,2}^{n-1}y_{2,3}y_{1,4}$ & $y_{1,2}^{n-1}y_{2,3}y_{2,4}$ & $y_{1,2}^{n-1}y_{2,3}y_{3,4}$
      \\
        $y_{1,2}^{n-1}y_{1,4}$ & $y_{1,2}^{n-1}y_{1,4}^2$ & $y_{1,2}^{n}y_{1,4}$ & $y_{1,2}^{n-1}y_{1,3}y_{1,4}$
      \\
        $y_{1,2}^{n-1}y_{2,4}$  & $-y_{1,2}^{n}y_{2,4}$ & $y_{1,2}^{n-1}y_{1,4}^2$ & $y_{1,2}^{n-1}y_{2,3}y_{2,4}$
      \\
       %  $y_{1,2}^{n-1}y_{3,4}$  & $-y_{1,2}^{n-1}y_{1,3}y_{3,4}$ & $-y_{1,2}^{n-1}y_{2,3}y_{3,4}$ & $y_{1,2}^{n-1}y_{3,4}^2$
      % \\
         $y_{1,2}^{n-2}y_{1,3}^2$  & $y_{1,2}^{n-2}y_{1,3}^2y_{1,4}$ & $y_{1,2}^{n-2}y_{1,3}^2y_{2,4}$ & $y_{1,2}^{n-2}y_{1,3}^2y_{3,4}$
      \\
        $y_{1,2}^{n-2}y_{1,3}y_{1,4}$  & $y_{1,2}^{n-2}y_{1,3}y_{1,4}^2$ & $-y_{1,2}^{n-1}y_{2,3}y_{1,4}$ & $y_{1,2}^{n-2}y_{1,3}^2y_{1,4}$
      \\
\hline
	\end{tabular}
   }
	\caption{Products $yy'$ for $n\geqslant 5$.}	
	\label{table:yn456}
\end{center}
\end{table}
\vspace{-0.8cm}

%%%%%%%%%%%%%%%%%%%%%%%%%%%%%%%%

\begin{table}[H]
   \begin{center}
        % \resizebox{\textwidth}{58.5mm}{
   \begin{tabular}{c|c|cccccc}
         \hline
         & \diagbox[width=23mm,height=5mm]{$y$}{$y'$}  & $y_{1,2}$ & $y_{3,4}$  & $y_{1,3}$ & $y_{2,4}$ & $y_{2,3}$ & $y_{1,4}$
         \\
         \hline
         $1$ & $y_{1,2}^{n-1}y_{1,3}$ & $1 $ & $5$ & $3$ & $20$ & $-2$ & $15$ 
         \\
         $2$ & $y_{1,2}^{n-2}y_{1,3}^2$ & $2$ & $21 $  & $1$ & $16$ & $3$ & $6$ 
         \\
         $3 $ & $y_{1,2}^{n-1}y_{2,3}$ & $3$  & $17$  & $-2$ & $8$ & $-1$ & $23$ 
         \\
         $4 $ & $y_{1,2}^{n-1}y_{1,4}$ & $4$ & $20$ & $-8$ & $-9$ & $-5$ & $14$ 
         \\
         $5 $ & $y_{1,2}^{n-2}y_{1,3}y_{1,4}$ & $5$  & $-10$ & $6$ & $-17$ & $-8$ & $21$ 
         \\
         $ 6$ & $y_{1,2}^{n-3}y_{1,3}^2y_{1,4}$ & $6$ & $20$ & $-8$ & $-11$ & $-5$ & $16$ 
         \\
         $7 $ & $y_{1,3}^{n-1}y_{1,4}$ & $5$ & $-12$ & $7$ & $-17$ & $-8$ & $22$ 
         \\
\hline
      \end{tabular}
        % }
      \caption{Products $y'y$ for $n\geqslant 6$ even.}	
      \label{table:y'yn even}
   \end{center}
   \end{table}
   \vspace{-0.8cm}

   \begin{table}[H]
      \begin{center}
           % \resizebox{\textwidth}{58.5mm}{
      \begin{tabular}{c|c|cccccc}
            \hline
            & \diagbox[width=23mm,height=5mm]{$y$}{$y'$}  & $y_{1,2}$ & $y_{3,4}$  & $y_{1,3}$ & $y_{2,4}$ & $y_{2,3}$ & $y_{1,4}$
            \\
            \hline
            $1$ & $y_{1,2}^{n-1}y_{1,3}$ & $1 $ & $-5$ & $2$ & $-15$ & $-3$ & $-20$ 
            \\
            $2$ & $y_{1,2}^{n-2}y_{1,3}^2$ & $ 2$ & $-21$ & $-3$ & $-6$ & $-1$ & $-16$ 
            \\
            $3 $ & $y_{1,2}^{n-1}y_{2,3}$ & $ 3$  & $-17$ & $1$ &  $-23$ & $2$ & $-8$ 
            \\
            $4 $ & $y_{1,2}^{n-1}y_{1,4}$ & $ 4$ & $-20$ & $5$ & $-14$ & $8$ & $9$ 
            \\
            $5 $ & $y_{1,2}^{n-2}y_{1,3}y_{1,4}$ & $5 $  & $10$ & $8$ & $-21$ & $-6$ & $17$ 
            \\
            $ 6$ & $y_{1,2}^{n-3}y_{1,3}^2y_{1,4}$ & $ 6$ & $-20$ & $5$ & $-16$ & $8$ & $11$ 
            \\
            $7 $ & $y_{1,3}^{n-1}y_{1,4}$ & $ 6$ & $-22$ & $7$ & $-16$ & $8$ & $12$ 
            \\
\hline
         \end{tabular}
           % }
         \caption{Products $y'y$ for $n\geqslant 5$ odd.}	
         \label{table:y'yn odd}
      \end{center}
      \end{table}
      \vspace{-0.8cm}

%%%%%%%%%
\subsection{\texorpdfstring{A basis of $M^1$}{A basis of M1}}
\label{sec:Basis of M1}

We present here the GAP code for computing a basis of the quadratic module $M^1$, defined at the beginning of Subsection \ref{subsec:Resolving datum fk4}. 
The code was provided by J.W. Knopper. 

\parskip 1ex
\parindent 0in
\begin{tcolorbox}[breakable,colback=white,width=\textwidth ,center,arc=0mm,size=fbox]
\begin{footnotesize}
\begin{Verbatim}[samepage=false]
# Knopper's code 
LoadPackage("GBNP");
A:=FreeAssociativeAlgebraWithOne(Rationals,"x12","x13","x23","x14","x24","x34");;
x12:=A.x12;; x13:=A.x13;; x23:=A.x23;; x14:=A.x14;; x24:=A.x24;; x34:=A.x34;; 
oA:=One(A);;
relationsA:=[x12^2, x13^2, x23^2, x14^2, x24^2, x34^2, x12*x23-x23*x13-x13*x12, 
 x23*x12-x12*x13-x13*x23, x12*x24-x24*x14-x14*x12, x24*x12-x12*x14-x14*x24, 
 x13*x34-x34*x14-x14*x13, x34*x13-x13*x14-x14*x34, x23*x34-x34*x24-x24*x23, 
 x34*x23-x23*x24-x24*x34, x12*x34-x34*x12, x13*x24-x24*x13, x14*x23-x23*x14];;
relsANP:=GP2NPList(relationsA);;
GBNP.ConfigPrint(A);
GA:=Grobner(relsANP);;

MD:=A^2;;
ab:=GeneratorsOfLeftModule(MD);
c2:=ab[1]; c1:=ab[2]; 
modrels:=[c1*x13, c1*x24, c2*x23, c2*x14, c1*x12+c2*x12, c1*x34+c2*x34];;
modrelsNP:=GP2NPList(modrels);
PrintNPList(modrelsNP);
      
GBNP.CheckHom:=function(G,wtv)
    local i,j,k,l,mon,h1,h2,ans;
    mon:=LMonsNP(G);
    ans:=GBNP.WeightedDegreeList(mon,wtv);
    for i in [1..Length(G)] do
        h1:=ans[i];
        l:=Length(G[i][1]);
        for j in [2..l] do
            mon:=G[i][1][j];
            h2:=0;
            for k in [1..Length(mon)] do
                if mon[k]>0 then
                # Don't count module generators, which have a negative index. 
                # Only count two-sided generators with index 1 or more.
                       h2:=h2+wtv[mon[k]];
                fi;
            od;
            if h2<>h1 then return(false); fi;
        od;
    od;
    Info(InfoGBNP,1,"Input is homogeneous");
    return(ans);
end;
      
GBNP.WeightedDegreeMon:=function(mon,lst)
    local i,ans;
    ans:=0;
    for i in mon do
        # Don't count module generators, which have a negative index. 
        # Only count two-sided generators with index 1 or more.
        if i>0 then
            ans:=ans+lst[i];
        fi;
    od;
    return(ans);
end;;
      
SetInfoLevel(InfoGBNP,1);
SetInfoLevel(InfoGBNPTime,1);
combinedrelsNP:=Concatenation(GA,modrelsNP);
GAT:=SGrobnerTrunc(combinedrelsNP, 9, [1,1,1,1,1,1]);
PrintNPList(GAT);
      
splitGAT:=function(GAT)
    local p, ts, rel, lm; 
    # p: list of module or prefix relations, ts: list of two-sided relations, 
    # rel: current relation, lm: leading monomial of current relation rel.
    p:=[];
    ts:=[];
    for rel in GAT do
        # get leading monomial.
        lm := rel[1,1];
        if Length(lm)>1 and lm[1]<0 then
            # module relations start with a negative generator.
            # if 1 is part of the GB then it does not have a generator, 
            # furthermore it is two-sided.
            Add(p, rel);
        else
            Add(ts, rel);
        fi;
    od;
    return rec(p:=p, ts:=ts);
end;;
      
split:=splitGAT(GAT);
GBR:=rec(p:=split.p, pg:=2, ts:=split.ts);
BQM:=BaseQM(GBR,6,2,0);;
PrintNPList(BQM);

[ 0, 1 ]
[ 1 , 0]
[ 0, x12 ]
[ 0, x23 ]
[ 0, x14 ]
[ 0, x34 ]
[ x13 , 0]
[ x24 , 0]
[ 0, x12x13 ]
[ 0, x12x23 ]
[ 0, x12x14 ]
[ 0, x12x24 ]
[ 0, x12x34 ]
[ 0, x23x14 ]
[ 0, x23x24 ]
[ 0, x23x34 ]
[ 0, x14x13 ]
[ 0, x14x34 ]
[ x13x24 , 0]
[ 0, x12x13x14 ]
[ 0, x12x13x24 ]
[ 0, x12x13x34 ]
[ 0, x12x23x14 ]
[ 0, x12x23x24 ]
[ 0, x12x23x34 ]
[ 0, x12x14x13 ]
[ 0, x12x14x34 ]
[ 0, x12x24x23 ]
[ 0, x12x24x34 ]
[ 0, x23x14x12 ]
[ 0, x23x14x34 ]
[ 0, x12x13x14x12 ]
[ 0, x12x13x14x13 ]
[ 0, x12x13x14x24 ]
[ 0, x12x13x14x34 ]
[ 0, x12x13x24x23 ]
[ 0, x12x13x24x34 ]
[ 0, x12x23x14x13 ]
[ 0, x12x23x14x24 ]
[ 0, x12x23x24x34 ]
[ 0, x12x14x13x34 ]
[ 0, x12x24x23x34 ]
[ 0, x12x13x14x12x23 ]
[ 0, x12x13x14x12x24 ]
[ 0, x12x13x14x12x34 ]
[ 0, x12x13x14x13x34 ]
[ 0, x12x13x14x24x23 ]
[ 0, x12x13x24x23x34 ]
[ 0, x12x13x14x12x23x34 ]
[ 0, x12x13x14x12x24x23 ]
\end{Verbatim}
\end{footnotesize}
\end{tcolorbox}

\subsection{\texorpdfstring{Koszul complex of $M^i$ for $i \in \{ 0, 1, 2, 3 \}$}{Koszul complex of Mi for i = 0,1,2,3}}
\label{sec:cpx}

We present here the GAP code for computing the differential of the Koszul complex of the quadratic modules $M^0=\Bbbk$ and $M^i$ for $i\in\llbracket 1,3\rrbracket$ defined in Subsection \ref{subsec:Resolving datum fk4}. 
We also present a basis of $\operatorname{H}_{n,m}(M^i)$ for some pairs $(n,m)$. 
In the following code, the matrix \begin{footnotesize}FF(i,n,m)\end{footnotesize} represents the linear map $d_{n+1,m-1}(M^i):\operatorname{K}_{n+1,m-1}(M^i)\to \operatorname{K}_{n,m}(M^i)$, 
\begin{footnotesize}Im(i,n,m)\end{footnotesize} is a basis of the space $B_{n,m}^{M^i}$
and \begin{footnotesize}Ker(i,n,m)\end{footnotesize} is a basis of the space $D_{n,m}^{M^i}$. 
Moreover, \begin{footnotesize}geneMH(i,n,m)\end{footnotesize} are some elements in $D_{n,m}^{M^i}$, and we can show that it represents a basis of $\operatorname{H}_{n,m}(M^i)$ since the dimension of the space spanned by $B_{n,m}^{M^i}$ and \begin{footnotesize}geneMH(i,n,m)\end{footnotesize} coincides with the dimension of $D_{n,m}^{M^i}$. 

\parskip 1ex
\parindent 0in
\begin{tcolorbox}[breakable,colback=white,width=\textwidth ,center,arc=0mm,size=fbox]
\begin{footnotesize}
\begin{Verbatim}[samepage=false]
LoadPackage("GBNP");
A:=FreeAssociativeAlgebraWithOne(Rationals,"x12","x13","x23","x14","x24","x34");;
x12:=A.x12;; x13:=A.x13;; x23:=A.x23;; x14:=A.x14;; x24:=A.x24;; x34:=A.x34;; 
oA:=One(A);;
relationsA:=[x12^2, x13^2, x23^2, x14^2, x24^2, x34^2, x12*x23-x23*x13-x13*x12, 
 x23*x12-x12*x13-x13*x23, x12*x24-x24*x14-x14*x12, x24*x12-x12*x14-x14*x24, 
 x13*x34-x34*x14-x14*x13, x34*x13-x13*x14-x14*x34, x23*x34-x34*x24-x24*x23, 
 x34*x23-x23*x24-x24*x34, x12*x34-x34*x12, x13*x24-x24*x13, x14*x23-x23*x14];;
# A/relationsA is the Fomin-Kirillov algbera on 4 generators.
relsANP:=GP2NPList(relationsA);;
GBNP.ConfigPrint(A);
GA:=Grobner(relsANP);; # GA is a Gröbner basis of the ideal in A.
# PrintNPList(GA);
C:=BaseQA(GA,6,0);;  # C is the set of standard words with respect to GA.
# PrintNPList(C);

f:=function(n)
    if n=0 then return 1;
    elif n=1 then return 6;
    elif n=2 then return 19;
    elif n=3 then return 42;
    elif n=4 then return 71;
    elif n=5 then return 96;
    elif n=6 then return 106;
    elif n=7 then return 96;
    elif n=8 then return 71;
    elif n=9 then return 42;
    elif n=10 then return 19;
    elif n=11 then return 6;
    elif n=12 then return 1;
    fi;
end;
# f(n) is the dimension of $A_n$.

g:=function(n)
    if n=-1 then return 0;
    elif n=0 then return f(0);
    elif n=1 then return Sum(List([0..1], s->f(s)));
    elif n=2 then return Sum(List([0..2], s->f(s)));
    elif n=3 then return Sum(List([0..3], s->f(s)));
    elif n=4 then return Sum(List([0..4], s->f(s)));
    elif n=5 then return Sum(List([0..5], s->f(s)));
    elif n=6 then return Sum(List([0..6], s->f(s)));
    elif n=7 then return Sum(List([0..7], s->f(s)));
    elif n=8 then return Sum(List([0..8], s->f(s)));
    elif n=9 then return Sum(List([0..9], s->f(s)));
    elif n=10 then return Sum(List([0..10], s->f(s)));
    elif n=11 then return Sum(List([0..11], s->f(s)));
    elif n=12 then return Sum(List([0..12], s->f(s)));
    fi;
end;
# g(n)-g(n-1)=f(n).

B:=FreeAssociativeAlgebraWithOne(Rationals,"y12","y13","y23","y14","y24","y34");;
y12:=B.y12;; y13:=B.y13;; y23:=B.y23;; y14:=B.y14;; y24:=B.y24;; y34:=B.y34;; 
oB:=One(B);;
relationsB:=[y12*y23+y23*y13, y13*y23+y23*y12, y12*y23+y13*y12, y12*y13+y23*y12, 
 y12*y24+y24*y14, y14*y24+y24*y12, y12*y24+y14*y12, y12*y14+y24*y12, y13*y34+y34*y14, 
 y14*y34+y34*y13, y13*y34+y14*y13, y13*y14+y34*y13, y23*y34+y34*y24, y24*y34+y34*y23, 
 y23*y34+y24*y23, y23*y24+y34*y23, y12*y34+y34*y12, y13*y24+y24*y13, 
 y23*y14+y14*y23];;
# B/relationsB is the quadratic dual of the Fomin-Kirillov algebra on 4 generators.
relsBNP:=GP2NPList(relationsB);;
wtv:= [1,1,1,1,1,1];;
GBNP.ConfigPrint(B);;
GB:=Grobner(relsBNP);; # GB is a Gröbner basis of the ideal in B.
# PrintNPList(GB);
D:= BaseQATrunc(GB,12,wtv);; 
for degpart in D do for mon in degpart do PrintNP([[mon],[1]]); od; od;
DT:=[]; 
for degpart in D do for mon in degpart do Append(DT,[[[mon],[1]]]); od; od;

S:=B^8;;
ab:=GeneratorsOfLeftModule(S);;
g8:=ab[1];; g7:=ab[2];; g6:=ab[3];; g5:=ab[4];; 
g4:=ab[5];; g3:=ab[6];; g2:=ab[7];; g1:=ab[8];;
modrels:=[g1*y12-g2*y34, g1*y34+g2*y12, g3*y12+g4*y34, g3*y34-g4*y12, g4*y13-g2*y24, 
 g4*y24+g2*y13, g3*y13-g1*y24, g3*y24+g1*y13, g1*y23+g4*y14, g1*y14-g4*y23, 
 g3*y23+g2*y14, g3*y14-g2*y23, g5*y14, g5*y24, g5*y34, g6*y13, g6*y23, g6*y34, g7*y12, 
 g7*y23, g7*y24, g8*y12, g8*y13, g8*y14 ];;
modrelsNP:=GP2NPList(modrels);;
# PrintNPList(modrelsNP);;
   
GBNP.CheckHom:=function(G,wtv)
    local i,j,k,l,mon,h1,h2,ans;
    mon:=LMonsNP(G);
    ans:=GBNP.WeightedDegreeList(mon,wtv);
    for i in [1..Length(G)] do
        h1:=ans[i];
        l:=Length(G[i][1]);
        for j in [2..l] do
            mon:=G[i][1][j];
            h2:=0;
            for k in [1..Length(mon)] do
                if mon[k]>0 then
                    h2:=h2+wtv[mon[k]];
                fi;
            od;
            if h2<>h1 then return(false); fi;
        od;
    od;
    Info(InfoGBNP,1,"Input is homogeneous");
    return(ans);
end;
   
GBNP.WeightedDegreeMon:=function(mon,lst)
    local i,ans;
    ans:=0;
    for i in mon do
        if i>0 then
            ans:=ans+lst[i];
        fi;
    od;
    return(ans);
end;;
   
SetInfoLevel(InfoGBNP,1);;
SetInfoLevel(InfoGBNPTime,1);;
combinedrelsNP:=Concatenation(GB,modrelsNP);;
GBT:=SGrobnerTrunc(combinedrelsNP, 15, [1,1,1,1,1,1]);;
# PrintNPList(GBT);

splitGBT:=function(GBT)
    local p, ts, rel, lm; 
    p:=[];
    ts:=[];
    for rel in GBT do
        lm := rel[1,1];
        if Length(lm)>1 and lm[1]<0 then
            Add(p, rel);
        else
            Add(ts, rel);
        fi;
    od;
    return rec(p:=p, ts:=ts);
end;;
   
split:=splitGBT(GBT);;
GBRM3:=rec(p:=split.p, pg:=8, ts:=split.ts);;
BQMM3:=BaseQM(GBRM3,6,8,500);;
# PrintNPList(BQMM3);

S:=B^7;;
ab:=GeneratorsOfLeftModule(S);;
g7:=ab[1];; g6:=ab[2];; g5:=ab[3];; g4:=ab[4];; g3:=ab[5];; g2:=ab[6];; g1:=ab[7];;  
modrels:=[g1*y14+g4*y14, g1*y24+g3*y24, g1*y34-g2*y34, g2*y13+g4*y13, g2*y23+g3*y23, 
 g3*y12-g4*y12, g5*y12-g6*y34, g1*y24+g5*y13, g5*y23+g7*y14, g5*y14-g7*y23, 
 g2*y13-g5*y24, g5*y34+g6*y12, g6*y13+g7*y24, g1*y14-g6*y23, g2*y23+g6*y14, 
 g6*y24-g7*y13, g1*y34-g7*y12, g3*y12-g7*y34];;
modrelsNP:=GP2NPList(modrels);;
# PrintNPList(modrelsNP);
   
SetInfoLevel(InfoGBNP,1);;
SetInfoLevel(InfoGBNPTime,1);;
combinedrelsNP:=Concatenation(GB,modrelsNP);;
GBT:=SGrobnerTrunc(combinedrelsNP, 15, [1,1,1,1,1,1]);;
# PrintNPList(GBT);
   
split:=splitGBT(GBT);;
GBRM2:=rec(p:=split.p, pg:=7, ts:=split.ts);;
BQMM2:=BaseQM(GBRM2,6,7,500);;
# PrintNPList(BQMM2);

MD:=B^2;;
ab:=GeneratorsOfLeftModule(MD);;
g2:=ab[1];; g1:=ab[2];;
modrels:=[g1*y12-g2*y12, g2*y13, g1*y23, g1*y14, g2*y24, g1*y34-g2*y34];;
modrelsNP:=GP2NPList(modrels);;
# PrintNPList(modrelsNP);

SetInfoLevel(InfoGBNP,1);;
SetInfoLevel(InfoGBNPTime,1);;
combinedrelsNP:=Concatenation(GB,modrelsNP);;
GBT:=SGrobnerTrunc(combinedrelsNP, 15, [1,1,1,1,1,1]);;
# PrintNPList(GBT);

split:=splitGBT(GBT);;
GBRM1:=rec(p:=split.p, pg:=2, ts:=split.ts);;
BQMM1:=BaseQM(GBRM1,6,2,400);;
# PrintNPList(BQMM1);

ff:=function(i,n)
    if i=0 and n=-1 then return 0;
    elif i=0 and n=0 then return 1;
    elif i=0 and n=1 then return 7;
    elif i=0 and n=2 then return 24;
    elif i=0 and n=3 then return 54;
    elif i=0 and n=4 then return 92;
    elif i=0 and n=5 then return 134;
    elif i=0 and n=6 then return 179;
    elif i=0 and n=7 then return 227;
    elif i=0 and n=8 then return 278;
    elif i=0 and n=9 then return 332;
    elif i=0 and n=10 then return 389;
    elif i=0 and n=11 then return 449;
    elif i=0 and n=12 then return 512;
    elif i=0 and n=13 then return 578; 

    elif i=1 and n=-1 then return 0;
    elif i=1 and n=0 then return 2;
    elif i=1 and n=1 then return 8;
    elif i=1 and n=2 then return 17;
    elif i=1 and n=3 then return 29;
    elif i=1 and n=4 then return 44;
    elif i=1 and n=5 then return 62;
    elif i=1 and n=6 then return 83;
    elif i=1 and n=7 then return 107;
    elif i=1 and n=8 then return 134;
    elif i=1 and n=9 then return 164;
    elif i=1 and n=10 then return 197;
    elif i=1 and n=11 then return 233;
    elif i=1 and n=12 then return 272;
    elif i=1 and n=13 then return 314;
	
    elif i=2 and n=-1 then return 0;
    elif i=2 and n=0 then return 7;
    elif i=2 and n=1 then return 31;
    elif i=2 and n=2 then return 74;
    elif i=2 and n=3 then return 128;
    elif i=2 and n=4 then return 185;
    elif i=2 and n=5 then return 245;
    elif i=2 and n=6 then return 308;
    elif i=2 and n=7 then return 374;
    elif i=2 and n=8 then return 443;
    elif i=2 and n=9 then return 515;
		
    elif i=3 and n=-1 then return 0;
    elif i=3 and n=0 then return 8;
    elif i=3 and n=1 then return 32;
    elif i=3 and n=2 then return 72;
    elif i=3 and n=3 then return 120;
    elif i=3 and n=4 then return 168;
    elif i=3 and n=5 then return 216;
    elif i=3 and n=6 then return 264;
    elif i=3 and n=7 then return 312;
    elif i=3 and n=8 then return 360;
    elif i=3 and n=9 then return 408;
    elif i=3 and n=10 then return 456;
    elif i=3 and n=11 then return 504;
    fi;
end;

FFM0:=function(j,i)
    local F,RDF,H,L,RFA,DFA,rra,dda,s,LAs,k,t;
    RDF:=List([ff(0,j-1)+1..ff(0,j+1)], p -> DT[p]);
    H:=List([1..6], s -> TransposedMat(MatrixQA(s,RDF,GB)));;
    L:=List([1..6], s -> List([ff(0,j)-ff(0,j-1)+1..ff(0,j+1)-ff(0,j-1)], 
     q -> List([1..ff(0,j)-ff(0,j-1)], p -> H[s][q][p])));;
    RFA:=List([g(i-1)+1..g(i)], p -> C[p]);
    DFA:=List([g(i-2)+1..g(i-1)], p -> C[p]);
    rra:=Length(RFA);
    dda:=Length(DFA);
    F:=[];
    for s in [1..6] do
        LAs:=0*[1..dda];
        for k in [1..dda] do 
            LAs[k]:=0*[1..rra];
            for t in [1..Length(MulQA(C[s+1], DFA[k], GA)[1])] do 
                LAs[k][Position(RFA,[[MulQA(C[s+1], DFA[k], GA)[1][t]],[1]])]:=
                 MulQA(C[s+1], DFA[k], GA)[2][t];
            od;
        od;
        F:=F+KroneckerProduct(L[s],LAs);
    od;
    return F;
end;;

FFM1:=function(j,i)
    local FF,RF,DF,rr,dd,RFA,DFA,rra,dda,s,LLs,LAs,k,t;
    RF:=List([ff(1,j-1)+1..ff(1,j)], p -> BQMM1[p]);
    DF:=List([ff(1,j)+1..ff(1,j+1)], p -> BQMM1[p]);
    rr:=Length(RF);
    dd:=Length(DF);
    RFA:=List([g(i-1)+1..g(i)], p -> C[p]);
    DFA:=List([g(i-2)+1..g(i-1)], p -> C[p]);
    rra:=Length(RFA);
    dda:=Length(DFA);
    FF:=[];
    for s in [1..6] do 
        LLs:=0*[1..rr];
        LAs:=0*[1..dda];
        for k in [1..rr] do
            LLs[k]:=0*[1..dd];
            for t in [1..Length(MulQM(RF[k], DT[s+1], GBRM1)[1])] do 
                LLs[k][Position(DF,[[MulQM(RF[k], DT[s+1], GBRM1)[1][t]],[1]])]:=
                 MulQM(RF[k], DT[s+1], GBRM1)[2][t];
            od;
        od;
        for k in [1..dda] do 
            LAs[k]:=0*[1..rra];
            for t in [1..Length(MulQA(C[s+1], DFA[k], GA)[1])] do 
                LAs[k][Position(RFA,[[MulQA(C[s+1], DFA[k], GA)[1][t]],[1]])]:=
                 MulQA(C[s+1], DFA[k], GA)[2][t];
            od;
        od;
        FF:=FF+KroneckerProduct(TransposedMat(LLs),LAs);
    od;
    return FF;
end;

FFM2:=function(j,i)
    local FF,RF,DF,rr,dd,RFA,DFA,rra,dda,s,LLs,LAs,k,t;
    RF:=List([ff(2,j-1)+1..ff(2,j)], p -> BQMM2[p]);
    DF:=List([ff(2,j)+1..ff(2,j+1)], p -> BQMM2[p]);
    rr:=Length(RF);
    dd:=Length(DF);
    RFA:=List([g(i-1)+1..g(i)], p -> C[p]);
    DFA:=List([g(i-2)+1..g(i-1)], p -> C[p]);
    rra:=Length(RFA);
    dda:=Length(DFA);
    FF:=[];
    for s in [1..6] do 
        LLs:=0*[1..rr];
        LAs:=0*[1..dda];
        for k in [1..rr] do
            LLs[k]:=0*[1..dd];
            for t in [1..Length(MulQM(RF[k], DT[s+1], GBRM2)[1])] do 
                LLs[k][Position(DF,[[MulQM(RF[k], DT[s+1], GBRM2)[1][t]],[1]])]:=
                 MulQM(RF[k], DT[s+1], GBRM2)[2][t];
            od;
        od;
        for k in [1..dda] do 
            LAs[k]:=0*[1..rra];
            for t in [1..Length(MulQA(C[s+1], DFA[k], GA)[1])] do 
                LAs[k][Position(RFA,[[MulQA(C[s+1], DFA[k], GA)[1][t]],[1]])]:=
                 MulQA(C[s+1], DFA[k], GA)[2][t];
            od;
        od;
        FF:=FF+KroneckerProduct(TransposedMat(LLs),LAs);
    od;
    return FF;
end;

FFM3:=function(j,i)
    local FF,RF,DF,rr,dd,RFA,DFA,rra,dda,s,LLs,LAs,k,t;
    RF:=List([ff(3,j-1)+1..ff(3,j)], p -> BQMM3[p]);
    DF:=List([ff(3,j)+1..ff(3,j+1)], p -> BQMM3[p]);
    rr:=Length(RF);
    dd:=Length(DF);
    RFA:=List([g(i-1)+1..g(i)], p -> C[p]);
    DFA:=List([g(i-2)+1..g(i-1)], p -> C[p]);
    rra:=Length(RFA);
    dda:=Length(DFA);
    FF:=[];
    for s in [1..6] do 
        LLs:=0*[1..rr];
        LAs:=0*[1..dda];
        for k in [1..rr] do
            LLs[k]:=0*[1..dd];
            for t in [1..Length(MulQM(RF[k], DT[s+1], GBRM3)[1])] do 
                LLs[k][Position(DF,[[MulQM(RF[k], DT[s+1], GBRM3)[1][t]],[1]])]:=
                 MulQM(RF[k], DT[s+1], GBRM3)[2][t];
            od;
        od;
        for k in [1..dda] do 
            LAs[k]:=0*[1..rra];
            for t in [1..Length(MulQA(C[s+1], DFA[k], GA)[1])] do 
                LAs[k][Position(RFA,[[MulQA(C[s+1], DFA[k], GA)[1][t]],[1]])]:=
                 MulQA(C[s+1], DFA[k], GA)[2][t];
            od;
        od;
        FF:=FF+KroneckerProduct(TransposedMat(LLs),LAs);
    od;
    return FF;
end;

FF:=function(ii,j,i)
    if ii=0 then return FFM0(j,i);
    elif ii=1 then return FFM1(j,i);
    elif ii=2 then return FFM2(j,i);
    elif ii=3 then return FFM3(j,i);
    fi;
end;

Im:=function(ii,j,i)
    local Imm;
    Imm:=TriangulizedMat(BaseMatDestructive(FF(ii,j,i)));
    return Imm;
end;

Ker:=function(ii,j,i)
    local Kerr;
    Kerr:=TriangulizedNullspaceMatDestructive(FF(ii,j-1,i+1));
    return Kerr;
end;

HXR:=function(ii,Uh,Vh,Wh,n,m,r)
    local hxr,Vhxr,CC,s,t,le,yy,VP,j,i,k;
    VP:=[];;
    hxr:=0*[1..Length(Uh)*f(r)];;
    Vhxr:=0*[1..Length(Uh)*f(r)];;
    CC:=List([g(m+r-1)+1..g(m+r)], p -> C[p]);;
    for s in [1..Length(Uh)] do
        for t in [1..f(r)] do 
            le:=Length(Uh[s]);;
            yy:=C[g(r-1)+t];;
            hxr[(s-1)*f(r)+t]:=0*[1..(ff(ii,n)-ff(ii,n-1))*f(m+r)];;
            VP:=0*[1..le];;
            Vhxr[(s-1)*f(r)+t]:=0*[1..le];;
            for j in [1..le] do
                VP[j]:=[ [  ], [  ] ];
                for i in [1..Length(Vh[s][j])] do 
                    VP[j]:=AddNP(VP[j],MulQA(C[g(m-1)+Vh[s][j][i]],yy,GA),1,
                     Wh[s][j][i]);
                od;
                Vhxr[(s-1)*f(r)+t][j]:=List([1..Length(VP[j][1])], k -> 
                 Position(CC, [ [VP[j][1][k]], [ 1 ] ] ));
                for k in [1..Length(VP[j][1])] do  
                    hxr[(s-1)*f(r)+t][f(m+r)*(Uh[s][j]-1)+Vhxr[(s-1)*f(r)+t][j][k]]:=
                     VP[j][2][k];
                od;
            od;
        od;
    od;
    return hxr;
end;

UU:=function(gene,ii)
    local Rest,Uh,Vh,Wh,Post,k,aa,Quo,Res,Sig,Qu,Re,Sg,i,j;
    Rest:=function(n)
        if n mod f(ii) > 0 then return n mod f(ii);
        else return f(ii);
        fi;
    end;
    Uh:=0*[1..Length(gene)];;
    Vh:=0*[1..Length(gene)];;
    Wh:=0*[1..Length(gene)];;
    Post:=[];;
    for k in [1..Length(gene)] do 
        Uh[k]:=[]; Vh[k]:=[]; Wh[k]:=[];
        aa:=gene[k];
        Post:=[1..Length(aa)];
        SubtractSet(Post, Positions(aa,0));
        Quo:=List([1..Length(Post)], s->(Post[s]-Rest(Post[s]))/f(ii)+1);
        Res:=List([1..Length(Post)], s->Rest(Post[s]));
        Sig:=List([1..Length(Post)], s->gene[k][Post[s]]);
        Qu:=Set(Quo);
        Re:=0*[1..Length(Qu)];
        Sg:=0*[1..Length(Qu)];
        for i in [1..Length(Qu)] do 
            Re[i]:=[];
            Sg[i]:=[];
            for j in [1..Length(Positions(Quo,Qu[i]))] do 
                Re[i][j]:=Res[Position(Quo,Qu[i])+j-1];	
                Sg[i][j]:=Sig[Position(Quo,Qu[i])+j-1];
            od;
        od;
        Uh[k]:=Qu;
        Vh[k]:=Re;
        Wh[k]:=Sg;
    od;
    return Uh;
end;

VV:=function(gene,ii)
    local Rest,Uh,Vh,Wh,Post,k,aa,Quo,Res,Sig,Qu,Re,Sg,i,j;
    Rest:=function(n)
        if n mod f(ii) > 0 then return n mod f(ii);
        else return f(ii);
        fi;
    end;
    Uh:=0*[1..Length(gene)];;
    Vh:=0*[1..Length(gene)];;
    Wh:=0*[1..Length(gene)];;
    Post:=[];;
    for k in [1..Length(gene)] do 
        Uh[k]:=[]; Vh[k]:=[]; Wh[k]:=[];
        aa:=gene[k];
        Post:=[1..Length(aa)];
        SubtractSet(Post, Positions(aa,0));
        Quo:=List([1..Length(Post)], s->(Post[s]-Rest(Post[s]))/f(ii)+1);
        Res:=List([1..Length(Post)], s->Rest(Post[s]));
        Sig:=List([1..Length(Post)], s->gene[k][Post[s]]);
        Qu:=Set(Quo);
        Re:=0*[1..Length(Qu)];
        Sg:=0*[1..Length(Qu)];
        for i in [1..Length(Qu)] do 
            Re[i]:=[];
            Sg[i]:=[];
            for j in [1..Length(Positions(Quo,Qu[i]))] do 
                Re[i][j]:=Res[Position(Quo,Qu[i])+j-1];	
                Sg[i][j]:=Sig[Position(Quo,Qu[i])+j-1];
            od;
        od;
        Uh[k]:=Qu;
        Vh[k]:=Re;
        Wh[k]:=Sg;
    od;
    return Vh;
end;

WW:=function(gene,ii)
    local Rest,Uh,Vh,Wh,Post,k,aa,Quo,Res,Sig,Qu,Re,Sg,i,j;
    Rest:=function(n)
        if n mod f(ii) > 0 then return n mod f(ii);
        else return f(ii);
        fi;
    end;
    Uh:=0*[1..Length(gene)];;
    Vh:=0*[1..Length(gene)];;
    Wh:=0*[1..Length(gene)];;
    Post:=[];;
    for k in [1..Length(gene)] do 
        Uh[k]:=[]; Vh[k]:=[]; Wh[k]:=[];
        aa:=gene[k];
        Post:=[1..Length(aa)];
        SubtractSet(Post, Positions(aa,0));
        Quo:=List([1..Length(Post)], s->(Post[s]-Rest(Post[s]))/f(ii)+1);
        Res:=List([1..Length(Post)], s->Rest(Post[s]));
        Sig:=List([1..Length(Post)], s->gene[k][Post[s]]);
        Qu:=Set(Quo);
        Re:=0*[1..Length(Qu)];
        Sg:=0*[1..Length(Qu)];
        for i in [1..Length(Qu)] do 
            Re[i]:=[];
            Sg[i]:=[];
            for j in [1..Length(Positions(Quo,Qu[i]))] do 
                Re[i][j]:=Res[Position(Quo,Qu[i])+j-1];	
                Sg[i][j]:=Sig[Position(Quo,Qu[i])+j-1];
            od;
        od;
        Uh[k]:=Qu;
        Vh[k]:=Re;
        Wh[k]:=Sg;
    od;
    return Wh;
end;

geneMH:=function(i,n,m)
    if i=0 and n=3 and m=3 then return 
    [Ker(0,3,3)[99], Ker(0,3,3)[378], Ker(0,3,3)[164], -Ker(0,3,3)[467],
     Ker(0,3,3)[219], -Ker(0,3,3)[40], Ker(0,3,3)[301], 
     Ker(0,3,3)[206]-Ker(0,3,3)[99]-Ker(0,3,3)[378]+Ker(0,3,3)[164]-Ker(0,3,3)[467]];
    elif i=0 and n=3 and m=5 then return [Ker(0,3,5)[79]];
    elif i=0 and n=4 and m=4 then return 
    [Ker(0,4,4)[550], Ker(0,4,4)[450]-Ker(0,4,4)[550]];
    elif i=0 and n=5 and m=11 then return [Ker(0,5,11)[90]];
    elif i=1 and n=1 and m=3 then return 
    [Ker(1,1,3)[15], Ker(1,1,3)[27], Ker(1,1,3)[53], -Ker(1,1,3)[67],
     Ker(1,1,3)[19], -Ker(1,1,3)[16], Ker(1,1,3)[22]];
    elif i=1 and n=1 and m=5 then return [Ker(1,1,5)[76]];
    elif i=1 and n=1 and m=7 then return [Ker(1,1,7)[64]];
    elif i=2 and n=1 and m=3 then return [Ker(2,1,3)[257]];
    elif i=2 and n=1 and m=5 then return [Ker(2,1,5)[908]];
    elif i=2 and n=2 and m=4 then return 
    [Ker(2,2,4)[783]-Ker(2,2,4)[784], Ker(2,2,4)[784]];
    elif i=2 and n=3 and m=3 then return 
    [Ker(2,3,3)[36]-Ker(2,3,3)[193]-Ker(2,3,3)[470]-Ker(2,3,3)[570]-Ker(2,3,3)[658],  
     Ker(2,3,3)[200], Ker(2,3,3)[197], -Ker(2,3,3)[16],
     Ker(2,3,3)[193], Ker(2,3,3)[470], Ker(2,3,3)[570], Ker(2,3,3)[658]];
    elif i=3 and n=3 and m=3 then return 
    [Ker(3,3,3)[179], -Ker(3,3,3)[185], -Ker(3,3,3)[174], Ker(3,3,3)[176],
     Ker(3,3,3)[355], Ker(3,3,3)[452], Ker(3,3,3)[540], Ker(3,3,3)[628]];
    else return [];
    fi;
end;
\end{Verbatim}
\end{footnotesize}
\end{tcolorbox}

\subsection{\texorpdfstring{A basis of $(M^2)^!$}{A basis of quadratic dual of M2}}
\label{Appendix:basis m2 d}

We present here the GAP code to compute a basis of $(M^2)^!_{-n}$ for $n$ less than some positive integer, where the quadratic module $M^2$ is defined at the beginning of Subsection \ref{subsec:Resolving datum fk4}. 
We also list the basis of $(M^2)^!_{-n}$ for $n\in \llbracket 0,3\rrbracket$. 
 
\parskip 1ex
\parindent 0in
\begin{tcolorbox}[breakable,colback=white,width=\textwidth ,center,arc=0mm,size=fbox]
\begin{footnotesize}
\begin{Verbatim}[samepage=false]
LoadPackage("GBNP");
B:=FreeAssociativeAlgebraWithOne(Rationals,"y12","y13","y23","y14","y24","y34");;
y12:=B.y12;; y13:=B.y13;; y23:=B.y23;; y14:=B.y14;; y24:=B.y24;; y34:=B.y34;; 
oB:=One(B);;
relationsB:=[y12*y23+y23*y13, y13*y23+y23*y12, y12*y23+y13*y12, y12*y13+y23*y12,
 y12*y24+y24*y14, y14*y24+y24*y12, y12*y24+y14*y12, y12*y14+y24*y12, y13*y34+y34*y14, 
 y14*y34+y34*y13, y13*y34+y14*y13, y13*y14+y34*y13, y23*y34+y34*y24, y24*y34+y34*y23, 
 y23*y34+y24*y23, y23*y24+y34*y23, y12*y34+y34*y12, y13*y24+y24*y13, 
 y23*y14+y14*y23];;
relsBNP:=GP2NPList(relationsB);;
wtv:= [1,1,1,1,1,1];;
GB:=Grobner(relsBNP);;
GBNP.ConfigPrint(B);;
PrintNPList(GB);

D:= BaseQATrunc(GB,12,wtv);; 
for degpart in D do 
    for mon in degpart do 
	PrintNP([[mon],[1]]); 
    od; 
od;
DT:=[];
for degpart in D do 
    for mon in degpart do 
        Append(DT,[[[mon],[1]]]); 
    od; 
od;

S:=B^7;
ab:=GeneratorsOfLeftModule(S);
g7:=ab[1]; g6:=ab[2]; g5:=ab[3]; g4:=ab[4]; g3:=ab[5]; g2:=ab[6]; g1:=ab[7];  
modrels:=[g1*y14+g4*y14, g1*y24+g3*y24, g1*y34-g2*y34, g2*y13+g4*y13, g2*y23+g3*y23, 
 g3*y12-g4*y12, g5*y12-g6*y34, g1*y24+g5*y13, g5*y23+g7*y14, g5*y14-g7*y23, 
 g2*y13-g5*y24, g5*y34+g6*y12, g6*y13+g7*y24, g1*y14-g6*y23, g2*y23+g6*y14, 
 g6*y24-g7*y13, g1*y34-g7*y12, g3*y12-g7*y34];
modrelsNP:=GP2NPList(modrels);
PrintNPList(modrelsNP);

GBNP.CheckHom:=function(G,wtv)
    local i,j,k,l,mon,h1,h2,ans;
    mon:=LMonsNP(G);
    ans:=GBNP.WeightedDegreeList(mon,wtv);
    for i in [1..Length(G)] do
        h1:=ans[i];
        l:=Length(G[i][1]);
        for j in [2..l] do
            mon:=G[i][1][j];
            h2:=0;
            for k in [1..Length(mon)] do
                if mon[k]>0 then
                    h2:=h2+wtv[mon[k]];
                fi;
            od;
            if h2<>h1 then return(false); fi;
        od;
    od;
    Info(InfoGBNP,1,"Input is homogeneous");
    return(ans);
end;
   
GBNP.WeightedDegreeMon:=function(mon,lst)
    local i,ans;
    ans:=0;
    for i in mon do
        if i>0 then
            ans:=ans+lst[i];
        fi;
    od;
    return(ans);
end;;

SetInfoLevel(InfoGBNP,1);
SetInfoLevel(InfoGBNPTime,1);
combinedrelsNP:=Concatenation(GB,modrelsNP);
GBT:=SGrobnerTrunc(combinedrelsNP, 15, [1,1,1,1,1,1]);
PrintNPList(GBT);

splitGBT:=function(GBT)
    local p, ts, rel, lm; 
    p:=[];
    ts:=[];
    for rel in GBT do
        lm := rel[1,1];
        if Length(lm)>1 and lm[1]<0 then
            Add(p, rel);
        else
            Add(ts, rel);
        fi;
    od;
    return rec(p:=p, ts:=ts);
end;;
   
split:=splitGBT(GBT);
GBR:=rec(p:=split.p, pg:=7, ts:=split.ts);
BQM:=BaseQM(GBR,6,7,500);;
PrintNPList(BQM);

[ 0, 0, 0, 0, 0, 0, 1 ]
[ 0, 0, 0, 0, 0, 1 , 0]
[ 0, 0, 0, 0, 1 , 0, 0]
[ 0, 0, 0, 1 , 0, 0, 0]
[ 0, 0, 1 , 0, 0, 0, 0]
[ 0, 1 , 0, 0, 0, 0, 0]
[ 1 , 0, 0, 0, 0, 0, 0]
[ 0, 0, 0, 0, 0, 0, y12 ]
[ 0, 0, 0, 0, 0, 0, y13 ]
[ 0, 0, 0, 0, 0, 0, y23 ]
[ 0, 0, 0, 0, 0, 0, y14 ]
[ 0, 0, 0, 0, 0, 0, y24 ]
[ 0, 0, 0, 0, 0, 0, y34 ]
[ 0, 0, 0, 0, 0, y12 , 0]
[ 0, 0, 0, 0, 0, y13 , 0]
[ 0, 0, 0, 0, 0, y23 , 0]
[ 0, 0, 0, 0, 0, y14 , 0]
[ 0, 0, 0, 0, 0, y24 , 0]
[ 0, 0, 0, 0, y12 , 0, 0]
[ 0, 0, 0, 0, y13 , 0, 0]
[ 0, 0, 0, 0, y14 , 0, 0]
[ 0, 0, 0, 0, y34 , 0, 0]
[ 0, 0, 0, y23 , 0, 0, 0]
[ 0, 0, 0, y24 , 0, 0, 0]
[ 0, 0, 0, y34 , 0, 0, 0]
[ 0, 0, y12 , 0, 0, 0, 0]
[ 0, 0, y23 , 0, 0, 0, 0]
[ 0, 0, y14 , 0, 0, 0, 0]
[ 0, 0, y34 , 0, 0, 0, 0]
[ 0, y13 , 0, 0, 0, 0, 0]
[ 0, y24 , 0, 0, 0, 0, 0]
[ 0, 0, 0, 0, 0, 0, y12^2 ]
[ 0, 0, 0, 0, 0, 0, y12y13 ]
[ 0, 0, 0, 0, 0, 0, y12y23 ]
[ 0, 0, 0, 0, 0, 0, y12y14 ]
[ 0, 0, 0, 0, 0, 0, y12y24 ]
[ 0, 0, 0, 0, 0, 0, y12y34 ]
[ 0, 0, 0, 0, 0, 0, y13^2 ]
[ 0, 0, 0, 0, 0, 0, y13y14 ]
[ 0, 0, 0, 0, 0, 0, y13y24 ]
[ 0, 0, 0, 0, 0, 0, y13y34 ]
[ 0, 0, 0, 0, 0, 0, y23^2 ]
[ 0, 0, 0, 0, 0, 0, y23y14 ]
[ 0, 0, 0, 0, 0, 0, y23y24 ]
[ 0, 0, 0, 0, 0, 0, y23y34 ]
[ 0, 0, 0, 0, 0, 0, y14^2 ]
[ 0, 0, 0, 0, 0, 0, y24^2 ]
[ 0, 0, 0, 0, 0, 0, y34^2 ]
[ 0, 0, 0, 0, 0, y12^2 , 0]
[ 0, 0, 0, 0, 0, y12y13 , 0]
[ 0, 0, 0, 0, 0, y12y23 , 0]
[ 0, 0, 0, 0, 0, y12y14 , 0]
[ 0, 0, 0, 0, 0, y12y24 , 0]
[ 0, 0, 0, 0, 0, y13y24 , 0]
[ 0, 0, 0, 0, 0, y23y14 , 0]
[ 0, 0, 0, 0, 0, y14^2 , 0]
[ 0, 0, 0, 0, 0, y24^2 , 0]
[ 0, 0, 0, 0, y12y34 , 0, 0]
[ 0, 0, 0, 0, y13^2 , 0, 0]
[ 0, 0, 0, 0, y13y14 , 0, 0]
[ 0, 0, 0, 0, y13y34 , 0, 0]
[ 0, 0, 0, 0, y14^2 , 0, 0]
[ 0, 0, 0, 0, y34^2 , 0, 0]
[ 0, 0, 0, y23^2 , 0, 0, 0]
[ 0, 0, 0, y23y24 , 0, 0, 0]
[ 0, 0, 0, y23y34 , 0, 0, 0]
[ 0, 0, 0, y24^2 , 0, 0, 0]
[ 0, 0, 0, y34^2 , 0, 0, 0]
[ 0, 0, y12^2 , 0, 0, 0, 0]
[ 0, 0, y12y34 , 0, 0, 0, 0]
[ 0, 0, y23^2 , 0, 0, 0, 0]
[ 0, 0, y23y14 , 0, 0, 0, 0]
[ 0, y13^2 , 0, 0, 0, 0, 0]
[ 0, y13y24 , 0, 0, 0, 0, 0]
[ 0, 0, 0, 0, 0, 0, y12^3 ]
[ 0, 0, 0, 0, 0, 0, y12^2y13 ]
[ 0, 0, 0, 0, 0, 0, y12^2y23 ]
[ 0, 0, 0, 0, 0, 0, y12^2y14 ]
[ 0, 0, 0, 0, 0, 0, y12^2y24 ]
[ 0, 0, 0, 0, 0, 0, y12^2y34 ]
[ 0, 0, 0, 0, 0, 0, y12y13^2 ]
[ 0, 0, 0, 0, 0, 0, y12y13y14 ]
[ 0, 0, 0, 0, 0, 0, y12y13y24 ]
[ 0, 0, 0, 0, 0, 0, y12y13y34 ]
[ 0, 0, 0, 0, 0, 0, y12y23y14 ]
[ 0, 0, 0, 0, 0, 0, y12y23y24 ]
[ 0, 0, 0, 0, 0, 0, y12y23y34 ]
[ 0, 0, 0, 0, 0, 0, y12y14^2 ]
[ 0, 0, 0, 0, 0, 0, y12y34^2 ]
[ 0, 0, 0, 0, 0, 0, y13^3 ]
[ 0, 0, 0, 0, 0, 0, y13^2y24 ]
[ 0, 0, 0, 0, 0, 0, y13^2y34 ]
[ 0, 0, 0, 0, 0, 0, y13y14^2 ]
[ 0, 0, 0, 0, 0, 0, y13y24^2 ]
[ 0, 0, 0, 0, 0, 0, y23^3 ]
[ 0, 0, 0, 0, 0, 0, y23^2y14 ]
[ 0, 0, 0, 0, 0, 0, y23y14^2 ]
[ 0, 0, 0, 0, 0, 0, y23y24^2 ]
[ 0, 0, 0, 0, 0, 0, y14^3 ]
[ 0, 0, 0, 0, 0, 0, y24^3 ]
[ 0, 0, 0, 0, 0, 0, y34^3 ]
[ 0, 0, 0, 0, 0, y12^3 , 0]
[ 0, 0, 0, 0, 0, y12^2y14 , 0]
[ 0, 0, 0, 0, 0, y12^2y24 , 0]
[ 0, 0, 0, 0, 0, y12y14^2 , 0]
[ 0, 0, 0, 0, 0, y13y24^2 , 0]
[ 0, 0, 0, 0, 0, y23y14^2 , 0]
[ 0, 0, 0, 0, 0, y14^3 , 0]
[ 0, 0, 0, 0, 0, y24^3 , 0]
[ 0, 0, 0, 0, y12y34^2 , 0, 0]
[ 0, 0, 0, 0, y13^3 , 0, 0]
[ 0, 0, 0, 0, y13^2y14 , 0, 0]
[ 0, 0, 0, 0, y13^2y34 , 0, 0]
[ 0, 0, 0, 0, y13y14^2 , 0, 0]
[ 0, 0, 0, 0, y14^3 , 0, 0]
[ 0, 0, 0, 0, y34^3 , 0, 0]
[ 0, 0, 0, y23^3 , 0, 0, 0]
[ 0, 0, 0, y23^2y24 , 0, 0, 0]
[ 0, 0, 0, y23^2y34 , 0, 0, 0]
[ 0, 0, 0, y23y24^2 , 0, 0, 0]
[ 0, 0, 0, y24^3 , 0, 0, 0]
[ 0, 0, 0, y34^3 , 0, 0, 0]
[ 0, 0, y12^3 , 0, 0, 0, 0]
[ 0, 0, y12^2y34 , 0, 0, 0, 0]
[ 0, 0, y23^3 , 0, 0, 0, 0]
[ 0, 0, y23^2y14 , 0, 0, 0, 0]
[ 0, y13^3 , 0, 0, 0, 0, 0]
[ 0, y13^2y24 , 0, 0, 0, 0, 0]
\end{Verbatim}
\end{footnotesize}
\end{tcolorbox}

\subsection{\texorpdfstring{Right action of $\FK(4)^!$ on $(M^2)^!$}{Right action of FK(4)! on (M2)!}}
\label{sec:products M2}

We list below the right action of some elements of $A^!$ on $(M^2)^!$, where $M^2$ is the quadratic (right) $A$-module defined at the beginning of Subsection \ref{subsec:Resolving datum fk4}.
In Tables \ref{table:product module M2 yy' n even 123}-\ref{table:product module M2 yy' n odd 456}, the entry appearing in the row indexed by $y$ and the column indexed by $y'$ is the product $yy'$. 
To reduce space, the integer $m\in\llbracket 1,24\rrbracket$, appearing in the third to fifth columns of Tables \ref{table:product module M2 yy' n even 123}-\ref{table:product module M2 yy' n odd 456} indicates the element $\mathsf{b}^{n+1}_m$, where $\mathsf{b}^n_m$ is the $m$-th element in \eqref{eq: basis1 M2} for $n\geqslant 4$ and $m\in \llbracket 1, 24\rrbracket$. 

%\vskip -0.5cm

\begin{table}[H]
   \begin{center}
       \resizebox{0.8\textwidth}{101mm}{
   \begin{tabular}{c|c|ccc}
         \hline
         & \diagbox[width=26mm,height=5mm]{$y$}{$y'$}  & $y_{1,2}$ & $y_{1,3}$  & $y_{2,3}$ 
         \\
         \hline
         $1$ & $g_1y_{1,2}^{n-1}y_{1,3}$ & $ -2$ & $5$ & $1 $ 
         \\
         $2$ & $g_1y_{1,2}^{n-1}y_{2,3}$ & $ -1$ & $-2$ & $5$ 
         \\
         $3$ & $g_1y_{1,2}^{n-1}y_{1,4}$ & $ -4$ & $-8$ & $-9$ 
         \\
         $4$ & $g_1y_{1,2}^{n-1}y_{2,4}$ & $-3 $ & $-7$ & $-11$ 
         \\
         $5$ & $g_1y_{1,2}^{n-2}y_{1,3}^2$ & $ 5$ & $1$ & $2$ 
         \\
         $6$ & $g_1y_{1,2}^{n-2}y_{1,3}y_{1,4}$ & $10 $ & $-13$ & $-6$ 
         \\
         $7$ & $g_1y_{1,2}^{n-2}y_{1,3}y_{2,4}$ & $9 $ & $-4$ & $-8$ 
         \\
         $8$ & $g_1y_{1,2}^{n-2}y_{1,3}y_{3,4}$ & $11 $ & $-3$ & $-7$ 
         \\
         $9$ & $g_1y_{1,2}^{n-2}y_{2,3}y_{1,4}$ & $7 $ & $11$ & $-3$ 
         \\
         $10$ & $g_1y_{1,2}^{n-2}y_{2,3}y_{2,4}$ & $ 6$ & $10$ & $-13$ 
         \\
         $11$ & $g_1y_{1,2}^{n-2}y_{2,3}y_{3,4}$ & $8 $ & $9$ & $-4$ 
         \\
         $12$ & $g_1y_{1,2}^{n-2}y_{1,4}^2$ & $12 $ & $14$ & $15$ 
         \\
         $13$ & $g_1y_{1,2}^{n-3}y_{1,3}^2y_{3,4}$ & $-13 $ & $-6$ & $-10$ 
         \\
         $14$ & $g_1y_{1,2}^{n-3}y_{1,3}y_{1,4}^2$ & $-15 $ & $12$ & $14$ 
         \\
         $15$ & $g_1y_{1,2}^{n-3}y_{2,3}y_{1,4}^2$ & $ -14$ & $-15$ & $12$ 
         \\
         $16$ & $g_2y_{1,2}^{n-1}y_{1,4}$ & $-17 $ & $-8$ & $-9$ 
         \\
         $17$ & $g_2y_{1,2}^{n-1}y_{2,4}$ & $ -16$ & $-7$ & $-11$ 
         \\
         $18$ & $g_2y_{1,2}^{n-2}y_{1,4}^2$ & $18 $ & $14$ & $15$ 
         \\
         $19$ & $g_3y_{1,3}^{n-1}y_{1,4}$ & $-10 $ & $-20$ & $6$ 
         \\
         $20$ & $g_3y_{1,3}^{n-1}y_{3,4}$ & $-11 $ &  $-19$ & $7$ 
         \\
         $21$ & $g_3y_{1,3}^{n-2}y_{1,4}^2$ & $-12 $ &  $21$ & $-15$ 
         \\
         $22$ & $g_4y_{2,3}^{n-1}y_{2,4}$ & $-6 $ & $-10$ & $-23$ 
         \\
         $23$ & $g_4y_{2,3}^{n-1}y_{3,4}$ & $-8 $ & $-9$ & $-22$ 
         \\
         $24$ & $g_4y_{2,3}^{n-2}y_{2,4}^2$ & $ -12$ & $-14$ & $24$ 
         \\
         \hline
         %%%%%%%%%%%%%%%
         & $ g_1 y_{1,2}^n $ & $g_1y_{1,2}^{n+1}$ & $1$ & $2$ 
         \\
         & $ g_1 y_{1,2}^{n-r}y_{3,4}^{r} $ & $(-1)^{r} g_1 y_{1,2}^{n-r+1}y_{3,4}^{r} $ & $\chi_r 14 -\chi_{r+1} 6 $ & $\chi_r 15-\chi_{r+1}10 $ 
         \\
         & $ g_1 y_{3,4}^n $ & $g_1 y_{1,2}y_{3,4}^n$ & $14$ & $15$ 
         \\
       & $g_2y_{1,2}^{n}$ & $g_2y_{1,2}^{n+1} $ & $14$ & $15$ 
         \\
       & $g_3y_{1,2}y_{3,4}^{n-1}$ & $ g_1y_{3,4}^{n+1}$ & $6$ & $10$ 
         \\
         & $g_3y_{3,4}^{n}$ & $g_3y_{1,2}y_{3,4}^n $ & $21$ & $-15$ 
         \\
          & $g_4y_{3,4}^{n}$ & $g_3 y_{1,2}y_{3,4}^n $ & $-14$ & $24$  
         \\
          & $g_5y_{1,2}^{n}$ & $g_5 y_{1,2}^{n+1} $ & $-4$ & $-8$ 
         \\
          & $g_5y_{1,2}^{n-1}y_{3,4}$ & $-g_5y_{1,2}^n y_{3,4} $ & $-11$ & $3$ 
         \\
         \hline
         %%%%%%%%%%%
         & $ g_1 y_{1,3}^n $ & $5$ & $g_1y_{1,3}^{n+1}$ &  $2$ 
         \\
         & $ g_1 y_{1,3}^{n-r}y_{2,4}^{r} $ & $\chi_r 12+\chi_{r+1}9 $ & $ (-1)^{r}g_1 y_{1,3}^{n-r+1}y_{2,4}^{r} $ & $\chi_r 15-\chi_{r+1} 8 $ 
         \\
         & $ g_1 y_{2,4}^n $ & $ 12$ & $g_1y_{1,3}y_{2,4}^n$ & $15$ 
         \\
    & $g_2y_{1,3}y_{2,4}^{n-1}$ & $9 $ & $-g_1 y_{2,4}^{n+1} $ & $-8$ 
         \\
         & $g_2y_{2,4}^{n}$ & $18 $ & $g_2 y_{1,3}y_{2,4}^n $ & $15$  
         \\
    & $g_3y_{1,3}^{n}$ & $-12 $ & $g_3 y_{1,3}^{n+1} $ & $-15$
         \\
         & $g_4y_{2,4}^{n}$ & $-12 $ & $-g_2y_{1,3}y_{2,4}^n  $ & $24$ 
         \\
         & $g_6y_{1,3}^{n}$ & $-7 $ & $ g_6 y_{1,3}^{n+1}$ & $3$ 
         \\
          & $g_6y_{1,3}^{n-1}y_{2,4}$ & $-13 $ & $-g_6y_{1,3}^{n}y_{2,4} $ & $-10$ 
         \\
         \hline
         %%%%%%%%%%
         & $ g_1 y_{2,3}^n $ & $5$ & $ 1$ & $g_1 y_{2,3}^{n+1}$
         \\
         & $ g_1 y_{2,3}^{n-r}y_{1,4}^{r} $ & $ \chi_r 12+\chi_{r+1} 7$ & $\chi_r 14+\chi_{r+1} 11$ & $(-1)^rg_1 y_{2,3}^{n-r+1}y_{1,4}^r $ 
         \\
         & $ g_1 y_{1,4}^n $ & $ 12$ & $14$ & $g_1 y_{2,3}y_{1,4}^n $
         \\
    & $g_2y_{2,3}y_{1,4}^{n-1}$ & $ 7$ & $11$ & $-g_1 y_{1,4}^{n+1}$ 
         \\
         & $g_2y_{1,4}^{n}$ & $18 $ & $14$ & $g_2 y_{2,3}y_{1,4}^n$ 
         \\
         & $g_3y_{1,4}^{n}$ & $ -12$ &  $21$ & $-g_2 y_{2,3}y_{1,4}^n$
         \\
    & $g_4y_{2,3}^{n}$ & $-12 $ & $-14$ & $g_4 y_{2,3}^{n+1}$ 
         \\
         & $g_5y_{2,3}^{n}$ & $ 9$ & $-4$ & $g_5 y_{2,3}^{n+1}$ 
         \\
          & $g_5y_{2,3}^{n-1}y_{1,4}$ & $-13 $ & $-6$ & $-g_5y_{2,3}^ny_{1,4} $ 
         \\
         \hline
      \end{tabular}
        }
       % }
      \caption{Products $yy'$ for $n\geqslant 4$ even.}	
      \label{table:product module M2 yy' n even 123}
   \end{center}
   \end{table}
   \vspace{-0.8cm}

%%%%%%%%%%%%%%%%%%%%%%%%%%%%%%%%
\begin{table}[H]
   \begin{center}
        % \resizebox{\textwidth}{58.5mm}{
   \begin{tabular}{c|c|ccc}
         \hline
         & \diagbox[width=27mm,height=5mm]{$y$}{$y'$}  & $y_{1,4}$ & $y_{2,4}$ & $y_{3,4}$
         \\
         \hline
         $1$ & $g_1y_{1,2}^{n-1}y_{1,3}$ & $6 $ & $7 $ & $8 $
         \\
         $2$ & $g_1y_{1,2}^{n-1}y_{2,3}$ & $9 $ & $ 10$ & $ 11$
         \\
         $3$ & $g_1y_{1,2}^{n-1}y_{1,4}$ & $12 $ & $3 $ & $ 6$
         \\
         $4$ & $g_1y_{1,2}^{n-1}y_{2,4}$ & $ -4$ & $12 $ & $10 $
         \\
         $5$ & $g_1y_{1,2}^{n-2}y_{1,3}^2$ & $3 $ & $ 4$ & $ 13$
         \\
         $6$ & $g_1y_{1,2}^{n-2}y_{1,3}y_{1,4}$ & $ 14$ & $ -9$ & $3 $
         \\
         $7$ & $g_1y_{1,2}^{n-2}y_{1,3}y_{2,4}$ & $10 $ & $14 $ & $7 $
         \\
         $8$ & $g_1y_{1,2}^{n-2}y_{1,3}y_{3,4}$ & $ -13$ & $-8 $ & $14 $
         \\
         $9$ & $g_1y_{1,2}^{n-2}y_{2,3}y_{1,4}$ & $15 $ & $-6 $ & $-9 $
         \\
         $10$ & $g_1y_{1,2}^{n-2}y_{2,3}y_{2,4}$ & $ 7$ & $ 15$ & $4 $
         \\
         $11$ & $g_1y_{1,2}^{n-2}y_{2,3}y_{3,4}$ & $11 $ & $-13 $ & $ 15$ 
         \\
         $12$ & $g_1y_{1,2}^{n-2}y_{1,4}^2$ & $3 $ & $ 4$ & $13 $
         \\
         $13$ & $g_1y_{1,2}^{n-3}y_{1,3}^2y_{3,4}$ & $-8 $ & $-11 $ & $12 $
         \\
         $14$ & $g_1y_{1,2}^{n-3}y_{1,3}y_{1,4}^2$ & $6 $ & $7 $ & $8 $
         \\
         $15$ & $g_1y_{1,2}^{n-3}y_{2,3}y_{1,4}^2$ & $9 $ & $10 $ & $11 $
         \\
         $16$ & $g_2y_{1,2}^{n-1}y_{1,4}$ & $18 $ & $16 $ & $6 $
         \\
         $17$ & $g_2y_{1,2}^{n-1}y_{2,4}$ & $-17 $ & $ 18$ & $10 $
         \\
         $18$ & $g_2y_{1,2}^{n-2}y_{1,4}^2$ & $16 $ & $17 $ & $13 $
         \\
         $19$ & $g_3y_{1,3}^{n-1}y_{1,4}$ & $21 $ & $9 $ & $19 $
         \\
         $20$ & $g_3y_{1,3}^{n-1}y_{3,4}$ & $-20 $ &  $ 8$ & $21 $
         \\
         $21$ & $g_3y_{1,3}^{n-2}y_{1,4}^2$ & $19 $ &  $ -4$ & $20 $
         \\
         $22$ & $g_4y_{2,3}^{n-1}y_{2,4}$ & $ -7$ & $ 24$ & $22 $
         \\
         $23$ & $g_4y_{2,3}^{n-1}y_{3,4}$ & $ -11$ & $ -23$ & $24 $
         \\
         $24$ & $g_4y_{2,3}^{n-2}y_{2,4}^2$ & $ -3$ & $ 22$ & $23 $
         \\
         \hline
         %%%%%%%%%%%%%%%
         & $ g_1 y_{1,2}^n $ & $3 $ & $4 $ & $ g_1y_{1,2}^n y_{3,4}$ 
         \\
         & $ g_1 y_{1,2}^{n-r}y_{3,4}^{r} $ & $\chi_r 3-\chi_{r+1} 8 $ & $\chi_r 4-\chi_{r+1} 11 $ & $ g_1 y_{1,2}^{n-r}y_{3,4}^{r+1}$
         \\
         & $ g_1 y_{3,4}^n $ & $ 3$ & $ 4$ & $g_1 y_{3,4}^{n+1} $ 
         \\
       & $g_2y_{1,2}^{n}$ & $ 16$ & $ 17$ & $g_1 y_{1,2}^n y_{3,4} $
         \\
       & $g_3y_{1,2}y_{3,4}^{n-1}$ & $ 8$ & $11 $ & $g_3 y_{1,2}y_{3,4}^{n} $
         \\
         & $g_3y_{3,4}^{n}$ & $ 19$ & $-4 $ &  $ g_3 y_{3,4}^{n+1}$
         \\
          & $g_4y_{3,4}^{n}$ & $-3 $ &  $22 $ & $g_4 y_{3,4}^{n+1} $
         \\
          & $g_5y_{1,2}^{n}$ & $10 $ & $14 $ & $g_5 y_{1,2}^{n}y_{3,4} $ 
         \\
          & $g_5y_{1,2}^{n-1}y_{3,4}$ & $-15 $ & $6 $ & $g_5 y_{1,2}^{n+1} $
         \\
         \hline
         %%%%%%%%%%%
         & $ g_1 y_{1,3}^n $ & $ 3$ & $g_1 y_{1,3}^n y_{2,4} $ & $13 $
         \\
         & $ g_1 y_{1,3}^{n-r}y_{2,4}^{r} $ & $\chi_r 3+\chi_{r+1} 10 $ & $g_1 y_{1,3}^{n-r}y_{2,4}^{r+1} $ & $\chi_r 13+\chi_{r+1} 7 $
         \\
         & $ g_1 y_{2,4}^n $ & $3 $ & $g_1y_{2,4}^{n+1} $ & $13 $
         \\
    & $g_2y_{1,3}y_{2,4}^{n-1}$ & $10 $ & $g_2 y_{1,3}y_{2,4}^{n} $ & $7 $
         \\
         & $g_2y_{2,4}^{n}$ & $16 $ & $g_2 y_{2,4}^{n+1} $ & $13 $
         \\
    & $g_3y_{1,3}^{n}$ & $19 $ & $ -g_1y_{1,3}^n y_{2,4}$ & $20 $
         \\
         & $g_4y_{2,4}^{n}$ & $-3 $ & $ g_4 y_{2,4}^{n+1}$ & $ 23$
         \\
         & $g_6y_{1,3}^{n}$ & $-15 $ & $g_6 y_{1,3}^ny_{2,4} $ & $9 $
         \\
          & $g_6y_{1,3}^{n-1}y_{2,4}$ & $-8 $ & $g_6 y_{1,3}^{n+1} $ & $12 $
         \\
         \hline
         %%%%%%%%%%%
         & $ g_1 y_{2,3}^n $ & $g_1 y_{2,3}^ny_{1,4} $ & $4 $ & $13 $
         \\
         & $ g_1 y_{2,3}^{n-r}y_{1,4}^{r} $ & $g_1 y_{2,3}^{n-r}y_{1,4}^{r+1} $ & $ \chi_r 4-\chi_{r+1} 6$ & $\chi_r 13-\chi_{r+1}9 $
         \\
         & $ g_1 y_{1,4}^n $ & $g_1 y_{1,4}^{n+1} $ & $ 4$ & $13 $
         \\
    & $g_2y_{2,3}y_{1,4}^{n-1}$ & $g_2 y_{2,3}y_{1,4}^n $ & $-6 $ & $-9 $
         \\
         & $g_2y_{1,4}^{n}$ & $g_2y_{1,4}^{n+1} $ & $17 $ & $ 13$
         \\
         & $g_3y_{1,4}^{n}$ & $g_3 y_{1,4}^{n+1}$ &  $ -4$ & $20 $
         \\
    & $g_4y_{2,3}^{n}$ & $ -g_1 y_{2,3}^n y_{1,4}$ & $22 $ & $23 $
         \\
         & $g_5y_{2,3}^{n}$ & $g_5 y_{2,3}^ny_{1,4} $ & $14 $ & $ 7$
         \\
          & $g_5y_{2,3}^{n-1}y_{1,4}$ & $g_5 y_{2,3}^{n+1} $ & $-11 $ & $12 $
         \\
         \hline
         \end{tabular}
      \caption{Products $yy'$ for $n\geqslant 4$ even.}	
      \label{table:product module M2 yy' n even 456}
      \end{center}
   \end{table}
   \vspace{-0.8cm}

%%%%%%%%%%%%%%%%%%%%%%%%%%%%%%%%
  \begin{table}[H]
   \begin{center}
        % \resizebox{\textwidth}{58.5mm}{
   \begin{tabular}{c|c|cccc}
         \hline
         & \diagbox[width=27mm,height=5mm]{$y$}{$y'$}  & $y_{1,2}$ & $y_{1,3}$  & $y_{2,3}$ 
         \\
         \hline
         $1$ & $g_1y_{1,2}^{n-1}y_{1,3}$ & $ -2$ & $5$ & $1 $ & 
         \\
         $2$ & $g_1y_{1,2}^{n-1}y_{2,3}$ & $-1 $ & $-2 $ & $5 $ & 
         \\
         $3$ & $g_1y_{1,2}^{n-1}y_{1,4}$ & $-4$ & $ -8$ & $-9 $ & 
         \\
         $4$ & $g_1y_{1,2}^{n-1}y_{2,4}$ & $ -3$ & $-7 $ & $-11 $ & 
         \\
         $5$ & $g_1y_{1,2}^{n-2}y_{1,3}^2$ & $ 5$ & $ 1$ & $2 $ & 
         \\
         $6$ & $g_1y_{1,2}^{n-2}y_{1,3}y_{1,4}$ & $10 $ & $-13 $ & $-6 $ & 
         \\
         $7$ & $g_1y_{1,2}^{n-2}y_{1,3}y_{2,4}$ & $9 $ & $-4 $ & $-8 $ & 
         \\
         $8$ & $g_1y_{1,2}^{n-2}y_{1,3}y_{3,4}$ & $11 $ & $-3 $ & $ -7$ & 
         \\
         $9$ & $g_1y_{1,2}^{n-2}y_{2,3}y_{1,4}$ & $7 $ & $11 $ & $ -3$ & 
         \\
         $10$ & $g_1y_{1,2}^{n-2}y_{2,3}y_{2,4}$ & $ 6$ & $ 10$ & $-13 $ & 
         \\
         $11$ & $g_1y_{1,2}^{n-2}y_{2,3}y_{3,4}$ & $8 $ & $9 $ & $ -4$ & 
         \\
         $12$ & $g_1y_{1,2}^{n-2}y_{1,4}^2$ & $ 12$ & $14 $ & $15 $ & 
         \\
         $13$ & $g_1y_{1,2}^{n-3}y_{1,3}^2y_{3,4}$ & $ -13$ & $-6 $ & $ -10$ & 
         \\
         $14$ & $g_1y_{1,2}^{n-3}y_{1,3}y_{1,4}^2$ & $-15 $ & $12 $ & $14 $ & 
         \\
         $15$ & $g_1y_{1,2}^{n-3}y_{2,3}y_{1,4}^2$ & $-14 $ & $-15 $ & $ 12$ & 
         \\
         $16$ & $g_2y_{1,2}^{n-1}y_{1,4}$ & $-17 $ & $ -8$ & $-9 $ & 
         \\
         $17$ & $g_2y_{1,2}^{n-1}y_{2,4}$ & $-16 $ & $-7 $ & $ -11$ & 
         \\
         $18$ & $g_2y_{1,2}^{n-2}y_{1,4}^2$ & $18 $ & $14 $ & $15 $ & 
         \\
         $19$ & $g_3y_{1,3}^{n-1}y_{1,4}$ & $ 4$ & $-20 $ & $9 $ & 
         \\
         $20$ & $g_3y_{1,3}^{n-1}y_{3,4}$ & $ 13$ &  $-19 $ & $10 $ & 
         \\
         $21$ & $g_3y_{1,3}^{n-2}y_{1,4}^2$ & $15 $ &  $ 21$ & $-14 $ & 
         \\
         $22$ & $g_4y_{2,3}^{n-1}y_{2,4}$ & $3 $ & $ 7$ & $ -23$ & 
         \\
         $23$ & $g_4y_{2,3}^{n-1}y_{3,4}$ & $13 $ & $6 $ & $-22 $ & 
         \\
         $24$ & $g_4y_{2,3}^{n-2}y_{2,4}^2$ & $ 14$ & $15 $ & $24 $ & 
         \\
         \hline
         %%%%%%%%%%%%%%%
         & $ g_1 y_{1,2}^n $ & $g_1y_{1,2}^{n+1}$ & $1$ & $2 $ & 
         \\
         & $ g_1 y_{1,2}^{n-r}y_{3,4}^{r} $ & $(-1)^rg_1 y_{1,2}^{n-r+1}y_{3,4}^{r} $ & $\chi_r 14-\chi_{r+1} 6 $ & $\chi_r 15-\chi_{r+1}10 $ & 
         \\
         & $ g_1 y_{3,4}^n $ & $-g_1y_{1,2}y_{3,4}^n $ & $-6 $ & $ -10$ & 
         \\
       & $g_2y_{1,2}^{n}$ & $g_2 y_{1,2}^{n+1} $ & $ 14$ & $15 $ & 
         \\
       & $g_3y_{1,2}y_{3,4}^{n-1}$ & $-g_1 y_{3,4}^{n+1} $ & $ -14$ & $-15 $ & 
         \\
         & $g_3y_{3,4}^{n}$ & $-g_3 y_{1,2}y_{3,4}^n $ &  $-19 $ & $10 $ & 
         \\
          & $g_4y_{3,4}^{n}$ & $-g_3 y_{1,2}y_{3,4}^n $ & $ 6$ &  $-22 $ & 
         \\
          & $g_5y_{1,2}^{n}$ & $g_5 y_{1,2}^{n+1} $ & $11 $ & $-3 $ & 
         \\
          & $g_5y_{1,2}^{n-1}y_{3,4}$ & $-g_5 y_{1,2}^n y_{3,4} $ & $-4 $ & $ -8$ & 
         \\
         \hline
         %%%%%%%%%%%
         & $ g_1 y_{1,3}^n $ & $-2$ & $g_1 y_{1,3}^{n+1} $ & $1 $ & 
         \\
         & $ g_1 y_{1,3}^{n-r}y_{2,4}^{r} $ & $-\chi_r 15-\chi_{r+1} 3 $ & $(-1)^rg_1 y_{1,3}^{n-r+1}y_{2,4}^r $ & $\chi_r 14-\chi_{r+1} 11 $ & 
         \\
         & $ g_1 y_{2,4}^n $ & $-3 $ & $-g_1 y_{1,3}y_{2,4}^n $ & $-11 $ & 
         \\
    & $g_2y_{1,3}y_{2,4}^{n-1}$ & $ -15$ & $g_1y_{2,4}^{n+1} $ & $14 $ & 
         \\
         & $g_2y_{2,4}^{n}$ & $-16 $ & $-g_2 y_{1,3}y_{2,4}^n $ & $ -11$ & 
         \\
    & $g_3y_{1,3}^{n}$ & $15 $ & $ g_3 y_{1,3}^{n+1}$ & $-14 $ & 
         \\
         & $g_4y_{2,4}^{n}$ & $3 $ & $g_2 y_{1,3}y_{2,4}^n $ & $-23 $ & 
         \\
         & $g_6y_{1,3}^{n}$ & $-8 $ & $g_6 y_{1,3}^{n+1} $ & $4 $ & 
         \\
         & $g_6y_{1,3}^{n-1}y_{2,4}$ & $10 $ & $-g_6 y_{1,3}^n y_{2,4} $ & $-6 $ & 
         \\
         \hline
         %%%%%%
         & $ g_1 y_{2,3}^n $ & $-1 $ & $ -2$ & $ g_1 y_{2,3}^{n+1} $ 
         \\
         & $ g_1 y_{2,3}^{n-r}y_{1,4}^{r} $ & $-\chi_r 14 -\chi_{r+1}4 $ & $-\chi_r 15-\chi_{r+1} 8 $ & $ (-1)^r g_1 y_{2,3}^{n-r+1}y_{1,4}^{r} $ 
         \\
         & $ g_1 y_{1,4}^n $ & $-4 $ & $-8 $ & $ -g_1y_{2,3} y_{1,4}^n $
         \\
    & $g_2y_{2,3}y_{1,4}^{n-1}$ & $-14 $ & $-15 $ & $g_1 y_{1,4}^{n+1} $ 
         \\
         & $g_2y_{1,4}^{n}$ & $-17 $ & $ -8$ & $-g_2 y_{2,3}y_{1,4}^n $
         \\
         & $g_3y_{1,4}^{n}$ & $4 $ &  $ -20$ & $ g_2y_{2,3}y_{1,4}^n$
         \\
    & $g_4y_{2,3}^{n}$ & $14 $ & $15 $ & $g_4 y_{2,3}^{n+1} $
         \\
         & $g_5y_{2,3}^{n}$ & $-11$ & $3 $ & $g_5 y_{2,3}^{n+1} $
         \\
          & $g_5y_{2,3}^{n-1}y_{1,4}$ & $6 $ & $10 $ & $-g_5 y_{2,3}^ny_{1,4} $
         \\
         \hline
      \end{tabular}
        % }
      \caption{Products $yy'$ for $n\geqslant 5$ odd.}	
      \label{table:product module M2 yy' n odd 123}
   \end{center}
   \end{table}
   \vspace{-0.8cm}

%%%%%%%%%%%%%%%%%%%%%%%%%%%%%%%%
   \begin{table}[H]
   \begin{center}
        % \resizebox{\textwidth}{58.5mm}{
   \begin{tabular}{c|c|cccccc}
         \hline
         & \diagbox[width=27mm,height=5mm]{$y$}{$y'$}  & $y_{1,4}$ & $y_{2,4}$ & $y_{3,4}$
         \\
         \hline
         $1$ & $g_1y_{1,2}^{n-1}y_{1,3}$ & $6 $ & $7 $ & $8 $ 
         \\
         $2$ & $g_1y_{1,2}^{n-1}y_{2,3}$ & $9 $ & $10 $ & $ 11$ 
         \\
         $3$ & $g_1y_{1,2}^{n-1}y_{1,4}$ & $12 $ & $3 $ & $ 6$ 
         \\
         $4$ & $g_1y_{1,2}^{n-1}y_{2,4}$ & $-4 $ & $12 $ & $10 $ 
         \\
         $5$ & $g_1y_{1,2}^{n-2}y_{1,3}^2$ & $ 3$ & $4 $ & $13 $ 
         \\
         $6$ & $g_1y_{1,2}^{n-2}y_{1,3}y_{1,4}$ & $14 $ & $ -9$ & $ 3$ 
         \\
         $7$ & $g_1y_{1,2}^{n-2}y_{1,3}y_{2,4}$ & $10 $ & $14 $ & $7 $ 
         \\
         $8$ & $g_1y_{1,2}^{n-2}y_{1,3}y_{3,4}$ & $-13 $ & $-8 $ & $ 14$ 
         \\
         $9$ & $g_1y_{1,2}^{n-2}y_{2,3}y_{1,4}$ & $15 $ & $ -6$ & $-9 $ 
         \\
         $10$ & $g_1y_{1,2}^{n-2}y_{2,3}y_{2,4}$ & $7 $ & $ 15$ & $4 $ 
         \\
         $11$ & $g_1y_{1,2}^{n-2}y_{2,3}y_{3,4}$ & $ 11$ & $-13 $ & $15 $ 
         \\
         $12$ & $g_1y_{1,2}^{n-2}y_{1,4}^2$ & $3 $ & $ 4$ & $13 $ 
         \\
         $13$ & $g_1y_{1,2}^{n-3}y_{1,3}^2y_{3,4}$ & $-8 $ & $ -11$ & $12 $ 
         \\
         $14$ & $g_1y_{1,2}^{n-3}y_{1,3}y_{1,4}^2$ & $6 $ & $ 7$ & $8 $ 
         \\
         $15$ & $g_1y_{1,2}^{n-3}y_{2,3}y_{1,4}^2$ & $9 $ & $10 $ & $11 $ 
         \\
         $16$ & $g_2y_{1,2}^{n-1}y_{1,4}$ & $18 $ & $16 $ & $6 $ 
         \\
         $17$ & $g_2y_{1,2}^{n-1}y_{2,4}$ & $-17 $ & $ 18$ & $10 $ 
         \\
         $18$ & $g_2y_{1,2}^{n-2}y_{1,4}^2$ & $ 16$ & $17 $ & $13 $ 
         \\
         $19$ & $g_3y_{1,3}^{n-1}y_{1,4}$ & $21 $ & $-3 $ & $ 19$ 
         \\
         $20$ & $g_3y_{1,3}^{n-1}y_{3,4}$ & $-20 $ &  $11 $ & $ 21$ 
         \\
         $21$ & $g_3y_{1,3}^{n-2}y_{1,4}^2$ & $19 $ &  $ -7$ & $ 20$ 
         \\
         $22$ & $g_4y_{2,3}^{n-1}y_{2,4}$ & $4 $ & $24 $ & $22 $ 
         \\
         $23$ & $g_4y_{2,3}^{n-1}y_{3,4}$ & $8 $ & $-23 $ & $24 $ 
         \\
         $24$ & $g_4y_{2,3}^{n-2}y_{2,4}^2$ & $-9 $ & $22 $ & $ 23$ 
         \\
         \hline
         %%%%%%%%%%%%%%%
         & $ g_1 y_{1,2}^n $ & $3 $ & $ 4$ & $ g_1 y_{1,2}^n y_{3,4}$
         \\
         & $ g_1 y_{1,2}^{n-r}y_{3,4}^{r} $ & $\chi_r 3-\chi_{r+1}8 $ & $\chi_r 4-\chi_{r+1}11 $ & $ g_1 y_{1,2}^{n-r}y_{3,4}^{r+1} $ 
         \\
         & $ g_1 y_{3,4}^n $ & $ -8$ & $-11 $ & $ g_1 y_{3,4}^{n+1} $
         \\
       & $g_2y_{1,2}^{n}$ & $16 $ & $ 17$ & $g_1 y_{1,2}^n y_{3,4} $ 
         \\
       & $g_3y_{1,2}y_{3,4}^{n-1}$ & $-3 $ & $-4 $ & $g_3 y_{1,2}y_{3,4}^{n} $ 
         \\
         & $g_3y_{3,4}^{n}$ & $-20 $ & $11 $ &  $ g_3 y_{3,4}^{n+1}$
         \\
          & $g_4y_{3,4}^{n}$ & $8 $ & $ -23$ &  $g_4y_{3,4}^{n+1}$
         \\
          & $g_5y_{1,2}^{n}$ & $ 15$ & $-6 $ & $g_5y_{1,2}^{n}y_{3,4}$
         \\
          & $g_5y_{1,2}^{n-1}y_{3,4}$ & $10 $ & $14 $ & $g_5y_{1,2}^{n+1}$
         \\
         \hline
         %%%%%%%%%%%
         & $ g_1 y_{1,3}^n $ & $6 $ & $ g_1 y_{1,3}^ny_{2,4} $ & $8 $  
         \\
         & $ g_1 y_{1,3}^{n-r}y_{2,4}^{r} $ & $ \chi_r 6-\chi_{r+1} 4$ & $ g_1 y_{1,3}^{n-r}y_{2,4}^{r+1} $ & $\chi_r 8+\chi_{r+1}10 $  
         \\
         & $ g_1 y_{2,4}^n $ & $-4 $ & $ g_1 y_{2,4}^{n+1} $ & $10 $ 
         \\
    & $g_2y_{1,3}y_{2,4}^{n-1}$ & $6 $ & $g_2y_{1,3}y_{2,4}^{n}$ & $8 $ 
         \\
         & $g_2y_{2,4}^{n}$ & $-17 $ & $g_2y_{2,4}^{n+1}$ & $10 $ 
         \\
    & $g_3y_{1,3}^{n}$ & $19 $ & $-g_1 y_{1,3}^n y_{2,4} $ & $20 $
         \\
         & $g_4y_{2,4}^{n}$ & $4 $ & $g_4 y_{2,4}^{n+1} $ & $22 $ 
         \\
         & $g_6y_{1,3}^{n}$ & $-11 $ & $ g_6 y_{1,3}^n y_{2,4}$ & $-15 $ 
         \\
          & $g_6y_{1,3}^{n-1}y_{2,4}$ & $14 $ & $g_6 y_{1,3}^{n+1} $ & $ 3$
         \\
         \hline
         %%%%%
         & $ g_1 y_{2,3}^n $ & $ g_1 y_{2,3}^ny_{1,4} $ & $10 $ & $11$ 
         \\
         & $ g_1 y_{2,3}^{n-r}y_{1,4}^{r} $ & $ g_1 y_{2,3}^{n-r}y_{1,4}^{r+1} $ & $\chi_r 10+\chi_{r+1} 3 $ & $\chi_r 11+\chi_{r+1}6 $
         \\
         & $ g_1 y_{1,4}^n $ & $ g_1 y_{1,4}^{n+1} $ & $3 $ & $6 $
         \\
    & $g_2y_{2,3}y_{1,4}^{n-1}$ & $g_2y_{2,3}y_{1,4}^{n} $ & $10$ & $ 11$
         \\
         & $g_2y_{1,4}^{n}$ & $g_2y_{1,4}^{n+1} $ & $16 $ & $ 6$
         \\
         & $g_3y_{1,4}^{n}$ & $ g_3y_{1,4}^{n+1}$ &  $-3 $ & $19 $
         \\
    & $g_4y_{2,3}^{n}$ & $-g_1 y_{2,3}^n y_{1,4} $ & $ 22$ & $23 $
         \\
         & $g_5y_{2,3}^{n}$ & $g_5y_{2,3}^n y_{1,4} $ & $ 8$ & $ -14$
         \\
          & $g_5y_{2,3}^{n-1}y_{1,4}$ & $g_5 y_{2,3}^{n+1} $ & $15 $ & $4 $ 
         \\
         \hline
      \end{tabular}
        % }
      \caption{Products $yy'$ for $n\geqslant 5$ odd.}	
      \label{table:product module M2 yy' n odd 456}
    \end{center}
   \end{table}
   \vspace{-0.8cm}

%%%%%%%%%%%%%%%%%%%%%%%%%%%%%%%%%%%%%%%%%%%%%%%%%%%%%%%%%
%%%%%%%%%%%%%%%%%%%%%%%%%%%%%%%%%%%%%%%%%%%%%%%%%%%%%%%%%%%%%%%%%%%%%%%%%%%%%%%%%%%%%
\bibliographystyle{model1-num-names}

\begin{bibdiv}
\begin{biblist}
\addcontentsline{toc}{section}{References}

\bib{MR2630042}{article}{
	author={Andruskiewitsch, Nicol\'{a}s},
	author={Schneider, Hans-J\"{u}rgen},
	title={On the classification of finite-dimensional pointed Hopf algebras},
	journal={Ann. of Math. (2)},
	volume={171},
	date={2010},
	number={1},
	pages={375--417},
	issn={0003-486X},
	review={\MR{2630042}},
	doi={10.4007/annals.2010.171.375},
}

\bib{MR1930968}{article}{
   author={Brenner, Sheila},
   author={Butler, Michael C. R.},
   author={King, Alastair D.},
   title={Periodic algebras which are almost Koszul},
   journal={Algebr. Represent. Theory},
   volume={5},
   date={2002},
   number={4},
   pages={331--367},
   issn={1386-923X},
   review={\MR{1930968}},
   doi={10.1023/A:1020146502185},
}

\bib{MR2586982}{article}{
   author={Cassidy, Thomas},
   title={Quadratic algebras with Ext algebras generated in two degrees},
   journal={J. Pure Appl. Algebra},
   volume={214},
   date={2010},
   number={7},
   pages={1011--1016},
   issn={0022-4049},
   review={\MR{2586982}},
   doi={10.1016/j.jpaa.2009.09.008},
}

\bib{MR4248207}{article}{
   author={Conner, Andrew},
   author={Goetz, Peter},
   title={Classification, Koszulity and Artin-Schelter regularity of certain
   graded twisted tensor products},
   journal={J. Noncommut. Geom.},
   volume={15},
   date={2021},
   number={1},
   pages={41--78},
   issn={1661-6952},
   review={\MR{4248207}},
   doi={10.4171/jncg/395},
}

\bib{cohenknopper}{article}{
		author={Cohen, A.M.},
		author={Knopper, J.W.},
		title={GBNP - a GAP package},
		journal={Version 1.0.3},
		date={2016},
		eprint={https://www.gap-system.org/Packages/gbnp.html},
	}

\bib{MR1667680}{article}{
		author={Fomin, Sergey},
		author={Kirillov, Anatol N.},
		title={Quadratic algebras, Dunkl elements, and Schubert calculus},
		conference={
		   title={Advances in geometry},
		}, 
		book={
		   series={Progr. Math.},
		   volume={172},
		   publisher={Birkh\"{a}user Boston, Boston, MA},
		}, 
		date={1999},
		pages={147--182},
	   review={\MR{1667680}},
	}

  \bib{GV16}{article}{
	author={Gra\~na, Mat\'{\i}as},
	%   author={Vendramin, Leandro},
	title={Nichols algebras of non-abelian group type: zoo examples},
	year={2016},
	eprint={http://mate.dm.uba.ar/~lvendram/zoo/},
    }

\bib{MR4102553}{article}{
   author={Herscovich, Estanislao},
   title={An elementary computation of the cohomology of the Fomin-Kirillov
   algebra with 3 generators},
   journal={Homology Homotopy Appl.},
   volume={22},
   date={2020},
   number={2},
   pages={367--386},
   issn={1532-0073},
   review={\MR{4102553}},
   doi={10.4310/hha.2020.v22.n2.a22},
}

	\bib{MR1772287}{article}{
		author={Kirillov, A. N.},
		title={On some quadratic algebras},
		conference={
			title={L. D. Faddeev's Seminar on Mathematical Physics},
		},
		book={
			series={Amer. Math. Soc. Transl. Ser. 2},
			volume={201},
			publisher={Amer. Math. Soc., Providence, RI},
		},
		date={2000},
		pages={91--113},
		review={\MR{1772287}},
		doi={10.1090/trans2/201/07},
	}

	\bib{MR3439199}{article}{
		author={Kirillov, Anatol N.},
		title={On some quadratic algebras I $\frac{1}{2}$: combinatorics of Dunkl
			and Gaudin elements, Schubert, Grothendieck, Fuss-Catalan, universal
			Tutte and reduced polynomials},
		journal={SIGMA Symmetry Integrability Geom. Methods Appl.},
		volume={12},
		date={2016},
		pages={Paper No. 002, 172},
		issn={1815-0659},
		review={\MR{3439199}},
		doi={10.3842/SIGMA.2016.002},
	}

\bib{MR1800714}{article}{
		author={Milinski, Alexander},
		author={Schneider, Hans-J\"{u}rgen},
		title={Pointed indecomposable Hopf algebras over Coxeter groups},
		conference={
		   title={New trends in Hopf algebra theory},
		   address={La Falda},
		   date={1999},
		}, 
		book={
		   series={Contemp. Math.},
		   volume={267},
		   publisher={Amer. Math. Soc., Providence, RI},
		},
		date={2000},
		pages={215--236},
	   review={\MR{1800714}},
	   doi={10.1090/conm/267/04272},
	}

\bib{MR2177131}{book}{
   author={Polishchuk, Alexander},
   author={Positselski, Leonid},
   title={Quadratic algebras},
   series={University Lecture Series},
   volume={37},
   publisher={American Mathematical Society, Providence, RI},
   date={2005},
   pages={xii+159},
   isbn={0-8218-3834-2},
   review={\MR{2177131}},
   doi={10.1090/ulect/037},
}

\bib{MR1269324}{book}{
   author={Weibel, Charles A.},
   title={An introduction to homological algebra},
   series={Cambridge Studies in Advanced Mathematics},
   volume={38},
   publisher={Cambridge University Press, Cambridge},
   date={1994},
   pages={xiv+450},
   isbn={0-521-43500-5},
   isbn={0-521-55987-1},
   review={\MR{1269324}},
   doi={10.1017/CBO9781139644136},
}
		


\end{biblist}
\end{bibdiv}

\end{document}